\documentclass[11pt,a4paper]{article}
\usepackage{a4wide}
\usepackage{amsmath, amsthm, amssymb, amsfonts, thm-restate}
\usepackage{hyperref}
\usepackage{enumerate}
\usepackage{xparse}
\usepackage{xspace}
\usepackage{todonotes}
\usepackage{tikz}
\usepackage{bbm}
\usepackage{cite}
\usepackage{pgf} 
\usepackage{graphicx}
\usepackage{euscript}
\usetikzlibrary{calc}
\usetikzlibrary{decorations.markings}
\usetikzlibrary{arrows.meta, shapes.misc, positioning,patterns,snakes}
\usepackage{microtype}
\usepackage{color,colortbl}
 \definecolor{darkgrey}{gray}{0.35}
\definecolor{grey}{gray}{0.86}
\definecolor{lightgrey}{gray}{0.91}
\definecolor{ballblue}{rgb}{0, 0.5,0.5}
\definecolor{lightballblue}{rgb}{0, 0.8,0.8}
\definecolor{dbblue}{rgb}{0, 0.4,0.4}
\usepackage[format=plain, 
justification=raggedright,singlelinecheck=false]{caption}
\usepackage{subcaption}
\usepackage{makecell}
\usepackage{tabularray, xcolor, ninecolors}

\newcounter{constant} 
 
\newcommand{\cost}[1]{\mathcal{C}(#1)}
\newcommand{\mcost}{\textnormal{mcost}}

\newtheorem{theorem}{Theorem}[section]
\newtheorem{lemma}[theorem]{Lemma}
\newtheorem{claim}[theorem]{Claim}
\newtheorem{proposition}[theorem]{Proposition}
\newtheorem{observation}[theorem]{Observation}
\newtheorem{corollary}[theorem]{Corollary}

\newtheorem{definition}[theorem]{Definition}
\theoremstyle{definition}
\newtheorem{remark}[theorem]{Remark}
\newtheorem{assumption}[theorem]{Assumption}
\newtheorem{setting}[theorem]{Setting}
\numberwithin{equation}{section}
\newcommand{\N}{\mathbb{N}}
\newcommand{\Z}{\mathbb{Z}}

\newcommand{\R}{\mathbb{R}}

\newcommand{\pr}{\mathbb{P}}
\newcommand{\E}{\mathbb{E}}

\newcommand{\ind}[1]{\mathbf{1}_{#1}}

\newcommand{\eps}{\varepsilon}
\newcommand{\ulw}{\underline{w}}
\newcommand{\olw}{\overline{w}}

\newcommand{\newnet}[3]{$(#1,#2,#3)$-net}
\newcommand{\mpar}{\textnormal{\texttt{par}}\xspace}
\newcommand{\lls}{{\,\ll_{\star}\,}}
\newcommand{\ggs}{{\,\gg_{\star}\,}}

\newcommand{\comment}[1]{}

\newcommand{\calA}{\mathcal{A}}
\newcommand{\calB}{\mathcal{B}}
\newcommand{\calC}{\mathcal{C}}

\newcommand{\calE}{\mathcal{E}}
\newcommand{\calF}{\mathcal{F}}
\newcommand{\calG}{\mathcal{G}}
\newcommand{\calH}{\mathcal{H}}

\newcommand{\calJ}{\mathcal{J}}

\newcommand{\calL}{\mathcal{L}}

\newcommand{\calN}{\mathcal{N}}

\newcommand{\calP}{\mathcal{P}}

\newcommand{\calR}{\mathcal{R}}
\newcommand{\calS}{\mathcal{S}}

\newcommand{\calU}{\mathcal{U}}
\newcommand{\calV}{\mathcal{V}}
\newcommand{\calW}{\mathcal{W}}
\newcommand{\calX}{\mathcal{X}}

\newcommand{\calZ}{\mathcal{Z}}

\tikzset{cross/.style={cross out, draw=black, minimum size=2*(#1-\pgflinewidth), inner sep=0pt, outer sep=0pt},
cross/.default={2pt}}

\title{Four universal growth regimes in degree-dependent first passage percolation on spatial random graphs I}

\author{J{\'u}lia Komj{\'a}thy\thanks{Delft University of Technology, j.komjathy@tudelft.nl}, John Lapinskas\thanks{University of Bristol, john.lapinskas@bristol.ac.uk}, Johannes Lengler\thanks{ETH Z{\"u}rich, johannes.lengler@inf.ethz.ch}, Ulysse Schaller\thanks{ETH Z{\"u}rich, ulysse.schaller@inf.ethz.ch U.S. was supported by the Swiss National Science Foundation [grant number 200021\_192079].}}

\begin{document}

\maketitle

\begin{abstract}
One-dependent first passage percolation is a spreading process on a graph where the transmission time through each edge depends on the direct surroundings of the edge. In particular, the classical iid transmission time $L_{xy}$ is multiplied by $(W_xW_y)^\mu$, a polynomial of the expected degrees $W_x, W_y$ of the endpoints of the edge $xy$, which we call the penalty function. Beyond the Markov case, we also allow any distribution for $L_{xy}$ with regularly varying distribution near $0$. 
We then run this process on three spatial scale-free random graph models: finite and infinite Geometric Inhomogeneous Random Graphs, and Scale-Free Percolation. In these spatial models, the connection probability between two vertices depends on their spatial distance and on their expected degrees. 

We show that as the penalty-function, i.e., $\mu$ increases, the transmission time between two far away vertices sweeps through four universal phases: \emph{explosive} (with tight transmission times), \emph{polylogarithmic}, \emph{polynomial} but strictly sublinear, and \emph{linear} in the Euclidean distance. The strictly polynomial growth phase here is a new phenomenon that so far was extremely rare in spatial graph models.
The four growth phases are highly robust in the model parameters and are not restricted to phase boundaries. Further, 
the transition points between the phases depend non-trivially on the main model parameters: the tail of the degree distribution, a long-range parameter governing the presence of long edges, and the behaviour of the distribution $L$ near $0$. 
In this paper we develop new methods to prove the upper bounds in all sub-explosive phases. Our companion paper complements these results by providing matching lower bounds in the polynomial and linear regimes. 

\end{abstract}

\section{Introduction}\label{sec:intro}

First passage percolation (FPP) is a natural way to understand geodesics in random metric spaces. 
Starting from some initial vertex at time $0$, the process spreads through the underlying graph so that the transmission time between any two vertices $x,y$ is the minimum sum of edge transmission times over all paths between $x$ and $y$.
In classical FPP, edge transmission times are independent and identically distributed random variables.
In the recent paper \cite{komjathy2020stopping} we introduced one-dependent FPP, where edge transmission times depend on the edge's direct surroundings in the underlying graph. There, we determined the phase transition for explosion (i.e., reaching infinitely many vertices in finite time). 
In this paper we study the sub-explosive regime, when explosion does not occur. We show that the process exhibits rich behaviour with several growth phases and non-smooth phase transitions between them. This holds across a large class of scale-free spatial random graph models (namely Scale-Free Percolation, Hyperbolic Random Graphs, and infinite and finite Geometric Inhomogeneous Random Graphs \cite{deijfen2013scale, bringmann2019geometric, krioukov2010hyperbolic}), and across all Markovian and non-Markovian transmission time distributions with reasonable limiting behaviour at zero.
\vskip0.5em

\textbf{Universality classes of transmission times.}
In one-dependent FPP, we set the transmission time through the edge $e=xy$ between vertices $x,y$ as the product of an independent and identically distributed (iid) random factor $L_{xy}$ and a factor $(W_xW_y)^{\mu}$ for $\mu \geq 0$, where $W_x$ and $W_y$ are (up to constant factors) the expected degrees of $x$ and $y$ in the graph models under consideration.\footnote{Using $W_x$ instead of the actual degree of $x$ is natural in this model. We are convinced that the same results would also hold if we took the actual degrees instead of their expectation. However, that would make the proofs more technical without giving much additional insight.} When $\mu \in(0,1)$, a high-degree vertex still causes more new infections per unit time than a low-degree vertex, but this effect is sublinear in the degree.
As $\mu$ increases and/or the parameters of the underlying graph change, we prove that the following four different phases occur for the transmission time between the vertex at $0$ and a far away vertex $x$:
\begin{itemize}
\setlength\itemsep{0em}
    \item[(i)] it converges to a limiting distribution that is independent of $|x|$ (\emph{explosive phase});
\end{itemize}
This was the main result of \cite{komjathy2020stopping}. The main result of this paper is to characterise the other phases by the growth of the transmission time between $0$ and $x$:
\begin{itemize}
\setlength\itemsep{0em}
    \item[(ii)] it grows at most \emph{polylogarithmically} in $|x|$, without being explosive;
    \item[(iii)] it grows \emph{polynomially} with exponent $0<\eta_0 < 1$;
    \item[(iv)] it grows \emph{linearly} with exponent $\eta_0=1$.
\end{itemize}
These phases are \emph{highly robust} in the parameters, they are not restricted to phase boundaries in either $\mu$ or the other model parameters. 
Moreover, all four phases can occur on a single underlying graph by changing the penalty exponent $\mu$ only; universally across distributions of $L_{xy}$ with regularly varying behaviour at $0$ (exponent $\beta$ in Table \ref{table:summary}), see Figure \ref{fig:simulation} for a visualisation. 
This rich behaviour arises despite the doubly-logarithmic graph distances in the underlying spatial graph models. By contrast, in other models the behaviour of transmission times in classical FPP is less rich, see Section \ref{paragraph:FPP-other-graphs} for the discussion.
\vskip0.5em

\begin{table}[t]
\begin{center}

\begin{tblr}{|c|c|c|}
\hline
\SetRow{grey}{  Graph param.} &{  1-FPP parameters }&{   Behaviour of 1-FPP transmission times} \\
\hline\hline
\SetCell[r=2]{c} {\textbf{Weak decay:}\\ $\tau\in(2,3)$\\$\alpha \in(1,2)$}
&
$\mu < \frac{3-\tau}{2\beta}$
&
\makecell{\textbf{Explosive}:\\ $d_{\calC}(0,x) = \Theta(1)$} \\
\cline{2-3}
&
$\mu > \frac{3-\tau}{2\beta}$
&
\makecell{\textbf{Polylogarithmic}:\\ $d_{\calC}(0,x) = O((\log|x|)^{\Delta_0+o(1)}), \Delta_0>1$}\\
\hline\hline
\SetCell[r=4]{c} {\textbf{Strong decay:}\\ $\tau\in(2,3)$\\$\alpha > 2$ }
&
$\mu < \frac{3-\tau}{2\beta}$
&
\makecell{\textbf{Explosive}:\\ $d_{\calC}(0,x) = \Theta(1)$ } \\
\cline{2-3}
&$\mu\in\big(\frac{3-\tau}{2\beta}, \frac{3-\tau}{\beta}\big)$
&
\makecell{\textbf{Polylogarithmic}:\\ $d_{\calC}(0,x) = O((\log|x|)^{\Delta_0+o(1)}), \Delta_0>1$} \\
\cline{2-3}
&
$\mu\in\big(\frac{3-\tau}{\beta}, \frac{3-\tau}{\min\{\beta, d(\alpha-2)\}} + \frac{1}{d}\big) $
&
\makecell{\textbf{Polynomial}:\\ $d_{\calC}(0,x) = |x|^{\eta_0\pm o(1)}, \eta_0<1$}\\
\cline{2-3}
&$\mu > \frac{3-\tau}{\min\{\beta, d(\alpha-2)\}} + \frac{1}{d} $
&
\makecell{\textbf{Linear:}\\ $d_{\calC}(0,x) = \Theta(|x|)$} \\
\hline
\end{tblr}

\end{center}
\caption{Summary of our main results. In 1-FPP, edge transmission times are $L_{xy} (W_xW_y)^\mu$ where $W_x, W_y$ are constant multiples of the expected degrees of the vertices $x,y$, and $L_{xy}$ is iid with distribution function that varies regularly near $0$ with exponent $\beta\in(0,\infty]$. The degree distribution follows a power law with exponent $\tau\in(2,3)$: graph distances are doubly-logarithmic in the underlying graph. The transmission time $d_\calC(0,x)$ between $0$ and a far away vertex $x$ sweeps through four different phases as the penalty exponent $\mu$ increases. For long-range parameter $\alpha\in(1,2)$, long edges between low-degree vertices maintain polylogarithmic transmission times (similar to long-range percolation), so increasing $\mu$ stops explosion but it has no further effect. When $\alpha>2$, these edges are sparser and a larger $\mu$ slows down 1-FPP, to polynomial but sublinear transmission times in an interval of length at least $1/d$ for $\mu$. Then, all long edges have polynomial transmission times in the distance they bridge.
For even higher penalty exponent $\mu$ the behaviour becomes similar to FPP on the grid $\Z^d$.  We give  the growth exponents $\Delta_0$ and $\eta_0$ explicitly  in~\eqref{eq:Delta_0} and~\eqref{eq:eta_0}.}
\label{table:summary}
\end{table}

\begin{figure}[!ht]
    \centering
    \begin{subfigure}[b]{0.49\textwidth}
         \centering
         \includegraphics[trim = 0 40 0 40,clip,width=\textwidth]{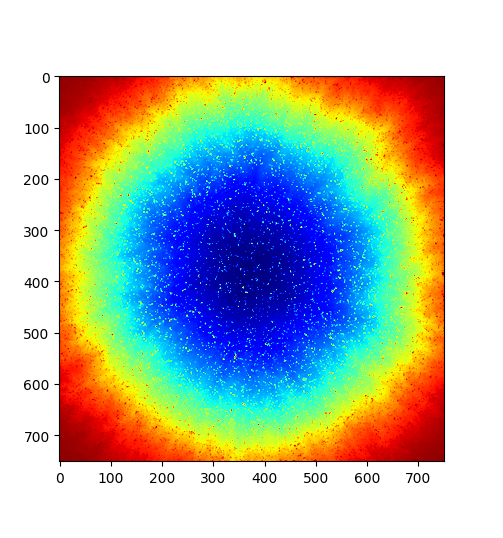}
         (a) Linear regime
     \end{subfigure}
     \hfill
     \begin{subfigure}[b]{0.49\textwidth}
         \centering
         \includegraphics[trim = 0 40 0 40,clip,width=\textwidth]{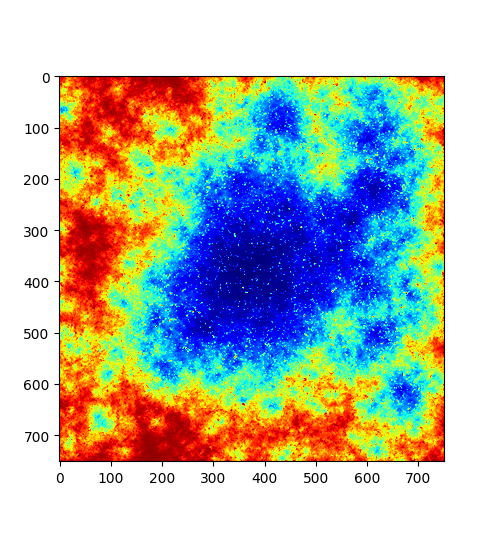}
         (b) Polynomial regime
     \end{subfigure}
     \hfill
     \begin{subfigure}[b]{0.49\textwidth}
         \centering
         \includegraphics[trim = 0 20 0 20,clip,width=\textwidth]{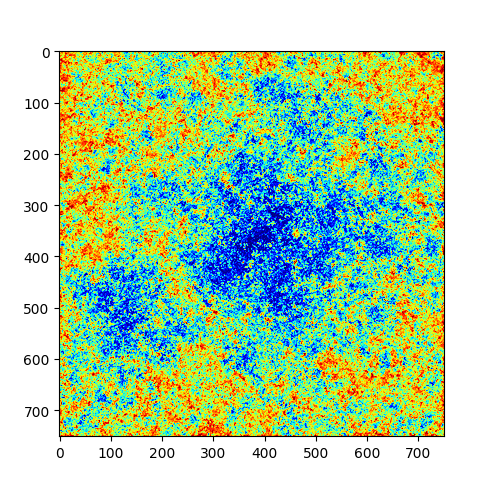}
         (c) Polylogarithmic regime
     \end{subfigure}
     \hfill
     \begin{subfigure}[b]{0.49\textwidth}
         \centering
         \includegraphics[trim = 0 20 0 20,clip,width=\textwidth]{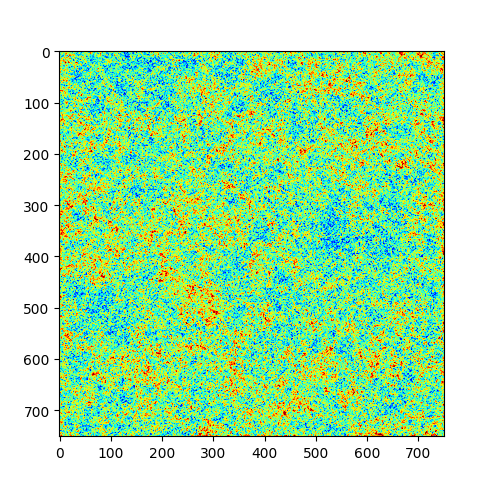}
         (d) Explosive regime
     \end{subfigure}\vspace{2ex}
     
    \caption{Heatmaps for the four different universality classes. The vertices are sorted by their transmission times from the origin (center vertex). The colours represent this ordering: blue infected first, red last. All four plots are generated on the \emph{same} underlying graph (with parameters $\tau=2.3$ and $\alpha=5$, and edge connection probabilities $p(u,v)=1 \wedge (w_u w_v / (\E[W]\|u-v\|^2))^5$), where the vertices are placed on a $750 \times 750$ grid in the 2-dimensional torus. The random factors $L_{xy}$ associated to each edge are also identical in all four plots, and follow an exponential distribution (i.e.\ $\beta=1$). The only varying parameter is the penalty exponent $\mu$, taking values (a) $\mu=2$ for the linear regime (b) $\mu=1$ for the polynomial regime (c) $\mu=0.5$ for the polylogarithmic regime and (d) $\mu=0$ for the explosive regime. In the linear regime, the late points are -- typically -- high degree vertices carrying high penalisation. We thank Zylan Benjert for generating the simulations and the pictures.}
    \label{fig:simulation}
\end{figure}

\textbf{Precise behaviour in the four phases.}
In this paper we prove the upper bounds on transmission times in the sub-explosive regime (phase (i) was previous work~\cite{komjathy2020stopping}). In phase (ii), we show that the transmission time is at most $(\log|x|)^{\Delta_0 + o(1)}$ with an explicit $\Delta_0 >1$ which we conjecture to be tight. In phases (iii) and (iv), we show that the transmission time is precisely $|x|^{\eta_0 \pm o(1)}$, where we give $\eta_0<1$ explicitly for phase (iii) and $\eta_0 =1$ for phase (iv). The companion paper~\cite{komjathy2022one2} contains the matching lower bounds for phases (iii)-(iv) as well as some additional results for phase (iv).  
We develop new techniques that allow us to treat upper bounds for all three sub-explosive phases \emph{simultaneously}, which we expect to be of independent interest. 
\vskip0.5em

\textbf{Motivation of the process from applications.}
One-dependent processes in general, and one-dependent FPP in particular allow for more realistic modelling of real phenomena.  In social networks, actual contacts and infections do not scale linearly with the degree~\cite{feldman2017high,kroy2023superspreading,wang2022effects, ke2021vivo}. 1-FPP type penalisation has frequently been used to model the sublinear impact of superspreaders~\cite{giuraniuc2006criticality,karsai2006nonequilibrium,miritello2013time,pu2015epidemic,yang2008optimal, baxter2021degree}, and in other contexts~\cite{bonaventura2014characteristic,ding2018centrality,lee2009centrality,zlatic2010topologically,andrade2009ising,hooyberghs2010biased}. Consistent with our model, all these applications assume a polynomial dependence with exponent in the range $\mu \in (0,1)$, where a high-degree vertex may cause more new infections per time than a low-degree vertex, but this effect is sublinear in the degree.

While our paper is theoretical, we do believe that a model with a rich phase space can have practical implications. In the spread of physical epidemics, while some diseases spread at an exponential rate, others spread at a polynomial rate, dominated by the local geometry. Examples of the latter include HIV/AIDS, Ebola, and foot-and-mouth disease, see the survey~\cite{polyepidemicsurvey} on polynomial epidemic growth. Classical epidemic models can typically only model either exponential or polynomial growth, not both. Arguably, $1$-FPP provides a natural explanation, since in $1$-FPP the transition can be driven by changes only to the transmission dynamics, not to the underlying network.

\vskip0.5em

\textbf{New methodology: moving to quenched to replace FKG-inequality.} 
In this paper we develop a general technique -- \emph{pseudorandom nets combined with multi-round exposure} -- that \emph{replaces the FKG-inequality} in problems concerning vertex and/or edge-weighted graph models where this inequality does not hold. 
Let us explain why the FKG-inequality fails in the context of 1-FPP. Typically, for upper bounds one constructs paths connecting $0$ and $x$ by revealing vertices and/or edges of the graph sequentially, which destroys the independence of edges. For graph distances in long-range percolation, the FKG-inequality resolves this problem \cite{biskup2004scaling}, but it already needs adjustments once vertex-weights are present \cite{gracar2022finiteness}. 
In 1-FPP, the existence of a long edge is positively correlated to its endpoints having large vertex weights, which is \emph{negatively} correlated to its other outgoing edges having short transmission times. 
To overcome this issue, we
move to the \emph{quenched setting} where we reveal the realisation of the whole weighted vertex set -- say $(\calV, \calW)=(V, w_V)$ -- and thus events concerning only edges become independent. Working with `arbitrary' realisations is impossible. We thus develop a `multi-scale control' on the realisation of the weighted vertex set, which we call \emph{pseudorandom nets} or just nets. 
A net $\calN$ is a subset of vertices, together with their weights, such that for every not-too-small radius $r$, every ``reasonable'' weight $w$, and every vertex $v\in\calN$,
the net has \textit{constant density} in $B_r(v)\times [w, 2w]$, shorthand for vertices of weight in $[w,2w]$ within Euclidean distance $r$ of $v$:
\begin{equation}\label{eq:net-heuristics}
    \frac{|\calN\cap B_r(v)\times[w, 2w]|}{\E\big[|\calV \cap B_r(v)\times [w,2w]|\big]} \in \Big(\frac{1}{16}, 8\Big).
\end{equation}  
We prove via a \emph{multi-scale analysis} that as $|x|\to \infty$, in a box containing $0$ and $x$, asymptotically almost every realisation of the weighted vertex set contains a net $\mathcal N$ with total density at least $1/4$. Then for every net, with a carefully chosen \emph{multi-round exposure} process we can define a coupling which lets us replace the FKG inequality, see Section~\ref{sec:exposure} for more details. We believe that this method is also useful for many other graph models, see shortly below.
\vskip0.5em
%
%

\begin{figure}[t]
    \subfloat[]{\includegraphics[width=0.45\textwidth]{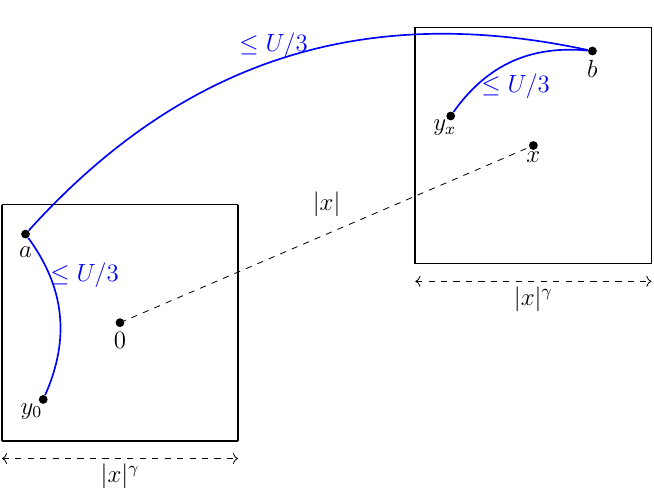}}
    \hspace{0.05\textwidth}
    \subfloat[]{\includegraphics[width=0.45\textwidth]{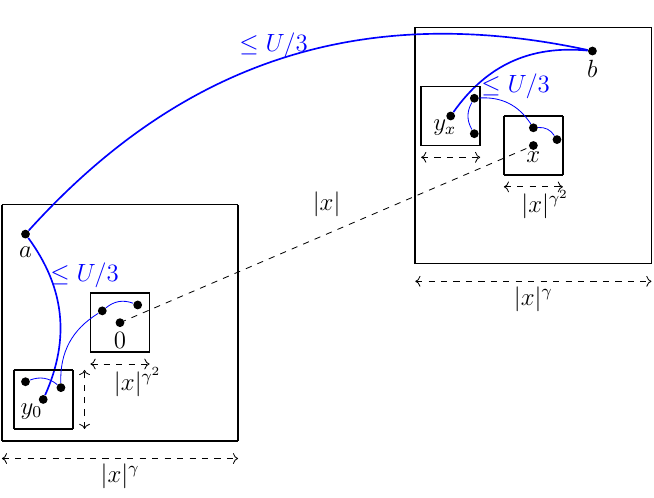}}
    
    \caption{The budget travel plan with 3-edge bridging-paths: (a) first and (b) second iteration.}
    \label{fig:intuition_hierarchy}
\end{figure}

\vskip0.5em

\textbf{Budget travel plan with 3-edge bridge-paths.}
Switching to the quenched setting allows to prove the upper bounds in all subexponential phases (ii)--(iv) all-at-once. Our construction of a connecting path overcomes the following problem: A long edge with a short transmission time typically occurs on typical high-degree vertices and thus all other outgoing edges from the same vertices have too long transmission times. The main idea resembles a `budget travel plan': when someone travels with a low budget, one takes the cheapest mode of transport to the airport within a $200$km radius that offers the cheapest flight landing within a $200$km radius of the destination, then takes the cheapest transport to the destination city.

Formally, we put balls of radius $|x|^\gamma$ for some $\gamma \in (0,1)$ around $0$ and around $x$, and we find a cheap $3$-edge path (``bridge'') $\pi_1 = y_0aby_x$ between these two balls \emph{using only vertices in the pseudorandom net}. The net guarantees enough vertices in each vertex-weight range of interest. We find atypical vertices high-weight vertices $a,b$ that are connected by an atypically cheap edge, that simultaneously have an atypically cheap edge to low-weight vertices $y_0, y_x$, respectively. (Here we use the common terminology of fast transmission corresponding to `cheap' cost.)  
Then we have replaced the task of connecting $0$ and $x$ by the two tasks of connecting $0$ with $y_0$ and $x$ with $y_x$, where the new `gaps' $|0-y_0|$ and $|x-y_x|$ are much smaller than $|x|$. The \emph{multi-round exposure} and the \emph{pseudorandom net} on the fixed vertex set together guarantee that we can iterate this process without running out vertices in the relevant weight-ranges, and without accumulated correlations in the presence of edges along the iteration (e.g.\ out of $y_0, y_x$). Iteration yields a set of multi-scale bridge-paths, which we call after Biskup a \emph{hierarchy} \cite{biskup2004scaling}. The construction in \cite{biskup2004scaling} also uses recursion, with one-edge bridges instead of three-edge bridges, and yields polylogarithmic graph distances in long-range percolation. The techniques in \cite{biskup2004scaling} would not work for 1-FPP because we need to balance distances vs costs vs the penalisation on high-weight vertices in very different regimes, and at the same time deal with edge-costs dependencies. Those can only be dealt with in the quenched setting.  
\vskip0.5em

\textbf{The phases.}
The cost (transmission time) of the bridge-paths $\pi$ in 1-FPP are either polynomial in the distance they bridge or constant.   When the cost is \emph{polynomial} -- with optimal exponent $\eta_0$ -- we are in the \emph{polynomial phase}. The cost of the first bridge $\pi_1$ then dominates the cost of the whole path, and we only carry out a constant number of iterations (irrespective of $|x|$). When bridge-paths with constant cost exist, we are in the \emph{polylogarithmic phase}. Then, the cost of all bridges together are negligible compared to the cost of the polylogarithmic number of gaps that remain after the last iteration. Here, we iterate until we can connect the remaining gaps via essentially constant cost paths. Connecting the gaps is a non-trivial task itself since GIRG/SFP does not contain nearest-neighbour edges. Solutions for filling gaps in \cite{biskup2004scaling} do not work in our setting due to the presence of vertex weights. 
Instead, we connect the gaps with \emph{`weight-increasing paths'} that crucially use that the underlying graphs are scale-free.
We give a more detailed discussion about the hierarchical construction at the beginning of Section~\ref{sec:hierarchy} and back-of-the-envelope calculations about how to obtain the precise growth exponents in phases (ii) and (iii) at the beginning of Section \ref{sec:choices} with proof sketches below Corollaries \ref{cor:computations-polylog} and \ref{cor:computations-polynomial}.
\vskip0.5em

\textbf{Robustness of our techniques.}
The technique of nets combined with multi-round edge-exposure is robust, and will be applicable elsewhere, for questions concerning \emph{first passage percolation, robustness to percolation (random deletion of edges), graph distances, SIR-type and other epidemic processes, rumour spreading}, etc.\ on a larger class of vertex-weighted graphs; including random geometric graphs, Boolean models with random radii, the age-dependent and the weight-dependent random connection model (mimicking spatial preferential attachment), scale-free Gilbert graph, and the models used here \cite{AieBonCooJanss08, CooFriePral12, gracar2019age, gracar2022chemical, GraHeyMonMor19, gracar2022finiteness, gracar2021percolation, hirsch2017gilbertgraph, JacMor15}, and can also be extended to dynamical versions of the above graph models on fixed vertex sets.
\vskip0.5em

\textbf{Two papers, two techniques and optimality.}
The `budget travel plan' together with the renormalisation group argument in~\cite{komjathy2022one2} reveals that the strategy of polynomial paths is essentially optimal: in this phase, all long edges have polynomial transmission time in the distance they bridge.  Our techniques for the lower bounds are entirely different and deserve their own exposition, hence we present them in the companion paper \cite{komjathy2022one2}.
\vskip0.5em

\subsubsection{Related work: phases of FPP in other models.}\label{paragraph:FPP-other-graphs}
The phase diagrams of transmission times in classical FPP are less rich. In particular, the strict polynomial phase is absent or restricted only to phase transition boundaries. 
Indeed, on sparse \emph{non-spatial} graph models with finite-variance degrees, both Markovian and non-Markovian classical FPP universally show Malthusian (exponential) growth \cite{bhamidi2017universality}. Transmission times between two uniformly chosen vertices are then \emph{logarithmic} in the graph size. Sparse \emph{spatial} graphs with finite-variance degrees (e.g. percolation, long-range percolation, random geometric graphs etc.) are typically restricted to linear graph distances/transmission times in the absence of long edges \cite{antal1996chemical, penrose2003random, cox1981some}, or to polylogarithmic distances in the presence of long edges \cite{biskup2004scaling, biskup2019sharp, hao2021graph}. 
In both spatial and non-spatial graph models with infinite-variance degrees, classical FPP typically either explodes or exhibits a smooth transition between explosion and doubly-logarithmic transmission times (which match the graph distances) \cite{adriaans2018weighted, jorritsma2020weighted, van2017explosion}; in particular, there is no analogue of phases (ii)--(iv). For  one-dependent FPP on non-spatial graphs there are strong indications that the process either explodes \cite{slangen19}, with the same criterion for explosion as for spatial graphs in \cite{komjathy2020stopping}, or becomes Malthusian \cite{Fransson1720143, fransson2022stochastic}, the latter implying logarithmic transmission times between two uniformly chosen vertices by the universality in \cite{bhamidi2017universality}, so only two phases can occur. The only graph model to exhibit a transition from a fast-growing phase to a slow-growing phase is long-range percolation, where the polynomial phase is restricted to the phase boundary in the long-range parameter $\alpha=2$ \cite{baumler2023distances}. 
Even in degenerate models (where the underlying graph is complete), long-range first passage percolation \cite{chatterjee2016multiple} is the only other model where a similarly rich set of phases is known to occur.
 Thus one-dependent FPP is the first process that displays a full interpolation between the four phases on a \emph{single non-degenerate graph model}. Moreover, the phase boundaries for one-dependent FPP depend non-trivially on the main model parameters: the degree power-law exponent $\tau$, the parameter $\alpha$ controlling the prevalence of long-range edges, and the behaviour of $L_{xy}$ near $0$ characterised by $\beta$, see Table \ref{table:summary} for our results, Table \ref{table:phases-top-FPP} for phases of growth in other models, and Section \ref{sec:discussion} for more details on related work.

\subsection{Graph Models}\label{sec:graph_model}

We consider simple and undirected graphs with vertex set $\calV \subseteq \R^d$. We use standard graph notation along with other common terminology, see Section~\ref{sec:notation}. 
We consider three random graph models: \emph{Scale-Free Percolation} (SFP), \emph{Infinite Geometric Inhomogeneous Random Graphs} (IGIRG)\footnote{They have also been called EGIRG, where E stands for extended~\cite{komjathy2020explosion}.}, and (finite) \emph{Geometric Inhomogeneous Random Graphs} (GIRG). The latter model contains \emph{Hyperbolic Random Graphs} (HypRG) as special case, so our results extend to HypRG. The main difference between SFP and IGIRG is the vertex set $\calV$. For SFP, we use $\calV := \mathbb Z^d$, with $d \in \mathbb{N}$. For IGIRG, a unit-intensity Poisson point process on $\mathbb R^d$ forms $\calV$. 
\begin{definition}[SFP, IGIRG, GIRG]\label{def:girg}
Let $d\in \N$, $\tau >2$, $\alpha\in(1,\infty)$, and $\overline{c}>\underline{c}>0$. Let $\ell:[1,\infty)\rightarrow(0,\infty)$ be function that varies slowly at infinity (see Section \ref{sec:notation}), and let
 $h:\R^d\times[1,\infty)\times[1,\infty)\rightarrow[0,1]$ be a function satisfying
\begin{align}\label{eq:connection_prob}
	\underline{c}\cdot\min\left\{1,
	\dfrac{w_1w_2}{|x|^d}\right\}^{\alpha}
	\le h(x,w_1,w_2)\le \overline{c}\cdot\min\left\{1,
	\dfrac{w_1w_2}{|x|^d}\right\}^{\alpha}.
 \end{align}
The vertex set and vertex-weights: For SFP, set $\calV := \mathbb Z^d$, for IGIRG, let $\calV$ be given by a Poisson point process on $\mathbb R^d$ of intensity one.\footnote{If we take an IGIRG and rescale the underlying space $\mathbb R^d$ by a factor $\lambda$, then we obtain a random graph which satisfies all conditions of IGIRGs except that the density of the Poisson point process is $\lambda^{-d}$ instead of one. Thus it is no restriction to assume density one.} For each $v\in\calV$, we draw a \emph{weight} $W_v$ independently from a probability distribution on $[1, \infty)$ satisfying
\begin{equation}\label{eq:power_law}
  F_W(w)=\mathbb{P}( W\le w)= 1-\ell(w)/w^{\tau-1}.
\end{equation} 
We denote $\widetilde \calV(G):=(\calV, \calW)$ the vertex set $\calV$ together with the random weight vector $\calW_{\calV}:=(W_v)_{v\in \calV}$, and $(V,w_V):=(V, (w_v)_{v\in V})$ a realisation of $\widetilde \calV:=\widetilde \calV(G)$, where $\tilde v:=(v, w_v)$ stands for a single weighted vertex. 

The edge set:  Conditioned on $\widetilde \calV=(V, w_V)$, consider all unordered pairs $\calV^{\scriptscriptstyle{(2)}}$ of $\calV$. Then every pair $xy\in \calV^{\scriptscriptstyle{(2)}}$ is present in $\calE(G)$ independently with probability $h(x-y,w_x,w_y)$. 

Finally, a GIRG $G_n$ is obtained as the induced subgraph $G[Q_n]$ of an IGIRG $G$ by the set of vertices in the cube $Q_n$ of volume $n$ centred at~$0$. We call $h$ the \emph{connection probability}, $d$ the \emph{dimension}, $\tau$ the \emph{power-law exponent}, and $\alpha$ the \emph{long-range parameter}.
\end{definition}
 The above definition essentially merges the Euclidean space and the vertex-weight space by considering vertices with weights as points in
$\R^d \times [1, \infty)$, i.e., we think of each vertex as a pair $\tilde v=(v, w_v)$, where $v\in \R^d$ is its spatial location and $w_v$ is its weight.

For finite GIRG models, we are interested in the behaviour as $n\to \infty$. Definition \ref{def:girg} leads to a slightly less general model than those e.g.\ in~\cite{bringmann2019geometric} and~\cite{komjathy2020stopping}. There, the original definition had a different scaling of the geometric space vs connection probabilities. However, the resulting graphs are identical in distribution after rescaling, see~\cite{komjathy2020stopping} for a comparison. Finally,~\cite{bringmann2019geometric} considered the torus topology on the cube, identifying ``left'' and ``right'' boundaries, but this does not make a difference for our results. Next we define $1$-dependent FPP on these graphs.

\begin{definition}[1-dependent first passage percolation (1-FPP)]\label{def:1-FPP}
Consider a graph $G = (\calV, \calE)$ where each vertex $v\in \calV$ has an associated vertex-weight $W_v$.  For every edge $xy\in \calE$, draw an i.i.d.\ copy $L_{xy}$ of a random variable $L$, and set the \emph{(transmission) cost} of an edge $xy$ as 
\begin{equation}\label{eq:cost}
\cost{xy}:=L_{xy}(W_xW_y)^{\mu},    
\end{equation}
for a fixed parameter $\mu>0$ called the \emph{penalty strength}. The costs define a \emph{cost distance} $d_{\calC}(x,y)$ between any two vertices $x$ and $y$, which is the minimal total cost of any path between $x$ and $y$ (see Section~\ref{sec:notation}). We call $d_{\calC}$ the 1-dependent first passage percolation.
\end{definition}
We usually assume that the cumulative distribution function (cdf) $F_L:[0,\infty)\rightarrow[0,1]$ of $L$ satisfies the following assumption, (with exceptions of this assumption explicitly mentioned):
\begin{assumption}\label{assu:L}
There exist constants $t_0,\,c_1,\,c_2,\,\beta>0$ such that
\begin{align}\label{eq:F_L-condition}
	c_1t^{\beta}\le F_L(t)\le c_2t^{\beta}\mbox{ for all }t\in[0,t_0].
\end{align}
\end{assumption}
Without much effort, one can relax Assumption~\ref{assu:L} to $\lim_{x\to0} \log F_L(x)/\log x=\beta$. We work with~\eqref{eq:F_L-condition} for the sake of readability. We discuss  extensions to $\alpha = \infty$ and $\beta = \infty$ separately in Section~\ref{sec:threshold}.
We call the set of parameters $\mpar := \{d, \tau, \alpha, \mu, \beta, \underline{c}, \overline{c}, c_1, c_2, t_0\}$ the \emph{model parameters}. We say that a variable is \emph{large} (or \emph{small}) relative to a collection of other variables when it is bounded below (or above) by some finite positive function of those variables and the model parameters. We restrict to $\tau \in (2,3)$,  (explicitly stated in the theorems), which ensures that there is a unique infinite component (or linear-sized ``giant'' component for finite GIRG)\footnote{For $\tau >3$, an infinite component only exists for high enough edge density, which is captured by $h$ in \eqref{eq:connection_prob}.} and that graph distances between vertices $x,y$ in the infinite/giant component grow like $d_G(x,y) \sim 2\log \log |x-y|/|\log(\tau-2)|$ in all three models~\cite{komjathy2020explosion,bringmann2016average,deijfen2013scale,van2017explosion}. We consider $\mu$ as the easiest parameter to change: increasing $\mu$ means gradually slowing down the spreading process around high-degree vertices, which corresponds to adjusting behaviour of individuals with high number of contacts. Hence, we will phrase our results from this perspective.

\subsection{Results}\label{sec:results}
In this paper, we focus on the sub-explosive parameter regime
\begin{equation}\label{eq:no_explosion}
   \mu>\frac{3-\tau}{2\beta}:=\mu_{\mathrm{expl}},
\end{equation}
since for $\mu < \mu_{\mathrm{expl}}$ we have shown in previous work~\cite{komjathy2020stopping} that the model is \emph{explosive}: the cost-distance of two vertices $x,y$ converges in distribution to an almost surely finite variable as $|x-y| \to \infty$, conditioned on $x$ and $y$ being in the infinite component.\footnote{The phase is called \emph{explosive} since the size of the the cost-ball of radius $r$ jumps from finite to infinite at some random finite threshold, called the \emph{explosion time}.} 
In other words, \eqref{eq:no_explosion} restricts us to the non-explosive phase. The following two quantities define the boundaries of the new phases:
\begin{align}\label{eq:mu_pol_log}
    \mu_{\log}:=\frac{3-\tau}{\beta}, \quad \mu_{\mathrm{pol}}:=\frac1d+\frac{3-\tau}{\min\{\beta, d(\alpha-2)\}} = \max\big\{1/d+\mu_{\log}, \mu_{\mathrm{pol}, \alpha}\big\},
    \end{align}
with $\mu_{\mathrm{pol},\alpha}:=\tfrac{1}{d}+\tfrac{3-\tau}{d(\alpha-2)}=\tfrac{\alpha-(\tau-1)}{d(\alpha-2)}$ and $\mu_{\mathrm{pol}, \beta}:=\tfrac{1}{d}+\tfrac{3-\tau}{\beta}$.
We also define two \emph{growth exponents}. 
If $\alpha\in(1,2)$ or $\mu\in(\mu_{\mathrm{expl}}, \mu_{\log})$, we define
\begin{align}\label{eq:Delta_0}
    \Delta_0 := \Delta_0(\alpha, \beta,\mu, \tau) := \frac{1}{1-\log_2(\min\{\alpha, \tau-1+\mu\beta\})} =\min\{\Delta_\alpha, \Delta_\beta\}> 1,
\end{align}
with $\Delta_\alpha=1/(1-\log_2\alpha)$ and $\Delta_\beta=1/(1-\log_2(\tau-1+\mu\beta)$.
$\Delta_0>1$ follows since when $\alpha\in(1,2)$ then $\Delta_\alpha>1$, while  when $\mu \in( \mu_{\mathrm{expl}}, \mu_{\log})$ then $\tau-1+\mu\beta >\tfrac{\tau+1}{2} >1$ and also $\tau-1+\mu\beta<2$, so $\log_2(\tau-1+\mu\beta)$ is positive but less than $1$. If both $\alpha>2$ and $\mu > \mu_{\log}$, we define
\begin{align}\label{eq:eta_0}
    \eta_0 := \eta_0(\alpha,\beta,\mu,\tau) := \begin{cases}
	1 & \mbox{ if $\mu>\mu_{\mathrm{pol}}$,}\\
	\min\left\{d(\mu-\mu_{\log}), \mu/\mu_{\mathrm{pol},\alpha}\right\} & \mbox{ if $\mu\le\mu_{\mathrm{pol}}$,}
	\end{cases}
\end{align}
and note that $\eta_0>0$ for all $\mu>\mu_{\log}$, and $\eta_0<1$ exactly when $\mu< \mu_{\mathrm{pol}}$ by \eqref{eq:mu_pol_log}. We often write $\eta_\beta:=d(\mu-\mu_{\log})$ and $\eta_\alpha:=\mu/\mu_{\mathrm{pol},\alpha}$.
The formulas can be naturally extended by taking limits and hold also when $\alpha = \infty$ or $\beta=\infty$, which we elaborate in Section~\ref{sec:threshold} below. 
We first formulate the main results for the infinite models IGIRG and SFP. We denote by~$\calC_{\infty}$ the unique infinite component of IGIRG/SFP.
\begin{restatable}{theorem}{PolylogRegime}
\label{thm:polylog_regime}
    Consider $1$-FPP in Definition \ref{def:1-FPP} on the graphs IGIRG or SFP of Definition \ref{def:girg} satisfying the assumptions given in \eqref{eq:power_law}--\eqref{eq:F_L-condition} with $\tau\in(2,3), \alpha>1, \mu>0$. When either $\alpha\in(1,2)$ or  $\mu\in(\mu_{\mathrm{expl}},\mu_{\log})$ or both hold, then for any $\eps>0$,
 \begin{align*}
        \lim_{|x|\to \infty}\pr\big( d_{\calC}(0,x) \le (\log |x|)^{\Delta_0+\varepsilon} \mid 0, x \in \calC_{\infty}  \big) =1.
    \end{align*}
\end{restatable}
The proof of Theorem~\ref{thm:polylog_regime} is valid when $\mu < \mu_{\mathrm{expl}}$, however, then the model is explosive \cite[Theorem~1.1]{komjathy2020stopping}, and the bound is not sharp. 
With the restriction $\mu>\mu_{\mathrm{expl}}$, we conjecture that Theorem~\ref{thm:polylog_regime} is actually sharp, i.e., that a corresponding lower bound with exponent $\Delta_0 - \eps$ also holds. See Section \ref{sec:discussion} below for results on polylogarithmic lower bounds. 
The exponent $\Delta_0>1$ intuitively corresponds to stretched exponential ball-growth, where the number of vertices in cost-distance at most $r$ scales as $\exp(r^{1/\Delta_0})$. Trapman in~\cite{trapman2010growth} showed that strictly exponential ball growth, i.e., $\Delta_0=1$ is possible for long-range percolation when $\alpha=1$ under additional constraints. 
This is consistent with our formula for $\Delta_0$, since $\Delta_0 \to 1$ as $\alpha\to 1$.
We leave the lower bound in this phase for future work. Slightly related is the work \cite{hao2021graph} that treats polylogarithmic graph distances in the same model class but in a different parameter regime (finite variance degrees), however, the proof techniques there neither extend to FPP nor to infinite variance degree underlying graphs.

\begin{remark}\label{remark:structure-polylog}\emph{Structure of near-optimal paths in the polylog phase.} The proof reveals two different types of paths with polylogarithmic cost-distances present in the graph. When $\alpha<2$, randomly occurring long edges on low-weight vertices cause the existence of paths of cost at most $(\log |x|)^{\Delta_\alpha+o(1)}$ with $\Delta_\alpha=1/(1-\log_2(\alpha))$. The closest long edge of order $|x|$ lands at distance $|x|^{\alpha/2}$ from $0$ and $x$ respectively, resulting in $\Delta_
\alpha$ after iteration. 
When $\mu<\mu_{\log}$, there are also paths using a cheap yet long edge (of order $|x|$) between two high-weight vertices (weight roughly $|x|^{d/2}$) that lie within distance $|x|^{(\tau-1+\mu\beta)/2+o(1)}$ from $0$ and $x$ respectively, and these cause the existence of paths of cost at most $(\log|x|)^{\Delta_\beta+o(1)}$ with $\Delta_\beta=1/(1-\log_2(\tau-1+\mu\beta))$. $\Delta_\beta$ is the outcome of an optimisation: we minimise the distance between the high-weight vertices to $0$ and $x$, while maintaining that an edge with constant cost exists between them. The minimal distance possible is of order $|x|^{(\tau-1+
\mu\beta)/2+o(1)}$:  the tail exponent $\tau-1$ of the weight distribution \eqref{eq:power_law}, and $\mu\beta$, the penalty exponent in \eqref{eq:cost} times the behaviour of the cdf of $L$ in \eqref{eq:F_L-condition} both play a role.
 \end{remark}
When we increase $\mu$ above $\mu_{\log}$ and $\alpha$ above $2$, we enter a new universality class and cost distances become  polynomial:
\begin{restatable}{theorem}{PolynomialRegime}
\label{thm:polynomial_regime}
	Consider $1$-FPP in Definition \ref{def:1-FPP} on the graphs IGIRG or SFP of Definition \ref{def:girg} satisfying the assumptions given in \eqref{eq:power_law}--\eqref{eq:F_L-condition} with $\tau\in(2,3)$.
	 When $\alpha>2$ and $\mu > \mu_{\log}$ both hold, then for any $\eps>0$,
\begin{align*}
    \lim_{|x|\to \infty}\pr\left(d_{\calC}(0,x) \le|x|^{\eta_0+\varepsilon}   \mid 0,x \in \calC_{\infty}\right) =1.
    \end{align*}
\end{restatable}
In the accompanying~\cite{komjathy2022one2} we prove the corresponding lower bound, which implies:
\begin{corollary}[Polynomial Regime]
\label{cor:polynomial_regime}
	Consider $1$-FPP in Definition \ref{def:1-FPP} on the graphs IGIRG or SFP satisfying the assumptions given in \eqref{eq:power_law}--\eqref{eq:F_L-condition} with $\tau\in(2,3)$.
	When $\alpha>2$ and $\mu > \mu_{\mathrm{log}}$ both hold, then for any $\eps>0$,
	 \begin{equation*}
    \lim_{|x|\to \infty}\pr\left( |x|^{\eta_0-\varepsilon}\le d_{\calC}(0,x) \le|x|^{\eta_0+\varepsilon}   \mid 0,x \in \calC_{\infty}\right) =1.
\end{equation*}
\end{corollary}
Corollary~\ref{cor:polynomial_regime} together with Theorem~\ref{thm:polylog_regime}  implies that the phase transition is proper at $\mu_{\log}$ and at $\alpha =2$: distances increase from at most polylogarithmic to polynomial. Moreover, when $\mu\ge \mu_{\mathrm{pol}}$ and the dimension $d\ge 2 $, in \cite{komjathy2022one2} we also prove \emph{strictly} linear cost-distances (both upper and lower bounds). This, together with Theorem~\ref{thm:polynomial_regime}, implies that there is another phase transition at $\mu_{\mathrm{pol}}$, from sublinear ($\eta_0 < 1$) to linear ($\eta_0=1$) cost-distances.  See Table~\ref{table:summary} for a summary. 
We find it remarkable that 1-FPP shows polynomial distances with exponent \emph{strictly less than one} in a spread-out parameter regime $\mu\in(\mu_{\log}, \mu_{\mathrm{pol}})$. This implies polynomial ball-growth faster than the dimension for 1-FPP, which is rare in spatial models, see Section~\ref{sec:discussion}.

\begin{remark}\label{remark:structure-polynom}\emph{Structure of near-optimal paths in the polynomial phase.} The proof reveals two different types of paths with polynomial cost-distances present in the graph. When $\mu\le \mu_{\mathrm{pol,\alpha}}$, there are a few very long edges (of order $|x|$) with endpoints polynomially near $0$ and $x$, emanating from vertices with weight $|x|^{1/(2\mu_{\mathrm{pol}, \alpha})}$, and these results in paths with cost at most $|x|^{\eta_\alpha+o(1)}$ (the second term in \eqref{eq:eta_0}). Since there are only few such edges, the optimisation effect of choosing the one with smallest cost is negligible and $\beta$ does not enter the formula. Further, when $\mu\le \mu_{\mathrm{pol},\beta}$ in \eqref{eq:mu_pol_log}, there are many long edges (of order $|x|$) with respective endpoints polynomially near $0$ and $x$ on vertices with weight roughly $|x|^{d/2}$, and when we optimise to choose the one with cheapest cost, the effect of $F_L$, i.e.,\ $\beta$ in \eqref{eq:F_L-condition}, enters the formula, and we obtain a path with cost at most $|x|^{d(\mu-\mu_{\log})+o(1)}$, the first term in \eqref{eq:eta_0}. The proof of the lower bound in \cite{komjathy2022one2} shows that in this phase \emph{all} long edges near $0,x$ have polynomial costs in the Euclidean distance they bridge, which explains the qualitative difference between 1-FPP and classical FPP.
\end{remark}

The next theorem describes in which sense the results stay valid for finite-sized models:
  
\begin{restatable}{theorem}{FiniteGraph}
\label{thm:finite_graph}
Consider 1-FPP in Definition \ref{def:1-FPP} on the graph GIRG of Definition \ref{def:girg} satisfying the assumptions given in \eqref{eq:power_law}--\eqref{eq:F_L-condition} with $\tau\in(2,3), \alpha>1, \mu>0$.  Let $\calC_{\max}^{(n)}$ be the largest component in $Q_n$. Let $u_n,v_n$ be two vertices chosen uniformly at random from $\calV\cap Q_n$. 

\noindent (i) When either $\alpha\in(1,2)$ or $\mu\in(\mu_{\mathrm{expl}},\mu_{\log})$ or both hold, then for any $\varepsilon>0$,    
	\begin{align}\label{eq:finite-polylog}
		 \lim_{n\to \infty}\pr\left( d_{\calC}(u_n,v_n) \le (\log |u_n-v_n|)^{\Delta_0+\varepsilon} \ \mid \  u_n,v_n \in \calC_{\max}^{\scriptscriptstyle{(n)}} \right)=1.
	\end{align}
(ii) When $\alpha>2$ and $\mu>\mu_{\log}$ both hold, then for any $\varepsilon>0$,
    \begin{align}\label{eq:finite-polynomial}
		 \lim_{n\to \infty}\pr\left( d_{\calC}(u_n,v_n) \le |u_n-v_n|^{\eta_0+\varepsilon} \ \mid \  u_n,v_n \in \calC_{\max}^{\scriptscriptstyle{(n)}} \right)=1.
	\end{align}
\end{restatable}
The lower bound in Corollary~\ref{cor:polynomial_regime} also transfers to finite GIRGs, since GIRG is defined as a subgraph of IGIRG. We refer to~\cite{komjathy2022one2} for details. The proofs of Theorems \ref{thm:polylog_regime}, \ref{thm:polynomial_regime}, and \ref{thm:finite_graph} also reveal that the paths realising the upper bounds deviate only sublinearly from the straight line between the two vertices, cf.\ Definition~\ref{def:deviation} and Lemmas~\ref{lem:polylog-deviation} and~\ref{lem:polynomial-deviation} for more details.

\subsubsection{Limit Cases and Extensions}\label{sec:threshold}

Theorems~\ref{thm:polylog_regime}--\ref{thm:finite_graph} can be extended to interesting cases that may informally be described as $\alpha = \infty$ or $\beta = \infty$. In the case $\alpha = \infty$, all connection probabilities are either constant or zero,
and we replace the condition~\eqref{eq:connection_prob} by
\begin{align}\label{eq:alpha_infty}
 h(x,w_1,w_2)
 \begin{cases} 
 \ = 0,\quad  & \text{if } \tfrac{w_1w_2}{|x|^d} < c', \\ 
 \ \ge \underline{c}\quad & \text{if }\tfrac{w_1w_2}{|x|^d} \ge c'',
\end{cases}
\end{align}
for some constants $\underline{c} \in(0,1] $ and $c'' \ge c' > 0$. For the sake of simplicity we will assume $c''=1$ in all our proofs, however the results still hold for general $c''$. Models satisfying \eqref{eq:alpha_infty} are called threshold (or zero temperature) models, and include \emph{hyperbolic random graphs} \cite{krioukov2010hyperbolic} when the dimension is one. The correspondence between GIRGs and threshold hyperbolic random graphs was established in~\cite[Theorem 2.3]{bringmann2019geometric}. For models where \eqref{eq:alpha_infty} holds, we extend the definitions~\eqref{eq:mu_pol_log}-\eqref{eq:Delta_0} in the natural way to $\alpha=\infty$, since $\lim_{\alpha\to \infty}\mu_{\mathrm{pol},\alpha}=1/d$:
\begin{align}\label{eq:alpha-infty-definitions}
    \mu_{\log}:=\frac{3-\tau}{\beta}, \quad \mu_{\mathrm{pol}}:= \frac{1}{d}+\frac{3-\tau}{\beta},\quad 
    \eta_0 := \begin{cases}
	1 & \mbox{ if $\mu>\mu_{\mathrm{pol}}$,}\\
	  d\cdot(\mu-\mu_{\log})  & \mbox{ if $\mu\le\mu_{\mathrm{pol}}$,}
	\end{cases}
\end{align}
and, when $\mu\in(\mu_{\mathrm{expl}},\mu_{\log})$,
\begin{align}\label{eq:alpha-infty-Delta_0}
    \Delta_0 := \frac{1}{1-\log_2(\tau-1+\mu\beta)} >0.
\end{align}
The case $\beta = \infty $ captures when the cdf of the edge transmission variable $L$ in \eqref{eq:F_L-condition} is flatter near $0$ than any polynomial, and we replace~\eqref{eq:F_L-condition} by the condition that 
\begin{align}\label{eq:beta_infty}
	\lim_{t\to 0} F_L(t)/t^{\beta} = 0 \mbox{ for all }0<\beta <\infty.
\end{align}
In particular, this condition is satisfied if $F_L$ has no probability mass around zero, for example\footnote{For $\mu=0, L\equiv 1$,  the cost-distance $d_{\calC}(x,y)$ then equals the \emph{graph-distance} between $x$ and $y$.~\cite{komjathy2022one2} contains as special cases the linear lower bound on graph-distances by Berger~\cite{berger2004lower} for long-range percolation (LRP) and  by Deprez, Hazra, and W\"uthrich~\cite{deprez2015inhomogeneous} for SFP, see \cite[Proposition 2.4]{komjathy2022one2}.} when $L \equiv 1$. When $\beta=\infty$, using that $\tau\in(2,3)$ we replace~\eqref{eq:mu_pol_log}-\eqref{eq:eta_0} naturally by
\begin{align}\label{eq:beta-infty-definitions}
    \mu_{\mathrm{expl}} := \mu_{\log}:=0, \quad \mu_{\mathrm{pol}}:= \frac{\alpha-(\tau-1)}{d(\alpha-2)}, \quad
    \eta_0 := \begin{cases}
	1 & \mbox{ if $\mu>\mu_{\mathrm{pol}}$,}\\
	\mu/\mu_{\mathrm{pol}} & \mbox{ if $\mu\le\mu_{\mathrm{pol}}$,}
	\end{cases}
\end{align}
and, when $\alpha\in(1,2)$,
\begin{align}\label{eq:beta-infty-Delta_0}
    \Delta_0 := \frac{1}{1-\log_2(\alpha)} >0.
\end{align}
Finally, when both $\alpha=\beta=\infty$ we replace~\eqref{eq:mu_pol_log} and \eqref{eq:eta_0} by
\begin{align}\label{eq:alpha-beta-infty-definitions}
    \mu_{\mathrm{expl}} := \mu_{\log}:=0, \quad \mu_{\mathrm{pol}}:= \tfrac{1}{d}, \quad \eta_0 := \min\{1,d\mu\},
\end{align}
and in that case we do not define $\Delta_0$, since the polylogarithmic case is vacuous when $\alpha=\beta=\infty$ (see Remark~\ref{rem:alpha-beta-infty-polyllog}).
Our main results still hold for these limit regimes. We remark that the corresponding lower bounds also hold~\cite[Theorem~1.10]{komjathy2022one2}. 
\begin{theorem}
[Extension to threshold GIRGs and $\beta=\infty$]~
\label{thm:threshold_regimes}
\begin{enumerate}[(a)]
    \item Theorems~\ref{thm:polylog_regime},~\ref{thm:polynomial_regime} and \ref{thm:finite_graph} still hold for $\alpha=\infty$ if we replace definitions~\eqref{eq:mu_pol_log}-\eqref{eq:eta_0} by definitions~\eqref{eq:alpha-infty-definitions}-\eqref{eq:alpha-infty-Delta_0}. 

    \item Theorems~\ref{thm:polylog_regime},~\ref{thm:polynomial_regime} and \ref{thm:finite_graph} still hold for $\beta=\infty$ if we replace definitions~\eqref{eq:mu_pol_log}-\eqref{eq:eta_0} by definitions~\eqref{eq:beta-infty-definitions}-\eqref{eq:beta-infty-Delta_0}.
    \item Theorems~\ref{thm:polynomial_regime} and \ref{thm:finite_graph} still hold for $\alpha=\beta=\infty$ if we replace definitions~\eqref{eq:mu_pol_log}-\eqref{eq:eta_0} by definition~\eqref{eq:alpha-beta-infty-definitions}.
\end{enumerate}
\end{theorem}
Theorem~\ref{thm:threshold_regimes}(a) implies the analogous result for hyperbolic random graphs (HypRG) by setting $d=1$ in \eqref{eq:alpha-infty-definitions}, except for some minor caveats. In Definition \ref{def:girg}, the number of vertices in GIRG is Poisson distributed with mean $n$, while in the usual definition of HypRG~\cite{krioukov2010hyperbolic,gugelmann2012random} and GIRG~\cite{bringmann2019geometric} the number of vertices is exactly $n$. In HypRG the vertex-weights have an $n$-dependent distribution converging to a limiting distribution \cite{komjathy2020explosion}. However, these differences may be overcome by coupling techniques presented in e.g.~\cite{komjathy2020explosion}: a model with exactly $n$ vertices can be squeezed between two GIRGs with Poisson intensity $1-\sqrt{4\log n/n}$ and  $1+\sqrt{4\log n/n}$, and one can couple $n$-dependent and limiting vertex-weights to each other, respectively, but we avoid spelling out the details and refer the reader to \cite[Claims 3.2, 3.3]{komjathy2020explosion}.

\subsection{Discussion}\label{sec:discussion}
Here we discuss our results in context with related results about (inhomogeneous) first passage percolation and graph distances on spatial random graphs. 

{\bf Long-range first passage percolation.} The work on long-range first passage percolation (LR-FPP) \cite{chatterjee2016multiple} is closest to our work. In that model, the underlying graph is the \emph{complete graph} of $\Z^d$, and the edge transmission time on any edge $uv$ is exponentially distributed with mean $|u-v|^{d\alpha'+o(1)}$, so $\beta=1$, the process is Markovian, and the penalty depends on the Euclidean distance of $u$ and $v$. This choice eliminates the correlations coming from the presence/absence of underlying edges, and the growth is strictly governed by the long-range transmission times. 
As $\alpha'$ grows, \cite{chatterjee2016multiple} finds the same sub-explosive phases for transmission times in LR-FPP that we find for 1-FPP in Table \ref{table:summary}. The main difference is that the explosive phase is absent in LR-FPP, and is replaced by a `super-fast' phase there where transmission times are $0$ almost surely. Moreover, the behaviour on phase boundaries are different. Using the symmetries in their model, \cite{chatterjee2016multiple} proves that whenever transmission times in LR-FPP are strictly positive, then they must be at least logarithmic. In contrast, in general 1-FPP on IGIRG and SFP we also see doubly-logarithmic distances, for example for graph-distances ($L\equiv1, \mu=0$). We summarise the results on LR-FPP in Table \ref{table:phases-top-FPP}.
Nevertheless, in 1-FPP, the cost function $\cost{xy}$ in \eqref{eq:cost} could also depend on $|x-y|$, i.e., take the form $L_{xy}(W_x W_y)^\mu|x-y|^\zeta$. The result of \cite{komjathy2020stopping} on explosion carries through to this case without much effort \cite{reubsaet2022}, with the model being explosive if and only if $\mu+\zeta/d < (3-\tau)/(2\beta)=\mu_{\mathrm{expl}}$, with $\tau\in(2,3)$. Extrapolating our results to the spatial penalty, we conjecture that we see polylogarithmic distances when $\mu+\zeta/d\in (\mu_{\mathrm{expl}}, \mu_{\log})$ and strictly polynomial distances when $\mu+\zeta/d\in (\mu_{\log}, \mu_{\mathrm{pol}})$.  
\vskip0.5em

\begin{table}[ht!]
\begin{tabular}{l||l|l|l|l}
\cellcolor{grey}\textbf{SFP/LRP} with & \cellcolor{grey}\textbf{Graph-distance} & \cellcolor{grey}\textbf{Growth} & \cellcolor{grey}{\textbf{Upper bound}} & \cellcolor{grey}{\textbf{Lower bound}}  
  
\tabularnewline\hline
$\tau\in(2,3)$ & $\tfrac{(2\pm o(1))\log \log |x|}{|\log(\tau-2)|}$ & doubly- & \cite{deijfen2013scale, van2017explosion} & {\cite{van2017explosion}} 
\tabularnewline
&& logarithmic &  &  
\tabularnewline\hline
$\tau>3$ and & $(\log |x|)^{\Delta\pm o(1)}$ & poly- &  {SFP: \cite{hao2021graph},} & SFP: \cite{hao2021graph} 
\tabularnewline
$\alpha \in (1,2)$ & for some $\Delta>0$ & logarithmic & LRP: \cite{biskup2004scaling,trapman2010growth} & \quad (partly open) 
\tabularnewline
&&&  &  LRP:\cite{biskup2004scaling, biskup2019sharp,trapman2010growth} 
\tabularnewline\hline
$\tau>3$ and & $|x|^{\eta\pm o(1)}$, & polynomial &SFP: open, & SFP: open  
\tabularnewline
$\alpha=2$ & for some $\eta<1$ && LRP: \cite{baumler2023distances} & LRP: \cite{baumler2023distances} 
\tabularnewline\hline
$\tau>3 $ and & $\Theta(|x|)$ & linear & partly open$^\ddag$ \cite{antal1996chemical} & \cite{berger2004lower, deprez2015inhomogeneous} 
\tabularnewline
$\alpha>2$ & & &  &
\tabularnewline\hline\hline
\cellcolor{grey}\textbf{LRFPP} with & \cellcolor{grey}\textbf{Cost-distance} &\cellcolor{grey}\textbf{Growth}&  \cellcolor{grey}\textbf{Upper bound}&\cellcolor{grey}\textbf{Lower bound} 
\tabularnewline\hline
$\alpha' < 1$ & $0$ & instantaneous & \multicolumn{2}{c}{\cite{chatterjee2016multiple}}

\tabularnewline\hline
$\alpha' \in(1, 2)$ & $(\log |x|)^{\Delta_{\alpha'}\pm o(1)}$ for & poly- & \multicolumn{2}{c}{\cite{chatterjee2016multiple}}
\tabularnewline
& $\Delta_{\alpha'}=1/(1-\log_2\alpha')$ &  logarithmic & \multicolumn{2}{c}{}

\tabularnewline\hline
$\alpha'\in (2,2+1/d)$ & $|x|^{d (\alpha'-2)\pm o(1)}$ & polynomial & \multicolumn{2}{c}{\cite{chatterjee2016multiple}}

\tabularnewline\hline
$\alpha'>2+1/d$ & $\Theta(|x|)$ & linear & \multicolumn{2}{c}{\cite{chatterjee2016multiple}}

\tabularnewline\hline\hline
\cellcolor{grey}\textbf{IGIRG/SFP} & \cellcolor{grey}\textbf{Cost-distance} & \cellcolor{grey}\textbf{Growth}& \cellcolor{grey}\textbf{Upper bound} & \cellcolor{grey}\textbf{Lower bound}
\tabularnewline
\cellcolor{grey}with~$\tau\in(2,3)$ &\cellcolor{grey}&\cellcolor{grey}& \cellcolor{grey} & \cellcolor{grey}
\tabularnewline\hline
$\mu<\mu_{\mathrm{expl}}$ & converges & explosion & {\cite{komjathy2020stopping}} &{\cite{komjathy2020stopping}}
\tabularnewline
& in distribution &&  &
\tabularnewline\hline
$\mu\in(\mu_{\mathrm{expl}},\mu_{\log})$ & $(\log |x|)^{\Delta_0+ o(1)}$, & poly- & Theorem~\ref{thm:polylog_regime} & open
\tabularnewline
or $\alpha\in(1,2)$ & $\Delta_0$ as in~\eqref{eq:Delta_0} & logarithmic & \ & 
\tabularnewline\hline
$\mu\in\big(\mu_{\log}, \mu_{\mathrm{pol}})$ & $ |x|^{\eta_0\pm o(1)}$, & polynomial & {Theorem~\ref{thm:polynomial_regime}} & {\cite{komjathy2022one2}}
\tabularnewline
and $\alpha>2$ & $\eta_0$ as in~\eqref{eq:eta_0} &&  &
\tabularnewline\hline
$\mu>\mu_{\mathrm{pol}}$ & $\Theta(|x|)$ & linear & {\cite{komjathy2022one2}}, $d\ge 2$ & {\cite{komjathy2022one2}}
\tabularnewline
and $\alpha>2$ &&& &
\end{tabular}
\caption[caption]{\raggedright Known results about the universality classes of graph-distances on long-range percolation \textbf{LRP}, scale-free percolation \textbf{SFP}, long-range first-passage percolation \textbf{LRFPP} and infinite geometric inhomogeneous random graphs \textbf{IGIRG}. \\$^\ddag$An upper bound is only  known for high enough edge-density or all nearest-neighbour edges present.}
\vskip-1.5em
\label{table:phases-top-FPP}
\end{table}

{\bf Qualitative difference between one-dependent FPP and graph distances.} Some phases of 1-FPP in Table \ref{table:summary} are also phases for \emph{graph-distances} in spatial models in general. 
However, while the polynomial phase is spread-out in 1-FPP, this phase is essentially absent for graph distances in models. Indeed, the polynomial phase occurs when long edges all have polynomial spreading times in the Euclidean distance they bridge, both in 1-FPP here and in LR-FPP in \cite{chatterjee2016multiple}. Thus, transmission times in 1-FPP are not equivalent to graph distances in any inhomogeneous percolation on the underlying graph. 
Table~\ref{table:phases-top-FPP} summarises known results on 1-FPP, LR-FPP, and graph distances in spatial graphs. Now we elaborate on each phase.
\vskip0.5em

\textbf{The polylogarithmic phase.}
Theorem \ref{thm:polylog_regime} proves polylogarithmic cost-distances in 1-FPP when $\tau\in(2,3)$, and either $\mu\in (\mu_{\mathrm{exp}}, \mu_{\mathrm{log}})$ or $\alpha\in(1,2)$. The results here, in \cite{komjathy2020stopping} and the accompanying \cite{komjathy2022one2} (Corollary \ref{cor:polynomial_regime}) together imply that $\mu_{\mathrm{exp}}$ and $\mu_{\mathrm{log}}$ are true phase-transition points, separating this phase from both the explosive and the polynomial phases.
Even though we do not have a matching lower bound, we conjecture that this phase is truly polylogarithmic, and the exponent $\Delta_0$ in \eqref{eq:Delta_0} is sharp. The exponent $\Delta_0$ also depends on the product $\mu\beta$, which does not allow to match it easily to exponents for graph-distances:
For long-range percolation, where each edge $(u,v)\in\Z^d\times \Z^d$ is present independently with probability $\Theta(|u-v|^{-d\alpha})$, Biskup and Lin \cite{ biskup2019sharp} show that graph distances grow polylogarithmically with exponent $\Delta_{\alpha} = 1/(1-\log_2 \alpha)$ when $\alpha\in(1,2)$. This coincides with our upper bound in Theorem \ref{thm:polylog_regime} if $\alpha\le \tau-1+\mu\beta$. The same type of paths are used in both cases, passing through only low-degree vertices (and typical edge-costs on them for 1-FPP). For scale-free percolation, Hao and Heydenreich proved in~\cite[Theorem~1.1]{hao2021graph} that graph distances are also polylogarithmic -- with non-sharp exponents -- in scale-free percolation when $\alpha\in(1,2)$ and additionally $\tau>3$  (using our parametrisation of SFP). 
In an ongoing work a subset of us determines the exact exponent for graph distances in SFP. The lower-bound methods in~\cite{biskup2019sharp} or~\cite{hao2021graph} do not transfer to 1-FPP.
\vskip0.5em

\textbf{The linear phase.} Linear distances are common in supercritical spatial graph models with bounded edge-lengths. For example, Random Geometric Graphs exhibit linear distances~\cite{penrose2003random}, and so does supercritical percolation on grids of dimension at least $2$ \cite{antal1996chemical}. Assuming high enough edge-density, a renormalisation argument to percolation on $\Z^d$ gives that SFP and LRP for $\tau>3$ and $\alpha >2$ also have at most linear graph-distances for $d\geq 2$. The corresponding lower bound was shown by Berger for LRP~\cite{berger2004lower} and by Deprez \emph{et al.}~for SFP~\cite{deprez2015inhomogeneous}. Our lower bound for 1-FPP contains these as special cases, and holds universally for classical FPP for any positive edge-transmission time-distribution~\cite[Corollary 1.12]{komjathy2022one2}. 
\vskip0.5em

\textbf{The strictly polynomial phase.} The phase where intrinsic distances scale as $|x-y|^{\eta_0+o(1)}$ with $\eta_0<1$ (the result of Theorem \ref{thm:polynomial_regime}) is quite rare in spatial settings and we only know two examples. One is in one-dimensional LRP at a boundary line in the parameter space, when $\alpha=2$~\cite{baumler2023distances}. The other is for long-range first passage percolation (LR-FPP) in \cite{chatterjee2016multiple}, mentioned at the beginning of this section. 
There are some similarities to 1-FPP: LR-FPP is Markovian, i.e., $\beta=1$ in \eqref{eq:F_L-condition}, and has  strictly polynomial growth when $\alpha'\in(2,2+1/d)$, see Table \ref{table:phases-top-FPP}. Using exponential $L$ in 1-FPP, the length of  the parameter interval $(\mu_{\log}, \mu_{\mathrm{pol}})$ with polynomial growth is also exactly $1/d$ for $\mu$ when $\alpha>2+1/d$, but it is longer when $\alpha<2+1/d$, which shows that the penalty $\alpha'$ of LR-FPP plays a slightly different role as the long-range parameter $\alpha$ in \eqref{eq:connection_prob} here.
\vskip0.5em

\textbf{Gaps at approaching the phase boundaries.} 
Here we discuss what happens as the  parameters $\tau, \alpha, \mu, \beta$ approach the phase boundaries of growth. 

Polylogarithmic distances with exponent $\Delta_0$ heuristically imply \emph{stretched exponential ball-growth}, where the number of vertices within intrinsic distance $r$ scales as $\exp(r^{1/\Delta_0})$.
Our upper bound exponent $\Delta_0=\min\{\Delta_\alpha, \Delta_\beta\}$ in \eqref{eq:Delta_0} approaches $1$ as $\alpha \downarrow 1$, and so does the exponent $\Delta_{\alpha}$ of LRP \cite{biskup2004scaling}, which also partly governs SFP. This means that as $\alpha\downarrow 1$ we approach exponential growth. 
 In LRP, strictly exponential ball growth only occurs when $\alpha=1$ and the connectivity function has a suitably chosen slowly varying correction term $\ell(\cdot)$, i.e., $p(x,y)= \ell(|x-y|)/|x-y|^{\alpha d}$, see  Trapman~\cite{trapman2010growth}. Strictly exponential growth is a natural barrier, since (age-dependent) branching processes with finite first moments exhibit at most exponential growth, 
and non-Markovian FPP can be dominated by such branching processes. 
Interestingly, when $\alpha>2$ and we approach the phase transition of explosion by letting $\mu\downarrow\mu_{\mathrm{expl}}=(3-\tau)/(2\beta)$, then $\Delta_0$ in \eqref{eq:Delta_0} does not converge to $1$, but to $1/(2-\log_2(\tau+1))=:\Delta_{\tau}$. So, for the whole range $\tau \in (2,3)$, $\Delta_{\tau}\ge 1/(2-\log_2(3)) > 2.4 > 1$. As $\tau \uparrow 3$, $\Delta_{\tau}$ approaches $\infty$, which is natural, since already graph-distances are linear when $\tau >3$ and $\alpha >2$~\cite{deprez2015inhomogeneous}. This leaves two possibilities: either our upper bound $\Delta_0$ is not sharp for $\alpha>2$; or the ball growth jumps directly from subexponential ($\Delta_0 >1$) into the explosive phase. If the latter should be the case, it would be interesting to understand better how such a jump could happen. Such jumps at phase boundaries may occur. This paper together with~\cite{komjathy2022one2} proves a gap in the \emph{polynomial regime} as $\tau$ crosses the threshold $3$.
The limits of $\lim_{\tau\uparrow 3}\mu_{\textrm{pol}}=1/d$ and $\lim_{\tau\uparrow 3}\eta_0=\mu d$ exist and are in $(0,1)$ in \eqref{eq:mu_pol_log} and \eqref{eq:eta_0}. So if we fix some $\mu < 1/d$ and let $\tau \uparrow 3$, the cost-distances grow polynomially with exponents bounded away from one (e.g., they approach $1/2$ from below for $\mu  = 1/(2d)$). But as soon as $\tau > 3$ is reached, the exponent ``jumps'' to $1$ and distances become strictly linear~\cite[Theorem~1.11]{komjathy2022one2}. So the parameter space is discontinuous in $\eta_0$ and $\mu_{\textrm{pol}}$ with respect to $\tau$. 

Some important questions are centred around such gaps. The \emph{gap conjecture} in geometric group theory is about the ball growth of finitely generated groups: it states that there are no groups with growth between polynomial and stretched exponential of order $\exp(\Theta(\sqrt{n}))$ \cite{grigorchuk1990growth}. While the polynomial side is understood by Gromov's theorem \cite{gromov1981groups}, the conjecture remains open. We find it intriguing to discover a deeper connection between phases of intrinsic growth in spatial random graphs (``stochastic lattices'') and group theory  (``deterministic lattices'').
\vskip0.5em

\textbf{Organisation.} We start by moving to the quenched setting. In Section \ref{sec:nets} we develop the pseudorandom nets, and in Section \ref{sec:exposure} the multi-round exposure of edges, with the main result in Proposition~\ref{prop:multi-round-exposure}. In Section \ref{sec:lemmas} we collect some connectivity-estimates that serve as our building blocks, while in Section \ref{sec:hierarchy} we carry out the `budget travel plan' and build a hierarchical path that only uses vertices of the pseudorandom net and connects vertices $y_0^\star, y_x^\star$ very close to $0$ and $x$, respectively. In Section \ref{sec:endpoints} we connect $0,x$ to these vertices with low cost, which is a nontrivial task itself, and prove the main theorems.
 
\subsubsection{Notation}\label{sec:notation}
Throughout, we consider simple and undirected graphs with vertex set $\calV \subseteq \R^d$. For a graph $G=(\calV,\calE)$ and a set $A\subseteq \R^d$, $G[A]$ stands for the induced subgraph of $G$ with vertex set $\calV \cap A$. For two vertices $x,y\in \calV$, we denote the edge between them by $xy$, and for a set $V\subseteq \calV$ we write $V^{(2)} := \{\{x,y\} : x,y\in V, x\neq y\}$ for the set of possible edges among vertices in $V$. For a path $\pi$, $\calE(\pi)$ is the set of edges forming $\pi$, and $|\pi|$ is the number of edges of $\pi$. Generally the size of a discrete set $S$ is $|S|$, while of a set $A\subseteq \R^d$, $\mathrm{Vol}(A)$ is its Lebesgue measure. Given a cost function $\calC: \calE \to [0,\infty]$ on the edges, the cost of a set of edges $\calP$ is $\cost{\calP}:=\sum_{e\in\calE(\calP)}\cost{e}$. We define $\calC(xx):=0$ for all $x\in \calV$. We define the \emph{cost-distance} between vertices $x$ and $y$ as
\begin{align}\label{eq:cost_distance}
    d_{\calC}(x,y):=\inf\{\cost{\pi} : \pi \textnormal{ is a path from } x \textnormal{ to } y \textnormal{ in $G$}\}.
\end{align}
We define the graph distance $d_G(x,y)$ similarly, where all edge-costs are set to $1$. 
We denote the Euclidean norm of $x \in \mathbb R^d$ by $|x|$, the Euclidean ball with radius $r\geq 0$ around $x$ by $B_r(x) := \{y\in \R^d : |x-y|\le r\}$, and the set of vertices in this ball by $\calB_r(x):=\{y\in\calV : |x-y|\le r\} = B_r(x) \cap \calV$. (The minimal notation difference is intentional). The \emph{graph-distance ball} and \emph{cost-distance ball} (or \emph{cost-ball} for short) around a vertex~$x$ are the vertex sets $\calB_r^G(x):=\{y\in\calV : d_G(x,y)\le r\}$ and $\calB_r^{\calC}(x):=\{y\in\calV : d_{\calC}(x,y)\le r\}$, respectively. We set $B_r := B_r(0)$, and do similarly for $\calB_r$, $\calB_r^G$, $\calB_r^{\calC}$ if $0$ is a vertex. We define $\partial B_r(x) := B_r(x)\setminus \{y\in \R^d : |x-y| < r\}$, and use similar definitions for $\partial \calB_r$, $\partial \calB_r^G$ and $\partial \calB_r^{\calC}$. In particular, $\partial \calB_1^G(v)$ is the set of neighbours of $v$.
In Definition \ref{def:girg}, $\widetilde v = (v,w_v)$ stands for a weighted vertex. For some $A \subset \R^d$, we simply write $\widetilde{v} \in A$ to indicate that $(v, w_v) \in A \times [1, \infty)$, i.e., we implicitly project $\widetilde{v}$ onto Euclidean space. 

The set of \emph{model parameters} are $\mpar:= \{d, \tau, \alpha, \mu, \beta, \underline{c}, \overline{c}, c_1, c_2, t_0\}$. For parameters $a,b >0$ (model parameters or otherwise), we use ``for all $a\ggs b$'' as shortcut for ``$\forall b>0:\, \exists a_0 = a_0(b):\, \forall a\ge a_0$''. We also say ``$a \ggs b$'' to mean that $a \ge a_0(b)$. We use $a \lls b$ analogously, and may use more than two parameters. For example, ``for $a\ggs b,c$'' means ``$\forall b,c>0:\, \exists a_0=a_0(b,c):\, \forall a \ge a_0$''. A measurable function $\ell:(0,\infty) \to (0,\infty)$ is said to be \emph{slowly varying} at infinity if $\lim_{x\to \infty} \ell(cx)/\ell(x) =1$ for all $c>0$. We denote by $\log$ the natural logarithm, by $\log_2$ the logarithm with base $2$, and by $\log^{*k}$ the $k$-fold iterated logarithm, e.g. $\log^{*3}x = \log\log\log x$. For $n \in \N$ we write $[n]:= \{1,\ldots,n\}$.

\vskip-2em

\section{Moving to the quenched setting: pseudorandom nets}\label{sec:nets}

In proving the upper bounds (Theorems~\ref{thm:polylog_regime} and~\ref{thm:polynomial_regime}), we will construct cheap paths along the lines of the `budget travel plan' in Section \ref{sec:intro}, which is an iterative scheme of finding long 3-edge bridge-paths to connect two far-away vertices. Since low-cost events in 1-FPP are not increasing, we develop a technique that replaces the FKG-inequality. Moving to the quenched setting, we will first expose all vertex positions and weights (above some threshold weight, in the case of IGIRG); then, low-cost edge existence events become independent. To be able to work with \emph{fixed realisations} of the vertex set, we find (with high probability as $|x|\to\infty$) a \emph{pseudorandom} subset of the vertices that behaves regularly enough inside a box around $0,x$, as in \eqref{eq:net-heuristics}, which we call a \emph{net}.
We formalise the notion of the nets now.

For a set $A\subset \R^d$ we write $\mathrm{Vol}(A)$ for its Lebesgue measure (volume), while for a discrete set $\calA\subseteq (0,\infty)$ we write $|\calA|$ for the cardinality (size) of the set.  Recall the slowly varying function $\ell(\cdot)$ from \eqref{eq:power_law}. Recall that weighted vertices are of the form $\tilde v=(v, w_v)$.
\begin{definition}[Net]\label{def:net}
    Let $G=(\calV,\calE)$ be an IGIRG or SFP in Definition \ref{def:girg}. Let $\calR \subseteq (0,\infty)$ be a set with $|\calR|<\infty$, let $w_0 \in [1,\infty)$, and let $f\colon \calR\to [w_0,\infty)$. Let $Q \subseteq \R^d$
    be Lebesgue measurable. An \emph{\newnet{\calR}{w_0}{f} for $Q$} is a set $\calN \subseteq \widetilde{\calV} \cap Q\times[1,\infty)$ of size at least $\mathrm{Vol}(Q)/4$ such that for all $\tilde v \in \calN$, $r \in \calR$, and all $w \in [w_0,f(r)]$,  
    \begin{equation}\label{eq:net-defining-crit}
      \left| \big\{\widetilde u\in \calN\cap B_r(v) \times [w/2, 2w] \big\}\right| \ge  r^d\cdot \ell(w)w^{-(\tau-1)}/(2d)^{d+\tau+5}.
    \end{equation}
\end{definition}
It may seem very strong that we require \eqref{eq:net-defining-crit} for infinitely many $w$. We discretise $[w_0, f(r)]$ into a finite set of subintervals $(I_j)_{j\le j_{\max}}$ in a smart way. Then we ensure that  \eqref{eq:net-defining-crit} holds with $[w/2, 2w]$ replaced by $I_j$ on the left hand side and the $\ell(w)w^{-(\tau-1)}$ replaced by $\mathbb P(W\in I_j)$ on the right hand side, and then this will imply that \eqref{eq:net-defining-crit} also holds for all values  $w\in [w_0, f(r)]$. Now we set $f(r)$ only a little less than the typical largest vertex weight in a ball of radius $r$ (roughly $r^{d/(\tau-1)}$), and  $w_0$ is a large constant, starting with $w_0$: 

\begin{definition}[Strong net]\label{def:net-constants} Fix $\ell(w)$ from \eqref{eq:power_law}.
    We define $w_0$ to be the smallest integer in $[1,\infty)$ such that for all $w \ge w_0$ and all $t \in [1/2,2]$,
    \begin{equation}\label{eq:nets-ell-bound-0}
        \ell(w)w^{-(\tau-1)} < 2^{-\tau-8} \quad \mbox{and}\quad 0.99 \le \ell(t w)/\ell(w) \le 1.01
    \end{equation}
    both hold.
     For all $\delta>0$, and $R>0$ we define the function
    \begin{equation}\label{eq:nets-f-def}
    f_{R,\delta}(r) = r^{\frac{d}{\tau-1}} \big(1 \wedge \inf\big\{\ell(x)\colon x \in [w_0,r^{d/(\tau-1)}]\big\}\big)^{\frac{1}{\tau-1}}\cdot \Big(\frac{1}{(2d)^{2\tau+d+8}\log(16 R/\delta)}\Big)^{\frac{1}{\tau-1}}.
    \end{equation}
     On fixing $\delta>0$, $w_0$ satisfying \eqref{eq:nets-ell-bound-0}, and $R:=|\calR|$ in \eqref{eq:nets-f-def}, we call an \newnet{\calR}{w_0}{f_{R,\delta}} in Definition \ref{def:net} for a set $Q \subseteq \R^d$ an \emph{$(\delta, \calR)$-net for $Q$}. 
\end{definition}
Note that $w_0$ satisfying \eqref{eq:nets-ell-bound-0} must exist since $\ell$ is a slowly-varying function, and so Potter's bound \cite{bingham1989regular} ensures the first inequality.
In order to find an $(\delta,\calR)$-net, and specify the values $r\in \calR$, for technical reasons it is convenient to assume that the values in $\calR$ grow at least exponentially with base depending on $|\calR|$. 
In the main proofs, we later choose $|\calR|$ to be a either a constant or doubly logarithmic in the distance of the two vertices we want to connect via a cheap path. The specific condition is the following.

\begin{definition}\label{def:well-spaced} Fix $\delta \in(0,1)$.
    We say a set $\calR \subseteq (0,\infty)$ is \emph{$\delta$-well-spaced} if $|\calR|=:R<\infty$ and the following hold, writing $\calR = \{r_1,\ldots,r_R\}$ with $r_1 < \ldots < r_R$:
      \begin{align}
        r_1 &\ge 24d\big(\log(4R/\delta)\big)^{1/d} \vee w_0^{(\tau-1)/d} \vee \inf\{r \colon f_{R,\delta}(r) \ge w_0\}; \label{eq:nets-small-r}\\
         \frac{r_i}{r_{i-1}} &\ge 6R^{1/d}\Big(\frac{\log(2R/\delta)}{\delta}\Big)^{1/d} \quad \forall i\in [2,R]. \label{eq:nets-radii}
    \end{align}
       
\end{definition}

We now state the main result of the section. Heuristically, a box $Q$ contains an $(\delta,\calR)$-net with sufficiently high probability, and we can also condition on the presence of a few vertices in the net. The condition $t\le 1/\delta$ in the following is added to avoid a vacuous statement.

\begin{restatable}{proposition}{NetsExist}\label{lem:nets-exist}
    Consider IGIRG or SFP with $\tau > 2$. Let $\delta \in (0,1/16), \xi > 0$, and $Q \subseteq \R^d$ be a cube of side length $\xi$. Let $R\in \N$ and $\calR=\{r_1, \ldots, r_R\}$ with $0<r_1<\ldots<r_R$ be a $\delta$-well-spaced set such that $r_R = \xi\sqrt{d}$. Let $x_1,\ldots,x_t \in Q$ with $t \le \min\{1/\delta,(r_1/4\sqrt{d})^d\}$. Then
 \begin{align}
    \pr&(\text{$Q$ contains an $(\delta,\calR)$-net $\calN$})\ge 1-\delta/R\label{eq:netsexist-1};\\
    \pr&(\text{$Q$ contains an $(\delta,\calR)$-net $\calN$}, x_1,\ldots,x_t \in \calN \mid x_1,\ldots,x_t \in \calV ) \ge 1-t\delta. \label{eq:strong-netsexist}
    \end{align}
\end{restatable}
The rest of this section is devoted to proving Proposition~\ref{lem:nets-exist}. All remaining definitions and lemmas are used only within this section. We now give the setting throughout this section. 

\begin{setting}\label{set:R-of-section}
Consider IGIRG or SFP with $\tau > 2$ in Definition \ref{def:girg}. Let $\delta \in (0,1/16)$, $\xi > 0$, $Q \subseteq \R^d$ be a $d$-dimensional cube of side length $\xi$, and $\calR = \{r_1,\ldots,r_R\} \subseteq (0,\infty)$ a $\delta$-well-spaced set with $r_1 < \ldots < r_R = \xi\sqrt{d}$. 
\end{setting}

Shortly we shall carry out a multi-scale analysis. We partition $Q\times [w_0, f(r_R)]$ into hyper-rectangles. On the weight-coordinate, we cover the interval $[w_0,f(r_R)]$ of weights with a set of disjoint intervals $(I_j)_{j=1, \ldots, j_{\max}}$ so that the first interval is of length $w_0$, and each consecutive interval is twice as long as the previous one.
On the space-marginal, we partition $Q$ into nested boxes $B$. The side lengths of these nested boxes will be close to $r_1, \ldots, r_R$, with some minor perturbation so that they can form a proper nested partition: we write $r_i' \approx r_i$ for the side length of the $i$'th level of boxes. 
A depiction and extended example can be found in Figure~\ref{fig:i-good} on page~\pageref{fig:i-good} below, after the formal definition.

After fixing this partitioning of $Q\times [w_0,f(r_R)]$, we look at $\widetilde {\calV}\cap (Q\times [w_0,f(r_R)])$. 
For each $i \in [R]$, we show that with probability close to $1$ there is a dense subset of ``good'' boxes $B$ of side length $r_i'$, in the sense that $B \times I_j$ contains the right number of vertices for all $I_j$ with $\max(I_j) \approx f(r_i)$. We choose $r_i'<r_i/\sqrt{d}$ to ensure that for all vertices $v$ in a box of side length $r_i'$, the entire box will be contained in $B_{r_i}(v)$ -- the ball of radius $r_i$ around $v$ -- this will allow us to take the net to be the set of all vertices which lie in good boxes of all side lengths $r_1',\ldots,r_R'$.
We start with the space marginal and now formally define the nested boxes.

\begin{definition}\label{def:nets-partition} Given $\calR=\{r_1, \ldots, r_R\}$, and $Q$ as in Setting~\ref{set:R-of-section}, an \emph{$\calR$-partition} of $Q$ is a collection of partitions $\widehat{\calP}(\calR):=\{\calP_1,\ldots,\calP_R\}$ of $Q$ into boxes with the following properties:
    \begin{enumerate}[(P1)]
        \item\label{item:P1} For all $i \in [R]$, every box in $\calP_i$ has the same side length $r_i'$ with 
        \begin{equation}\label{eq:ri-bound}
        r_i'\in [r_i/(2\sqrt{d}), r_i/\sqrt{d}].
        \end{equation}
        \item\label{item:P2} For all $i \in [R-1]$, every box in $\calP_{i+1}$ is partitioned into exactly $(r_{i+1}')^d/(r_i')^d$ boxes in $\calP_i$.
             \item\label{item:P3} We have $\calP_R = \{Q\}$.
             \item\label{item:P4} For $x \in Q, i \!\in\! [R]$, write $B^i(x)$ for the box in $\calP_i$ containing $x$. Then $B^i(x) \subseteq B_{r_i}(x)$.
 \end{enumerate}
\end{definition}
Observe that \emph{(P2)} ensures that the partition $\calP_i$ is a refinement of the partition of $\calP_{i+1}$, i.e., that every box in $\calP_{i+1}$ can be partitioned exactly into sub-boxes in $\calP_i$. Also, $r_R=\xi\sqrt{d}$ ensures that \emph{(P1)} and \emph{(P3)} can be simultaneously satisfied for $i=R
$.
\begin{claim}\label{lem:nets-partition-exists} Suppose $\calR$ and $Q$ satisfy Setting~\ref{set:R-of-section}. Then an $\calR$-partition $\widehat{\calP}(\calR)$ of $Q$ exists.
\end{claim}

\begin{proof}
    We prove that given  $r_1, \ldots, r_R $, there exist side lengths $r_1',\ldots,r_R'$ that satisfy \emph{(P\ref{item:P1}) -- (P\ref{item:P4})}, i.e., that $r_{i+1}'/r_i'$ is an integer, \eqref{eq:ri-bound} holds, and $r_R'=\xi$. Clearly \eqref{eq:ri-bound} implies that for any vertex $v$ in a box $B$ of side-length $r_i'$, $B\subseteq B_{r_i}(v)$, hence \emph{(P\ref{item:P4})} follows directly once we satisfy \eqref{eq:ri-bound} and allocate box-boundaries uniquely.
    We proceed by induction on $i$, starting from $i=R$ and decreasing $i$. We take $r_{R}' := r_{R}/\sqrt{d}=\xi$, then \eqref{eq:ri-bound} is satisfied immediately and \emph{(P\ref{item:P2})--(P\ref{item:P3})} are vacuous. Suppose we have found $r_i',\ldots,r_{R}'$ satisfying \emph{(P\ref{item:P1})--(P\ref{item:P3})} for some $2 \le i \le R$. Let 
\begin{equation}\label{eq:def-ri-prime}
        r_{i-1}' = \frac{r_i'}{\lceil \sqrt{d} r_{i}'/r_{i-1} \rceil}.
    \end{equation}
    This choice of $r_{i-1}'$ divides $r_i'$, hence \emph{(P2)} can be satisfied, and $r_{i-1}' \le r_i'/(\sqrt{d}r_i'/r_{i-1}) =  r_{i-1}/\sqrt{d}$. Moreover,
\begin{equation}\label{eq:nets-partition-exists}
        r_{i-1}' \ge \frac{r_i'}{1 + \sqrt{d} r_i'/r_{i-1}} = \frac{r_{i-1}}{r_{i-1}/r_i'+\sqrt{d}}.
    \end{equation}
   Since \eqref{eq:ri-bound} holds for $i$ (by the inductive assumption), $r_i' \ge r_i/2\sqrt{d}$. Since $R$ is well-spaced, $r_{i-1} \le r_i/2$ by~\eqref{eq:nets-radii}; hence $r_{i-1}/r_i' \le \sqrt{d}$. It follows from~\eqref{eq:nets-partition-exists} that $r_{i-1}' \ge r_{i-1}/2\sqrt{d}$, and so \eqref{eq:ri-bound} holds also for $i-1$ and the induction  is advanced.
       
Given these $r_1',\ldots,r_R'$, we find an $\calR$-partition of $Q$ by taking $\calP_R = \{Q\}$ and iteratively forming each layer $\calP_{i-1}$ by taking the unique partition of each box in $\calP_i$ into $(r_i')^d/(r_{i-1}')^d$ sub-boxes of side length $r_{i-1}'$. We first define each partition box to be of the form $\prod_{j=1}^d[a_j, b_j)$, this allocates each point except $d$ of the $d-1$-dim faces of $\partial Q$ uniquely. 
Finally, we allocate the points $x\in\partial Q$ in $\calP_i$ to the box in $\calP_i$ that contains $x$ in its closure, this box is unique except on $d-2$ dimensional faces. Here we again use half-open $d-1$-dim boxes to determine the $d-2$ dim boundaries, and so on until only the corner-points are left which we allocate arbitrarily (but consistently across different $i$).
\end{proof}
We continue with the weight-marginal and cover the interval $[w_0,f(r_i)]$ of weights with a collection of intervals, forming later the weight-coordinate of the hyper-rectangles:
\begin{definition}[Base-$2$-cover]\label{def:base-2-cover}
    Given a closed interval $J = [a,b] \subset \R_+$, let $j_{\max}:=\lfloor\log_2(b/a)\rfloor+1$, and define $I_j:=[2^{j-1} a, 2^{j}a)$ for $j\in[j_{\max}]$. 
    Then $J\subseteq \bigcup_{j=1}^{j_{\max}}I_j$
    and we call $I=\{I_j\}_{j\le j_{\max}}$ the \emph{base-$2$-cover} of $[a,b]$.
    For each $x\in [a,b]$ we define $I(w):=\{I_j: w\in I_j\}$ to be the unique interval that contains $x$, and write $I(w)=[w_-, w_+)$ for its endpoints.
\end{definition}
For each $w\!\in\![a,b]$, $w/2\le w_-$ and $ w_+\le 2x$, so $I(w)\subseteq [w/2, 2w]$; by the definition of $I_{j_{\max}}$, we also have $b \in I_{j_{\max}}$. We define the hyperrectangle-covering of the box $Q$ including vertex-weights now. Recall vertex-weight distribution $W$ from \eqref{eq:power_law}, and $f(r)$ from Def.~\ref{def:net}.

\begin{definition}[Hyperrectangles]\label{def:hyperrectangles}
Consider Setting~\ref{set:R-of-section} and Definitions~\ref{def:nets-partition}, \ref{def:base-2-cover}. Let $\widehat {\calP}(\calR):=\{\calP_1, \ldots, \calP_R\}$ be an $\calR$-partition of the cube $Q$ with $\calR=\{r_1, \ldots, r_R\}$, $r_i'$ be the side-lengths in $\calP_i$, and let $I=\{I_j\}_{j\le j_{\max}}$ be a base-$2$-cover of $[w_0, f(r_R)]$. 
    Let $j_{\star}(i)$ be the index of the interval that contains $f(r_i)$, i.e.,
    \begin{equation}\label{eq:jstar}
        f(r_i) \in I(f(r_i))=:I_{j_\star(i)}.
    \end{equation}  
    Then we say that the collection $\calH(\calR):=\big\{ B_i\times I_j: B_i\in \calP_i, 1 \le j \le j_{\star}(i) \big\}$
    is a \emph{hyperrectangle-cover of $Q\times [w_0, f(r_R)]$}. For all $i \in [R]$ and all $A \subset [w_0,f(r_R)]$, we define
    \begin{equation}\label{eq:measure-hyperrectangle}
        \mu_i(A) :=(r_i')^d\cdot \mathbb P(W\in A).
    \end{equation}
\end{definition}
When we cover with boxes in $\calP_i$ on the spatial coordinate, the number $j_{\star}(i)$ of weight intervals in $\calH(\calR)$ depends on $i$. In particular, for smaller side-length we do not include intervals of very large weights. This is because there are too few (or no) vertices of large weight in a typical box of small side-length, so we cannot control their number. We illustrate a hyperrectangle cover on Figure \ref{fig:i-good} in dimension $1$.
In GIRG, $\mu_i(A)$ is the expected number of vertices with weights in $A$ in any box in $\calP_i$. In SFP, $\mu_i(A)$ is only roughly the expectation,  since e.g. a box touching the boundary $\partial Q$ in $\calP_i$ may not contain exactly $(r_i')^d$ vertices. By Definition~\ref{def:nets-partition} \emph{(P\ref{item:P4})}, all vertices in the hyperrectangle $B^i(v) \times I_j$ are within distance $r_i$ of $v$. Hence once we control the number of vertices in a dense set of hyperrectangles in all partitions $i\in[R]$,  we can find a net.
We now define a hyperrectangle being ``good'', with respect to a realisation of $\widetilde{\calV}$. Recall that $I(w)$ denotes the interval $I_j$ that contains $w$ in Definition \ref{def:base-2-cover}. 

\begin{definition}\label{def:nets-good-box} 
Consider the setting of Def.~\ref{def:hyperrectangles}, and let a $ \calH(\calR)$ be a hyperrectangle-cover of $Q\times [w_0, f(r_R)]$. Consider a realisation of the weighted vertex set  $\widetilde {\calV} =\big((v, w_v)\big)_{v\in \calV}$.

We recursively define when we call a vertex $\widetilde v\in \widetilde{\calV}$ and a box $B\in \widehat{\calP}(\calR)$ good. Every vertex is $1$-good. For all $i \in [R]$, we say a vertex $\tilde v=(v, w_v)\in \widetilde{\calV}$  is \emph{$i$-good} if the boxes $B^1(v),\ldots,B^{i-1}(v)$ are all good (which we define next). Denote the set of $i$-good (weighted) vertices by $\widetilde{\calG}_i:=\{\widetilde v \in {\widetilde \calV} \ i\text{-good}\}$ and $\calG_i:=\{v: \widetilde v \in \widetilde{\calG}_i\}$.
 We say that a box $B \in \calP_i$ is $i$-\emph{good} or simply good if the following conditions all hold:
    \begin{enumerate}[(B1)]
         \item\label{item:B1} Either $i=1$, or $B$ contains at least $1-\tfrac{2\delta}{R}\cdot (r_i'/r_{i-1}')^d$ many $i-1$-good sub-boxes $B'\in\calP_{i-1}$. 
        \item\label{item:B2} The total number of $i$-good vertices in $B$ satisfies 
   \begin{equation}\label{eq:i-good-total}
       |\calG_i \cap B | \in \Big[\Big(\frac{1}{2} - \frac{2(i-1)\delta}{R}\Big)(r_i')^d,  2(r_i')^d\Big].
   \end{equation}
        \item\label{item:B3} For all $w \in [w_0,f(r_i)]$, the number of $i$-good vertices in $B$ with weight in $I(w)$ satisfies 
   \begin{equation}\label{eq:i-good-weight}
     |\widetilde{\calG_i}\cap (B\times I(w))| \in \Big[\frac{1}{8} \Big(1 - \frac{2i\delta}{R}\Big)\mu_i(I(w)),  8\mu_i(I(w))\Big].
   \end{equation}
   \end{enumerate}
Finally, we say that the realisation $\widetilde {\calV}$ is \emph{good} wrt the hyperrectangle-cover $\widetilde {\calP}(\calR)$ if $Q$ is $R$-good. 
\end{definition}
The above definition is \emph{not circular}; the definition of $i$-good vertices depends only on the definition of good boxes in $\calP_{i-1}$, i.e., one level lower, and then the definition of a good box in $\calP_i$ depends only on the number of $i$-good vertices in it (and their weights) and its number of good subboxes in $\calP_{i-1}$. For $i=1$, the (longer) definition of $1$-good vertices is vacuous, so every vertex is indeed $1$-good which we emphasised in the definition.
Further, $\calG_R\subseteq \calG_{R-1}\subseteq \ldots \subseteq  \calG_1=\calV\cap Q$, since each $i$-good vertex is also $(i\!-\!1)$-good for all $i\le R$. See Figure~\ref{fig:i-good} for a graphical depiction of $i$-good vertices and boxes.
\begin{figure}[t]
    \centering
    \includegraphics[trim=0.75cm 0.1cm 0.2cm 0.1cm,clip,width=0.7\textwidth]{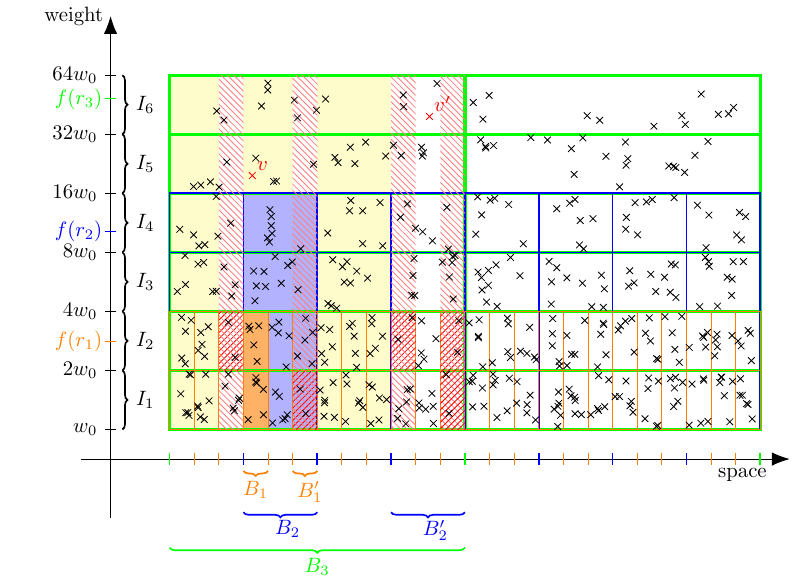}
    \caption[Hyperrectangle-cover and definition of $i$-good boxes.]{Hyperrectangle-cover and definition of $i$-good boxes. In this figure, $d=1$, $R=3$, $r_2'/r_1' = 3$, and $r_3'/r_2' = 4$, and the requirement of \emph{(B1)} for $i > 1$ is ``all but at most \textit{one} sub-box $B' \in \calP_{i-1}$ of $B$ is good''. The hyperrectangle-cover is denoted by coloured-boundary rectangles. The spatial dimension on the $x$ axis is covered by nested intervals, where (blue) boxes in $\calP_2$ contain $3$ level-1 (orange) boxes and (green) boxes in $\calP_3$ contain $4$ level-2 boxes. The weight dimension on the $y$ axis is covered by a base-2-cover $I_1,\ldots,I_6$. Hyperrectangles above $f(r_1)$ (e.g. $B_1\! \times\! I_3$) and above $f(r_2)$ (e.g. $B_2 \!\times\! I_5$), are not included in $\calP_1, \calP_2$, respectively, since they contain too few vertices for concentration. 
    Good boxes are shaded and bad boxes are hatched or get no colour. Box $B_1$ is good because its two hyper-rectangles $B_1\! \times\! I_1$ and $B_1\! \times\! I_2$ (filled orange) contain the right number of vertices, making all vertices in $B_1$ $2$-good, including those with weights above $I_1 \cup I_2$. Box $B_1'$ is bad (light hatching), since it contains too few vertices in $B_1\!\times\!I_1$ (cross-hatching). Box $B_2$ is good, because it only contains one bad sub-box ($B_1'$) in $\calP_1$, and because its four hyperrectangles $B_2\! \times\! I_1,\ldots,B_2 \!\times\! I_4$ (filled blue) all contain the right number of $2$-good vertices in total. Since $B_1$ and $B_2$ are both good, vertex $v$ is $3$-good. Box $B_2'$ is ``doubly'' bad (filled white): it contains two level-$1$ bad sub-boxes, and the hyperrectangle $B_2' \times I_3$ contains too few $1$-good vertices. Thus no vertex in $B_2'$ is $3$-good, including $v'$. Still, $B_3$ is $3$-good: it contains enough $3$-good vertices in total, and only one bad level-$2$ sub-box ($B_2'$).  }
    \label{fig:i-good}
\end{figure}
Before we relate goodness to our overall goal of finding an $(\delta,\calR)$-net, we give some easy algebraic bounds which we need multiple times in the rest of the section.  We defer the proof to the appendix. Recall that $I(w)=I_j$ iff $w\in I_j$ (cf. Def.~\ref{def:base-2-cover}).  
\begin{restatable}{claim}{ClaimPolylogRegime}
\label{claim:nets-mu-bound}
Consider Setting~\ref{set:R-of-section} and Definitions~\ref{def:nets-partition}, \ref{def:hyperrectangles}. Suppose $\widehat{\calP}(\calR)=
\{\calP_1,\ldots,\calP_R\}$ is an $\calR$-partition of $Q$, and for all $i \in [R]$, let $r_i'$ be the side length of boxes in $\calP_i$. Then for all $i \in [R]$ and all $w \in [w_0,f(r_R)]$, we have 
    \begin{align}
        r_i^d\ell(w)w^{-(\tau-1)}/(2d)^{\tau+d+1} &\le \mu_i(I(w))= (r_i')^d \cdot \pr(W\in I(w)) \le 2^\tau r_i^d\ell(w)w^{-(\tau-1)}, \label{eq:mui-bound}\\
        r_i^d \ell(f(r_i)) f(r_i)^{-(\tau-1)} &\ge (2d)^{2\tau+d+8}\log(16R/\delta). \label{eq:z-bound}
    \end{align}
\end{restatable}

We now show that given that the box $Q$ is good with respect to a hyperrectangle cover, we can find an $(\delta,\calR)$-net for $Q$ (see Def.~\ref{def:net}, \ref{def:net-constants}). 

\begin{lemma}\label{lem:good-implies-net}
Consider Setting~\ref{set:R-of-section} and Definitions~\ref{def:nets-partition}, \ref{def:hyperrectangles}, \ref{def:nets-good-box}, i.e., consider a hyper-rectangle cover $ {\calH}(\calR)$ of $Q\times [w_0, f(r_R)]$. Consider a realisation of $\widetilde {\calV}$ for which $Q$ is $R$-good. Then $\widetilde{\calG}_R$, the set of all $R$-good vertices, forms an $(\delta,\calR)$-net for $Q$.
\end{lemma}
\begin{proof}
    Suppose that $Q$ is $R$-good. The side length of $Q$ equals $r_R'$ by \emph{(P3)} in Def.~\ref{def:nets-partition}, and $\delta\in(0,1/16)$ in Setting \ref{set:R-of-section}, hence we may apply  \emph{(B\ref{item:B2})} in Def.~\ref{def:nets-good-box} for $i=R$ to get
        \[
        |\widetilde{\calG}_R| \ge \Big(\frac{1}{2} - \frac{2(R-1)\delta}{R}\Big)\mathrm{Vol}(Q) \ge \Big(\frac{1}{2} - 2\delta\Big)\mathrm{Vol}(Q) > \mathrm{Vol}(Q)/4,
    \] 
    hence the cardinality assumption in Definition~\ref{def:net} is satisfied for $\widetilde{\calG}_R$.
    To show that $\widetilde{\calG}_R$ satisfies Definitions \ref{def:net} with $f$ from Definition \ref{def:net-constants}, we first show that for all $v\in \calG_R$,
      \begin{equation}\label{eq:good-also}
        \calG_i \cap B^i(v)= \calG_R \cap B^i(v).
     \end{equation}
 Given an $i$-good vertex $u \in B^i(v)$ we show that $u$ is also $R$-good, i.e.,\ that $B_j(u)$ is good for all $j \in [R]$. By the definition of $i$-goodness (Def.~\ref{def:nets-good-box}), $B_j(u)$ is good for all $j \le i-1$, and $B_R(u) = Q$ is good by hypothesis. Consider now a  $j \in [i+1, R-1]$. By Def.~\ref{def:nets-partition}, the partition $\calP_i$ is a refinement of the partition $\calP_j$, so $B^j(u) = B^j(v)$. Since $v$ is $R$-good, it follows that $B^j(u)$ is good. So, $B^j(u)$ is good for all $j \in [R]$, so $u$ is $R$-good, showing \eqref{eq:good-also}.

Recall now that $B^i(v)$ is the box in $\calP_i$ containing $v$, and $I(w)\subseteq [w/2, 2w]$ is the interval containing $w$ in the base-$2$-cover of $[w_0, f(r_i)]$.
We now show that 
\begin{align*}
| \widetilde {\calG}_R \cap (B_{r_i}(v) \times [w/2,2w])| \ge |\widetilde {\calG}_R \cap (B^i(v)\times I(w))|=|\widetilde {\calG}_i \cap (B^i(v)\times I(w))|.
\end{align*}
Indeed, $B^{i}(v)\subseteq B_{r_i}(v)$ by Def.~\ref{def:nets-partition} \emph{(P3)} and $I(w)\subseteq [w/2, 2w]$ by Def.~\ref{def:base-2-cover}, and since all $i$-good vertices in $B^i(v)$ are also $R$-good by \eqref{eq:good-also}, the last inequality follows.
Now we  apply, on the rhs $|\widetilde{\calG_i} \cap (B^i(v) \times I(w))|$ above, the lower bound from Def.~\ref{def:nets-good-box} \emph{(B\ref{item:B3})}, i.e., \eqref{eq:i-good-weight}, to obtain
\begin{align*}
    \big| \widetilde{\calG}_R\cap (B^{i}(v)\times [w/2,2w])\big|         &\ {\buildrel \eqref{eq:i-good-weight} \over \ge}\  \frac18\Big(1-\frac{2i\delta}{R}\Big) \mu_i(I(w)) \\&\  {\buildrel \eqref{eq:mui-bound} \over \ge} \frac{1}{8} \Big(1-\frac{2i\delta}{R}\Big)r_i^d\ell(w)w^{-(\tau-1)}/(2d)^{\tau+d+1}.
\end{align*}
Observing that $\delta<1/16$ and $i\le R$ ensures that the prefactor on the rhs of the last row is at least $1/8\cdot 1/2=1/2^4$, establishing \eqref{eq:net-defining-crit} for all $w\le f(r_i)$, as required.
\end{proof}

A lower bound on the probability that any given box in an $\calR$-partition is good, together with Claim~\ref{lem:nets-partition-exists} and Lemma~\ref{lem:good-implies-net}, will yield the proof of Proposition~\ref{lem:nets-exist}. The bound is  by induction on $i$ together with Chernoff bounds.
Recall $I$ and $I(w)$ from Def.~\ref{def:base-2-cover}, applied to the interval $[w_0,f(r_R)]$ for $\calR=\{r_1, \ldots, r_R\}$.
Recall that \eqref{eq:i-good-weight} of Def.~\ref{def:nets-good-box} \emph{(B3)} is required only when $w\in [w_0, f(r_i)]$, and that $j_\star(i)$ in \eqref{eq:jstar} is  the index of $I_j$ that contains $f(r_i)$.
We now describe a gradual revealment of vertex-weights.   

\begin{definition}\label{def:filtration}
Consider Setting~\ref{set:R-of-section} and Definitions~\ref{def:nets-partition}, \ref{def:hyperrectangles}.
    Suppose $\calH(\calR)$ is a hyperrectangle-cover of $Q\times [w_0, f(r_R)]$. Let $i \in [R]$ and $B \in \calP_i$, and let $\widetilde{\calV}$ be a realisation of the weighted vertices in Definition \ref{def:girg}. We define
   \begin{align}
 \calF_i(B):=
 \begin{cases}
    \calV\cap B  & \text{when }i=1 \text{ and } B\in \calP_1,\\
    \calF_1(B)\cup \Big(\widetilde{\calV} \cap (B\times \cup_{j\le j_\star(i-1)}I_j)\Big) & \text{when } i>1\text{ and } B\in \calP_i.\label{eq:fib-2}
 \end{cases}
 \end{align}
\end{definition}
$\calF_1(B)$ reveals the number and positions of vertices in $B$, while $\calF_i(B)$ reveals the precise weights \textit{only} of vertices whose weight falls in the interval $\cup_{j\le j_\star(i-1)}I_j \supseteq [w_0, f(r_{i-1})]$. The index shift in \eqref{eq:fib-2}, and the fact that $\calR$ is $\delta$-well-spaced, means that vertex weights between $w_0 2^{j_\star(i-1)}\approx f(r_{i-1})$ and $f(r_i)$ are not revealed in $\calF_i(B)$. Also, vertex weights in $[1,w_0]$ are not revealed at all; since $w_0$ is large, $\pr(W\le w_0)$ is large and most vertex weights will not be revealed by exposing $\calF_i(B)$.
The filtration generated by $\calF_i(B)$ determines whether or not boxes in $ \cup_{j\le i-1} \calP_{j}$ are good, and whether or not a \emph{vertex} is $i$-good (see Def.~\ref{def:nets-good-box}). So, $\calF_i(B)$ determines whether or not $B\in \calP_i$ satisfies  Def.~\ref{def:nets-good-box} \emph{(B\ref{item:B1}) -(B\ref{item:B2})}, but it leaves \emph{(B\ref{item:B3})} undecided for weights slightly below $f(r_i)$. The next lemma treats \emph{(B\ref{item:B3})}, with $\calF_i(B)$ exposed.
\begin{lemma}\label{lem:nets-good-likely-b3}
Consider Setting~\ref{set:R-of-section} and Definitions~\ref{def:nets-partition}, \ref{def:hyperrectangles}.
    Let $\calH(\calR)$ be a hyperrectangle-cover of the cube $Q$. Let $i \in [R]$ and let $B \in \calP_i$. Let $F$ be a realisation of $\calF_i(B)$ that satisfies Definition \ref{def:nets-good-box} (B\ref{item:B1}) and (B\ref{item:B2}) for $B$. Then independently of other boxes in $\calP_i$, uniformly for all such $F$, 
    \begin{equation}\label{eq:b-good-b3}
        \pr(B\mbox{ is good}\mid \calF_i(B) = F) \ge 1 - \delta/(2R).
    \end{equation}
\end{lemma}
\begin{proof}
    Recall $I_j$ from Def.~\ref{def:base-2-cover}, and let $B\in \calP_i$.   
    Let $a(I_j) := (1-2i\delta/R)\mu_i(I_j)/8, b(I_j) := 8\mu_i(I_j)$ the required lower and upper bounds in \eqref{eq:i-good-weight}. 
    Let $\Xi_j(B)=|\widetilde{\calG_i}\cap (B\times I_j)|$;
    thus (B3) holds for $B\times I_j$ iff $\Xi_j(B) \in [a(I_j),b(I_j)]$.
    Since $B$ satisfies \emph{(B1)--(B2)} on $F$, by a union bound,
    \begin{equation}\label{eq:nets-exist-b3-sum}
        \pr\big(B\mbox{ is good}\mid \calF_i(B) = F\big) \ge 1 - \sum_{I_j: j\le j_\star(i)} \pr\big(\Xi_j(B) \notin [a(I_j),b(I_j)] \mid \calF_i(B) = F \big).
    \end{equation}
    We proceed by bounding each term above. 
    By the definition of $\calF_i(B)$ in \eqref{eq:fib-2}, we already exposed $\Xi_j(B)$ when $i > 1$ and $j\le j_*(i-1)$; the latter is equivalent to $\min(I_j) \le f(r_{i-1})$.
    
    \medskip\noindent\textbf{Case 1: $\boldsymbol{i \!>\! 1}$ and $\boldsymbol{j\!\le\! j_\star(i\!-\!1)}$.} We first show that $\Xi_j \ge a(I_j)$ holds deterministically on $\{\calF_i(B)=F\}$. The goodness of each sub-box $B'\in \calP_{i-1}$ of $B\in \calP_i$ \emph{is revealed} by $\calF_i(B)$. If $B'\in \calP_{i-1}$ is a good box, all vertices in $\calG_{i-1} \cap B'$ are also $i$-good by Definition \ref{def:nets-good-box}. So
    \begin{equation}\label{eq:nets-good-likely-1}
        \Xi_j(B) = |\widetilde{\calG_i}\cap (B\times I_j)|\ge \sum_{B'\in \calP_{i-1}\colon B' \text{ good}} |\widetilde{\calG}_{i-1}\cap (B'\times I_j)|.
    \end{equation}
    Since $j\le j_\star(i-1)$, $\min(I_j) \le f(r_{i-1})$, so we may apply \eqref{eq:i-good-weight} to the good subboxes:
    \begin{equation}\label{eq:nets-good-likely-2}
       |\widetilde{\calG}_{i-1}\cap (B'\times I_j)|\ge  \Big(1-\frac{2(i-1)\delta}{R}\Big)\frac{\mu_{i-1}(I_j)}{8} = \Big(1-\frac{2(i-1)\delta}{R}\Big)\frac{\mu_i(I_j)}{8}\Big(\frac{r_{i-1}'}{r_i'}\Big)^d,
    \end{equation}
    where $\mu_i(I_j)/\mu_{i-1}(I_j)=(r_i'/r_{i-1}')^d$  follows from  \eqref{eq:measure-hyperrectangle}.
    By \emph{(B\ref{item:B1})} holding on $F$, $B$ contains at least $(1-2\delta/R)(r_i'/r_{i-1}')^d$ good sub-boxes in $\calP_{i-1}$. Combining that with~\eqref{eq:nets-good-likely-1}--\eqref{eq:nets-good-likely-2} yields 
    \begin{equation*}
    \begin{aligned}
        \Xi_j(B)&\ge \Big(1-\frac{2\delta}{R}\Big) \cdot \Big(1-\frac{2(i-1)\delta}{R}\Big) \frac{\mu_i(I_j)}{8}
        \ge \Big(1-\frac{2i\delta}{R}\Big) \frac{\mu_i(I_j)}{8} = a(I_j).
        \end{aligned}
    \end{equation*}
    We show $\Xi_j(B)=|\widetilde{\calG_i}\cap (B\times I_j)| \le b(I_j)$ also holds a.s.. $B$ contains $(r_i'/r_{i-1}')^d$ sub-boxes in $\calP_{i-1}$ by Def.~\ref{def:nets-partition} \emph{(P\ref{item:P2})}. If a sub-box is bad, it contains no $i$-good vertices. If it is good, \emph{(B\ref{item:B3})} holds and it contains at most $8\mu_i(I_j)(r_{i-1}'/r_i')^d$ $i$-good vertices with weights in $B\!\times\! I_j$. We obtain $\Xi_j(B) \le 8\mu_i(I_j) = b(I_j)$. So overall we have shown that
    \begin{equation}\label{eq:nets-exist-b3-term-1}
        \mbox{if $i>1$ and $j \le j_\star(i-1)$:}\quad \pr\big(\Xi_j(B) \notin [a(I_j),b(I_j)] \mid \calF_i(B) = F\big) = 0.
    \end{equation}

    \noindent\textbf{Case 2: $\boldsymbol{i > 1}$ and $\boldsymbol{j> j_\star(i-1)}$.} Define the set   
      \begin{equation}\label{eq:calS}
        \calS=\widetilde{\calG_i}\cap \Big(B\times \big([1,w_0) \cup \bigcup\nolimits_{j> j_\star(i-1)} I_j\big)\Big),
    \end{equation}
  the $i$-good vertices in $B$ with weights \emph{not revealed} by $\calF_i(B)$. $\calF_i(B)$ does reveal the positions of vertices in $B$, and all weighted vertices not in $\calS$; thus $\calF_i(B)$ reveals $|\calS|$, and the position of vertices in $\calS$.  
By Def.~\ref{def:girg}, vertex weights are independent of vertex positions and of each other; hence for a vertex $v \in \widetilde{\calV}$, conditioning on $\calF_i(B)$ is equivalent to either exposing its weight (if $w_v\in \cup_{j\le j_\star(i-1)}I_j$) or conditioning on $w_v \notin \cup_{j\le j_\star(i-1)} I_j$ (otherwise). Since $|\calS|$ is determined by $\calF_i(B)$, $\Xi_j(B)$ is binomially distributed on $\calF_i(B)$, when $j> j_\star(i-1)$, with parameters
    \begin{equation}\label{eq:withinS-distribution}
        \Xi_j(B) \mid \calF_i(B)\ {\buildrel d \over =} \ \mathrm{Bin}\Big(|\calS|, \mathbb P(W\in I_j)/\mathbb P(W\notin \cup_{j\le j_\star(i-1)} I_j) \Big).
    \end{equation} 
   We next bound the conditional expectation of $\Xi_j(B)\mid \calF_i(B)$, starting with the upper bound. 
Using the lower bound~\eqref{eq:nets-ell-bound-0} on $w_0$,
   the success probability of the binomial in \eqref{eq:withinS-distribution} is
  \begin{align}\label{eq:successprob}
        \frac{\pr(W\in I_j)}{\mathbb P(W\notin \cup_{j\le j_\star(i-1)} I_j)} \le \frac{\pr(W\in I_j)}{\pr(W\in [1,w_0))} = \frac{\pr(W\in I_j)}{1-\ell(w_0)w_0^{\tau-1}}\le 2\pr(W\in I_j).
    \end{align}
    Since Def.~\ref{def:nets-good-box} \emph{(B\ref{item:B2})} holds for $B$ on $\calF_i(B)=F$, by~\eqref{eq:i-good-total} there are at most $2(r_i')^d$ $i$-good vertices in $B$, so $|\calS| \le 2(r_i')^d$.
    Recalling the definition of $\mu_i(I_j)$ from \eqref{eq:measure-hyperrectangle}, we thus obtain
    \begin{equation}\label{eq:nets-b3-bound-1}
        \E\big[\Xi_j(B) \mid \calF_i(B) = F\big] \le 4(r_i')^d\pr(W\in I_j) = 4\mu_i(I_j) = b(I_j)/2.
    \end{equation}
    We next prove the corresponding  lower bound. By Def.~\ref{def:nets-partition} \emph{(P\ref{item:P2})}, $B$ contains $(r_i'/r_{i-1}')^d$ sub-boxes in $\calP_{i-1}$. Clearly by \eqref{eq:calS}, $|\calS|=|\calG_i\cap B|-|\cup_{j\le j_\star(i-1)} \widetilde{\calG_i}\cap (B\times I_j)|$, both terms revealed by $F$. We can bound $|\calG_i\cap B|$ from below using \eqref{eq:i-good-total} in Def.~\ref{def:nets-good-box} \emph{(B\ref{item:B2})}. Since there are no $i$-good vertices in bad boxes $B'\in \calP_{i-1}$, we can bound $|\cup_{j\le j_\star(i-1)} \widetilde{\calG_i}\cap (B\times I_j)|$ from above by applying \eqref{eq:i-good-weight} to each good sub-box $B'\in \calP_{i-1}$ of $B$, yielding 
    \begin{equation}\label{eq:nets-b3-bound-2}
        |\calS|\ge \Big(\frac{1}{2}-\frac{2(i-1)\delta}{R}\Big)(r_i')^d - \Big(\frac{r_i'}{r_{i-1}'}\Big)^d\cdot\sum_{j \le j_\star(i-1)} 8\mu_{i-1}(I_j).
    \end{equation}
Using that $I_j=[2^{j-1}w_0, w_0)$, Claim~\ref{claim:nets-mu-bound} with $w=2^{j-1}w_0$ yields
    \begin{align}
    \begin{split}\label{eq:geometric}
        \big(r_i'/r_{i-1}'\big)^d\sum_{j \le j_\star(i-1)} 8\mu_{i-1}(I_j)&\le  2^{\tau+3}(r_i')^d\sum_{j \le j_\star(i-1)} \ell(2^{j-1}w_0)(2^{j-1}w_0)^{-(\tau-1)}\\
        &< 2^{\tau+3}(r_i')^d\sum_{j = 0}^\infty \ell(2^jw_0)(2^jw_0)^{-(\tau-1)},
    \end{split}
    \end{align}
  where we switched indices in the last row.  By the lower bound~\eqref{eq:nets-ell-bound-0} on $w_0$ and since $\tau>2$, for all $j \ge 0$ we have $\ell(2^{j+1}w_0)(2^{j+1}w_0)^{-(\tau-1)} < \tfrac{2}{3}\ell(2^jw_0)(2^jw_0)^{-(\tau-1)}$, so the sum on the rhs is bounded above termwise by a geometric series. It follows from~\eqref{eq:nets-b3-bound-2} that
    \[
        |\calS| \ge \Big(\frac{1}{2} - \frac{2(i-1)\delta}{R} - 2^{\tau+5}\ell(w_0)w_0^{-(\tau-1)}\Big)(r_i')^d \ge  (r_i')^d/4,
    \]
    where we used in the last step that $i \le R$ and $\delta < 1/16$, and the lower bound~\eqref{eq:nets-ell-bound-0} on $w_0$. The success probability of the binomial in \eqref{eq:withinS-distribution} is at least $\pr(W\in I_j)$. We defined $a(I_j)= (1-2i\delta/R)\mu_i(I_j)/8$ at the beginning of the proof, so
    \begin{equation}\label{eq:nets-b3-bound-3}
        \E(\Xi_j(B) \mid \calF_i(B) = F) \ge (r_i')^d\pr(W\in I_j)/4 = \mu_i(I_j)/4 \ge 2a(I_j).
    \end{equation}
    Combining \eqref{eq:nets-b3-bound-1} with \eqref{eq:nets-b3-bound-3} yields that  $\mathbb E(\Xi_j(B)\mid \calF_i(B)=F) \in [2a(I_j), b(I_j)/2]$, which allows us to bound $\pr(\Xi_j(B) \notin [a(I_j),b(I_j)])$ with standard Chernoff bounds. By~\eqref{eq:nets-b3-bound-1} and Theorem~\ref{thm:chernoff} applied with $\varepsilon=1/2$, we have shown that
    \begin{equation}\label{eq:nets-exist-b3-term-2a}
    \begin{aligned}
        \mbox{for $i > 1, j > j_\star(i-1)$:}\quad \pr\big(\Xi_j(B) \notin [a(I_j), b(I_j)] \mid \calF(B) = F\big) &\le 2\exp(-a(I_j)/6) \\
        & \le 2\exp(-\mu_i(I_j)/96). 
        \end{aligned}
    \end{equation}
    
    \medskip\noindent
    \textbf{Case 3: $\boldsymbol{i\!=\!1}$.} When $i=1$, we set $j_\star(i-1):=0$ naturally, since in $\calF_1(B)$ we revealed the total number of vertices in $B\in \calP_1$, which are all $1$-good by Def.~\ref{def:nets-good-box}. Conditioned on $\calF_1(B)=F$, \eqref{eq:i-good-total} in \emph{(B\ref{item:B2})} is satisfied and $(r_1')^d/4 \le |\calS| \le 2(r_1')^d$ directly. The rest of our previous calculations from Case 2 with $j>j_\star(0)=0$ all carry through for estimating the lhs of \eqref{eq:i-good-weight} in \emph{(B\ref{item:B3})}. We obtain that \eqref{eq:nets-exist-b3-term-2a} holds also for $i=1$ and all $j\ge j_\star(0)=0$.

       \medskip\noindent\textbf{Combining the cases:} By~\eqref{eq:nets-exist-b3-term-1}, \eqref{eq:nets-exist-b3-term-2a}, and Case 3,  for all $i$ and $j \le j_\star(i)$ the bound in \eqref{eq:nets-exist-b3-term-2a} holds. 
    Combining that with~\eqref{eq:nets-exist-b3-sum}, for all $B\in \calP_i$ and $ \calF_i(B)=F$ satisfying \emph{(B1),(B2)}, 
    \begin{equation*}
    \begin{aligned}
        p_i:=&\ \pr\big(B\in \calP_i \mbox{ is $i$-good}\mid \calF_i(B) = F\big) \ge 1 - 2\sum_{j \le j_\star(i)} \exp(-\mu_i(I_j)/96)\\
        &\ge 1 - 2\sum_{j \le j_\star(i)} \exp\Big({-}(2d)^{-(\tau+d+8)}r_i^d\ell(2^{j-1}w_0 )(2^{j-1}w_0)^{-(\tau-1)}\Big),
        \end{aligned}
    \end{equation*}
where the second line follows from  $\mu_i(I_j)=\mu_i(I(2^{j-1}w_0))$ in Claim~\ref{claim:nets-mu-bound}.
    By the lower bound~\eqref{eq:nets-ell-bound-0} on $w_0$ and the fact that $\tau>2$, for all $j\ge 1$ we have $ \ell(2^{j-1}w_0)(2^{j-1}w_0)^{-(\tau-1)} \ge \tfrac{3}{2}\ell(2^{j}w_0)(2^{j}w_0)^{-(\tau-1)}$. Using the same geometric-sum argument as below~\eqref{eq:geometric} except now from the reversed viewpoint, writing $z:=w_02^{j_\star(i)-1}$, the lower endpoint of $I_{j_\star(i)}$, we have
    \begin{align}\label{eq:nets-likely-pen}
        p_i &\ge 1 - 2\sum_{t=0}^\infty \exp\Big({-}\frac{1}{(2d)^{\tau+d+8}}r_i^d\Big(\frac{3}{2}\Big)^t\ell(z)z^{-(\tau-1)}\Big).
    \end{align}
    Since $z$ is the lower endpoint of $I_{j_\star(i)}=I(f(r_i))$, we have $z \le f(r_i) \le 2z$.  Hence by~\eqref{eq:nets-ell-bound-0}
    \[
        \ell(z)z^{-(\tau-1)} \ge 2^{-\tau}\ell(f(r_i))f(r_i)^{-(\tau-1)}.
    \]
    Using this bound in \eqref{eq:nets-likely-pen} and combining it with \eqref{eq:z-bound} from Claim~\ref{claim:nets-mu-bound}, we obtain that
    \begin{align*}
        p_i \ge 1 - 2\sum_{t=0}^\infty \exp\Big({-}\Big(\frac{3}{2}\Big)^t\log(16R/\delta)\Big)
        \ge 1 - \frac{\delta}{8R} - 2\sum_{t=1}^\infty (\delta/16R)^t \ge 1 - \frac{\delta}{2R},
    \end{align*}
    where we used that $(3/2)^t \ge t$ for all $t \ge 1$, and $\delta < 1/16$. Independence across boxes in the same $\calP_i$ is immediate, since whether $B$ satisfies \emph{(B\ref{item:B3})} conditioned on $\calF_i(B)=F$ satisfying \emph{(B\ref{item:B1}), (B\ref{item:B2})} only depends on vertices in $B$.  
\end{proof}
The next lemma gets rid of the conditioning in Lemma \ref{lem:nets-good-likely-b3} on the filtration:
\begin{lemma}\label{lem:nets-good-likely}
Consider Setting~\ref{set:R-of-section} and Definitions~\ref{def:nets-partition}, \ref{def:hyperrectangles}.
    Let $\calH(\calR)$ be a hyperrectangle-cover of the cube $Q$.
 Let $t \le (r_1/4\sqrt{d})^d$, and let $x_1,\ldots,x_t$ be a (possibly empty) sequence of points in $\R^d$. Then for each $B\in \cup_i\calP_i$
 \begin{equation}\label{eq:nets-good-likely}
  \pr\big(B \mbox{ is good}\mid x_1,\ldots,x_t \in \calV\big) \ge 1 - \delta/R.
 \end{equation}
\end{lemma}
\begin{proof}
    We prove the statement by induction on $i$, the base case being $i=1$. Consider a box $B \in \calP_1$. By Def.~\ref{def:nets-good-box}, \emph{(B\ref{item:B1})} holds vacuously. We next consider \emph{(B\ref{item:B2})}. Every vertex in $B$ is $1$-good, and so $|\calG_1\cap B|=|\calV \cap B|$. In SFP, $\calV$ is deterministic and $|\calG_1\cap B|\in[(r_1')^d/2, 2(r_1')^d]$ holds with certainty by Def.~\ref{def:nets-partition} \emph{(P\ref{item:P1})}. In (I)GIRG, $|\calG_1\cap B|$ is a Poisson variable with mean $(r_1')^d$, so by a standard Chernoff bound (Theorem~\ref{thm:chernoff} with $\varepsilon=1/2$), and using $r_i' \in [r_i/(2\sqrt{d}), r_i/\sqrt{d}]$ in Def.~\ref{def:nets-partition} \emph{(P\ref{item:P1})}, and the bound \eqref{eq:nets-small-r} on $r_1$ in Def.~\ref{def:well-spaced},
  \begin{align*}
         p_1':=\pr\Big( |\calG_1\cap B| \in[t+\tfrac12(r_1')^d, t+\tfrac32(r_1')^d]\mid x_1, \dots, x_t\in \calV\Big) \ge 1 - 2\mathrm{e}^{-(r_1/24d)^d} \ge 1-\delta/(2R),
    \end{align*}
 so the lower bound in \eqref{eq:i-good-total} in Def.~\ref{def:nets-good-box} \emph{(B\ref{item:B2})} is satisfied.  Moreover, since $t \le (r_1/4\sqrt{d})^d$, by Def.~\ref{def:nets-partition} \emph{(P\ref{item:P1})}, $\tfrac32(r_1')^d + t \le 2(r_1')^d$, the upper bound in Def.~\ref{def:nets-good-box} \emph{(B\ref{item:B2})} also holds, so independently for all boxes $B\in \calP_1$, and regardless on where $x_1, \ldots, x_t$ fall,
    \begin{align*}
        \pr\big(B \in \calP_1 \mbox{ satisfies \emph{(B2)}}\mid x_1,\ldots,x_t \in \calV\big) \ge p_1'\ge 1 - \delta/(2R).
    \end{align*}
Lemma~\ref{lem:nets-good-likely-b3} ensures that \emph{(B\ref{item:B3})} holds with probability at least $1-\delta/(2R)$ conditioned on any realisation where \emph{(B\ref{item:B1}),(B\ref{item:B2})} holds. A union bound proves \eqref{eq:nets-good-likely} for $B\in \calP_1$.  
   
    Now we advance the induction. Suppose that \eqref{eq:nets-good-likely} holds for each $B\in\cup_{j\le i-1}\calP_i$, and let $B \in \calP_i$. $B$ contains $(r_i'/r_{i-1}')^d$ sub-boxes in $\calP_{i-1}$ (by Def.~\ref{def:nets-partition}), and by induction these sub-boxes of $B$ are good \emph{independently} (regardless of the positions and weights of $x_1,\ldots,x_t \in\calV$), so the number of bad sub-boxes of $B$ is binomial  with mean at most $(r_i'/r_{i-1}')^d\delta/R$. 
    Let 
    \begin{equation}\label{eq:aib}
        \calA_{i}(B):=\Big\{\big|\{B'\in \calP_{i-1}\colon B'\subseteq B,\ B' \mbox{ not $(i-1)$-good}\}\big| < 2(r_i'/r_{i-1}')^d\delta/R\Big\}.
    \end{equation}
Then, $\calA_{i}(B)$ implies Def.~\ref{def:nets-good-box}\emph{(B\ref{item:B1})}. A Chernoff bound (Theorem~\ref{thm:chernoff} with $\varepsilon=1$) yields that
    \[
        \pr\Big(\neg \calA_{i}(B) \mid x_1,\ldots,x_t\in\calV\Big) \le \exp\Big({-}\frac{\delta}{3R}\cdot \Big(\frac{r_i'}{r_{i-1}'}\Big)^d \Big).
    \]
    By Def.~\ref{def:nets-partition}~\emph{(P\ref{item:P1})}, $(r_i'/r_{i-1}')^d \ge 2^{-d} (r_i/r_{i-1})^d$, so by Def.~\ref{def:well-spaced} \eqref{eq:nets-radii}, regardless of the positions of $x_1,\ldots,x_t\in\calV$, and \emph{independently} across different boxes in $\calP_i$:  
 \begin{align}\label{eq:induction-error-1}
          \pr\Big(\neg \calA_{i}(B)\mid x_1,\ldots,x_t\in\calV\Big) &\le  \exp\Big({-}\frac{\delta}{3R} \cdot 3^d R \cdot  \frac{\log (2R/\delta)}{\delta}\Big)
        \le \frac{\delta}{2R}.
    \end{align}
    We now show that $\calA_{i}(B)$ implies \emph{(B\ref{item:B2})} as well, by inductively applying \eqref{eq:i-good-total} to the good sub-boxes of $B$. Consider an $(i-1)$-good vertex $v$ in a good sub-box $B'\in \calP_{i-1}$ of $B$. Since $v$ is $(i-1)$-good, $B^1(v),\ldots,B^{i-2}(v)$ must all be good; since $B^{i-1}(v)=B'$ is also good, it follows that $v$ is in fact $i$-good. Thus, for all good $B'\in \calP_{i-1}:$ $\calG_{i-1}\cap B'=\calG_{i}\cap B'$ holds. Since $B'$ is $(i\!-\!1)$-good, it satisfies~\eqref{eq:i-good-total}, and we obtain:
 \[
    |\calG_{i}\cap B| \ge \sum_{\text{good } B'\subset B} |\calG_{i-1}\cap B'|\ge \big|\{B' \in \calP_{i-1}\colon B'\subset B,\ B'\mbox{ good}\}\big|\Big(\frac{1}{2} - \frac{2(i-2)\delta}{R}\Big)(r_{i-1}')^d.
 \]
 On $\calA_i(B)$ in \eqref{eq:aib}, there are $(1-2\delta/R)(r_i'/r_{i-1}')^d$ good sub-boxes of $B$, so a.s. on $\calA_i(B)$
    \begin{align*}
        |\calG_{i}\cap B| &\ge  \Big(1-\frac{2\delta}{R}\Big)\Big(\frac{r_i'}{r_{i-1}'}\Big)^d \Big(\frac{1}{2} - \frac{2(i-2)\delta}{R}\Big)(r_{i-1}')^d \ge \Big(\frac{1}{2}-\frac{2(i-1)\delta}{R}\Big)(r_i')^d.
    \end{align*}
     Vertices in bad sub-boxes cannot be $i$-good by Definition \ref{def:nets-good-box}. So \eqref{eq:i-good-total} similarly implies that on $\calA_{i}(B)$ there are at most $2(r_i')^d$ $i$-good vertices in $B$, and thus $\calA_i(B)$ in \eqref{eq:aib} implies both \emph{(B\ref{item:B1})--(B\ref{item:B2})} for $B$; it follows from~\eqref{eq:induction-error-1} that
    \begin{equation}\label{eqn:boxes-B1B2}
        \pr\big(\mbox{\emph{(B\ref{item:B1})} and \emph{(B\ref{item:B2})} hold for }B \mid x_1,\ldots,x_t \in \calV\big) \ge 1 - \delta/(2R).
    \end{equation}
    By Lemma~\ref{lem:nets-good-likely-b3}, $B$ is good (i.e., \emph{(B\ref{item:B3})} also holds) with probability at least $1-\delta/(2R)$ for all $\calF_{i-1}(B)=F$  with \emph{(B\ref{item:B1})--(B\ref{item:B2})} holding for $B$; these events and $x_1,\ldots,x_t \in \calV$ are all determined by $\calF_{i-1}(B)$. So, a union bound on \eqref{eq:b-good-b3} and \eqref{eqn:boxes-B1B2} yields that  independently across boxes in $\calP_i$, and regardless of the vertices $x_1,\ldots,x_t\in \calV$, \eqref{eq:nets-good-likely} holds.
This advances the induction and finishes the proof.
\end{proof}
\begin{proof}[Proof of Proposition \ref{lem:nets-exist}]
    Recall Setting~\ref{set:R-of-section}, and consider an $\calR$-partition $\calP_1,\ldots,\calP_R$ of $Q$, (which exists by Claim~\ref{lem:nets-partition-exists}), and let $\calH(\calR)$ be the associated hyperrectangle cover of $Q\times [w_0, f(r_R)]$.
    By Lemma~\ref{lem:nets-good-likely}, conditioned on $x_1,\ldots,x_t \in \calV$, each box $B \in \calP_1 \cup \dots \cup \calP_R$ is good with probability at least $1 - \delta/R$. Let $\calA$ be the event that all boxes in $\{B^i(x_j)\colon i \in [R], j \in [t]\}$ are good; a union bound over $i=1, \ldots, R$ and $1,\ldots,t$ implies that
    \[
        \pr(\calA\mid x_1,\ldots,x_t\in\calV) \ge 1-t\delta.
    \]
    In particular, if $\calA$ occurs then 
    $B^R(x_1)=Q$ is also good, so by Lemma~\ref{lem:good-implies-net}, the set $\widetilde{\calG}_R=:\calN$ of all $R$-good vertices forms an $(\delta,\calR)$-net of $Q$ with $x_1,\ldots,x_t \in \calN$ as required, showing \eqref{eq:strong-netsexist}. 
    To obtain \eqref{eq:netsexist-1}, note that Lemma~\ref{lem:nets-good-likely} implies that $Q\subset \calP_R$ is good with probability at least $1-\delta/R$, and then again Lemma \ref{lem:good-implies-net} finishes the proof.
\end{proof}
\textbf{Weak nets.} After having established pseudorandom nets in a general and strong form (possibly better for re-use), we now give a relaxed version that suffices here. The following version -- \emph{weak nets} -- only retains three parameters: the box $Q$; an error parameter $\eps >0$; and a lower bound $w_1$ on the weights considered for the net. 
Recall the definition of an $(\calR,w_0, f)$-net from Def.~\ref{def:net} and the choice of $w_0$ and of the function $f(r)=f_{R, \delta}(r)$ from Def.~\ref{def:net-constants}, yielding an $(\delta,\calR)$-net.

\begin{definition}[Weak net]\label{def:weak-net}
    Let $Q \subseteq \R^d$ be a box of side length $\xi$, $\eps >0$, and $w_1 \geq w_0$. A \emph{weak $(\eps,w_1)$-net} for $Q$ is a set $\calN \subseteq \widetilde{\calV} \cap Q\times[1,\infty)$ of size at least $\mathrm{Vol}(Q)/4$ such that for all $\widetilde v \in \calN$, \emph{all} $r \in [(\log\log\xi\sqrt{d})^{4/\eps}, \xi\sqrt{d}]$ and all $w \in [w_1, r^{d/(\tau-1)-\eps}]$:
    \begin{equation}\label{eq:net-defining-crit-eps}
        |\calN\cap (B_r(v)\times [w/2, 2w])| \ge r^{d(1-\eps)}\cdot \ell(w)w^{-(\tau-1)}.
    \end{equation}
\end{definition}
In a weak $(\eps,w_1)$-net we allow an error of order $r^{-d\eps}$ on the rhs of \eqref{eq:net-defining-crit-eps} compared to the constant factor in a $(\delta,\calR)$-net in \eqref{eq:net-defining-crit}, and here the smallest radius $r$ also grows with $\xi$. 

\begin{lemma}\label{lem:weak-nets-exist}
    Consider IGIRG or SFP with $\tau\!>\!2$ in Def.~\ref{def:girg}. Then for all $\eps\in(0,1/2)$, and for all $\xi$ sufficiently large relative to $\eps$,  and $t \le \min\{1/\eps,\log\log\xi\}$ the following holds. Consider a cube $Q \subseteq \R^d$ of side length $\xi$, and let $x_1,\ldots,x_t \in Q$, and $w_0$ from Def.~\ref{def:net-constants}, then
    \begin{align*}
        \pr(\text{$Q$ contains a weak $(\eps,w_0)$-net $\calN$, and }  x_1,\ldots,x_t \in \calN \mid x_1,\ldots,x_t \in \calV ) \ge 1-t\eps. 
    \end{align*}
\end{lemma}
The condition $t\le 1/\eps$ is there to avoid a vacuous statement, and below we set $t=2$. Note that the condition~\eqref{eq:net-defining-crit-eps} never counts vertices of weight less than $w_1/2$. So, we can decide whether a weak $(\eps,w_1)$-net $\calN$ exists by uncovering only the set of vertices of weight at least $w_1/2$ (beyond $x_1,\ldots,x_t \in \calN$). For IGIRG, this set is independent of the set of vertices of weight smaller than $w_1/2$. This is the main reason for introducing the parameter $w_1$.
\begin{proof}[Proof of Lemma~\ref{lem:weak-nets-exist}]
    Let $\eps\in(0,1/2)$, set $\eta := 1 - \eps/2$, and define $\calR = \{r_1,\ldots,r_R\}$ as
    \begin{align}
       R &:= 2 + \lfloor(\log\log\xi\sqrt{d} - \log^{*4}\xi\sqrt{d}-\log\tfrac{4}{\eps})/\log(1/\eta)\rfloor,     \label{eq:weak-net-choices-1}\\
       r_i &:= (\xi\sqrt{d})^{\eta^{R-i}}, \quad\mbox{for} \quad i\in[R].    \label{eq:weak-net-choices-2}
    \end{align}
     We next prove that $\calR$ is $\eps$-well-spaced (Def. \ref{def:well-spaced}). Let $1-a\in[0,1)$ be the fractional part of the expression inside the $\lfloor \cdot \rfloor$ in \eqref{eq:weak-net-choices-1}. Then using that $\lfloor x\rfloor=x-1+a$,
    \begin{equation}\label{eq:weak-nets-exist-r1}
        r_1 = (\xi\sqrt{d})^{\eta^{R-1}} = (\xi\sqrt{d})^{\eta^{a}4\log^{*3}\xi\sqrt{d}/(\eps\log\xi\sqrt{d})} = (\log\log\xi\sqrt{d})^{(1-\eps/2)^{a}4/\eps}.
    \end{equation}
    Since $\xi$ is large relative to $\eps$, Def.~\ref{def:well-spaced} \eqref{eq:nets-small-r} holds for $\eps=\delta$. For all $i \in [R]$, since $\eta=1-\eps/2$:
    \[
        r_i/r_{i-1} = (\xi\sqrt{d})^{\eta^{R-i}(1-\eta)} \ge (\xi\sqrt{d})^{\eta^{R-1}\eps/2} = r_1^{\eps/2} = (\log\log\xi\sqrt{d})^{2(1-\eps/2)^{a}}.
    \]
    Since $a\le 1$, $2(1-\eps/2)^{a}>1$, for all $\eps<1/2$, so Def.~\ref{def:well-spaced} \eqref{eq:nets-radii} holds even in $d=1$, and $\calR$ is $\eps$-well-spaced as claimed. Moreover, since $t \le \log\log\xi$ by hypothesis, by~\eqref{eq:weak-nets-exist-r1} we have $t \le (r_1/4\sqrt{d})^d$. By Proposition~\ref{lem:nets-exist}, with probability at least $1-t\eps$, conditioned on $x_1,\ldots,x_t\in \calV$, $Q$ contains a (strong) $(\eps,\calR)$-net $\calN$, using Def.~\ref{def:net-constants}.
    We now prove that a strong net is also a weak net. Let $r \in [(\log\log\xi\sqrt{d})^{4/\eps}, \xi\sqrt{d}]$ and $w \in [w_0,r^{d/(\tau-1)-\eps}]$ as in Def.~\ref{def:weak-net} (this interval is non-empty since $\tau\in(2,3), \eps<1/2$), and consider a vertex $v \in \calN$. By~\eqref{eq:weak-nets-exist-r1} since $a\ge 0$, we have $r_1 \le r \le r_R$. Let $r_j$ be such that $r\in [r_j,r_{j+1})$; thus $r^{\eta} \le r_j \le r$. Since $\xi$ is large relative to $\eps$, we have $w \le f_{R,\eps}(r_j)$ (using Def.~\ref{def:net-constants}) for $f$); thus by the definition of a strong $(\eps, \calR)$-net in Def.~\ref{def:net}, $|\calN \cap (B_{r_j}(v)\times[w/2, 2w])|\ge \ell(w)w^{-(\tau-1)}r_j^d/(2d)^{d+\tau+5}$. Since $\xi$ is large relative to $\eps$ and $r \ge r_j \ge r^{\eta}$,  the required inequality in Def.~\ref{def:weak-net} follows since
    \[
       |\calN \cap (B_{r}(v)\times[w/2, 2w])| \ge  \ell(w)w^{-(\tau-1)}r^{d\eta}/(2d)^{d+\tau+5} \ge \ell(w)w^{-(\tau-1)}r^{d(1-\eps)}.
\qedhere    \]
\end{proof}
\vskip-2em

\section{Multi-round exposure with dependent edge-costs}\label{sec:exposure}
Now with the pseudorandom nets at hand, we may switch to the quenched setting and reveal the realisation of the vertex set $\widetilde\calV=(V,w_V)$. We shall now reveal edges adaptively to construct a fast-transmission path between $0,x$, according to the `budget travel plan' in Section \ref{sec:intro}, see Fig.~\ref{fig:intuition_hierarchy}. 
In particular, we need to find low-cost edges in spatial regions which depend on the previous low-cost edges we have found. When studying graph distances in Biskup~\cite{biskup2004scaling} this is not a major obstacle, but with the presence of edge-costs we run into conditioning issues. To overcome these, we develop a multiple-round exposure -- essentially an elaborate edge-sprinkling method on the quenched vertex set -- where in each round we reveal multiple edges, which we explain now heuristically. 
We construct \emph{$1/R$-percolated SFP/IGIRGs} $G_1, \ldots, G_R$ on the fixed vertex set $(V,w_V)$, where in each $G_i$ the probability that an edge $uv$ exists is given by $\pr(uv\in\calE(G_i)) = h(u-v,w_u,w_v)/R$, where $h$ is the connection probability from~\eqref{eq:connection_prob}. To each edge $uv$ in $G_i$ we also assign $L_{uv}^{(i)}$, i.i.d.\ from the r.v.~$L$ determining the costs in \eqref{eq:cost}. For each pair of vertices $u,v\in (V, w_V)$ we also draw a r.v.~$Z_{uv}$ uniformly in $[R]$. Then we construct a single weighted graph $G$ on $(V, w_V)$ from the collection $(G_1, G_2, \ldots, G_R)$ by setting $\calE(G):=\{uv : uv \textnormal{ is present in } G_{Z_{uv}}\}$ with transmission costs $\cost{uv} = L_{uv}^{(Z_{uv})}\cdot(w_uw_v)^{\mu}$. We show that this is an alternative construction of graphs in Def.~\ref{def:girg}. Given $(V, w_V)$, we then use edges in $G_i\cap G$ in the $i$th round of edge-revealment, independent of edges in earlier rounds. 

In a classical random graph setting, multiple-round exposure means to couple the base graph model $\calG$ to a suite of sparser but independent random subgraphs $\calG_1,\ldots,\calG_R \subseteq \calG$, and taking the $i$th edge of a path from the $i$th ``round of exposure'' $\calG_i$. In the edge-weighted setting we design a (slightly more restrictive) construction incorporating \emph{independent} edge-cost variables on $\calG_1,\ldots,\calG_R \subseteq \calG$ that we describe in Prop.~\ref{prop:multi-round-exposure} after some preliminary definitions. Recall that for a set $V$, we write $V^{(2)} := \{\{x,y\}\colon x,y \in V, x\ne y\}$ for the set of \emph{possible} edges of a graph with vertex set $V$. For future re-usability, we formulate Prop.~\ref{prop:multi-round-exposure} in a  general class of random graph models, as set out below.

\begin{definition}\label{def:CIRG}
    A \emph{conditionally-independent edge-weighted vertex-marked random graph model (CIRG model)} $\calG$ consists of distributions for a random vertex set $\calV$ which is a.s.\ countable, a random parameter set $\calW$, a random edge set $\calE \subseteq \calV^{(2)}$, and random edge costs $\cost{xy}$ for each possible edge $\{x,y\} \in \calV^{(2)}$. Conditioned on any given realisation of $(\calV,\calW)=(V, w_V)$, all costs $\cost{xy}$ and all events $xy \in \calE$ are independent across $\{x,y\} \in \calV^{(2)}$. If $G \sim \calG$, we say that $G$ is a realisation of a \emph{CIRG}, or simply a CIRG. We write $\{\calG\mid V, w_V\}$ for the distribution of $G$ conditioned on $(\calV, \calW)=(V, w_V)$.
\end{definition}

First passage percolation (1-FPP) on GIRG and SFP in Def.~\ref{def:girg}--\ref{def:1-FPP} are both CIRG models, with $\calV$ either a PPP on $\R^d$ or $\Z^d$, and $\calW$ the vertex-weights, and $\calC(xy)$ in \eqref{eq:cost} the edge-weights. For other graph models $\calW$ could contain extra randomness. We now define the analogue of a single ``round of exposure'' for a CIRG model:
\begin{definition}[$\theta$-percolated CIRG]\label{def:percolated}
    Let $G \sim \calG$ be a CIRG from Definition~\ref{def:CIRG}. Then for all $\theta\in(0,1)$ the \emph{$\theta$-percolation} of $G$ is the subgraph $G^{\theta}$ of $G$ which includes each  $e\in \calE(G)$ independently with probability $\theta$, and we write $\calG^{\theta}$ for its law. We call $\calG^{\theta}$ the \emph{$\theta$-percolation} of $\calG$, and $\theta$ the \emph{percolation probability}.
\end{definition}

\begin{remark}[$\theta$-percolated CIRGs are CIRGs]\label{remark:cirg} An alternative construction of $\calG^\theta$ is to sample the realisation of $(V,w_V)\sim (\calV, \calW)$ first, and then sample each edge $e$ with probability $\theta\mathbb P(e\in \calE(G)\mid V, w_V)$.
  So the CIRG model class is closed under $\theta$-percolation.
  \end{remark}
We now set out a specific coupling between a base CIRG model and percolated CIRGs, that will serve as graphs forming the rounds of exposure. Recall that $[r]:=\{1,2, \ldots, r\}$.

\begin{definition}[Exposure setting of $G$]\label{def:exposure-setting}
    Let $\calG$ be a CIRG model from Def.~\ref{def:CIRG}. Fix $r\!\in\!\N$ and  $\theta_1,\ldots,\theta_r \in [0,1]$ satisfying $\sum_{i\in[r]}\theta_i\!=\! 1$. We define the \emph{exposure setting} of $\calG$ with \emph{percolation probabilities} $\theta_1,\ldots,\theta_r$ as follows.
    Reveal the realisation of $(\calV, \calW)=:(V, w_V)$, and let $(Z_{uv})_{uv\in V^{(2)}}$ be iid random variables with $\pr(Z_{uv}=i)=\theta_i$ for all $i \in [r]$. Take $G_1^\star, \ldots, G_{r}^\star$ to be conditionally iid realisations of CIRGs, with shared distributions $G_i^\star \sim \{\calG \mid V, w_V\}$, and respective edge costs $\calC_i(e)$ for $e\in \calE(G_i^\star)$, independent across $i\le r$.  Let $G_i^{\theta_i}$ be the subgraph of $G_i^\star$ with edge set $\calE(G_i^{\theta_i}):=\{e \in \calE(G_i^\star)\colon Z_{e}=i\}$ and edge costs $\{\calC_i(e): e\in \calE(G_i^{\theta_i})\} $.
\end{definition}
The following claim reconstructs $G$ from the percolated versions.
\begin{claim}[Realisation of a CIRG in the exposure setting]\label{claim:exposure-coupling}
    Let $\calG$ be a CIRG model from Def.~\ref{def:CIRG}. Let $\theta_1,\ldots,\theta_r$ be the percolation probabilities, and consider $(G_i^{\theta_i})_{i\le r}$ in Definition \ref{def:exposure-setting}. Then marginally, each $G_i^{\theta_i}$ is a $\theta_i$-percolated CIRG. Define now $G$ as the graph with vertex set and parameters $(V,w_V)$, and with edge set $\calE(G):= \cup_{i\in[r]}\calE(G_i^{\theta_i})$, and with edge costs $\{\calC(e):=\calC_{Z_e}(e): e\in \calE(G)\}$. Then $G \sim \{\calG\mid V, w_V\}$.
\end{claim}
\begin{proof}
    That $G_i^{\theta_i}$ is marginally a $\theta_i$-percolated CIRG, i.e., that it has law $\calG^{\theta_i}$, is immediate since $\pr(e\in \calE(G_i^{\theta_i})\mid V, w_V)=\pr(Z_e=i)\cdot \pr(e\in G_i^\star\mid V, w_V)$, and now one can integrate over the realisations $(V, w_V)$. To see that $G$ has distribution $\{\calG\mid V, w_V\}$ we argue as follows. Since $Z_{uv}$ takes a single value in $[r]$ each possible edge $e=uv$ appears in at most one of $G_1^{\theta_1}, \ldots, G_r^{\theta_r}$. Hence the union $\cup_{i\in[r]}\calE(G_i^{\theta_i})=\calE(G)$ is \emph{disjoint}, and using that $G_1^\star, \ldots, G_r^\star$ all have law $\{\calG \mid V, w_V\}$,
    \[ \begin{aligned}\pr\big(uv \in \calE(G) \mid V, w_V\big) &= \sum_{i\in[r]} \pr(Z_{uv}=i)\pr\big(uv \in \calE(G_i^\star) \mid V, w_V\big)\\
    &= \sum_{i\in[r]} \theta_i\pr\big(uv \in  \calE(G_1^\star)\mid V, w_V \big) = \pr\big(uv \in \calE(G_1^\star)\mid V, w_V\big), \end{aligned} \]
and $G_1^\star\sim \{\calG \mid V, w_V\}$.  Further, edges are present in $G$ independently (conditional on $(V,w_V)$) since the variables $Z_e$ and $\calE(G_i^\star)$ are conditionally independent. 
\end{proof}
For two collections of random variables we write $\calX\in \sigma(\calX')$, if all elements in $\calX$ are measurable with respect to the $\sigma$-algebra generated by the variables in $\calF'$, i.e., they are a deterministic function of elements in $\calX'$. In the next definition we formalise multi-round exposure with edge-cost constraints, in the setting of CIRGs with $(\calV, \calW)=(V, w_V)$ already exposed, which guarantees that edge presence and edge costs are independent by Def.~\ref{def:CIRG}.

\begin{definition}[Iterative cost construction]\label{def:iter-construct}
    Let $G_1,\ldots,G_r$ be edge-weighted CIRGs on a common realisation of the vertex set and parameters $(\calV, \calW)=(V,w_V)$. Write $\calC_i$ for the cost function of $G_i$. An \emph{iterative cost construction} $\mathrm{Iter}$ is a sequence of tuples $(G_1,\calE_1,\calF_1,\calU_1),\ldots,$ $(G_r,\calE_r,\calF_r,\calU_r)$ satisfying the following properties:
    \begin{enumerate}[(i)]
        \item\label{item:iter1} $\calF_i$ is a finite collection of tuples (``allowed sets'') of potential edges without repetition (i.e.\ of $\binom{V}{2}$), in $\sigma(V, w_V, \calE_1,\ldots,\calE_{i-1})$.  In the $i$th round of revealment, we will reveal a tuple of edges $\calE_i \in \calF_i$ chosen among $\calF_i$:
        \item\label{item:iter2} Define the \emph{round-$i$ marginal cost} of an edge by
        \begin{align}\label{eq:marginal-cost}
            \mcost_i(e) = \begin{cases}
                0 & \mbox{ if $e$ appears in some tuple $\calE_1,\ldots,\calE_{i-1}$,}\\
                \calC_i(e) & \mbox{ otherwise.}
            \end{cases}
        \end{align}
     \item\label{item:iter3}   Either $\calE_i=\mathtt{None}$ or $\calE_i$ is a tuple of edges of $G_i$ together with their round-$i$ marginal cost of the form $\calE_i=((e_1,\mcost_i(e_1)), \ldots  (e_t,\mcost_i(e_t)))$ for some $t$.
        \item\label{item:iter4} Each $\calU_i = \calU_i(V,w_V,\calE_1,\ldots,\calE_i)$ (``cost constraint'') is a (list of) event(s) or constraint(s), measurable wrt $\sigma(V, w_V, \calE_1,\ldots,\calE_i)$. 
        \item\label{item:iter5} $\calE_i$ is chosen in the following way. We first fix a deterministic ordering of all tuples on $V^{\scriptscriptstyle{(2)}}$ that may appear in $\calF_i$ (the \emph{canonical ordering}). Then we define $\calE_i$ to be the first element $(e_1,\ldots,e_t)$ in this ordering such that $e_j \in \calE(G_i) \cup \calE_1 \cup \dots \cup \calE_{i-1}$ for all $j$ and $\calU_i(V,w_V,\calE_1,\ldots,\calE_i)$ occurs. If no such set of edges exists, we set $\calE_k = \textnormal{\texttt{None}}$ for all $k \ge i$.
    \end{enumerate}
    We call $r$ the number of \emph{rounds}, and $G_i$ the \emph{round-$i$ graph}. The construction is \emph{successful} if $\calE_i \ne \textnormal{\texttt{None}}$ for all $i \in [r]$. We define the event of seeing a given outcome up until round $i$ as:
    \begin{equation}\label{eq:ait-event}
        \calA_{\mathrm{Iter}}(V,w_V,E_1,\ldots,E_i) := \{(\calV,\calW) = (V,w_V)\} \cap \bigcap_{j\in [i]} \{\calE_j = E_j\};
    \end{equation}
    we omit $V$ and $w_V$ from $\calA_{\mathrm{Iter}}$ when they are clear from context.
\end{definition}
In further sections, $\calU_i$ represents the upper bounds we require on the round-$i$ marginal costs of the edges that we reveal in the $i$'th round, and $\calU_i$ may depend on $(\mathrm{mcost}_j(\calE_j))_{j\le i}$. The marginal cost \eqref{eq:marginal-cost} of an edge drops to zero when an edge has been already chosen in previous rounds, so bounding total marginal cost rather than total cost corresponds to counting the cost of each chosen edge exactly once; this is fine in our application, since if an edge appears twice in a walk then we can pass to a cheaper sub-path in which it appears at most once.

We illustrate Definition~\ref{def:iter-construct} on toy example. Take $G_1$ and $G_2$ to be two conditionally independent realisations of SFP in a box $B$ with the same $(V, w_V)$, equipped with edge-costs as in \eqref{eq:cost}. Take $\calF_1$ to be the set of all possible pairs of vertex-disjoint triangles in $B$, i.e., $\calF_1$ is the collection of all $6$ possible edges of the form $\{v_1v_2, v_2v_3, v_3v_1, v_4v_5, v_5v_6, v_6v_4\}$ with $v_1,\ldots,v_6$ all distinct. Take $\calU_1$ to be the constraint that each of the two triangles individually has total cost at most $1$. The canonical ordering is arbitrary and typically not important. The first round reveals edge presence and costs in $G_1$: it reveals tuples in $\calF_1$ sequentially in the canonical ordering until it finds $\calE_1$, two (random) triangles $T_1$ and $T_2$ in $G_1$, or else outputs $\mathtt{None}$. Take $\calF_2$ to be the collection of all possible paths of any length from $T_1$ to $T_2$, and take $\calU_2$ to be the constraint that $\calE_2\in \calF_2$ has total marginal cost at most $1$. The second round reveals edge presence and costs in $G_2$: it reveals tuples in $\calF_2$ sequentially. If successful, we have found two triangles of cost at most $1$ joined by a path $\pi_{T_1T_2}$ of marginal cost at most $1$. The path $\pi_{T_1T_2}$ may reuse an edge from $T_1$ or $T_2$, nevertheless, the total cost of the construction is still at most $3$ as each edge's cost only counts once in the construction. 

 The next proposition replaces the FKG-inequality for iterative cost constructions. Below, edges in the $\theta_i$-percolated CIRGs $(H_i)_{i\le r}$ are conditionally \emph{independent}, in contrast to Def.~\ref{def:exposure-setting} where edges in $(G_i^{\theta_i})_{i\le r}$ are dependent through $Z_e$. The proof is via coupling: when two graph-collections use the \emph{same vertex set}, one can carry out the \emph{same} iterative cost construction on them, using the same $(\calF_i, \calU_i)$ as long as tuples chosen in previous rounds agree.   
\begin{proposition}[Multi-round exposure]\label{prop:multi-round-exposure}
    Let $G\sim\{\calG\mid V, w_V\}$ be a CIRG model with a fixed realisation $(V,w_V)$ of $(\calV,\calW)$. Let $\theta_1,\ldots,\theta_r \in [0,1]$ with $\sum_{i\in[r]}\theta_i = 1$, and let $H_1,\ldots,H_r$ be independent with distributions $\{\calG^{\theta_i}\mid V, w_V\}$ for all $i\le r$ as in Definition~\ref{def:percolated}. Consider an iterative construction $\mathrm{Iter}_G = (G,\calE_1^G,\calF_1,\calU_1),\ldots,(G,\calE_r^G,\calF_r,\calU_r)$ with all rounds using $G$, and let $\mathrm{Iter}_H = (H_1,\calE_1^H,\calF_1,\calU_1),\ldots,(H_r,\calE_r^H,\calF_r,\calU_r)$ be the iterative cost construction on $H_1,\ldots,H_r$ using the same $(\calF_i,\calU_i)_{i\le [r]}$. Then, taking the minimum over all possible realisations $(E_1,\ldots,E_{r-1})$ of $(\calE_1^H, \ldots, \calE_r^H)$ not containing $\mathtt{None}$,
 \begin{align}
    \begin{split}\label{eq:multi-round-exp-goal}
        &\pr\big(\mathrm{Iter}_G \textnormal{ succeeds} \mid  V,w_V\big) \\
        &\qquad\qquad \ge 
        \min_{E_1,\ldots,E_{r-1} \nsupseteq \mathtt{None}} \prod_{i\in[r]}\pr\big(\calE_i^H \ne \mathtt{None} \mid \calA_{\mathrm{Iter}_H}(V,w_V,E_1,\ldots,E_{i-1})\big).
    \end{split}
    \end{align}
\end{proposition}
\begin{proof}[Proof of Proposition \ref{prop:multi-round-exposure}]
    Let $G \sim \{\calG\mid V, w_V\}$. By repeated conditioning, and taking the minimum over all possible successful realisations, 
    \begin{align}
    \begin{split}\label{eq:multi-round-exp-0}
        &\pr\big(\mathrm{Iter}_G \textnormal{ succeeds} \mid V,w_V\big)
        = \prod_{i \in [r]} \pr\big(\calE_i^G \ne \mathtt{None} \mid V,w_V, \calE_1,\ldots,\calE_{i-1} \nsupseteq \mathtt{None}  \big)\\
        &\qquad\qquad \ge \min_{E_1,\ldots,E_{r-1} \nsupseteq \mathtt{None}} \prod_{i\in[r]}\pr\big(\calE_i^G \ne \mathtt{None} \mid \calA_{\mathrm{Iter}_G}( V, w_V, E_1,\ldots,E_{i-1})\big).
    \end{split}
    \end{align}
      Claim~\ref{claim:exposure-coupling} ensures that using the edge-disjoint graphs $G_i^{\theta_i}, i\le r$, we can realise $G$ as
       \begin{equation}\label{eq:G-construction}
      G = \cup_{i\le r} G_i^{\theta_i},
      \end{equation}
      where marginally $G_i^{\theta_i}\,{\buildrel d \over =}\, H_i$, but $G_i^{\theta_i}$ are dependent through the variables $(Z_{uv})_{uv\in V^{(2)}}$.
     Let $\mathrm{Iter}_G^{-} = (G_1^{\theta_1},\calE_1^-,\calF_1,\calU_1),\ldots,(G_r^{\theta_r},\calE_r^-,\calF_r,\calU_r)$ be the iterative cost construction on $(G_1^{\theta_1},\ldots,G_r^{\theta_r})$ that uses the same possible tuples and constraints $(\calF_i, \calU_i)$ as $\mathrm{Iter}_G$, but differs in that $\calE_i^{-}$ is chosen from $G_i^{\theta_i}$ in  round $i$ (while $\calE_i^G$ is chosen from $G$). By~\eqref{eq:multi-round-exp-0} and \eqref{eq:G-construction},
    \begin{align}\label{eq:multi-round-exp-1}
        \pr\big(\mathrm{Iter}_G \textnormal{ succeeds} \mid (V,w_V)\big) \ge \min_{E_1,\ldots,E_{r-1} \nsupseteq \mathtt{None}} \prod_{i\in[r]}\pr\big(\calE_i^- \ne \mathtt{None} \mid \calA_{\mathrm{Iter}_G^-}(E_1,\ldots,E_{i-1})\big).
    \end{align}  
 In order to prove~\eqref{eq:multi-round-exp-goal} from~\eqref{eq:multi-round-exp-1}, for each $i \in [r]$ and each possible outcome $E_1,\ldots,E_{i-1} \nsupseteq \mathtt{None}$, for each $i\le r$ we will couple $G_i^{\theta_i}$ conditioned on $\calA_{\mathrm{Iter}_G^-}(E_1,\ldots,E_{i-1})$ to $H_i$ conditioned on $\calA_{\mathrm{Iter}_H}(E_1,\ldots,E_{i-1})$; satisfying for all $e \in V^{(2)}$:
 \begin{align}
   \{e \in \calE(H_i)\}\cap\{ e\in \calE(G_i^{\theta_i})\}& \Longrightarrow  \calC_{H_i}(e)= \calC_{i}(e), \label{eq:multi-round-exp-coupling-1a}\\
         \{ e \in \calE(H_i) \} &\subseteq \{e \in \calE(G_i^{\theta_i}) \cup \calE_1^- \cup \dots \cup \calE_{i-1}^-\}. \label{eq:multi-round-exp-coupling-1b}
         \end{align}
  In words, if an edge appears in both graphs $H_i$ and $G_i^{\theta_i}$ then its cost agrees, and if an edge is in $H_i$ then either $e$ has been chosen already in previous rounds or $e$ is also in $G_i^{\theta_i}$. 
   Under the conditioning, $\calE_j^H = \calE_j^- =: E_j \ne \texttt{None}$ for all $j \le i-1$. Suppose $\calE_i^H =: E_i \ne \texttt{None}$. Then by~\eqref{eq:multi-round-exp-coupling-1b} every set $E \in E_i \subseteq \calE(H_i) \cup E_1 \cup \dots \cup E_{i-1}$ is also contained in $\calE(G_i^{\theta_i}) \cup E_1 \cup \dots \cup E_{i-1}$, and by \eqref{eq:marginal-cost} their round-$i$ marginal costs in $\mathrm{Iter}_G^-$ are equal to those in $\mathrm{Iter}_H$. Hence $E_i$ provides a valid choice for $\calE_i^-$, and $\calE_i^- \ne \texttt{None}$ holds also. Thus
    \begin{equation}\label{eq:multi-round-exp-2}
        \pr\big(\calE_i^- \ne \mathtt{None} \mid \calA_{\mathrm{Iter}_G^-}(E_1,\ldots,E_{i-1})\big) \ge \pr\big(\calE_i^H \ne \mathtt{None} \mid \calA_{\mathrm{Iter}_H}(E_1,\ldots,E_{i-1})\big),
    \end{equation}
    and the result follows from~\eqref{eq:multi-round-exp-1}.
    Now we provide the coupling achieving \eqref{eq:multi-round-exp-coupling-1a}--\eqref{eq:multi-round-exp-coupling-1b}. 
     Def.~\ref{def:exposure-setting} uses the independent graphs $G_i^\star\sim\{\calG \mid V, w_V\}$, and obtains $G_i^{\theta_i}$ as a dependent thinning of $G_i^\star$ using $(Z_{uv})_{uv \in V^{(2)}}$ (independent across different $uv$). For each $uv\in V^{(2)}$ sample iid uniform $U_{uv}^{(i)}\sim U[0,1]$ and realise the presence of $uv$ in $H_i$ (resp.\ $G_i^{\theta_i}$) as 
     \begin{equation}\label{eq:partial-coupling}
     \ind{uv\in H_i}=\ind{uv\in G_i^\star}\ind{U_{uv}^{(i)}\le \theta_i}, \qquad \ind{uv\in G_i^{\theta_i}}=\ind{uv\in G_i^\star}\ind{Z_{uv}=i}.
     \end{equation}
 Then $H_1,\ldots,H_r$ are \emph{independent} $\theta_i$-percolations of $G_1^\star,\ldots,G_r^\star$ respectively, and $H_1,\ldots,H_r$ are themselves independent as required in the statement, since $G_1^\star,\ldots,G_r^\star$ are independent conditionally on $(V, w_V)$. Note that \eqref{eq:partial-coupling} is only a \emph{partial coupling}, since we can still specify the joint distribution of $(U_{uv}^{(i)})_{i\le r}, Z_{uv}$.
 By this partial coupling, $H_i$ and $G_i^{\theta_i}$ are both subgraphs of $G_i^\star$, and hence if an edge $e$ is in both subgraphs, then $\calC_{H_i}(e)=\calC_i(e)$ (the cost of $e$ in $G_i^\star$), so~\eqref{eq:multi-round-exp-coupling-1a} holds. We now make this into a full coupling to satisfy \eqref{eq:multi-round-exp-coupling-1b}. 
 Fix $i$ and $E_1,\ldots,E_{i-1}$. For \eqref{eq:multi-round-exp-coupling-1b}, we first claim the following distributional identities hold:
  \begin{align}
            \calE(H_i)\mid \calA_{\mathrm{Iter}_H}(V, w_V, E_1,\ldots,E_{i-1})\ &{\buildrel d \over =}\ \calE(H_i)\mid (V, w_V),\label{eq:dist-id1}\\
            \calE(G_i^\star)\mid \calA_{\mathrm{Iter}_G^-}(V, w_V, E_1, \ldots, E_{i-1})\  &{\buildrel d \over = }\  \calE(G_i^\star)\mid (V, w_V).\label{eq:dist-id2}
        \end{align}
Indeed, by Def.~\ref{def:iter-construct} and $\mathrm{Iter_H}$ in Prop.~\ref{prop:multi-round-exposure}, the event $\calA_{\mathrm{Iter}_H}(V, w_V, E_1,\ldots,E_{i-1})$ is measurable wrt $\sigma(H_1,\ldots,H_{i-1} \mid V, w_V)$, and $H_1, \ldots, H_{i-1}$ are \emph{independent} of $H_i$ conditioned on $(V, w_V)$, so \eqref{eq:dist-id1} holds. By \eqref{eq:partial-coupling} and by the conditional independence of $G_1^\star, \ldots, G_i^{\star}$,  $\calA_{\mathrm{Iter}_G^-}(V, w_V, E_1, \ldots, E_{i-1})$ is in $\sigma(G_1^{\theta_1}, \ldots, G_{i-1}^{\theta_{i-1}}\mid V, w_V) \subseteq \sigma(G_1^{\star}, \ldots, G_{i-1}^{\star}, (Z_e)_{e\in V^{(2)}} \mid V, w_V)$, so \eqref{eq:dist-id2} follows similarly.
By Strassen's theorem, to prove that a coupling in~\eqref{eq:multi-round-exp-coupling-1b} exists, it now suffices to prove that for all $k\ge 1$ and all $e_1,\ldots,e_k \in V^{(2)}\setminus (E_1 \cup \dots \cup E_{i-1})$,
    \begin{equation}\label{eq:multi-round-exp-coupling-3}
    \begin{aligned}
        \pr\big(e_1,\ldots,e_k \in \calE(G_i^{\theta_i}) &\mid \calA_{\mathrm{Iter}_{G^-}}(V, w_V, E_1,\ldots,E_{i-1})\big) \\
        &\quad\ge \pr\big(e_1,\ldots,e_k \in \calE(H_i) \mid \calA_{\mathrm{Iter}_{H}}(V, w_V, E_1,\ldots,E_{i-1})\big)\\
        &\quad=\theta_i^k \cdot \pr\big(e_1,\ldots,e_k \in \calE(G_i^\star) \mid \calA_{\mathrm{Iter}_{G}^-}(V, w_V, E_1,\ldots,E_{i-1})\big),
    \end{aligned}
    \end{equation}
    where the last row follows from \eqref{eq:partial-coupling}--\eqref{eq:dist-id2}.
    Dividing both sides by the second factor of the rhs and applying~\eqref{eq:partial-coupling} yields
    \begin{equation*}
    \begin{aligned}
        \pr\big(Z_{e_1}=i,\ldots,Z_{e_k} = i \mid \{e_1,\ldots,e_k \in \calE(G_i^\star)\} \cap \calA_{\mathrm{Iter}_{G}^-}(V, w_V, E_1,\ldots,E_{i-1})\big) \ge \theta_i^k.
        \end{aligned}
    \end{equation*}
    Since $\sigma(G_i^\star \mid V,w_V)$ is independent of $\sigma(G_1^\star,\dots,G_{i-1}^\star,(Z_e)_{e \in V^{(2)}}\mid V,w_V)$, we see it suffices to prove that for all $k\ge 1$ and all $e_1,\dots,e_k \in V^{(2)} \setminus (E_1 \cup \dots \cup E_{i-1})$,
    \begin{equation}\label{eq:multi-round-exp-coupling-4}
    \begin{aligned}
        \pr\big(Z_{e_1}=i,\ldots,Z_{e_k} = i \mid \calA_{\mathrm{Iter}_{G}^-}(V, w_V, E_1,\ldots,E_{i-1})\big) \ge \theta_i^k.
        \end{aligned}
    \end{equation}
    We next express $\calA_{\mathrm{Iter}_{G}^-}(V, w_V, E_1,\ldots,E_{i-1})$ in terms of simpler events, using Def.~\ref{def:iter-construct}\eqref{item:iter5}.
    Let $\underline t_{j,s}$ be the $s$-th tuple in the canonical ordering of $\calF_j(V,w_V,E_1,\ldots,E_{j-1})$, with $E_j=:\underline t_{j,s_j^\star}$ the $s^\star_j$th tuple.
    For each tuple $\underline t_{j,s}$ we collect the edges in $\{e_1,\dots, e_k\}\cap \underline t_{j,s}$, and define
 \begin{equation}\label{eq:a-and-c}
 \begin{aligned}
   A_{j,s}'&:=\{\exists e \in \underline t_{j,s}\cap \{e_1, \dots, e_k\}: Z_{e}\neq j\},\\
   B_{j,s}&:= \{\exists e\in \underline t_{j,s}, e\notin\{e_1, \dots e_k\}, Z_e\neq j\},\\
 C_{j,s}&:=  \big(\neg\,\calU_j(V,w_V,E_1,\ldots,E_{j-1},\underline t_{j,s}) \cup \{\exists e\in \underline t_{j,s}: e \notin \calE(G_j^\star)\} \big),
\end{aligned}
 \end{equation}
 with the idea that if $t_{j,s}\cap \{e_1, \dots, e_k\}=\emptyset$ then $A_{j,s}'=\Omega$. Let now $A=\{Z_{e_1}=i, \dots, Z_{e_k}=1 \}$, then clearly  $A\subseteq A_{j,s}'$. Further, for $j\le i-1$,  the edges and costs of $G_i^\star$ are (conditionally on $(V, w_V)$) independent from those of $G_1^\star,\ldots,G_{i-1}^\star$, so $C_{j,s}$ is independent of $A$. $B_{j,s}$ is also independent of $A$ since it concerns different edges than $\{e_1,\dots, e_k\}$. Then, by Def.~\ref{def:iter-construct}\eqref{item:iter5}, $\underline t_{j,s}$ is not chosen as $E_j$ iff $A'_{j,s}\cup B_{j,s}\cup C_{j,s}$  holds, while $\underline t_{j,s_{j}^\star}=E_j$ is chosen in round $j$ if $\neg A_{j,s_j^\star}\cap \neg B_{j,s_j^\star}\cap \neg C_{j,s_j^\star}$ holds. 
 Hence
    \begin{equation*}
\calA_{\mathrm{Iter}_{G}^-}(V, w_V, E_1,\ldots,E_{i-1}) = \bigcap_{j \le i-1}  \Big(\neg A_{j,s_j^\star}\cap \neg B_{j,s_j^\star}\cap \neg C_{j,s_j^\star} \cap \bigcap_{s < s_j^\star} \big( A'_{j,s}\cup B_{j,s}\cup C_{j,s} \big)\Big).
    \end{equation*}
    Since $e_1,\dots,e_k \notin E_1,\dots,E_{i-1}$, the event $A_{j,s_j^\star}$ is also independent of $Z_{e_1},\dots,Z_{e_k}$, so the event $B:=\cap_{j\le i-1}(\neg A_{j,s_j^\star}\cap \neg B_{j,s_j^\star}\cap \neg C_{j,s_j^\star}) $ above is independent of $Z_{e_1},\dots,Z_{e_k}$ and thus of $A$, and since $A\subseteq \cap_{j\le i-1}\cap_{s<s_j^\star}A_{j,s}'$, 
    the lhs of ~\eqref{eq:multi-round-exp-coupling-4} equals
    \[
\frac{\pr(A\cap B)}{\pr(B\cap \cap_{s\le s_j^\star} (A_{j,s}'\cup B_{j,s}\cup C_{j,s}))} \ge \frac{\pr(A\cap B)}{\pr(B)}= \pr(A)=\theta_i^k,
    \]
    since $Z_{e_\ell}$ are independent across $\ell\le k$.
    This yields the rhs of \eqref{eq:multi-round-exp-coupling-4}, finishing the proof.
\end{proof}

\section{Building blocks: finding cheap edges}\label{sec:lemmas}
In this section, we return from CIRGs to GIRGs and state a few ancillary claims that we shall use to construct the different parts of the low-cost path between $0$ and $x$. We work in the quenched setting with the realisation of vertices and their weights $\widetilde{\calV}=(V, w_V)$ \emph{exposed}, taking the role of $(V,w_V)$ for CIRGs of Definition \ref{def:CIRG}.
All claims here concern $\theta$-percolated SFP/IGIRG as in Definition \ref{def:percolated}, so that we can later use them on the graphs $H_i$ of the multi-round exposure Proposition \ref{prop:multi-round-exposure}. Since the proofs are straightforward but contain some long calculations, we defer them to Appendix~\ref{app:proof-building-lemmas}. We first set out some common notation for Sections \ref{sec:lemmas} and~\ref{sec:hierarchy}.

\begin{setting}(The setting)\label{set:hierarchy-common}
Consider $1$-FPP in Definition \ref{def:1-FPP} on the graphs IGIRG or SFP satisfying the assumptions given in \eqref{eq:power_law}--\eqref{eq:F_L-condition} with $d\ge 1, \alpha \in (1,\infty], \tau\in(2,3)$.
 Let $\underline{c}$, $\overline{c}$, $h$, $L$, $c_1$, $c_2$, and $\beta$ be as in~\eqref{eq:power_law}--\eqref{eq:F_L-condition}; we allow $\beta = \infty$ and $\alpha=\infty$.
    Fix a realisation $(V, w_V)$ of $\widetilde{\cal V}$, and let $G \sim \{\calG\mid V, w_V\}$, and for a $\theta\in(0,1]$, let $G'$ be a $\theta$-percolation $G'$ of $G$. For brevity we write $\pr( \cdot \mid V,w_V)$ for $\pr( \cdot \mid \widetilde{\calV}(G') = (V,w_V))$. Let $x \in V$, and let $Q$ be a cube of side length $\xi$ containing $0$ and $x$. Let $\delta \in (0,1)$, $w_0\ge 1$, and assume that $(V, w_V)$ is such that $Q$ contains a weak $(\delta/4,w_0)$-net $\calN$ with $0,x \in \calN$ given in Definition \ref{def:weak-net}. Finally, let $\gamma \in(0,1)$.
\end{setting}

We now define a function of crucial importance for the optimisation of the exponents $\Delta_0, \eta_0$ in \eqref{eq:Delta_0} \eqref{eq:eta_0}. The first claim joins two Euclidean balls with a low-cost edge in the net with endpoints having specified weights. 
For all $\gamma > 0$ and all $\eta, z \ge 0$, we define
\begin{align}\label{eq:Lambda-def}
    \Lambda(\eta, z):= 2d\gamma-\alpha(d-z)-z(\tau-1)+\big(0\wedge\beta(\eta-\mu z)\big).
\end{align}

\begin{restatable}[Single bridge-edge]{claim}{ClaimSingleBridgeEdge}
\label{claim:cheap-bridge}
    Consider Setting~\ref{set:hierarchy-common}. 
    Let $z \in [0,d]$ satisfy $2d\gamma > z(\tau-1)$. Let $c_H, \eta \ge 0$. Suppose that $0 < \delta\lls \gamma,\eta, z, c_H,\mpar$, and that $D \ggs\eta, z,c_H,\delta,w_0$. Assume that $D^\gamma \in[  (\log\log\xi\sqrt{d})^{16/\delta}, \xi\sqrt{d} ]$ and that $x, y \in \calN$ satisfy $|x-y| \le c_HD$, and let $\ulw\in[w_0 \vee 4(c_H+2)^d \vee 4000c_1^{-1/(\mu\beta)}, D^{\delta}]$ satisfy $F_L((\ulw/4000)^{\mu})\ge 1/2$. For $v \in \{x,y\}$, define
    \begin{align}\label{eq:calZ-def}
        \calZ(v)=\calZ_{\gamma,z,\ulw}(v):=\calN \cap \big(B_{D^\gamma}(v)\times [\ulw D^{z/2}/2, 2\ulw D^{z/2}]\big).
    \end{align}
    Again for each $v \in \{x,y\}$, let $Z_v \subseteq \calZ(v)$ with $|Z_v| \ge |\calZ(v)|/4$. Let
    \begin{align}\label{eq:N-gamma-eta-z}
        N_{\eta,\gamma,z, \ulw} (Z_x,Z_y) &:=\left\{ (a,b) \in Z_x\times Z_y, ab\in \calE(G'), \calC(ab) \le (\ulw/10)^{3\mu}D^\eta \right\}.
    \end{align}
    Then, (and also for $\alpha=\infty$ and/or $\beta=\infty$ under the convention that $\infty \cdot 0 = 0$ in~\eqref{eq:Lambda-def}),
    \begin{equation}\label{eq:lem-cheap-bridge}
        \pr\big(N_{\eta,\gamma,z, \ulw}(Z_x,Z_y) = \emptyset \mid V,w_V\big) \le \exp\Big({-}\theta \ulw^{-2(\tau-1)}D^{\Lambda(\eta,z)-2\gamma d\delta/3}\Big).
    \end{equation}
\end{restatable}

Since $D\ggs \delta$, we can always ensure that the interval for $\ulw$ is non-empty.

The next claim finds a low-cost edge from a fix vertex in $\calN$ with weight roughly $M$ to some nearby vertex in $\calN$ with weight roughly $K$. We will use this claim later in two different ways, either $K$ being much lower than $M$; or $K$ being somewhat larger than $M$. 
\begin{restatable}[Single cheap edge nearby]{claim}{ClaimSingleCheapEdgeNearby}
\label{claim:cheap-edge-to-nice-vertex}\label{lem:CETNV}
    Consider Setting~\ref{set:hierarchy-common}. 
    Let $M > 1$, and let $x\in\calN$ be a vertex with $w_x\in [\tfrac12M, 2M]$. Let $U,D > 0$ and $K>w_0$, and define the event
    \begin{equation}\label{eq:akdu} \calA_{K,D,U}(x):=\left\{ \exists y\in \calN \cap (B_D(x)\times[\tfrac12K,2K]): xy\in \calE(G'), \ \cost{xy} \le U  \right\}. 
    \end{equation}
    Suppose that $\delta \lls \mpar$, that $K,M,D \ggs \delta,w_0$, and that 
    \begin{align}    
        (\log\log\xi\sqrt{d})^{16/\delta} &\le (D \wedge (KM)^{1/d})/4^{1/d} \le \xi\sqrt{d},\label{eq:KM-xi}\\
        K &\le D^{d/(\tau-1)-\delta} \wedge M^{1/(\tau-2+\delta\tau)}.\label{eq:K-D-M}
    \end{align}
    Then if $\beta = \infty$ and $U(KM)^{-\mu}\ggs \mpar$, or alternatively if $\beta < \infty$,  then     \begin{equation}\label{eq:cheap-edge-to-nice-vertex}
    \begin{aligned}
        \pr\big(\calA_{K,D,U}(x) &\mid V,w_V\big) 
        \ge 1 - \exp\Big({-}\theta K^{-(\tau-1)}(D^d \wedge KM)^{1-\delta}(1 \wedge (U(KM)^{-\mu})^\beta \Big).
        \end{aligned}
    \end{equation}
\end{restatable}

The third claim builds cheap \emph{weight-increasing paths}, from a low-weight vertex in $\calN$ to a high-weight vertex in $\calN$. The proof is via repeated application of Claim~\ref{claim:cheap-edge-to-nice-vertex} (see Appendix~\ref{app:proof-building-lemmas}).
\begin{restatable}[Weight-increasing paths]{claim}{ClaimWeightIncreasingPath}
\label{claim:cheap-path-to-larger-weight}
    Consider Setting~\ref{set:hierarchy-common}. 
    Let $M > 1$, and let $y_0$ be a vertex in $\calN$ with weight in $[\tfrac12M, 2M]$. Let $K,D>1$, $U\ge K^{2\mu}$, and let
    \begin{align}\label{eq:cond-q} 
        q := \left\lceil\frac{\log(\log K / \log M)}{\log(1/(\tau-2+2d\tau\delta))}\right\rceil.
    \end{align}
    Let $\calA_{\pi(y_0)}$ be the event that $G'$ contains a path $\pi = y_0y_1\ldots y_q$ contained in $\calN \cap B_{qD}(y_0)$ such that $W_{y_q} \in [\tfrac12K,2K]$ and $\cost{\pi} \le q U$. Suppose that $\delta\lls \mpar$, that $K,M,D\ggs \theta,\delta,w_0$, and that $M \le K \le D^{d/2}$, $D \le \xi\sqrt{d}$, and $(M/2)^{2/d} \ge (\log\log\xi\sqrt{d})^{16/\delta}$. Then if $\beta = \infty$ and $U(KM)^{-\mu} \ggs \mpar$, or if $\beta < \infty$,  then   
    \begin{align}\label{eq:weight-increasing-path-error}
        \pr\big(\calA_{\pi(y_0)} \mid V,w_V\big) \ge 1-\exp({-}\theta M^\delta).
    \end{align}
\end{restatable}

The last claim allows us to find a common neighbour for two vertices with roughly the same weight if the distance between them is not too large with respect to their weights.
\begin{restatable}[Common neighbour]{claim}{ClaimCommonNeighbour}
\label{claim:common-neighbour}
    Consider Setting~\ref{set:hierarchy-common}. 
    Let $\delta \lls \mpar$, let $c_H>0$, and let $D\ge w_0^{2/d}$ with $D \ggs c_H,\delta$ and $D\in[(\log\log\xi\sqrt{d})^{16/\delta}, \xi\sqrt{d}]$. Let $x_0,x_1\in\calN$ be vertices with $w_{x_0}, w_{x_1}\in[D^{d/2}, 4D^{d/2}]$ at distance $|x_0-x_1| \le c_H D$, and let $\calA_{x_0\star x_1}$ be the event that $x_0$ and $x_1$ have a common neighbour in $G'$, $v \in \calN \cap B_D(x_0)$ with $\cost{x_0v} + \cost{vx_1} \le D^{2\mu d}$. Then
        \begin{equation}\label{eq:common-neighbor}
        \pr\big(\calA_{x_0\star x_1} \mid V,w_V\big) \ge 1 - \exp\Big({-}\theta^2D^{(3-\tau-2\delta)d/2}\Big).
    \end{equation}
\end{restatable}

\section{Budget travel plan: hierarchical bridge-paths}\label{sec:hierarchy}

In this section, we present the main construction for the upper bounds in Theorems~\ref{thm:polylog_regime} and~\ref{thm:polynomial_regime}.  This construction is a ``hierarchy'' of cheap bridging paths connecting $x$ and $y$ that we heuristically described in Section \ref{sec:intro} as the `budget travel plan'. Here we elaborate more on the heuristics before diving into proofs. 

Let $U$ be either polynomial in $|x-y|$ (when proving Theorem \ref{thm:polynomial_regime}) or sub-logarithmic in  $|x-y|$ (when proving Theorem~\ref{thm:polylog_regime}). We first find a 3-edge \emph{bridging-path} $\pi_1=x'aby'$ of cost at most $U$ between two vertices $x'$ and $y'$ with weights $w_{x'}, w_{y'}\in [w_{H_1},4w_{H_1}]$, such that $|x-x'|$ and $|y-y'|$ are both at most $|x-y|^\gamma$ for some $\gamma \in (0,1)$, see Figure~\ref{fig:intuition_hierarchy}(a). 
This reduces the original problem of connecting $x$ and $y$ to two instances of connecting two vertices at distance $|x-y|^{\gamma}$, at the additional cost of $U$. We then work recursively, applying the same procedure to find a bridging-path with endpoints near $x$ and $x'$ and another one with endpoints near $y'$ and $y$, with all four distances at most $|x-y|^{\gamma^2}$, and both bridging-paths having cost at most $U$, obtaining the second level of the hierarchy, see Figure~\ref{fig:intuition_hierarchy}(b). The endpoints of the bridging paths always have weight in $[w_{H_1},4w_{H_1}]$, hence iteration is possible. By repeating the process $R$ times we obtain a ``broken path'' of bridging-paths of cost $U(1+2+\dots +2^R)$ and $2^R$ gaps of length $|x|^{\gamma^R}$ between the bridging-paths. 
We call this ``broken path'' a \textit{hierarchy} after Biskup, who developed the one-edge bridge construction for graph distances in long range percolation in \cite{biskup2004scaling}. 
There are two reasons for having a \emph{bridging-path} instead of a single bridge-edge. Firstly, a typical single bridge-edge $ab$ has very high weights $w_a, w_b$, and even the cheapest edge out of $a$ and $b$ has high cost, which would cause high costs when filling the gaps. So instead we find an a-typical path of the form $(x'aby')$, with all three edges of cost $U/5$, and $x',y'$ having low weight in $[w_{H_1}, 4 w_{H_1}]$, giving a bridging path of length three. 
Secondly, to fill the $2^R$ gaps after $R$ iterations whp, the failure probability of finding a connecting path has to be extremely low, $o(2^{-R})$. 
In most regimes this is impossible via short paths (e.g. length two) and low enough failure probability. Instead, we find weight increasing paths $\pi_{x'x''}$ and $\pi_{y'y''}$ (as in Claim~\ref{claim:cheap-path-to-larger-weight}) of cost at most $U/5$ from each vertex $x'$ and $y'$ of the bridge paths $(x'aby')$ to respective vertices $x'', y''$ still near $a$ and $b$ but with much higher weights in $[w_{H_2}, 4w_{H_2}]$. The concatenated paths $(\pi_{x''x'},x'aby',\pi_{y'y''})$ then themselves form a second hierarchy (now with bridging paths of more than $3$ edges). Connecting all the new $2^R$ gaps whp is possible via paths of length two and cost $U'$, which is polynomial in the distance $|x-y|^{\gamma^{R}}$, using Claim~\ref{claim:common-neighbour}.
In the \emph{polylogarithmic case}, $U$ and $U'$ are sublogarithmic, and the factor $2^R$ is of order $(\log |x-y|)^{\Delta_0 + o(1)}$ and dominates the overall cost. The bottleneck in this regime is the number of gaps, whereas the bridge-paths have negligible costs.
In the \emph{polynomial case}, however, the cost of the first bridge $U=|x-y|^{\eta_0 + o(1)}$ dominates, and all other costs (even with the factors $2^i$) are negligible in comparison, causing the total cost to be polynomial in $|x-y|$. In both cases we could use and optimise level-dependent costs $U_i$, but that does not improve the statements of Theorems \ref{thm:polylog_regime}, \ref{thm:polynomial_regime}.

Let us now formally define the concept of a \emph{hierarchy} including edge-costs. 
In the rest of the paper, the symbol $\sigma$ denotes an index $\sigma=\sigma_1\sigma_2\ldots\sigma_R\in\{0,1\}^R$ indicating the place of a vertex in the hierarchy. This can be viewed as the Ulam-Harris labelling of the leaves of a binary tree of depth $R$, e.g. $\sigma = 1001$ corresponds to the leaf that we reach by starting at the root and then moving to the right child, the left child twice, and the right child again. We denote the string formed by concatenating $\sigma'$ to the end of $\sigma$ by $\sigma \sigma'$. We ``pad'' strings of length less than $R$ by adding  copies of their last digit or its complement via the $T$ and $T^c$ operations we now define (and discuss further below):

\begin{definition}[Binary strings]\label{def:strings}
For $\sigma=\sigma_1\ldots \sigma_i\in\{0,1\}^i$ for some $i\ge 1$, we define $\sigma T:=\sigma_1\ldots \sigma_i\sigma_i\in \{0,1\}^{i+1}$, while $\sigma T_0:=\sigma$, and $\sigma T_k:=(\sigma T_{k-1})T$ for any $k\ge 2$. Let $0_i:=0T_{i-1}$ and $1_i:=1T_{i-1}$ be the strings consisting of $i$ copies of $0$ and $1$, respectively. 
Fix an integer $R \ge 1$. Define the \emph{equivalence relation} $\sim_T$ on $\cup_{i=1}^R\{0,1\}^i$, where $\sigma\sim_T \sigma'$ if either $\sigma T_k=\sigma'$ or $\sigma'T_k=\sigma$ for some $k\ge 0$, with  $\{\sigma\}$ be the equivalence class of $\sigma$. Let
\begin{align*}
    \Xi_{i}:=\{ \sigma \in \cup_{j=i}^R\{0,1\}^j: \sigma_{i-1} \neq \sigma_i, \sigma_{j}=\sigma_i \ \forall j\ge i\}, \qquad \Xi_0:=\{\emptyset\},
\end{align*}
with $\emptyset$ the empty string. We say that $\{\sigma\}$ \emph{appears first on level $i$} if any (the shortest) representative of the class $\{\sigma\}$ is contained in $\Xi_i$. 

For $\sigma=\sigma_1\ldots \sigma_i\in\{0,1\}^i$ for some $i\ge 1$, we define $\sigma T^c:=\sigma_1\ldots \sigma_i(1-\sigma_i)\in \{0,1\}^{i+1}$. For $\sigma\in \Xi_i$, we say that 
$(\sigma T_{j-1})T^c\in \{0,1\}^{i+j}$ is the \emph{level-$(i+j)$ sibling} of $\{\sigma\}$. We say that two strings in level $i$ are \emph{newly appearing cousins} on level $i$ if they are of the forms $\sigma01$ and $\sigma 10$ for some $\sigma\in \{0,1\}^{i-2}$. 
\end{definition}

The inverse of the operator $T$ "cuts off" all but one of the identical last digits from a $\sigma \in \{0,1\}^R$, hence, each class $\{\sigma\}$ has exactly one representative in $\{0,1\}^R$, and the number of equivalence classes is $2^R$. For $i>1$, there are exactly $2^{i-1}$ equivalence classes that first appear on level $i$ (i.e. the shortest representative of the class is in $\Xi_i$), and (since $0, 1\in \Xi_1$) the total number of equivalence classes that appear until level $i$ is $2^i$. 
To show an example of the sibling relationship, e.g. $01111\sim 01$ belongs to $\Xi_2$, and the level-$3$ sibling of $\{01\}$ is $010$, and the level-$(2+j)$ sibling of $\{01\}$ is $01_j0$. Similarly, $010$ and $001$ are newly appearing level-$3$ cousins, and on level $i$, there are $2^{i-2}$ pairs of newly appearing cousins.

The hierarchy embeds each equivalence class $\{\sigma\}\in\cup_{j=1}^{R}\{0,1\}^R$ into the (weighted) vertex set of $G$ so that all cousins are joined by low-cost ``bridge'' paths, all siblings are close in Euclidean space, $0^R=x$ and $1^R=y$ are the vertices we start with, and the weights of all other vertices in the embedding are constrained. 
The Euclidean distances between siblings/cousins will decay doubly-exponentially in $i$. We formalise the embedding in the following definition. 

\begin{definition}[Hierarchy]\label{def:hierarchy}
    Consider Setting~\ref{set:hierarchy-common}.
    Let $y_0, y_1 \in\widetilde\calV$, $U,\olw,c_H\ge1$, and $R\ge2$ be an integer. Consider a set of vertices $\{y_{\sigma}\}_{\sigma\in\{0,1\}^R}$, divided into \emph{levels} $\calL_i:=\{y_{\sigma} \colon \sigma\in\Xi_i\}$ for $i\in \{1,\ldots,R\}$, satisfying that $y_{\sigma}=y_{\sigma'}$ if $\sigma\sim_T \sigma'$.
  We say that $\{y_{\sigma}\}_{\sigma\in\{0,1\}^R}\subset \widetilde \calV$ is a \emph{$(\gamma, U, \olw, c_H)$-hierarchy of depth $R$} with $\calL_1=\{y_0, y_1\}$ if it satisfies the following properties:
    \begin{enumerate}[(H1)]
        \item\label{item:H1} $W_{y_{\sigma}}\in [\olw, 4\olw]$ for all $\sigma\in\{0,1\}^R \setminus \Xi_1$.
        
        \item  \label{item:H2} $|y_{\sigma0}-y_{\sigma1}|\le c_H|y_0-y_1|^{\gamma^i}$ for all  $\sigma\in\{0,1\}^i$, $i=0,\ldots,R-1$.
        
        \item\label{item:H3}  There is a set $\{P_{\sigma} : \sigma\in\{0,1\}^i, 0\le i \le R-2\}$ of paths in $G$ such that for all $0 \le i \le R-2$ and all $\sigma \in \{0,1\}^i$, $P_{\sigma}$ connects $y_{\sigma01}$ to $y_{\sigma10}$. Moreover, we can partition $\bigcup_{\sigma\in\{0,1\}^R} \calE(P_\sigma)$ into sets $\{\calE^-(P_\sigma)\colon\sigma\in\{0,1\}^R\}$ in such a way that for all $\sigma$, we have $\calE^-(P_\sigma) \subseteq \calE(P_\sigma)$ and $\calC(\calE^-(P_\sigma)) \le U$. These paths $P_{\sigma}$ are called \emph{bridges}.
    \end{enumerate}
    Given a set $\calN \subseteq \widetilde{\calV}$, we say that a hierarchy $\{y_\sigma\}_{\sigma \in \{0,1\}^R}$ is \emph{fully contained in $\calN$} if both $\{y_\sigma\}_{\sigma \in \{0,1\}^R} \subseteq \calN$, and every vertex on the paths $P_{\sigma}$ in (H\ref{item:H3}) lies in $\calN$.
\end{definition}

    \begin{figure}[t]
    \centering
    \begin{tikzpicture}[scale=0.65]
    	\draw[dashed] (0,9)--(20,9);
    	\draw[dashed] (0,6)--(20,6);
    	\draw[dashed] (0,3)--(20,3);
    	\draw[dashed] (0,0)--(20,0);
    	\draw (0,7.5) node[left] {$\calL_1$};
    	\draw (0,4.5) node[left] {$\calL_2$};
    	\draw (0,1.5) node[left] {$\calL_3$};
    	
    	\filldraw (1.5,7.5) circle (3pt) node[above] {$y_0=y_{\{0\}}$};
    	\filldraw[red] (1.5,4.5) circle (2pt) node[above] {${\color{black}y_{00}=y_{\{0\}}}$};
    	\filldraw[red] (1.5,1.5) circle (2pt) node[above] {${\color{black}y_{000}=y_{\{0\}}}$};
    		\draw[red] (1.5, 7.5) -- (1.5, 4.5);
    	\draw[red] (1.5, 4.5) -- (1.5, 1.5);
    	\filldraw (19,7.5) circle (3pt) node[above] {$y_1=y_{\{1\}}$};
    		\filldraw[red] (19,4.5) circle (2pt) node[above] {${\color{black}y_{11}=y_{\{1\}}}$};
    		\filldraw[red] (19,1.5) circle (2pt) node[above] {${\color{black}y_{111}=y_{\{1\}}}$};
    			\draw[red] (19, 7.5) -- (19, 4.5);
    		\draw[red] (19, 4.5) -- (19, 1.5);
    	\draw[dotted] (1.5,7.5)--(19,7.5);
    	\draw (10,7.5) node[below] {$=\xi$};
    	
    	\filldraw (7,4.5) circle (3pt) node[above] {$y_{01}=y_{\{01\}}$};
    	\filldraw[red] (7,1.5) circle (2pt) node[above]{${\color{black}y_{011}=y_{\{01\}}}$};
    	\draw[red] (7, 4.5) -- (7, 1.5);
    	
        \filldraw (13,4.5) circle (3pt) node[above] {$y_{10}=y_{\{10\}}$};
         \filldraw[red] (13,1.5) circle (2pt) node[above] {${\color{black}y_{100}=y_{\{10\}}}$};
         \draw[red] (13, 4.5) -- (13, 1.5);
        \draw[dotted] (1.5,4.5)--(7,4.5);
    	\draw (4,4.5) node[below] {$\le 2\xi^{\gamma}$};
    
    	\draw[dotted] (13,4.5)--(19,4.5);
    	\draw (16,4.5) node[below] {$\le 2\xi^{\gamma}$};
    	\draw[thick, blue] (7,4.5) .. controls (10,3) .. (13,4.5);
    	\draw (10,3.5) node[above] {$d_{\calC}\le U$};
    	
    	\filldraw (3.5,1.5) circle (3pt) node[above] {$y_{001}$};
    	\filldraw (5,1.5) circle (3pt) node[above] {$y_{010}$};
    	\filldraw (15,1.5) circle (3pt) node[above] {$y_{101}$};
    	\filldraw (17,1.5) circle (3pt) node[above] {$y_{110}$};
    	\draw[dotted] (1.5,1.5)--(3.5,1.5);
    	\draw (2,1.5) node[below] {\tiny $\le 2\xi^{\gamma^2}$};
    	\draw[dotted] (5,1.5)--(7,1.5);
    	\draw (6,1.5) node[below] {\tiny $\le 2\xi^{\gamma^2}$};
    	\draw[dotted] (13,1.5)--(15,1.5);
    	\draw (14,1.5) node[below] {\tiny $\le 2\xi^{\gamma^2}$};
    	\draw[dotted] (17,1.5)--(19,1.5);
    	\draw (18,1.5) node[below] {\tiny $\le 2\xi^{\gamma^2}$};
    	\draw[thick, blue] (3.5,1.5) .. controls (4.25,1) .. (5,1.5);
    	\draw (4,1) node[below] {\tiny $d_{\calC}\le U$};
    	\draw[thick,blue] (15,1.5) .. controls (16,1) .. (17,1.5);
    	\draw (16,1) node[below] {\tiny $d_{\calC}\le U$};
    \end{tikzpicture}
    \caption{{A schematic representation of a $(\gamma, U, \olw, 2)$-hierarchy of depth $R\!=\!3$.  The horizontal axis represents the ($1$-dimensional, Euclidean) distances between the vertices, while the vertical axis shows the level of the hierarchy. The weights of all vertices except $y_0$ and $y_1$ are in the interval $[\olw, 4\olw]$. On level $1$, only the initial vertices $y_0, y_1$ appear. We "push down" $y_0=y_{00}, y_{1}=y_{11}$ to level $2$ (red)  and we find them their respective level-$2$ sibling vertices $y_{01}$ and $y_{10}$ within Euclidean distance $2\xi^\gamma$, so that there is path of cost at most $U$ between $y_{01}, y_{10}$ (blue). Then, we "push down" to level $3$ all vertices that appeared at or before level $2$, i.e., $y_{000}, y_{011}, y_{100}, y_{111}$ (red), and find for each of them their level-$3$ siblings, i.e., $y_{001}, y_{010}, y_{101}, y_{110}$, so that each vertex is within Euclidean distance $\le 2 \xi^{\gamma^2}$ from its level-$3$ sibling, and that there is a (blue) path of cost at most $U$ between the newly appearing cousins $y_{001}, y_{010}$ and between $y_{101}, y_{110}$. An intuitive representation is in Figure \ref{fig:intuition_hierarchy}.}  
    }
    \label{fig:hierarchy}
    \end{figure}
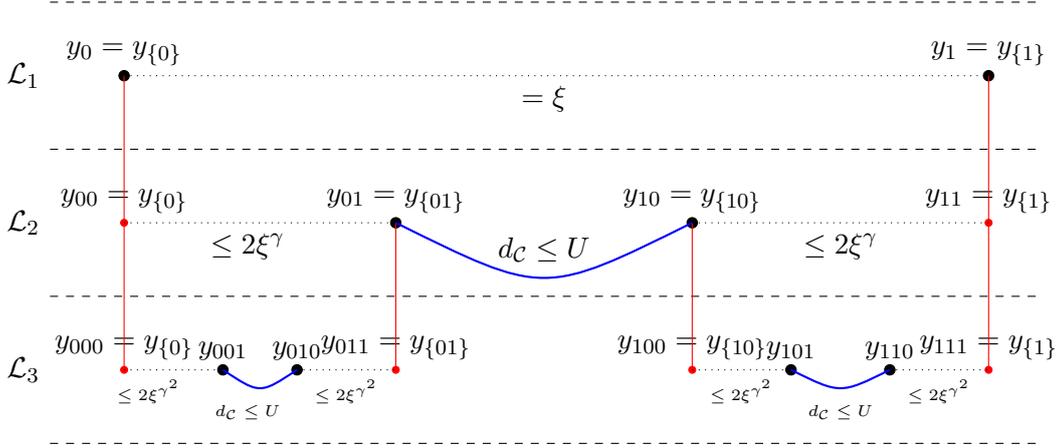

Condition~\emph{(H\ref{item:H3})} is slightly weaker than requiring each bridge to have cost at most $U$. We shall construct the hierarchy via an iterative construction in Def.~\ref{def:iter-construct}, using one round to embed each level $\calL_i$. We shall use Prop.~\ref{prop:multi-round-exposure} to estimate the success probability of the whole construction, which requires that we use marginal costs, and this gives the definition of $\calE^-(P_\sigma)$ in~\eqref{eq:marginal-cost}. 
Using marginal costs causes no problem, since our goal is to find a path $\pi$ between $y_0$ and $y_1$ of low cost. A path $\pi$ uses every edge in it once, so all the bridges $P_\sigma$ together will contribute to the cost of $\pi$ at most 
\begin{align*}
\calC(\pi) = \sum_{\sigma\in\{0,1\}^i, 0\le i \le R-2} \calC(\calE^-(P_\sigma)) \le (2^{R-1}-1) U.
\end{align*}
Later we also need that the the hierarchy stays close to the straight line segment between the starting vertices. To track this, we have the following definition:
\begin{definition}\label{def:deviation}
    Given $u,v \in \R^d$, let $S_{u,v}$ denote the line segment between $u,v$. For $x\in \R^d$ we define the deviation $\mathrm{dev}_{uv}(x):=\min\{|x-y|: y\in S_{u,v}\}$. Given a set of vertices $\calH$ in $\R^d$, we define the \emph{deviation of $\calH$ from $S_{uv}$} as $\mathrm{dev}_{uv}(\calH):=\max\{\mathrm{dev}_{uv}(x) \colon x\in \calH \}$. Finally, for a path $\pi=(x_1\dots x_k)$, let  the \emph{deviation of $\pi$} be $\mathrm{dev}(\pi) := \max\{\mathrm{dev}_{x_1x_k}(x_i) \colon i \in [k]\}$, i.e., the deviation of its vertex set from the segment between the endpoints.
\end{definition}

Next we describe the procedure used to find the hierarchy in $G$. We iteratively embed the levels $\Xi_i$ into the vertex set $\widetilde\calV$. We first embed $\Xi_1$, by setting $0\mapsto y_0$ and $1\mapsto y_1$, i.e., $\calL_1:=\{y_0, y_1\}$, the two given starting vertices. Observe that this embedding trivially satisfies condition \emph{(H\ref{item:H2})} for $i=0$, (i.e., $\sigma=\emptyset$ in \emph{(H\ref{item:H2})}) for all $c_H\ge1$. Conditions \emph{(H\ref{item:H1})} and \emph{(H\ref{item:H3})} do not concern $y_0$ and $y_1$. In round $i+1$ we then embed all $\sigma \in \Xi_{i+1}$. 
Given the embedding of $\cup_{j\le i}\Xi_j$ of vertices in level $\cup_{j\le i}\calL_j$, we will embed $\sigma\in \Xi_{i+1} $ by finding $\{y_{\sigma}\}_{\sigma\in\Xi_{i+1}}=\calL_{i+1}$ as follows. For each sibling pair $\sigma 0, \sigma 1 \in \{0,1\}^i$, by the equivalence relation $\sim_T$ in Def.~\ref{def:strings}, $y_{\sigma 00} = y_{\sigma 0}$ and $y_{\sigma 11} = y_{\sigma 1}$. We then search for a pair of vertices $a$ and $b$ close to $y_{\sigma 00}$ and $y_{\sigma 11}$ respectively, so that $ab$ is a low-cost edge (typically covering a large Euclidean distance), and both $a$ and $b$ have a low-cost edge to a nearby vertex with weight in $[\olw,4\olw]$; we embed these latter two vertices as $y_{\sigma 01}$ and $y_{\sigma10}$. The path $(y_{\sigma01}aby_{\sigma10})$ then constitutes the bridge-path $P_{\sigma}$ required by \emph{(H\ref{item:H3})}. See Figure \ref{fig:hierarchy} for a visual explanation. We formalise our goal for this iterative construction of bridges in the following definition and remark. 

\begin{definition}[Valid bridges]\label{def:valid-bridges}
    Consider Setting~\ref{set:hierarchy-common} and the notion of bridges in Definition~\ref{def:hierarchy}, and let $S$ be a set of edges of $G$. For any $D,U>0, w\ge 1$, we say that a path $P \subseteq \calN$ with endpoints $y,y'$ is a \emph{$(D,U,w)$-valid} bridge for $x_0$ and $x_1$ with respect to $S$ if:
    \begin{align}
        &w_y, w_{y'} \in [w,4w],  \label{eq:bridge-condition-b1}\\
        &|x_0-y|\le D, \quad \mbox{and}\quad |x_1-y'| \le D, \label{eq:bridge-condition-b2}\\
        & \calC(P\setminus S)\le U. \label{eq:bridge-condition-b3}
    \end{align}
\end{definition}

\begin{observation}\label{obs:valid-bridges}
    Consider Setting~\ref{set:hierarchy-common}. Fix any ordering on $\{0,1\}^R$, and let $\{y_\sigma\}_{\sigma \in \{0,1\}^R} \subseteq \calN$, $U > 0$, and $\olw \ge 1$. For all $0 \le i \le R-2$ and $\sigma \in \{0,1\}^i$, set $y_{\sigma'} := y_\sigma$ whenever $\sigma'\sim_T\sigma$. For all $i \le R-2$, let $D_{i} =|y_0-y_1|^{\gamma^{i}}$. 
  Suppose that for all $0 \le i \le R-2$ and all $\sigma \in \{0,1\}^i$, there exists a bridge $P_{\sigma}$ with endpoints $y_{\sigma01}$ and $y_{\sigma10}$ that is $(c_HD_{i+1}, U, \olw)$-valid for $y_{\sigma0},y_{\sigma1} \in \{0,1\}^{i+1}$ with respect to $S=\bigcup_{\sigma' < \sigma}\calE(P_{\sigma'})$. Then $\{y_\sigma\}_{\sigma \in \{0,1\}^R}$ is a $(\gamma,U,\olw, c_H)$-hierarchy of depth $R$ with first level $\{y_0,y_1\}$ (i.e., satisfying Definition \ref{def:hierarchy}).
\end{observation}
\begin{proof}
Conditions \emph{(H\ref{item:H1})} of Definition~\ref{def:hierarchy} is immediate from the weight constraint \eqref{eq:bridge-condition-b1}. 
\emph{(H\ref{item:H2})} holds for the following reason. For $\sigma0, \sigma1 \in \{0,1\}^{i+1}$, $P_\sigma$ being a $(c_HD_{i+1}, U, \olw)$ valid bridge for $y_{\sigma0}, y_{\sigma1}$ implies by \eqref{eq:bridge-condition-b2} and $y_{\sigma0}=y_{\sigma00}, y_{\sigma1}=y_{\sigma11}$ that both $|y_{\sigma00}- y_{\sigma01}|, |y_{\sigma10}- y_{\sigma11}|$ are at most $c_H|y_0-y_1|^{\gamma^{i+1}}$ for all $\sigma\in \{0,1\}^i$. Setting now either $\sigma':=\sigma0$ or $\sigma':=\sigma1$, this is equivalent to 
$|y_{\sigma'0}-y_{\sigma'1}|\le c_H|y_0-y_1|^{\gamma^{i+1}}$ for all $\sigma'\in \{0,1\}^{i+1}, i\ge 0$, and this exactly corresponds to \emph{(H\ref{item:H2})}, since the inequality in \emph{(H\ref{item:H2})} holds for $i=0$ trivially.
Finally, condition \emph{(H\ref{item:H3})} follows from \eqref{eq:bridge-condition-b3} by setting $\calE^-(P_{\sigma}) := \calE(P_\sigma) \setminus \bigcup_{\sigma' < \sigma}\calE(P_{\sigma'})$.
\end{proof}

We now lower-bound the probability of finding a valid bridge between two fixed vertices. 

\begin{lemma}[3-edge bridges]\label{lem:triple-edge-bridge}
    Consider Setting~\ref{set:hierarchy-common}. 
    Let $z \in [0,d]$ and let $c_H,\eta \ge 0$.
    Suppose that $\delta\lls \gamma,\eta,z,c_H,\mpar$ and that $D\ggs \gamma,\eta,z,c_H,\delta,w_0$. Suppose further that $D^\gamma \in[4^{1/d}(\log\log\xi\sqrt{d})^{16/\delta}, \xi\sqrt{d} ]$ and that $\theta D^{\Lambda(\eta,z)-\sqrt{\delta}} > 1$. Suppose that $x_0, x_1 \in \calN$ satisfy $|x_0-x_1| \le c_HD$, and let $\ulw \ggs \delta,w_0$ satisfy $\ulw\in[(\log\log\xi\sqrt{d})^{16d/\delta}, D^{\delta}]$. Let $\calA(x_0,x_1)$ denote the event that $G'$ contains a bridge $P$ that is $(2D^\gamma,3\ulw^{3\mu}D^{\eta},\ulw)$-valid for $x_0$ and $x_1$ with respect to $\emptyset$, and $\mathrm{dev}_{x_0x_1}(P)\le 2D^\gamma$. Finally, suppose that
    \begin{align}\label{eq:p(.)}
        p(D, \ulw, \theta, \eta, z):= \theta \ulw^{-(\tau-1)}\left(D^{d\gamma} \wedge \ulw^2 D^{z/2}\right)^{1-\delta}\left(1 \wedge \ulw^{\mu\beta}D^{\eta\beta-\mu\beta z/2}\right) \ge 20^{\tau+\mu\beta}.
    \end{align}
    Then, with $\Lambda(\eta,z)$ from~\eqref{eq:Lambda-def},
    \begin{align}\label{eq:bridge-proba}
    	   \pr\big(\calA(x_0,x_1) \mid V,w_V\big)
    	   \ge 1 - 3\exp\left(-\big(\theta D^{\Lambda(\eta,z)-\sqrt{\delta}}\big)^{1/4}\right).
    \end{align}
   With the convention that $\infty \cdot 0 = 0$ in~\eqref{eq:Lambda-def}, the statement is also valid when $\alpha=\infty$ or $\beta=\infty$.
\end{lemma}

\begin{proof}[Proof of Lemma \ref{lem:triple-edge-bridge}]
First, when $\beta=\infty$ and $\eta < \mu z$ then $\Lambda(\eta,z)=-\infty$ in \eqref{eq:Lambda-def}, and the condition $\theta D^{\Lambda(\eta,z)-\sqrt{\delta}} > 1$ cannot hold. Hence, we can wlog assume that if $\beta=\infty$ then $\eta \ge \mu z$.
We will first apply Claim~\ref{claim:cheap-edge-to-nice-vertex} to show that most vertices of weight roughly $\ulw D^{z/2}$ close to $x_0$ and $x_1$ are ``good'', i.e., they have a cheap edge to a vertex with weight in $[\ulw,4\ulw]$. We will then apply Claim~\ref{claim:cheap-bridge} to find a cheap edge between some pair of good vertices.

Formally, let $I^+ = [5\ulw D^{z/2}, 20\ulw D^{z/2}]$ and $I^- = [\ulw, 4\ulw]$. Note that $I_+\cap I_-=\emptyset$ for all $z\in[0,d]$. As in Claim~\ref{claim:cheap-edge-to-nice-vertex}, for all $v \in \calN$ let $\calA_{2\ulw,D^\gamma,\ulw^{3\mu}D^{\eta}}(v)=:\calA(v)$ be the event that there is an edge of cost at most $\ulw^{3\mu}D^{\eta}$ in $G'$ from $v$ to a vertex $y \in \calN \cap (B_{D^\gamma}(v) \times I^-)$. Let
\begin{align}\label{eq:Zi-new}
    Z_i := \big\{v \in \calN \cap (B_{D^\gamma}(x_i) \times I^+) \colon \calA(v)\mbox{ occurs}\big\}, \qquad i \in \{0,1\}.
\end{align}
As in~\eqref{eq:N-gamma-eta-z} of Claim~\ref{claim:cheap-bridge}, let $N_{\eta,\gamma,z,10\ulw}(Z_0,Z_1)$ be the set of all edges between $Z_0$ and $Z_1$ of cost at most $\ulw^{3\mu}D^\eta$ and then $I^+$ exactly corresponds to the weight interval $[5\ulw D^{z/2}, 20\ulw D^{z/2}]$ as required for $Z_0\subseteq \calZ(x_0), Z_1\subseteq \calZ(x_1)$ in \eqref{eq:calZ-def}. With $Z_i$ in \eqref{eq:Zi-new}, we now show that 
\begin{equation}\label{eq:triple-bridge-main}
    \pr\big(\calA(x_0,x_1) \mid V,w_V\big) \ge \pr\big(N_{\eta,\gamma,z,\ulw}(Z_0,Z_1) \ne \emptyset \mid V,w_V\big).
\end{equation}
Indeed, suppose there exists $(a,b) \in N_{\eta,\gamma,z,10\ulw}(Z_0,Z_1)$. Since $a \in Z_0$, there exists $x \in \calN \cap (B_{D^\gamma}(a)\times I^-)$ such that $(x,a)$ is an edge of cost at most $\ulw^{3\mu}D^\eta$. Likewise, since $b \in Z_1$, there exists $y \in \calN \cap (B_{D^\gamma}(b) \times I^-)$ such that $(y,b)$ is an edge of cost at most $\ulw^{3\mu}D^\eta$. Since $a \in B_{D^\gamma}(x_0)$ and $b \in B_{D^\gamma}(x_1)$, by the triangle inequality,  $x \in B_{2D^\gamma}(x_0)$ and $y \in B_{2D^\gamma}(x_1)$. Thus $xaby$ is a $(2D^\gamma, 3\ulw^{3\mu}D^\eta, \ulw)$-valid bridge with $\mathrm{dev}_{x_0x_1}\le 2D^\gamma$, as required by $\calA(x_0,x_1)$, showing \eqref{eq:triple-bridge-main}.

Now, for each $i \in \{0,1\}$, using \eqref{eq:calZ-def}, we set $\calZ(x_i) = \calN \cap (B_{D^\gamma}(x_i) \times I^+)$. For \eqref{eq:lem-cheap-bridge} to hold we need that $|Z_i| \ge |\calZ(x_i)|/4$. We prove this by showing that any given vertex in $v \in \calZ(x_i)$ lies in $Z_i$ with probability at least $1/2$, by recalling that in \eqref{eq:Zi-new}, $\calA(v)=\calA_{2\ulw,D^\gamma,\ulw^{3\mu}D^{\eta}}(v)=\calA_{K,D,U}(v)$ in Claim \ref{claim:cheap-edge-to-nice-vertex}. Hence we set $K = 2\ulw, M = 10\ulw D^{z/2}, U = \ulw^{3\mu}D^{\eta}$, $D_{\ref{claim:cheap-edge-to-nice-vertex}}=D^\gamma$, and all other variables to match their current values. We check the requirements of Claim~\ref{claim:cheap-edge-to-nice-vertex}:

By hypothesis in the statement of Lemma \ref{lem:triple-edge-bridge}, $\delta$ is small; $\ulw,D\ggs \delta,w_0$, and $D\ggs \gamma$. Since $M,K \ge \ulw$, it follows that $M,K,D^\gamma \ggs \delta, w_0$, as required above \eqref{eq:KM-xi}. Condition \eqref{eq:KM-xi} itself holds since  $(D^\gamma \wedge (20\ulw^2 D^{z/2})^{1/d}/4^{1/d} \ge D^\gamma/4^{1/d} \wedge \ulw^{1/d} \ge (\log\log\xi\sqrt{d})^{16/\delta}$ by hypothesis, and $(D^\gamma \wedge (20\ulw^2 D^{z/2})^{1/d}/4^{1/d} \le D^\gamma \le \xi\sqrt{d}$ by hypothesis. Condition \eqref{eq:K-D-M} holds since $K=2\ulw \le 2D^{\delta}$ by hypothesis, so since $\delta\lls \gamma$ and $D\ggs\delta$, we have $2\ulw \le D^{\gamma(d/(\tau-1)-\delta)}$, and similarly since $\tau\in(2,3)$ and $\delta \lls \mpar$, $K=2\ulw \le (2\ulw)^{1/(\tau-2+\tau\delta)} \le (10\ulw D^{z/2})^{1/(\tau-2+\delta\tau)}=M^{1/(\tau-2+\delta \tau)}$. Finally, if $\beta=\infty$, then below \eqref{eq:K-D-M} we need to check $U(KM)^{-\mu}\ggs \mpar$. Since wlog we assumed that  $\eta \ge \mu z$, clearly $\eta \ge \mu z/2$. Therefore, $U(KM)^{-\mu}=\ulw^{3\mu}D^{\eta}(20\ulw^2 D^{z/2})^{-\mu} = 20^{-\mu}\ulw^{\mu}D^{\eta-\mu z/2} \ge (\ulw/20)^{\mu}$, which is large since $\ulw$ is large by hypothesis.
Hence, all conditions of Claim~\ref{claim:cheap-edge-to-nice-vertex} are met and \eqref{eq:cheap-edge-to-nice-vertex} applies, and substituting $K = 2\ulw, M = 10\ulw D^{z/2}, U = \ulw^{3\mu}D^{\eta}$ there, the exponent on the rhs of \eqref{eq:cheap-edge-to-nice-vertex} in our setting becomes 
\[
    -2^{-(\tau-1)} \theta \ulw^{-(\tau-1)}\big(D^{d\gamma} \wedge 20\ulw^2D^{z/2}\big)^{1-\delta}\big(1 \wedge (\ulw/20)^{\mu\beta}D^{\eta\beta-\mu\beta z/2}\Big),
\]
where we recognise that this matches $p(\cdot)$ from \eqref{eq:p(.)} up to a factor of at most $20^{1-(\tau+\mu\beta)}$. Since we assumed $p(\cdot)\ge 20^{\tau+\mu\beta}$ in \eqref{eq:p(.)}, for any vertex $v \in \calZ(x_i)=\calN \cap (B_{D^\gamma}(x_i) \times I^+)$, 
\begin{equation}\label{eq:triple-bridge-binom-prob}
    \pr\big(v \in Z_i \mid V,w_V\big) = \pr(\calA(v) \mid V, w_V) \ge 1 - e^{-20} > 1/2.
\end{equation}
Since $I^+$ and $I^-$ are disjoint, the events $\calA(v), \calA(v')$ are functions of disjoint edge sets and are therefore mutually independent conditioned on $(V, w_V)$. Hence, for $i \in \{0,1\}$, $|Z_i|$ is dominated below by a binomial variable with mean $|\calZ(x_i)|/2$. By the standard Chernoff bound (Theorem~\ref{thm:chernoff} with $\lambda=1/2$), 
\begin{equation}\label{eq:triple-bridge-chernoff}
    \pr\big(|Z_i| < |\calZ(x_i)|/4 \mid V,w_V\big) \le e^{-|\calZ(x_i)|/16}.
\end{equation}
To bound $|\calZ(x_i)|$ below in~\eqref{eq:triple-bridge-chernoff}, we will use that $x_0,x_1 \in \calN$ and $\calN$ is a weak $(\delta/4,w_0)$-net as assumed in Setting \ref{set:hierarchy-common}, and apply \eqref{eq:net-defining-crit-eps}. We check if the conditions to apply \eqref{eq:net-defining-crit-eps} in Def.~\ref{def:weak-net} hold. Since $\calZ(x_i)=\calN \cap (B_{D^\gamma}(x_i) \times [5\ulw D^{z/2}, 20\ulw D^{z/2}])$, we set there $r=D^\gamma$ and $w=10\ulw D^{z/2}$, and we must bound $10\ulw D^{z/2}$ above and below. Recall that by hypothesis, $\theta D^{\Lambda(\eta,z)-\sqrt{\delta}} > 1$; this implies $\Lambda(\eta,z) > 0$ and hence $2d\gamma > z(\tau-1)$ using \eqref{eq:Lambda-def}. Since $\delta\lls \gamma,z$, we may therefore assume $z/2 \le d\gamma/(\tau-1) - 2\delta$. Also, we assumed $\ulw \le D^\delta$, so $10\ulw D^{z/2} \le 10D^{d\gamma/(\tau-1)-\delta} \le (D^{\gamma})^{d/(\tau-1)-\delta/4}$, where the second inequality holds since $\gamma<1$ and $D\ggs \delta$. Moreover, since $\ulw\ggs w_0$, we have $10\ulw D^{z/2} \ge w_0$. Thus all conditions in Def.~\ref{def:weak-net} are met, and \eqref{eq:net-defining-crit-eps} here becomes
\[
    |\calZ(x_i)| \ge D^{d\gamma(1-\delta/4)}\ell(10\ulw D^{z/2})(10\ulw D^{z/2})^{-(\tau-1)}\ge D^{d\gamma(1-\delta/4) - (\tau-1+\delta/4)(\delta + z/2)},
\]
where the second inequality holds by Potter's bound since $D^{\delta} \ge \ulw\gg \delta$. 
The exponent of $D$ on the right-hand side is
\[
    d\gamma-(\tau-1)z/2-\delta(d\gamma/4+z/8+\tau-1+\delta/4) \ge  d\gamma/2-(\tau-1)z/2 \ge \Lambda(\eta,z)/4,
\]
where we used $\delta\lls \gamma$ and then the formula of $\Lambda(\eta,z)$ in~\eqref{eq:Lambda-def}.
So, $|\calZ(x_i)| \ge D^{\Lambda(\eta,z)/4}$ in~\eqref{eq:triple-bridge-chernoff}, and since $D\ggs \delta$,
\begin{equation}\label{eq:triple-bridge-N-bound}
    \pr\big(|Z_i| < |\calZ(x_i)|/4 \mid V,w_V\big) \le \exp(-D^{\Lambda(\eta,z)/4}/16)\le \exp(-(\theta D^{\Lambda(\eta,z) - \sqrt{\delta}})^{1/4}).
\end{equation}
Returning to the event $\calA(x_0, x_1)$ in \eqref{eq:triple-bridge-main},
let $\calA'$ be the event that $|Z_i| \ge |\calZ(x_i)|/4$ for each $i \in \{0,1\}$, and suppose that $\calA'$ occurs. We apply Claim~\ref{claim:cheap-bridge}, conditioned on the values of $Z_0$ and $Z_1$, to lower-bound the rhs of~\eqref{eq:triple-bridge-main}. In the statement of Claim~\ref{claim:cheap-bridge}, we will take $x = x_0$, $y = x_1$, $Z_x = Z_0$, $Z_y = Z_1$, $\ulw_{\ref{claim:cheap-bridge}} = 10\ulw$, and all other variables to match their current values. The event $N_{\eta,\gamma,z,10\ulw}(Z_0,Z_1) \ne \emptyset$ of~\eqref{eq:triple-bridge-main} requires a low-cost edge between the set $Z_0$ and $Z_1$, connecting vertices with weights in $I_+$. Given $(V, w_V)$, the existence of such an edge $(u,v)$ is independent of the events $\calA(u), \calA(v)$ since in $\calA(\cdot)$ the other endpoint of the edge has weight  $I_-$, and $I_+\cap I_-=\emptyset$. We now check the requirements of Claim~\ref{claim:cheap-bridge}: it requires $z\in[0,d]$ that we assumed, and $2d\gamma > z(\tau-1)$. The latter holds since here we assume $\theta D^{\Lambda(\eta,z)-\sqrt{\delta}} > 1$ implying that $\Lambda(\eta,z) > 0$, so $2d\gamma > z(\tau-1)$ then follows from \eqref{eq:Lambda-def}. Second, here we assume $\ulw \ge (\log\log\xi\sqrt{d})^{16d/\delta} \ge (\log\log D^{\gamma})^{16d/\delta}$, and also $D\ggs \gamma,c_H,w_0$. So $\ulw \ge w_0 \vee 4(c_H+2)^d \vee 4000$ and $F_L((\ulw/4000)^{\mu})\ge 1/2$ as required above \eqref{eq:calZ-def}. The requirement on $D^\gamma$ here is more restrictive than in Claim \ref{claim:cheap-bridge}, so all requirements hold. Then, since here we have $10\ulw$, \eqref{eq:lem-cheap-bridge} turns into the following, which we then estimate by using that $\ulw \le D^\delta$, that $\delta\lls \mpar$ and that $D\ggs \delta$,
\begin{align*}
    \pr\big(N_{\eta,\gamma,z,10\ulw}(Z_0,Z_1) &= \emptyset \mid \calA',V,w_V\big) \le \exp\Big({-}\theta (10\ulw)^{-2(\tau-1)}D^{\Lambda(\eta,z)-2\gamma d\delta/3}\Big) \\
    &\le \exp\Big({-}\theta D^{\Lambda(\eta,z) - \sqrt{\delta}}\Big)
\le\exp\Big({-}\big(\theta D^{\Lambda(\eta,z) - \sqrt{\delta}}\big)^{1/4}\Big),
\end{align*}
since we assumed $\theta D^{\Lambda(\eta,z)-\sqrt{\delta}} > 1$.
Since $\calA'=\{Z_0\ge |\calZ(x_0)|/4,  Z_1\ge |\calZ(x_1)|/4\}$, combining this with~\eqref{eq:triple-bridge-N-bound} and a union bound, the result in \eqref{eq:bridge-proba} follows.
\end{proof}

We now construct a hierarchy by repeatedly applying Lemma~\ref{lem:triple-edge-bridge} to find a set of valid bridges as in Observation~\ref{obs:valid-bridges}, using an iterative construction (Def.~\ref{def:iter-construct}) to mitigate independence issues. We now define the second function that will be crucial in determining the optimal exponents $\Delta_0, \eta_0$ in \eqref{eq:Delta_0}, \eqref{eq:eta_0}.  For all $\eta > 0$, $z \ge 0$, and integers $R \ge 2$, we define
\begin{align}\label{eq:Phi-def}
    \Phi(\eta,z) &:= \Big[d\gamma \wedge \frac{z}{2}\Big] + \Big[0 \wedge \beta\Big(\eta - \frac{\mu z}{2}\Big)\Big].
\end{align}
Recall the $(\gamma, U, \olw, c_H)$-hierarchy of depth $R$ from Def.~\ref{def:hierarchy}, and $\Lambda(\eta,z)$ from \eqref{eq:Lambda-def}.
\begin{lemma}[Hierarchy with low weights $\ulw$]\label{lem:hierarchy-intermediate-weights}
    Consider Setting~\ref{set:hierarchy-common}, and let $y_0,y_1\in\calN$ with $|y_0-y_1|=\xi$. 
    Let $z\in[0,d]$, $\eta \ge 0$, and let $0<\delta\lls \gamma,\eta, z,\mpar$ be such that 
   $\Lambda(\eta,z) \ge 2\sqrt{\delta}$ and either $z=0$ or $\Phi(\eta,z)\ge\sqrt{\delta}$.
Let  $\xi\ggs \gamma,\eta,z,\delta, w_0$. Let $R\ge2$ be an integer satisfying
    \begin{align}
    \xi^{\gamma^{R-1}} &\ge (\log\log\xi\sqrt{d})^{16d/\delta^2} \qquad \mbox{and}\qquad R/\theta \le (\log\log\xi)^{1/\sqrt{\delta}},\label{eq:R-condition-in-prop}
\\
\text{and let }\qquad \ulw&:= \xi^{\gamma^{R-1}\delta}.\label{eq:ulw-in-prop}
    \end{align}
 Let $\calX_{low\textnormal{-}hierarchy}(R,\eta,y_0,y_1)$ be the event that $G'$ contains a $(\gamma,3\ulw^{3\mu}\xi^{\eta},\ulw,2)$-hierarchy $\calH_{low}$ of depth $R$ with first level $\calL_1 = \{y_0,y_1\}$,  fully contained in $\calN$, with $\mathrm{dev}_{y_0y_1}(\calH_{low})\le 4\xi^\gamma$. 
    Then 
\begin{align}\label{eq:hierarchy-intermediate-prob}
        \pr\big(\calX_{low\textnormal{-}hierarchy}(R,\eta,y_0,y_1)\mid V,w_V\big) \ge 1 - \exp\big({-}(\log\log\xi)^{1/\sqrt{\delta}}\big)=:1-\mathrm{err}_{\xi,\delta}.
    \end{align}
   With the convention that $\infty \cdot 0 = 0$, the statement is also valid when $\alpha=\infty$ or $\beta=\infty$.
\end{lemma}

\begin{proof}
To construct a $(\gamma,3\ulw^{3\mu}\xi^{\eta},\ulw,2)$-hierarchy in $\calN$, we will use an iterative cost construction of $R-1$ rounds from Def.~\ref{def:iter-construct} on $G'$. Recall from Setting \ref{set:hierarchy-common} that $G'$ with given $V, w_V$ is a $\theta$-percolated CIRG. By Remark \ref{remark:cirg}, $G'$ is a CIRG itself (also when $\theta=\theta_n$) with distribution $\{\calG^{\theta}|V,w_V\}$. In the $i$'th round we will construct all bridges of the $i$'th level of the hierarchy at once, using Lemma~\ref{lem:triple-edge-bridge} $2^{i-1}$ times to find each bridge in the level. We will use Prop.~\ref{prop:multi-round-exposure} to deal with conditioning between rounds, and union bounds to deal with conditioning within rounds. 
For $2 \le i \le R$, in the $i$'th round we we will set the constraint $\calU_i$ so that the chosen set $\calE_i$ in Def.~\ref{def:iter-construct}\eqref{item:iter5} consists of a $(2\xi^{\gamma^{i-1}}, 3\ulw^{2\mu}\xi^{\eta}, \ulw)$-valid bridge for $\tilde y_{\sigma 0}$ and $\tilde y_{\sigma1}$ for all $\sigma \in \{0,1\}^{i-2}$,
where $\tilde y_{\sigma 0};$ and $\tilde y_{\sigma1}$ are (all) endpoints of bridges from the previous levels.
In other words, $\calE_i$ will contain all the necessary bridges at the $i$'th level of the hierarchy for $\calX_{low\textnormal{-}hierarchy}(R,\eta,y_0,y_1):=\calX_{low\textnormal{-}h}$. Since $\calL_1=\{y_0, y_1\}$ contains no bridges yet, we denote by $(G',\calE_2,\calF_2,\calU_2),\ldots,(G',\calE_R,\calF_R,\calU_R)$ the consecutive rounds of the iterative construction.
    
We next inductively define the valid edge tuples $\calF_i$ in Def.~\ref{def:iter-construct}\eqref{item:iter1} and the cost constraint event $\calU_i$ in Def.~\ref{def:iter-construct}\eqref{item:iter5}; and the vertices $\tilde y_{\sigma}$ for all $\sigma \in \{0,1\}^{i}$.
Assume that $E_1, \dots, E_{i-1}$ is already given, i.e., we constructed $(\calL_j)_{j\le i-1}$. For each $\sigma \in \{0,1\}^{i-2}$ consider the vertices $\tilde y_{\sigma0}, \tilde y_{\sigma1}\in \{0,1\}^{i-1}$. 
Set $D_{i}=\xi^{\gamma^i}$ as in Observation \ref{obs:valid-bridges}, and write $\calP(\sigma)$ for the set of all possible paths (i.e., sequence of vertices) contained in $\calN$ between all $y\in \calN \cap (B_{2D_{i-1}}(\tilde y_{\sigma0})\times [\ulw, 4\ulw])$ and all $y'\in \calN \cap (B_{2D_{i-1}}(\tilde y_{\sigma1})\times [\ulw, 4\ulw])$ (so that if $\widetilde P_{\sigma}\in\calP(\sigma)$, then $\widetilde P_{\sigma}$ satisfies both \eqref{eq:bridge-condition-b1}, \eqref{eq:bridge-condition-b2}
in Def.~\ref{def:valid-bridges}). Given $V, w_V, E_1, \dots, E_{i-1}$, define now a tuple $\underline t$ to be level-$i$ admissible if it contains exactly one such potential path from $\calP(\sigma)$ for each $\sigma\in\{0,1\}^{i-2}$, and let $\calF_i$ be the collection of all level-$i$ admissible tuples, with an arbitrary ordering.
Let $\calU_i$ be the event that each potential path in the chosen tuple is present in the graph under consideration, and has round-$i$ marginal cost at most $3\ulw^{3\mu}\xi^{\eta}$ in \eqref{eq:marginal-cost}.     
Following Def.~\ref{def:iter-construct}\eqref{item:iter5} and Prop \ref{prop:multi-round-exposure}, we set $\calE_i^{G'}$ to be the first tuple of bridges (in the ordering of $\calF_i$) for which $\calU_i$ occurs, and we set $\calE_i^{G'}$ to $\mathtt{None}$ if no such tuple exists. Moreover, we define $\tilde y_{\sigma00} := \tilde y_{\sigma0}$, $\tilde y_{\sigma11} := \tilde y_{\sigma1}$, and $\tilde y_{\sigma01}$ and $\tilde y_{\sigma10}$ to be the endpoints of the bridge $\widetilde P_\sigma$ in the chosen tuple $\calE_i^{G'}$, or $\texttt{None}$ if $\calE_i^{G'}=\texttt{None}$. 
This gives the iterative cost construction $I_{G'} = ((G',\calE_i^{G'},\calF_i,\calU_i)\colon i \in \{2,\ldots,R\})$.
Note that the criteria above for $\calP(\sigma), \calF_i, \calU_i$ exactly matches Observation \ref{obs:valid-bridges} with $c_H\!=\!2$ and marginal cost of each bridge $\widetilde P_\sigma$ at most $U\!=\!3\ulw^{3\mu}\xi^{\eta}$, implying \eqref{eq:bridge-condition-b3}, i.e., each chosen bridge $\widetilde P_\sigma\in E_i$ is $(2\xi^{\gamma^{i-1}}, 3\ulw^{3\mu}\xi^{\eta}, \ulw)$-valid for $y_{\sigma0},y_{\sigma1} \in \{0,1\}^{i-1}$ wrt the chosen edges in earlier rounds, i.e., wrt $S=\cup_{j\le i-1}E_j$. Since all vertices in $\{P_\sigma\}_\sigma$ are contained in a $2(\xi^\gamma+\xi^{\gamma^2}+\dots+\xi^{\gamma^{R-1}})\le 4\xi^\gamma$ ball around $y_0$ and $y_1$, respectively, the deviation requirement is also satisfied, and so by Obs.~\ref{obs:valid-bridges}, if $I_G'$ succeeds then $\{\tilde y_\sigma\}_{ \sigma\in\{0,1\}^R}$ is a $(\gamma,3\ulw^{3\mu}\xi^{\eta},\ulw,2)$-hierarchy as needed in $\calX_{low\textnormal{-}h}$.
   
Following Prop.~\ref{prop:multi-round-exposure}, let $r=R-1$ and $\theta_i\equiv 1/(R\!-\!1)$ there, and let $H_2, \ldots, H_{R}$ be independent $1/(R\!-\!1)$-percolations of $\{\calG^{\theta} \mid V, w_W\}$, i.e., with distribution $\{\calG^{\theta/(R-1)}\mid V, w_W\}$ from Def.~\ref{def:percolated}.   Then, the iterative cost construction $I_H = ((H_i,\calE_i^H,\calF_i,\calU_i)\colon i \in \{2,\ldots,R\})$ on $H_2,\ldots,H_{R}$ in Prop.~\ref{prop:multi-round-exposure} uses the same $\calF_i, \calU_i$ as $I_{G'}$.  Let us denote by $\calA_{I_H}(V, w_V, E_2,\ldots,E_{i-1})$ the event that $\calE_j^H = E_{j}$ for $2\!\le\! j\!\le\! i-1$. So applying Prop.~\ref{prop:multi-round-exposure}, \eqref{eq:multi-round-exp-goal} turns into
\begin{equation}
    \begin{aligned}\label{eq:hierarchy-intermediate-weights-0}
       \pr\big(\calX_{low\textnormal{-}h} &\mid V,w_V\big)\ge \pr\big(I_{G'} \textnormal{ succeeds} \mid V,w_V\big) \\
       &\ge 
        \min_{E_2,\ldots,E_{R} \ne \mathtt{None}} \prod_{i=2}^R\pr\big(\calE_i^H \ne \mathtt{None} \mid \calA_{I_H}(V, w_V, E_2,\ldots,E_{i-1})\big)\\
        &\ge1-\sum_{i=2}^R\max_{E_2,\ldots,E_{i-1} \ne \mathtt{None}}\pr\big(\calE_i^H = \mathtt{None} \mid \calA_{I_H}(V, w_V, E_2,\ldots,E_{i-1})\big),
  \end{aligned}
\end{equation}
by a union bound over all rounds.
We now break the rhs of~\eqref{eq:hierarchy-intermediate-weights-0} down into bridge existence events under simpler conditioning. 
For each $2 \le i \le R$ and $\sigma0, \sigma1\in \{0,1\}^{i-1}$, let 
\begin{equation}\label{eq:Ai-event-def}
\calA_{i}(\tilde y_{\sigma0},\tilde y_{\sigma1}):=\{ \exists \widetilde P_\sigma \in E_i: (2\xi^{\gamma^{i-1}},3\ulw^{2\mu}\xi^{\gamma^{i-2}\eta}, \ulw)\text{-valid for } \tilde y_{\sigma0}, \tilde y_{\sigma1} \text{ wrt } S=\emptyset\}.
\end{equation}
This is a stronger condition than what is required for a $(\gamma, 3\ulw^{3\mu}\xi^\eta, \ulw, 2)$-hierarchy to exist in Obs.~\ref{obs:valid-bridges}, since $\gamma^{i-2}\eta\le\eta$ and validity with respect to $\emptyset$ implies validity with respect to any set of edges. Conditioned on $ \calA_{I_H}(V, w_V, E_2,\ldots,E_{i-1})$ so that none of the $(E_j)_{j\le i-1}$ equals $\mathtt{None}$, the event $\calE_{i}^H = \mathtt{None}$ occurs only if for some pair $\sigma0, \sigma1 \in \{0,1\}^{i-1}$, the event $\neg\calA_{i}(\tilde y_{\sigma 0}, \tilde y_{\sigma1})$ occurs; hence by a union bound,~\eqref{eq:hierarchy-intermediate-weights-0} implies   \begin{align}\label{eq:hierarchy-intermediate-weights-2}
\pr\big(\calX_{low\textnormal{-}h}\! \mid\! V,w_V\big) \ge 1\!-\!\sum_{i=2}^R 2^{i-2}\!\!\!\!\!\!\!\!\!\max_{\substack{\sigma \in \{0,1\}^{i-2}\\E_2,\ldots,E_{i-1} \ne \mathtt{None}}}\!\!\!\!\!\pr\big(\neg\calA_{i}(\tilde y_{\sigma 0},\tilde y_{\sigma 1})\! \mid\! \calA_{I_H}(V, w_V, E_2,\ldots,E_{i-1})\big).
    \end{align}
    Recall that conditioned on $(V,w_V)$, the graphs $H_2, \dots H_{R-1}$ are iid $\{\calG^{\theta/(R-1)}\mid V, w_W\}$. So, the events in $\calA_{I_H}(V, w_V, E_2,\ldots,E_{i-1})$ are contained in the $\sigma$-algebra generated by $H_2,\ldots,H_{i-1}$, i.e., independent of $H_i$ and thus of $\neg\calA_{i}(\tilde y_{\sigma 0},\tilde y_{\sigma 1})$. Hence~\eqref{eq:hierarchy-intermediate-weights-2} simplifies to
    \begin{align}\label{eq:hierarchy-intermediate-weights-3}
        \pr\big(\calX_{low\textnormal{-}h} \mid V,w_V\big) \ge 1-\sum_{i=2}^R 2^{i-2}\max_{\tilde y_{\sigma0},\tilde y_{\sigma1}\ne\mathtt{None}}\pr\big(\neg\calA_i(\tilde y_{\sigma 0},\tilde y_{\sigma 1}) \mid V,w_V\big),
    \end{align}
    where the maximum is taken over all possible values of $(\tilde y_{\sigma0},\tilde y_{\sigma1})$ occurring in non-\texttt{None} $E_{i-1}$. 
    Finally, we will upper-bound the probabilities on the rhs of~\eqref{eq:hierarchy-intermediate-weights-3} using Lemma~\ref{lem:triple-edge-bridge}. Let $2 \le i \le R$, let $\sigma \in \{0,1\}^{i-2}$, and let $\tilde y_{\sigma0},\tilde y_{\sigma1}$ be a possible non-\texttt{None} realisation of the embedding.    Recall $D_i=\xi^{\gamma^{i}}$. Then the event \eqref{eq:Ai-event-def} requires a $(D_{i-2}^\gamma, 3\ulw^{3\mu}D_{i-2}^\eta, \ulw)$-valid bridge $P_\sigma$, which formally matches Lemma~\ref{lem:triple-edge-bridge} with $D:=D_{i-2}, \tilde y_{\sigma0}:=x_0, \tilde y_{\sigma1}:=x_1$ there and the graph $H_i\sim \{\calG^{\theta/(R-1)}\mid V, w_V\}$ in place of $G'$ there, i.e., with $\theta_{\ref{lem:triple-edge-bridge}}:=\theta/(R-1)$.

    We check the conditions of Lemma~\ref{lem:triple-edge-bridge} in order of their appearance.
     $z\in[0,d], \eta,\delta>0$ and $\delta\lls \gamma,\eta, z$ is assumed both here and there. The assumption $\xi\ggs \gamma,\eta, z, \delta, w_0$ here implies $D_{i-2}\ggs \gamma,\eta, z, \delta, w_0$ since by \eqref{eq:R-condition-in-prop}  $D_{i-2}^\gamma  \ge \xi^{\gamma^{R-1}} \ge (\log\log\xi\sqrt{d})^{16d/\delta^2}$. The latter also implies the requirement on $D^\gamma$ in Lemma \ref{lem:triple-edge-bridge}. Similarly, the upper bound requirement holds since $D_{i-2}^{\gamma} \le \xi \le \xi\sqrt{d}$. We now check whether $\theta D^{\Lambda(\eta,z)-\sqrt{\delta}} > 1$ holds in Lemma \ref{lem:triple-edge-bridge} for our choices.
     Since we assumed here $\Lambda(\eta,z) \ge 2\sqrt{\delta}$, and also \eqref{eq:R-condition-in-prop}, we estimate
     \begin{equation}\label{eq:inter-weights-theta-R}
        \tfrac{\theta}{R-1}D_{i-2}^{\Lambda(\eta,z)-\sqrt{\delta}} \ge (\log\log\xi)^{-1/\sqrt{\delta}}\cdot D_{i-2}^{\sqrt{\delta}} \ge (\log\log\xi)^{15/\sqrt{\delta}} > 1.
     \end{equation}
 Next we need to check whether $x_0=\tilde  y_{\sigma0}, x_1=\tilde y_{\sigma1}$ satisfies $|x_0- x_1|\le c_HD_{i-2}$. This is true since $\tilde y_{\sigma0},\tilde y_{\sigma1}$ are possible non-\texttt{None} values coming from chosen tuples $E_1, \dots, E_{i-1}$; and by construction of $\calP(\sigma)$ above, we required that $|\tilde y_{\sigma'00}-\tilde y_{\sigma'01}|, |\tilde y_{\sigma'10}-\tilde y_{\sigma'11}|\le 2D_{i-1}$ for all $\sigma'\in \{0,1\}^{i-2}$, which, when shifting indices yields exactly that $|\tilde y_{\sigma0}-\tilde y_{\sigma1}|\le 2D_{i-2}$ for all $\sigma\in\{0,1\}^{i-2}$. Next we check the criterion on $\ulw$ in Lemma \ref{lem:triple-edge-bridge}. Here, $\ulw$ is defined in \eqref{eq:ulw-in-prop}, hence, using \eqref{eq:R-condition-in-prop}, $\ulw = \xi^{\gamma^{R-1}\delta} \ge (\log\log\xi\sqrt{d})^{16d/\delta}$ as required. This also implies $\ulw \ggs \delta, w_0$ since $\xi \ggs \delta, w_0$. 
 Moreover, $\ulw = \xi^{\gamma^{R-1}\delta} \le D_{i-2}^{\delta}=\xi^{\gamma^{i-2}\delta}$ holds since $i-2\le R-2$ and $\gamma<1$. Next, we check \eqref{eq:p(.)}, which can be lower bounded by omitting the prefactor $\ulw^{\mu\beta}$ in the last factor (the minimum):
      \begin{align}\label{eq:p(.)-bound}
            p(D_{i-2},\ulw,\tfrac{\theta}{R-1},\eta,z) \ge \tfrac{\theta}{R-1}\ulw^{-(\tau-1)} \left(D_{i-2}^{d\gamma } \wedge \ulw^2 D_{i-2}^{z/2}\right)^{1-\delta} D_{i-2}^{[0 \wedge \beta(\eta-\mu z/2)]}.
        \end{align}
  We distinguish cases wrt $z$ to handle the minimum in the middle of the rhs. If $z=0$, then $\ulw^2D_i^{z/2} = \ulw^2 = \xi^{2\gamma^{R-1}\delta} \le \xi^{\gamma^{i-1}d} = D_{i-2}^{\gamma d}$, where the inequality holds because $i\le R$ and $\delta \lls \gamma$. Moreover $0 \le \eta-\mu z/2$ in that case, so when $z=0$, equation~\eqref{eq:p(.)-bound} becomes
        \begin{align}\label{eq:p-bound-z-0}
            p(D_{i-2},\ulw,\tfrac{\theta}{R-1},\eta,z) \ge \tfrac{\theta}{R-1}\ulw^{2(1-\delta)-(\tau-1)} = \tfrac{\theta}{R-1}\ulw^{3-\tau-2\delta} \ge \tfrac{\theta}{R-1}\ulw^{\sqrt{\delta}},
        \end{align}
        where the last inequality holds because $\delta \lls \mpar$. 
               If, however, $z\neq 0$, then we assumed that  $\Phi(\eta,z)\ge\sqrt{\delta}$ in \eqref{eq:Phi-def}. Using again $\ulw\ge1$, we lower bound~\eqref{eq:p(.)-bound} in this case
         \begin{align*}
            p(D_{i-2},\ulw,\tfrac{\theta}{(R-1)},\eta,z) \ge \tfrac{\theta}{R-1}\ulw^{-(\tau-1)} D_{i-2}^{(1-\delta)[d\gamma  \wedge z/2] + [0 \wedge \beta(\eta-\mu z/2)]}\ge \tfrac{\theta}{R-1} D_{i-2}^{\Phi(\eta,z)-\delta(\tau-1+d)},
        \end{align*}
        where we used that $\ulw\le D_{i-2}^\delta$ implies $\ulw^{-(\tau-1)} \ge D_{i-2}^{-\delta(\tau-1)}$ and $d\gamma \wedge z/2 \le d$ (since $\gamma<1$) to obtain the last inequality.
        Since $\delta$ is small, $\Phi(\eta,z)\ge\sqrt{\delta}$, and $\ulw \le D_{i-2}^{\delta}$, this implies
        \begin{align}\label{eq:p-bound-z-not-0}
            p(D_{i-2},\ulw,\tfrac{\theta}{R-1},\eta,z) \ge \tfrac{\theta}{R-1} D_{i-2}^{\delta} \ge \tfrac{\theta}{R-1}\ulw \ge \tfrac{\theta}{R-1}\ulw^{\sqrt{\delta}},
        \end{align}
the same lower bound as in \eqref{eq:p-bound-z-0} for $z=0$. Thus for all $z\in[0,d]$, using \eqref{eq:R-condition-in-prop} and \eqref{eq:ulw-in-prop},
        \begin{align*}
             p(D_{i-2},\ulw,\tfrac{\theta}{R-1},\eta,z) \ge \tfrac{\theta}{R-1}\ulw^{\sqrt{\delta}} \ge \tfrac{\theta}{R-1}(\log\log\xi)^{16/\sqrt{\delta}} \ge (\log\log\xi)^{15/\sqrt{\delta}} \ge 20^{\tau+\mu\beta},
                     \end{align*}
        where the last inequality holds because $\xi \ggs \delta,\mpar$. With this, all conditions of Lemma~\ref{lem:triple-edge-bridge} are satisfied, so combining \eqref{eq:bridge-proba} with~\eqref{eq:hierarchy-intermediate-weights-3} and then using the lower bound in \eqref{eq:inter-weights-theta-R} yields \begin{equation}\label{eq:hierarchy-intermediate-weights-4}
    \begin{aligned}
        \pr\big(&\calX_{low\textnormal{-}h} \mid V,w_V\big) \ge 1 - 3\sum_{i=2}^{R}2^{i-2}\exp\Big({-}\Big[\tfrac{\theta}{R-1} D_{i-2}^{\Lambda(\eta,z)-\sqrt{\delta}}\Big]^{1/4}\Big)\\
        & \ge 1 - 3\sum_{i=2}^{R}2^{i-2}\cdot \exp\big({-}(\log\log\xi)^{3/\sqrt{\delta}}\big)\ge 1 - 2^{R+1}\exp\big({-}(\log\log\xi)^{3/\sqrt{\delta}}\big).
        \end{aligned}
    \end{equation}
  Finally, in \eqref{eq:R-condition-in-prop} the estimate $R\le(\log\log\xi)^{1/\sqrt{\delta}}$ can be used to upper bound $2^{R+1}$, yielding the required inequality in \eqref{eq:hierarchy-intermediate-prob}.
\end{proof}

Lemma~\ref{lem:hierarchy-intermediate-weights} constructed a hierarchy with bridge endpoints $\tilde y_\sigma$ of weight roughly $\ulw = \xi^{\gamma^{R-1}\delta}$. This weight is too low to connect the final gaps (siblings) in the hierarchy via short paths. The next lemma extends this hierarchy to a new one with endpoints $y_\sigma$ of weight roughly $\olw := \xi^{d\gamma^{R-1}/2}$, where connecting the gaps is possible. The proof follows a very similar structure as for Lemma~\ref{lem:hierarchy-intermediate-weights}, with just two rounds of exposure, so we defer it to Appendix~\ref{app:two-rounds-proof}. 
Recall the $(\gamma, U, \olw, c_H)$-hierarchy of depth $R$ from Def.~\ref{def:hierarchy}, $\Lambda(\eta,z), \Phi(\eta,z)$ from \eqref{eq:Lambda-def} and \eqref{eq:Phi-def}, respectively.

\begin{restatable}[Hierarchy with high weights $\olw$]{lemma}{LemmaHierarchyHighWeights}
\label{lem:hierarchy-final-weights}
    Consider Setting~\ref{set:hierarchy-common}, and let $y_0,y_1\in\calN$ with $|y_0-y_1|=\xi$. 
    Let $z\in[0,d], \eta \ge 0$,
    and let $0<\delta\lls \gamma,\eta,z,\mpar$ be such that $\Lambda(\eta,z) \ge 2\sqrt{\delta}$ and either $z=0$ or $\Phi(\eta,z)\ge\sqrt{\delta}$.
    Let $\xi\ggs \gamma,\eta,z,\theta, \delta,w_0$. Let $R\ge2$ be an integer satisfying $\xi^{\gamma^{R-1}} \ge (\log\log\xi\sqrt{d})^{16d/\delta^2}$ and $R\le (\log\log\xi)^2$, and set
    \begin{align}\label{eq:c_H}
        \olw:= \xi^{\gamma^{R-1}d/2}, \qquad c_H:=8\bigg(1+\left\lceil\frac{\log(d/\delta)}{\log(1/(\tau-2+2d\tau\delta))}\right\rceil\bigg).
    \end{align}
    Let $\calX_{high\textnormal{-}hierarchy}(R,\eta,y_0,y_1)$
    be the event that $G'$ contains a $(\gamma,c_H\olw^{4\mu}\xi^{\eta},\olw,c_H)$-hierarchy $\calH_{high}$ of depth $R$ with first level $\calL_1 = \{y_0,y_1\}$, fully contained in $\calN$, and $\mathrm{dev}_{y_0y_1}(\calH_{high})\le 2c_H\xi^\gamma$. 
    Then
\begin{align}\label{eq:hierarchy-final-prob}
        \pr\big(\calX_{high\textnormal{-}hierarchy}(R,\eta,y_0,y_1) \mid V,w_V\big) \ge
    	1 - \exp\big({-}(\log\log\xi)^{13}\big);
    \end{align}
    under the convention that $\infty \cdot 0 = 0$, the statement is also valid when $\alpha=\infty$ or $\beta=\infty$.
\end{restatable}

The hierarchy constructed in Lemma~\ref{lem:hierarchy-final-weights} is a "broken path" formed by the bridge paths between the starting vertices $y_0, y_1\in \calN$. The next proposition connects the endpoints of bridge-paths and constructs a connected path via common neighbours using Claim \ref{claim:common-neighbour}, but not yet between $y_0, y_1$, only between $y_{0_{R-1}1}$ and $y_{1_{R-1}0}$, the closest vertices to $y_0, y_1$ in the hierarchy constructed in Lemma~\ref{lem:hierarchy-final-weights}. Connecting $y_0$ to $y_{0_{R-1}1}$ and $y_1$ to  $y_{1_{R-1}0}$ needs different techniques, since $y_0,y_1$ have typically lower weights than $\olw$ in \eqref{eq:c_H}, see Section \ref{sec:endpoints}. The proof, which again uses two rounds of exposure and otherwise consists of plugging in appropriate values for Lemma \ref{lem:hierarchy-final-weights} and Claim \ref{claim:common-neighbour}, can be found in Appendix~\ref{app:two-rounds-proof}.

\begin{restatable}[Path from hierarchy]{proposition}{PropositionPathFromHierarchy}
\label{prop:path-from-hierarchy}
    Consider Setting~\ref{set:hierarchy-common}, and let $y_0,y_1\in\calN$ with $|y_0-y_1|=\xi$. 
    Let $z\in[0,d], \eta \ge 0$.
    Let $0<\delta \lls \gamma,\eta,z, \mpar$ be such that $\Lambda(\eta,z) \ge 2\sqrt\delta$ and either $z=0$ or $\Phi(\eta,z)\ge\sqrt{\delta}$. Let $\xi\ggs \gamma,\eta,z,\theta,\delta,w_0$. 
    Let $R\ge2$ be an integer satisfying $\xi^{\gamma^{R-1}} \ge (\log\log\xi\sqrt{d})^{16d/\delta^2}$ and $R \le (\log\log\xi)^2$, let $\olw:= \xi^{\gamma^{R-1}d/2}$, and $c_H$ be as in~\eqref{eq:c_H}. Let $\calX_{high\textnormal{-}path}=\calX_{high\textnormal{-}path}(R,\eta,y_0,y_1)$ be the event that $G'$ contains a path $\pi_{y_0^\star, y_1^\star}$ fully contained in $\calN$ between some vertices $y^{\star}_0 \in \calN \cap (B_{c_H\xi^{\gamma^{R-1}}}(y_0) \times [\olw, 4\olw])$ and $y^{\star}_1 \in \calN \cap (B_{c_H\xi^{\gamma^{R-1}}}(y_1)\times [\olw, 4\olw])$ with cost $\cost{\pi_{y_0^\star y_1^\star}}\le c_H 2^R\olw^{4\mu}\xi^{\eta}$ and $\mathrm{dev}_{y_0y_1}(\pi_{y_0^\star y_1^\star})\le 3c_H\xi^\gamma$.     
    Then 
\begin{align}\label{eq:high-path}
&\pr\big(\calX_{high\textnormal{-}path} \mid V,w_V\big) \ge
    	1 - 2\exp\big({-}(\log\log\xi)^{13}\big);
    \end{align}
    under the convention that $\infty \cdot 0 = 0$, the statement is also valid when $\alpha=\infty$ or $\beta=\infty$.
\end{restatable}

\subsection{Cost optimisation of the constructed paths }\label{sec:choices}
In this section we apply Proposition~\ref{prop:path-from-hierarchy} and optimise the cost of the path $\pi_{y_0^\star y_1^\star}$ constructed there, yielding either polylogarithmic (Corollary~\ref{cor:computations-polylog}) or polynomial cost-distances (Corollary~\ref{cor:computations-polynomial}). The cost of $\pi_{y_0^\star y_1^\star}$ will dominate the cost of the eventual path between $0, x$. These corollaries are rather immediate: we choose appropriate values of $\gamma,\eta,z,R$, apply Proposition~\ref{prop:path-from-hierarchy}, and read off the cost of $\pi_{y_0^\star y_1^\star}$. There are four possible optimal choices of $\gamma,\eta, z,R$ depending on the model parameters, and verifying that the conditions of Setting~\ref{set:hierarchy-common} and Proposition~\ref{prop:path-from-hierarchy} hold for these choices and calculating the resulting path's cost requires some work. Thus, we defer a formal proof of Corollaries~\ref{cor:computations-polylog} and~\ref{cor:computations-polynomial} to Appendix~\ref{app:path-proofs}, and instead focus on why these four optimisers arise and what they mean on a qualitative level.

Thus, in Prop.~\ref{prop:path-from-hierarchy}, disregarding constant factors, our goal is to minimise the cost bound $\calC(\pi_{y_0^\star y_1^\star})\le  2^R\olw^{4\mu}|x|^\eta = 2^R|x|^{2\gamma^{R-1}d\mu + \eta}$ by choosing $\gamma, \eta,z,R$ optimally. To put our choices into context, we first summarise the construction of the path in Proposition~\ref{prop:path-from-hierarchy} (drawing on the proofs of Lemmas~\ref{lem:triple-edge-bridge}--\ref{lem:hierarchy-final-weights}). We first embed a hierarchy $\{y_\sigma\}_{\sigma \in \{0,1\}^R}$ by embedding its bridge-paths. Each bridge of $\{y_\sigma\}_{\sigma \in \{0,1\}^R}$ from $y_{\sigma 01}$ to $y_{\sigma 10}$ contains a single long edge obtained via Claim~\ref{claim:cheap-bridge} from a vertex near $y_{\sigma 01}$ to a vertex near $y_{\sigma 10}$ (see Lemma~\ref{lem:triple-edge-bridge}), which is then extended to first a 3-edge then a multiple-edge bridge-path using Claims~\ref{claim:cheap-edge-to-nice-vertex} and~\ref{claim:cheap-path-to-larger-weight}. We informally refer to this long edge as a \emph{bridging edge}. Having obtained the hierarchy, we then join pairs of level-$R$ siblings (which we informally call \emph{gaps}) using Claim~\ref{claim:common-neighbour}; combined with the bridge-paths, this forms $\pi_{y_0^\star y_1^\star}$. The roles of $\gamma, \eta,z, R$ in the construction are as follows:
$R$ is the depth of the hierarchy, and hence controls the number $2^R$ of gaps, while $\gamma$ controls the Euclidean length of the gaps and hence also the cost of joining them, with the total cost of joining a single gap being roughly $\olw^{4\mu} = |x|^{2d\mu\gamma^{R-1}}$. The exponent $\eta$ controls the cost of bridge-paths in the hierarchy. When $\eta > 0$  the total cost is dominated by the cost of the very first bridge-edge of length $\Theta(|x|)$ of cost $|x|^\eta$ and when $\eta = 0$ the total cost of all bridge-paths is negligible compared to the cost of joining gaps. Finally, $z$ controls the weights of the endpoints of bridging edges relative to the distance they span: when $z>0$, each endpoint of a round-$i$ bridging edge has weight roughly $(|x|^{d\gamma^{i-1}})^{z/2}$ (see Claim \ref{claim:cheap-bridge}); when $z=0$, these weights are subpolynomial in $|x|$ and the same across all~rounds.
The final cost of $\pi_{y_0^\star y_1^\star}$ is thus roughly $2^R|x|^{2\mu\gamma^{R-1}} + |x|^{\eta}$, which we have bounded above by roughly $2^R|x|^{2\mu\gamma^{R-1}}|x|^\eta$ in Proposition~\ref{prop:path-from-hierarchy} for convenience. 

From the many constraints in Proposition~\ref{prop:path-from-hierarchy}, the following are relevant when optimising the cost of the path. The requirement $\Lambda(\eta,z)>0$ ensures that low-cost bridging edges exist (Claim~\ref{claim:cheap-bridge}). The requirement that either $z=0$ or $\Phi(\eta, z) > 0$ ensures that among the many potential bridging edges a few can be extended to low-cost 3-edge bridge-paths in Lemma \ref{lem:hierarchy-intermediate-weights}. The requirement $\gamma<1$ ensures that boxes where we search for the bridging edge shrink in size, while $z \le d$ is a formal requirement for applying Claim~\ref{claim:cheap-bridge} to find bridging edges, which we tolerate because increasing $z$ above $d$ will never be optimal. Heuristically, the effect of increasing $z$ is to increase the probability that a given bridging edge exists at the price of increasing its expected cost; at $z=d$ the existence probability is already in the interval $[\underline{c}, \overline{c}]$ and cannot be increased further, (however the penalty would increase and the number of combinatorial option decrease by increasing $z$, which is never optimal).
The other constraints of Proposition~\ref{prop:path-from-hierarchy} and Setting~\ref{set:hierarchy-common} (such as $2d\gamma < \tau-1$ and $R \le (\log \log |x|)^2$) never turn out to be tight for optimal choices of $\eta, R,\gamma, z$.
Recall $\mu_{\log}, \mu_{\mathrm{pol}}$ from \eqref{eq:mu_pol_log}.

\begin{restatable}[Path with polylogarithmic cost]{corollary}{ComputationsPolyLog}\label{cor:computations-polylog}
Consider $1$-FPP in Definition \ref{def:1-FPP} on the graphs IGIRG or SFP satisfying the assumptions given in \eqref{eq:power_law}--\eqref{eq:F_L-condition} with $d\ge 1, \alpha \in (1,\infty], \tau\in(2,3), \mu>0$.
Let $\underline{c},\overline{c},h,L,c_1,c_2,\beta$ be as in~\eqref{eq:connection_prob}--\eqref{eq:F_L-condition}, 
we allow $\beta = \infty$ and/or $\alpha=\infty$.
Let $q,\eps,\zeta\in(0,1)$, let $0<\delta\lls \eps,q,\mpar$, and let $w_0>1$. Fix a realisation $(V, w_V)$ of $\widetilde{\cal V}$.  Let $x \in V$ with $|x|\ggs q, \delta,\eps,\zeta, w_0, \mpar$. Let $Q$ be a cube of side length $|x|$ containing $0$ and $x$, and assume that $(V, w_V)$ is such that $Q$ contains a weak $(\delta/4,w_0)$-net $\calN$ with $0,x \in \calN$ given in Definition \ref{def:weak-net}. Let $G \sim \{\calG\mid V, w_V\}$.
Let $\calX_{\mathrm{polylog}}(0,x)$ be the event that $G$ contains a path $\pi$, fully contained in $\calN$, with endpoints say $y_0^\star, y_x^\star$, with the following properties:
\begin{align}
&w_{y_0^\star}, w_{y_x^\star} \in [\olw,4\olw], \quad \mbox{where} \quad \olw \in [\log\log|x|, (\log|x|)^\eps],\label{eq:weight-crit-cor1}\\
&y_0^\star \in B_{\olw^{3/d}}(0) \qquad\qquad \mbox{and} \qquad y_x^\star \in B_{\olw^{3/d}}(x) \label{eq:dist-crit-cor1},\\
&\calC(\pi) \le (\log |x|)^{\Delta_0+\eps}, \qquad \mbox{and} \qquad \mathrm{dev}_{0x}(\pi)\le \zeta |x|,\label{eq:cost-crit-cor1}
\end{align}
where $\Delta_0$ is defined in ~\eqref{eq:Delta_0}, \eqref{eq:alpha-infty-Delta_0} or~\eqref{eq:beta-infty-Delta_0} depending on whether $\alpha,\beta<\infty$, $\alpha=\infty$ or $\beta=\infty$.
If either $\alpha\in(1,2)$ or $\mu\in(\mu_{\mathrm{expl}},\mu_{\log})$ or both hold, then $\pr(\calX_{\mathrm{polylog}}(0,x) \mid V,w_V) \ge 1-q$.
\end{restatable}
\begin{proof}[Sketch of proof]
Corollary~\ref{cor:computations-polylog} covers the polylogarithmic regime, which corresponds to solutions where $\eta = 0$ is possible -- such solutions exists when either $\alpha\in(1,2)$ or $\mu<\mu_{\log}$. When $\eta=0$, the cost of the path $\pi_{y_0^\star,y_x^\star}$ is dominated by the cost $2^R|x|^{2\mu\gamma^{R-1}}$ of joining gaps. Given $\gamma$, this has minimum $2^{(1+o(1))R}$ when setting $R = (1-o(1))\log\log |x|/\log(1/\gamma)$. To minimise the cost further, we must therefore minimise $\gamma\in(0,1)$ subject to the constraints $z\in[0,d], \Lambda(0,z)>0$, and either $z=0$ or $\Phi(0,z)>0$. This problem turns out to have two potentially optimal solutions corresponding to two possible strategies for finding bridging edges, with the optimal choice depending on the values of $\alpha,\tau,\beta,\mu$. One possible solution -- which only exists when $\alpha\in(1,2)$ -- takes $\gamma = \alpha/2 + o(1)$ and $z=0$, so that bridging edges are unusually long-range edges between pairs of low-weight vertices, yielding total path cost  $(\log |x|)^{\Delta_\alpha}$ with $\Delta_\alpha=1/(1-\log_2\alpha)$, see Claim~\ref{claim:polylog-low-alpha}.
The other possible solution -- which only exists when $\mu<\mu_{\log}$ -- takes $\gamma = (\tau-1+\mu\beta)/2 + o(1)$ and $z=d$, so that bridging edges are unusually low-cost edges between pairs of high-weight vertices and the total path cost is $(\log x)^{\Delta_\beta}$ with $\Delta_\beta=1/(1-\log_2(\tau-1+\mu\beta))$, see Claim~\ref{claim:polylog-low-mu}. The proof is in Appendix \ref{app:path-proofs}.
\end{proof}
\begin{remark}\label{rem:alpha-beta-infty-polyllog}
  If both $\alpha=\beta=\infty$, then the conditions of Corollary~\ref{cor:computations-polylog} cannot be satisfied. Indeed, when $\alpha=\infty$ then $\alpha\in (1,2)$ is not satisfied. Since $\mu_{\mathrm{expl}}= \mu_{\log} = 0$ by~\eqref{eq:beta-infty-definitions} when $\alpha=\beta=\infty$, so $\mu\in(\mu_{\mathrm{expl}},\mu_{\log})$ can also not be satisfied. 
\end{remark}

\begin{restatable}[Path with polynomial cost]{corollary}{ComputationsPolynomial}\label{cor:computations-polynomial}
  Consider $1$-FPP in Definition \ref{def:1-FPP} on the graphs IGIRG or SFP satisfying the assumptions given in \eqref{eq:power_law}--\eqref{eq:F_L-condition} with $d\ge 1, \alpha \in (1,\infty], \tau\in(2,3), \mu>0$.
Let $\underline{c},\overline{c},h,L,c_1,c_2,\beta$ be as in~\eqref{eq:connection_prob}--\eqref{eq:F_L-condition}, 
we allow $\beta = \infty$ and/or $\alpha=\infty$.
Let $q,\eps,\zeta\in(0,1)$, and let $0<\delta\lls \eps,q,\mpar$, and $w_0>1$. Fix a realisation $(V, w_V)$ of $\widetilde{\cal V}$.  Let $x \in V$ with $|x|\ggs q,\delta,\eps, \zeta,w_0,\mpar$. Let $Q$ be a cube of side length $|x|$ containing $0$ and $x$, and assume that $(V, w_V)$ is such that $Q$ contains a weak $(\delta/4,w_0)$-net $\calN$ with $0,x \in \calN$ given in Definition \ref{def:weak-net}. Let $G \sim \{\calG\mid V, w_V\}$.
Let $\calX_{\mathrm{pol}}(0,x)$ be the event that $G$ contains a path $\pi$, fully contained in $\calN$, with endpoints say $y_0^\star, y_x^\star$, with the following properties:
\begin{align}
&w_{y_0^\star}, w_{y_x^\star} \in [\olw,4\olw], \quad \mbox{where} \quad \olw \in [\log\log|x|, |x|^\eps],\label{eq:weight-crit-cor2}\\
&y_0^\star \in B_{\olw^{3/d}}(0) \qquad\qquad \mbox{and} \qquad y_x^\star \in B_{\olw^{3/d}}(x) \label{eq:dist-crit-cor2},\\
&\calC(\pi) \le |x|^{\eta_0+\eps}, \qquad \mbox{and} \qquad \mathrm{dev}_{0x}(\pi)\le \zeta |x|,\label{eq:cost-crit-cor2}
\end{align}
where $\eta_0$ is defined in~\eqref{eq:eta_0}, \eqref{eq:alpha-infty-definitions}, \eqref{eq:beta-infty-definitions}, or~\eqref{eq:alpha-beta-infty-definitions} depending on $\alpha,\beta<\infty$, $\alpha=\infty$, $\beta=\infty$, or $\alpha=\beta=\infty$.
  If both $\alpha>2$ and $\mu \in(\mu_{\log}, \mu_{\mathrm{pol}}]$ hold then $\pr(\calX_{\mathrm{pol}}(0,x) \mid V,w_V) \ge 1-q$.
\end{restatable}
\begin{proof}[Sketch of proof]
Corollary~\ref{cor:computations-polynomial} covers the polynomial regime, which corresponds to solutions where only $\eta > 0$ is possible, i.e. when $\alpha>2$ and $\mu>\mu_{\log}$. Here, on taking $R$ to be a suitably large constant, the cost bound on the path $2^R|x|^{2\mu\gamma^{R-1}}|x|^\eta=|x|^{\eta+o(1)}$, which is roughly the cost of the very first bridging edge. Our goal is thus to minimise $\eta$ under the constraints that $\Lambda(\eta,z)>0$, $z\in[0,d]$, $\gamma\in(0,1)$, and either $z=0$ or $\Phi(\eta,z)>0$.
\eqref{eq:Lambda-def} and \eqref{eq:Phi-def} show that both $\Phi$ and $\Lambda$ are increasing functions of $\gamma$; thus we can take $\gamma = 1-o(1)$. As in the polylogarithmic regime, this minimisation problem has two potentially optimal solutions. One possible solution -- which exists when $\mu\le\mu_{\mathrm{pol},\beta}$ -- takes $z=d$ and gives $\eta = \mu d-(3-\tau)d/\beta + o(1)$, so that bridging edges are unusually low-cost edges between pairs of high-weight vertices. The total path cost is then $|x|^{\eta_\beta+o(1)}$ with $\eta_\beta=d(\mu-(3-\tau)/\beta)$ (see Claim~\ref{claim:polynomial-small-mu}). The other possible solution -- which exists when $\mu\le \mu_{\mathrm{pol},\alpha}$ -- takes $z$ to be as small as possible, so that bridging edges are unusually long-range edges between pairs of relatively low-weight vertices. However, when $\alpha>2$, there are no bridging-edges between constant weight vertices, and the minimal $z$ where bridging-edges appear is $z = d(\alpha-2)/(\alpha - (\tau - 1)) + o(1)=1/\mu_{\mathrm{pol},\alpha}+o(1)$, i.e., between vertices of weight $|x|^{1/(2\mu_{\mathrm{pol},\alpha})+o(1)}$. This gives cost-exponent  $\eta_\alpha := \mu/\mu_{\mathrm{pol},\alpha}$ and total cost $|x|^{\mu/\mu_{\mathrm{pol},\alpha} + o(1)}$ (see Claim~\ref{claim:polynomial-large-mu}). Whenever a solution exists among the above two possibilities, it gives an exponent $\eta$ at most $1$. So, whenever $\mu\le \max\{\mu_{\mathrm{pol},\alpha}, \mu_{\mathrm{pol},\beta}\}$, we obtain the cost bound $|x|^{\min\{\eta_\beta, \eta_\alpha\}+o(1)}$, which gives the
definition of $\eta_0$ in \eqref{eq:eta_0}.
The proof is in Appendix \ref{app:path-proofs}.
\end{proof}
The goal of this section has been to prove Corollaries~\ref{cor:computations-polylog} and~\ref{cor:computations-polynomial}; now that this has been achieved, notation internal to this section will no longer be used.
\section{Connecting the endpoints $0, x$ to the path}\label{sec:endpoints}
The final step is to connect the initial vertices $0$ and $x$ to the respective endpoints $y_0^\star$ and $y_x^\star$ of the path constructed in Section~\ref{sec:hierarchy}. For $d\ge 2$, we consider the graph $G_M$ induced by the vertices of weight in $[M,2M]$ for some large constant $M$. By a result in our companion paper, this graph has an infinite component $\calC_\infty^M$ where cost-distances scale linearly with the Euclidean distance. We connect $0, y_0^\star$ to respective nearby vertices $u_0, u_0^\star\in G_M$, and then use that the cost-distance $d_{\calC}(u_0, u_0^\star)=\Theta(|u_0-u_0^\star|$) within $G_M$. We do the same for $y_x^\star$ and $x$. For $d=1$, we take a similar approach with larger value of $M$: we need $M$ to depend on $|x|$ to guarantee that $G_M$ contains a large connected subgraph in the section between $0$ and $x$.

The construction of the path $\pi_{y_0^\star, y_x^\star}$ already revealed information about the graph: the vertices $y_0^\star, y_x^\star$ are the outcomes of a selection procedure that might influence the graph around them. So, we ensure that cost-distances are linear in $G_M$ simultaneously for all `candidate' vertices for $u_0$ and $u_0^\star$ in Lemma~\ref{lem:new-external}.

When the graph is finite (e.g. GIRG $G_n$ in Def. \ref{def:girg}), we additionally use that (near)-shortest paths within $G_M$ have very small deviation from the straight line, so that when two vertices are not too close to the boundary of the box $Q_n$, the constructed path stays in $Q_n$. For this we use that the constructed paths have low deviation, see Def.~\ref{def:deviation}. 
We define the setting of this section. As in Theorems \ref{thm:polylog_regime}--\ref{thm:polynomial_regime}, we aim to bound $d_\calC(0,x)$. 

\begin{setting}\label{set:joining-common}
 Consider $1$-FPP in Definition \ref{def:1-FPP} on the graphs IGIRG or SFP satisfying the assumptions given in \eqref{eq:power_law}--\eqref{eq:F_L-condition} with $d\ge 1, \alpha \in (1,\infty], \tau\in(2,3), \mu>0$.
Let $\underline{c},\overline{c},h,L,c_1,c_2,\beta$ be as in~\eqref{eq:connection_prob}--\eqref{eq:F_L-condition}, 
we allow $\beta = \infty$ and/or $\alpha=\infty$. Let $G\sim \calG$, let $\calF_{0,x} := \{0,x \in \calV\}$, and let $\calC_\infty$ be the (unique) infinite component of~$G$.  

Let $q \in (0,1)$, and let $M \ggs q,\mpar$. Let $I_M := [M,2M]$, and let $G_M=(\calV_M,\calE_M)$ be the subgraph of $G$ with vertices $\calV_M:=\{(v,w_v)\in\widetilde \calV\colon w_v\in I_M\}$ and edges $\calE_M:=\{uv \colon u,v\in \calV_M,\, \cost{uv}\le M^{3\mu}\}$. Let $\calC_\infty^M$ be the infinite component of $G_M$ if it exists and is unique, and define $\calC_\infty^M:=\texttt{None}$ otherwise.
\end{setting}

The next claim is a technical necessity to remove conditioning on membership of $\calC_\infty$ later.

\begin{restatable}{claim}{twoinCinfty}\label{claim:two-in-Cinfty}
    Consider Setting~\ref{set:joining-common}. There exists $\rho > 0$ such that for all distinct $a,b \in \R^d$, we have $\pr(a,b \in \calC_\infty \mid a,b \in \calV) \ge \rho$.
\end{restatable}
\begin{proof}
    For $d \ge 2$, this is~\cite[Claim~3.10]{komjathy2022one2}. We provide a proof for $d=1$ on page \pageref{proof:claim-rho} in Appendix~\ref{app:1d-endpoints}.
\end{proof}
The next lemma, that we prove in the companion paper \cite{komjathy2022one2}, is the main tool of the section. We first need some definitions.
We use the notation $\pi_{a,b}$ for a path between vertices $a$ and $b$. Let $r, \kappa, \zeta, C>0$ and $z\in \R^d$.
We say that a set of vertices $\calH\subseteq \calV_M$ is \emph{$r$-strongly dense} around $z\in \R^d$ in $\calV_M$ if the following event holds:
 \begin{equation}\label{eq:a-dense}
         \calA_\mathrm{dense}(\calH, \calV_M, r,z):=\Big\{\forall y \in B_r(z): \big|B_{r^{1/3}}(y) \cap \calH\big| \ge |B_{r^{1/3}}(y) \cap \calV_M|/2\Big\}.
        \end{equation}
We say that a set of vertices $\calH\subseteq \calC_\infty^M$ shows \emph{$r$-strongly $\kappa$-linear} distances with deviation $\zeta$ in $\calC_\infty^M$ around $z\in \R^d$ if the following event holds:
 \begin{equation}\label{eq:a-linear}
 \begin{aligned}
        \calA_\mathrm{linear}(\calH, \calC_\infty^M,r,\kappa,\zeta,C, z):=&\Big\{\forall a \in \calB_r(z) \cap \calH,\  \forall b \in \calH: \exists\mbox{ a path } \pi_{a,b} \subseteq \calC_\infty^M \mbox{ with}\\
        &\quad\calC(\pi_{a,b})\le \kappa |a-b|+C,\  \mathrm{dev}(\pi_{a,b})\le \zeta |a-b| + C\Big\}.
        \end{aligned}
        \end{equation}
Finally, we say that a set $\calH$ is $(r,C)$-near to $z\in \calV$ if the following event holds:  
 \begin{equation}
 \begin{aligned}\label{eq:a-near}
 \calA_{\mathrm{near}}(\calH, r,C,z):=&\Big\{\exists \mbox{ a path } \pi_{z,a} \mbox{ to some } a\in B_r(z)\cap \calH\mbox{ with}\\
 &\qquad \qquad\qquad\qquad\calC(\pi_{z,a})\le C, \ \mathrm{dev}(\pi_{z,a})\le C \Big\}.
  \end{aligned}
 \end{equation}

\begin{lemma}\label{lem:new-external}
    Consider Setting~\ref{set:joining-common} and assume $d\ge 2$. Let $M,r_1,r_2,C,\kappa > 0$ and $q,\zeta \in (0,1)$. 
    Whenever $C \ggs r_2$ and $r_1,r_2\ggs M,\zeta,q,\mpar$, and $\kappa\ggs M$,
    then a.s.\ 
    $\calC_\infty^M \neq \normalfont{\texttt{None}}$ and there is an infinite-sized vertex set $\calH_\infty\subseteq \calC_\infty^M$ determined by $(\widetilde \calV, \calE(G_M))$ so that $G_M[\calH_\infty]$ is connected, and for all $z\in \R^d$,
    \begin{align}
    &\pr(\calA_\mathrm{dense}(\calH_\infty,\calV_M,r_1,z))\ge 1-q/10, \qquad\pr(\calA_\mathrm{near}(\calH_\infty, r_2,C,z)\mid z\in \calC_\infty) \ge 1-q/10, \label{eq:dense-near}\\
     &\pr(\calA_{\mathrm{linear}}(\calH_\infty, \calC_\infty^M, r_2,\kappa, \zeta,C,z))\ge1-q/10.\label{eq:linear-again}
    \end{align}
    The statement remains valid conditioned on $\calF_{y,z}=\{y,z \in \calV\}$; moreover, the constraints on $C,r,M,\kappa$ are uniform over $\{\calF_{y,z}: y,z\in \R^d\}$.
\end{lemma}

\begin{proof}[Proof]
The lemma follows from results in~\cite{komjathy2022one2}. There, we show that $\calH_\infty$ exists, and is infinite and connected in $G_M$ in Corollary~3.9(ii).
The $r_2$-strong $\kappa$-linearity comes from Corollary~3.9(iv) in~\cite{komjathy2022one2} applied with $r_{3.9}=r_2$ and $C_{3.9}=C$, and the $(r_2,C)$-near property comes from in~\cite[Claim 3.11]{komjathy2022one2}. Moreover, we can apply ~\cite[Corollary~3.9(iii)]{komjathy2022one2} with $r_{3.9}=r_1$ to get the $r_1$-dense property with $(\log r_1)^2$ instead of $r_1^{1/3}$ and arbitrary density $1-\eps$ instead of $1/2$ in \eqref{eq:a-dense}. This is a strictly stronger statement since we can cover any ball of radius $r_1^{1/3}$ with balls of radius $(\log r_1)^2$, at the cost of increasing the fraction $\eps$ of non-covered vertices by a $d$-dependent factor.
\end{proof}

For $d=1$, the graph $G_M$ does not have an infinite component for any $M$ and the proof techniques in Lemma~\ref{lem:new-external} do not apply. Instead, on page~\pageref{proof:lem:new_external} in Appendix~\ref{app:1d-endpoints} we directly prove the following analogous statement for $G_M$ in a \emph{finite} interval.

\begin{restatable}{lemma}{newexternaloned}\label{lem:new-external-1d}
    Consider Setting~\ref{set:joining-common} with $d = 1$. Let $q,\zeta \in (0,1)$, $r_M := e^{(\log M)^2}$, $\kappa_M := 2M^{3\mu+2}$ and $C_M := M^{2(\tau-1)+3\mu}$. Let $z \in \R^d$, and $\calH_M := B_{2r_M}(z) \cap \calV_M$. Then whenever $M \ggs q,\mpar$,
    \begin{align}
    &\pr(\calA_\mathrm{dense}(\calH_M,\calV_M,r_M,z)) = 1, \qquad\pr(\calA_\mathrm{near}(\calH_M, C_M,C_M,z)\mid z\in \calC_\infty) \ge 1-q/10, \label{eq:dense-near-1d}\\
     &\pr(\calA_{\mathrm{linear}}(\calH_M, \calH_M, r_M,\kappa_M, 0,2\kappa_M,z))\ge 1-q/10.\label{eq:linear-again-1d}
    \end{align}
    The statement remains valid conditioned on $\calF_{y,z}=\{y,z \in \calV\}$; moreover, the constraints on $r_M$ are uniform over $\{\calF_{y,z}: y,z\in \R^d\}$.
\end{restatable}

We use the next claim to connect the endpoints $y_0^\star, y_x^\star$ of the path $\pi_{y_0^\star, y_x^\star}$ in Corollaries \ref{cor:computations-polylog} and \ref{cor:computations-polynomial} to $\calH_\infty \subseteq \calC_\infty^M$ from Lemma~\ref{lem:new-external} (when $d \ge 2$) and $\calH_{M}$ from Lemma~\ref{lem:new-external-1d} (when $d=1$).

\begin{restatable}{claim}{ClaimJoinCm}
\label{claim:join-Cm}
    Consider Setting~\ref{set:joining-common} and any $d\ge 1$. Let $w\ggs q,\mpar$ with $w \ge M^{8(\tau-1)}$, let $r := w^{3/d}$ and let $z \in \R^d$. Let $\calH\subseteq \calV_M$ be a random vertex set which depends only on $(V,w_V,\calE_M)$ and which satisfies $\pr(\calA_{\mathrm{dense}}(\calH,\calV_M,r,z) \mid \calF_{0,x}) \ge 1-q/10$. Let 
    \begin{align*}
    \calA_{\mathrm{down}}(w,z):=\Big\{ \forall y \in \widetilde \calV\cap (B_{r}(z)\times[w, 4w]): \exists u \in \calH \cap B_{r^{1/3}}(y),\ yu\in\calE,\ \calC(yu) \le w^{2\mu} \Big\}.
     \end{align*}
     Then for all $z \in \R^d$, $\pr(\calA_{\mathrm{down}}(w,z)\mid\calF_{0,x}) \ge 1 - q/3$.
\end{restatable}

The proof is in Appendix \ref{app:proof-join-Cm}, but here we briefly explain why the conditions are also satisfied for $d=1$. Fixing $w$ and requiring $w^{3/d}=r$ gives an equation for $M$ in terms of $w$ in Lemma \ref{lem:new-external-1d}:  $w^{3/d}=r_M=e^{(\log M)^2}$, which implies that $M=e^{\sqrt{\log(w)\cdot 3/d}}$. Then $M^{8(\tau-1)}$ is a subpolynomial function of $w$, and so the requirement $w\ge M^{8(\tau-1)}$ is satisfied for all sufficiently large $w\ggs \mpar$.

We are now ready to prove the main results. Recall $\mu_{\log}, \mu_{\mathrm{pol}}$ from \eqref{eq:mu_pol_log}. The following two lemmas contain the statements for the infinite graph cases.
\begin{lemma}\label{lem:polylog-deviation}
    Consider Setting~\ref{set:joining-common}. Suppose that either $\alpha\in(1,2)$ or $\mu\in(\mu_{\mathrm{expl}},\mu_{\log})$ or both hold, and let $\Delta_0$ be as defined in~\eqref{eq:Delta_0}, \eqref{eq:alpha-infty-Delta_0}, or~\eqref{eq:beta-infty-Delta_0}, depending on whether $\alpha,\beta<\infty$, $\alpha=\infty$, or $\beta=\infty$. For every $q, \eps,\zeta >0$ there exists $C>0$ such that the following holds. For any $x \in \R^d$ let $\calA_{\mathrm{polylog}}$ be the event that $G$ contains a path $\pi_{0,x}$, with endpoints $0$ and $x$, of cost $\calC(\pi_{0,x})\le (\log |x|)^{\Delta_0+\eps} + C$ and deviation  $\mathrm{dev}(\pi_{0,x})\le \zeta|x| + C$. Then $\pr(\calA_{\mathrm{polylog}} \mid 0,x \in \calC_\infty) \ge 1-q$.
\end{lemma}

\begin{proof} We first prove the result for $d\ge 2$, then describe the necessary modifications for $d=1$. Let $\rho > 0$ be as in Claim~\ref{claim:two-in-Cinfty} and let $\delta \lls \eps,q,\rho,\mpar$ and $w_0 >1$. We want to apply Corollary~\ref{cor:computations-polylog}, which holds for sufficiently large $|x|$. Thus there exists $r_{\ref{cor:computations-polylog}} \ggs q,\delta,\eps,\zeta,w_0,\mpar$ such that Corollary~\ref{cor:computations-polylog} is applicable whenever $|x| \ge r_{\ref{cor:computations-polylog}}$. We may also assume $r_{\ref{cor:computations-polylog}} \ggs \kappa$. To cover $|x| < r_{\ref{cor:computations-polylog}}$, for all $v\in \calC_\infty \cap B_{r_{\ref{cor:computations-polylog}}}(0)$, pick the cheapest path $\pi_{0,v}$ from $0$ to $v$. Then $R_1 := \max\{\calC(\pi_{0,v}): v\in \calC_\infty \cap B_{r_{\ref{cor:computations-polylog}}}(0)\}$ and $R_2 := \max\{\mathrm{dev}(\pi_{0,v}): v\in \calC_\infty \cap B_{r_{\ref{cor:computations-polylog}}}(0)\}$ are almost surely finite random variables, and since we may assume $C \ggs r_{\ref{cor:computations-polylog}},q,\mpar$, we have $\pr(R_1,R_2 \le C \mid 0,x\in \calC_\infty) \ge 1-q$, as required. So from now on we may assume $|x| \ge r_{\ref{cor:computations-polylog}}$.

Let $\olw$ be as in \eqref{eq:weight-crit-cor1} in Corollary~\ref{cor:computations-polylog} and let $r_1:=\olw^{3/d}$. Let $M,r_2,\kappa > 0$ satisfy $\olw, C \ggs r_2 \ggs M,\zeta$, $q, \rho, \eps,\mpar$ and $C \ggs r_{\ref{cor:computations-polylog}} \ggs \kappa \ggs M$ as in Lemma~\ref{lem:new-external}, and note that $|x|\ge r_{\ref{cor:computations-polylog}}$ implies $\olw,r_1 \ggs \kappa$.
Let $Q$ be a cube of side length $|x|$ containing 0 and $x$, and let $\calA_{\mathrm{net}}$ be the event that $Q$ contains a weak $(\delta/4, w_0)$-net (as in Definition \ref{def:weak-net}) which contains $0$ and $x$. Apply Corollary~\ref{cor:computations-polylog} with $\eps_{\ref{cor:computations-polylog}} = \eps/2$ and $q_{\ref{cor:computations-polylog}}:=q\rho/5$ to obtain $\calX_{\mathrm{polylog}}(0,x)$. Then consider the intersection of the following events from Corollary \ref{cor:computations-polylog}, Lemma~\ref{lem:new-external} (defined in \eqref{eq:a-linear}, \eqref{eq:a-near}) and  Claim~\ref{claim:join-Cm}:
\begin{align}
\begin{split}\label{eq:big-intersection}
 &\calA:=\calA_{\mathrm{net}}\cap
    \calX_{\mathrm{polylog}}(0,x)\\
    &\quad\cap \bigcap_{v\in\{0,x\}}\Big(\calA_\mathrm{near}(\calH_\infty, r_2,C/4,v)\cap \calA_\mathrm{linear}(\calH_\infty, \calC_{\infty}^M, r_2,\kappa,\zeta,C/4,v)\cap \calA_{\mathrm{down}}(\olw,v)\Big).
\end{split}   
\end{align}
When $\calA$ occurs, $\calX_{\mathrm{polylog}}(0,x)$ gives a path between endpoints $y_0^\star, y_x^\star$ with weights in $[\olw, 4\olw]$ and within distance $r_1=\olw^{3/d}$ from $0,x$, with cost $\calC(\pi_{y_0^\star, y_x^\star})\le (\log |x|)^{\Delta_0+\eps/2}$ and deviation $\mathrm{dev}_{0x}(\pi_{y_0^\star, y_x^\star}) \le \zeta |x|$, respectively, see \eqref{eq:weight-crit-cor1}--\eqref{eq:cost-crit-cor1}. 
Then, the events $\calA_{\mathrm{down}}(\olw,0), \calA_{\mathrm{down}}(\olw,x)$  from Claim~\ref{claim:join-Cm} applied respectively to $y_0^\star, y_x^\star$ give us two paths $\pi_{y_0^\star, u_0^\star}$ and $\pi_{y_x^\star, u_x^\star}$ with $u_0^\star, u_x^\star\in \calH_\infty$ and within respective distance $\olw^{1/d}$ from $y_0^\star, y_x^\star$, and cost at most $\overline w^{2\mu}$. Further, the events $\calA_{\mathrm{near}}(\calH_\infty, r_2,C/4,0)$ and $\calA_{\mathrm{near}}(\calH_\infty, r_2,C/4,x)$ in \eqref{eq:a-near} also give us two paths $\pi_{0, u_0}$ and $\pi_{x, u_x}$, with respective endpoints $ u_0, u_x\in \calH_\infty$ within distance $r_2 \le \olw^{3/d}$ from $0,x$ respectively, and cost at most $C/4$ each. Finally, since $u_0, u_0^\star, u_x, u_x^\star\in \calH_\infty$, and $u_0, u_x$ is within distance $r_2$ from $0,x$, respectively, the events $\calA_\mathrm{linear}(\calH_\infty, \calC_{\infty}^M, r_2,\kappa,\delta,C/4,v), v\in\{0,x\}$ in \eqref{eq:a-linear} ensure that there exist paths $\pi_{u_0, u_0^\star}$ and $\pi_{u_x, u_x^\star}$ in $G$ that have cost at most $\kappa|u_v-u_v^\star|+C/4 \le \kappa3\olw^{3/d}+C/4$ since $|u_v-u_v^\star|\le 3r_1 = 3\olw ^{3/d}$ and deviation at most $\zeta|u_v-u_v^\star|+C/4\le \zeta 3\olw^{3/d}+C/4$.
The concatenated path is $\pi_{0,x}:=\pi_{0,u_0} \pi_{u_0,u_0^\star} \pi_{u_0^\star, y_0^\star} \pi_{y_0^\star, y_x^\star} \pi_{y_x^\star, u_x^\star} \pi_{u_x^\star, u_x} \pi_{u_x, x}$. Then, since $\olw\le (\log |x|)^{\eps/2}$ in \eqref{eq:weight-crit-cor1}, we can estimate the cost, and using that the vertices of the paths $\pi_{0,u_0},\pi_{u_0,u_0^\star}, \pi_{y_0^\star, y_x^\star}, \pi_{y_x^\star, u_x^\star},\pi_{u_x^\star, u_x}, \pi_{u_x, x}$ are all within distance $3r_1=3\olw^{3/d}$ from $0$ and $x$ respectively, we can bound cost and deviation as
\begin{equation}\label{eq:cost-calculation-long}
\begin{aligned}  
  \calC(\pi_{0,x})&\le 2\cdot C/4+ 2\olw^{2\mu}+ 2(\kappa3\olw^{3/d}+C/4) + (\log |x|)^{\Delta_0+\eps/2}\le (\log |x|)^{\Delta_0+\eps} +C,\\
  \mathrm{dev}(\pi_{0,x})&\le \max\big\{\mathrm{dev}_{0,x}(\pi_{y_0^\star, y_x^\star}), 2\olw^{3/d}+\zeta3\olw^{3/d}+C/4\big\}\le \zeta |x|+C, 
\end{aligned}
\end{equation}
using $|x| \ge r_{\ref{cor:computations-polylog}} \ggs \eps, \kappa, \mpar$. Thus $\calA\subseteq\calA_{\mathrm{polylog}}$.  
A union bound on the complement of the events in \eqref{eq:big-intersection} from Lemma \ref{lem:weak-nets-exist} with $t=2$ and $\eps_{\ref{lem:weak-nets-exist}}:=\delta/4$ for $\calA_{\mathrm{net}}$, Corollary \ref{cor:computations-polylog}, Lemma~\ref{lem:new-external} with $q_{\ref{lem:new-external}}:=q\rho$, and Claim~\ref{claim:join-Cm} with $q_{\ref{claim:join-Cm}}:=q\rho$ gives
\begin{align*}
\pr(\neg\calA_{\mathrm{polylog}}\mid 0,x\in\calC_\infty) &\le \pr(\neg\calA \mid 0,x\in\calC_\infty) =\frac{\pr(\neg\calA \cap \{0,x\in\calC_\infty\}\mid 0,x\in \calV)}{\pr(0,x\in\calC_\infty\mid 0,x\in\calV)}\\ &\le \frac{q\rho}{\pr(0,x\in\calC_\infty\mid 0,x\in\calV)}.
\end{align*}
The result therefore follows from Claim~\ref{claim:two-in-Cinfty}.

When $d=1$, we construct $\pi_{0,x}$ in exactly the same way as below \eqref{eq:big-intersection}, using Lemma~\ref{lem:new-external-1d} with $M_{\ref{lem:new-external-1d}} := \exp(\sqrt{(3/d)\log\olw})$ (so that $r_{M_{\ref{lem:new-external-1d}}} = \olw^{3/d}$), in place of Lemma~\ref{lem:new-external}. We may assume $|x| \ggs \eps,\zeta$ as for $d\ge 2$. Note that $M_{\ref{lem:new-external-1d}} \le \exp(\sqrt{\log\log|x|})$ since $\olw \le (\log|x|)^{\eps}$ and $\eps \lls \mpar$, and in particular we may assume $\olw \ge M_{\ref{lem:new-external-1d}}^{8(\tau-1)}$ as in Claim~\ref{claim:join-Cm} since $\olw \ggs \eps, \mpar$, see also the computation below Claim \ref{claim:join-Cm}. Using $|x|\ggs\eps,\zeta$, this implies the costs of all our subpaths counted in \eqref{eq:cost-calculation-long} except $\pi_{y_0^\star,y_x^\star}$ are negligible compared to the $(\log |x|)^{\Delta_0+\eps/2}$ cost of $\pi_{y_0^\star,y_x^\star}$, as in the $d \ge 2$ case, and likewise that the deviation of these subpaths from the line segment $S_{0x}$ is negligible compared to $\zeta |x|$. The cost and deviation of $\pi_{y_0^\star,y_x^\star}$ are bounded using Corollary~\ref{cor:computations-polylog} exactly as in the $d \ge 2$ case in \eqref{eq:cost-calculation-long}.
\end{proof}
\vskip-1em
\begin{lemma}\label{lem:polynomial-deviation}
    Consider Setting~\ref{set:joining-common}. Suppose that $\alpha>2$ and $\mu>\mu_{\log}$, and let $\eta_0$ be as defined in~\eqref{eq:eta_0}, \eqref{eq:alpha-infty-definitions},~\eqref{eq:beta-infty-definitions}, or \eqref{eq:alpha-beta-infty-definitions}, depending on $\alpha,\beta<\infty$, $\beta <\alpha=\infty$, $\alpha <\beta=\infty$, or $\alpha=\beta =\infty$. For every  $q, \eps,\delta >0$ there is $C>0$ such that the following holds. 
    For any $x \in \R^d$ let $\calA_{\mathrm{pol}}$ be the event that $G$ contains a path $\pi_{0,x}$, with endpoints $0$ and $x$, of cost $\calC(\pi_{0,x})\le |x|^{\eta_0+\eps} + C$ and deviation  $\mathrm{dev}(\pi_{0,x})\le \zeta|x| + C$. Then $\pr(\calA_{\mathrm{pol}} \mid 0,x \in \calC_\infty) \ge 1-q$.
\end{lemma}
\begin{proof}
The proof when $\mu\in(\mu_{\log}, \mu_{\mathrm{pol}}]$ is identical to the proof of Lemma~\ref{lem:polylog-deviation}, except that the event $\calX_{\mathrm{polylog}}(0,x)$ from Corollary~\ref{cor:computations-polylog} is replaced by $\calX_{\mathrm{pol}}(0,x)$ from Corollary~\ref{cor:computations-polynomial}.
     
When $\mu>\mu_{\mathrm{pol}}$ and $d=1$,  we have $\eta_0=1$ and we can prove that the cost distance is at most $|x|^{1+\eps}$ by using  Lemma \ref{lem:new-external-1d} directly as follows.
We set $r_M=|x|$, which gives, using $r_M=\exp((\log M)^2)$, the value $M=\exp(\sqrt{\log |x|})$, which is slowly varying in $|x|$. Lemma \ref{lem:new-external-1d} defines $\calH_M:=B_{2r_M}\cap \calV_M$, and with $C_M=M^{2(\tau-1)+3\mu}$ and $\kappa_M=2M^{3\mu+2}$, it states that 
\begin{align*}
\pr\Big(\calA_\mathrm{near}(\calH_M, C_M,C_M,0) &\cap \calA_\mathrm{near}(\calH_M, C_M,C_M,x)\\
&\cap \calA_{\mathrm{linear}}(\calH_M, \calH_M, r_M,\kappa_M, 0,2\kappa_M,0) \mid 0,x \in \calC_\infty\Big)\ge 1-3q/10.
\end{align*}
The first two events $\calA_\mathrm{near}(\calH_M, C_M,C_M,z)$ with $z\in\{0,x\}$ guarantee that we find two paths $\pi_{0,y^\star_0}$ and $\pi_{x, y_x^\star}$ with cost at most $C_M$ from $0$ and $x$ to respective vertices $y_0^\star, y_x^\star\in \calV_M$, that are fully contained in $B_{C_M}(0)$ and $B_{C_M}(x)$, respectively. Here, $C_M=M^{2(\tau-1)+3\mu}\le |x|^{\eps/2}$ for sufficiently large $|x|$. Then, since $y_0^\star$ is within distance $C_M<r_M$ from $0$, the third event $\calA_{\mathrm{linear}}$ guarantees a path between $y_0^\star$ and every vertex in $\calH_M=\calV_M\cap B_{2|x|}(0)$ with $\kappa_M$-linear cost, in particular there is such a path $\pi_{y_0^\star, y_x^\star}$ between $y_0^\star$ and $y_x^\star$. Let $\pi_{0,x}:=\pi_{0,y_0^\star}\pi_{y_0^\star, y_x^\star}\pi_{y_x^\star, x}$ be the concatenation of these paths. Since the distance  $|y_0^\star- y_x^\star|\le 2C_M+|x|$, the cost and deviation  of this path is, using that $\kappa_M=2M^{3\mu+2}\le |x|^{\eps/2}$ for $|x|$ large,
\begin{align*}
\calC(\pi_{0,x})&=\calC(\pi_{0,y_0^\star})+\calC(\pi_{y_x^\star, x})+ \calC(\pi_{y_0^\star, y_x^\star}) \le 2C_M + \kappa_M|y_0^\star-y_x^\star| +2\kappa_M\\
&\le 2|x|^{\eps/2} + |x|^{\eps/2} (|x|+2|x|^{\eps/2}) +2 |x|^{\eps/2}\le |x|^{1+\eps},
\\
\mathrm{dev}(\pi_{0,x})&\le \max\{\mathrm{dev}_{0,x}(\pi_{0,y_0^\star}), \mathrm{dev}_{0,x}(\pi_{y_0^\star, y_x^\star}), \mathrm{dev}_{0,x}(\pi_{y_x^\star},x)\} \\
&\le \max\{C_M , 0|y_0^\star-y_x^\star|+2\kappa_M\} \le |x|^\eps,
\end{align*}
for $|x|$ large enough. For small $|x|$ we can absorb the costs and deviation in the constant $C$.
This proves the lemma when $d=1$ and $\mu>\mu_{\mathrm{pol}}$ with $\eta_0=1$. 

When $\mu>\mu_{\mathrm{pol}}$ and $d\ge 2$, a straightforward adaptation of the above proof for $d=1$ could in principle also be used in higher dimensions. 
However, with some more effort one can get rid of the extra $+\eps$ in the exponent, and prove fully linear cost distances. We prove this stronger version (without the $|x|^\eps$ factor) in~\cite[Theorems 1.8, 1.10]{komjathy2022one2}. 
\end{proof}

\subsection{Proof of the Main Theorems}\label{sec:main_proof}

The proofs of Theorems~\ref{thm:polylog_regime} and~\ref{thm:polynomial_regime} follow directly from Lemmas~\ref{lem:polylog-deviation} and~\ref{lem:polynomial-deviation}, respectively, and so do their extensions to $\alpha=\infty$ and/or $\beta=\infty$ in Theorem~\ref{thm:threshold_regimes}. It remains to prove Theorem~\ref{thm:finite_graph}, including its extension to $\alpha=\infty$ and/or $\beta=\infty$.

\begin{proof}[Proof of Theorem~\ref{thm:finite_graph}]
Following Def.~\ref{def:girg}, let $G_n$ be a finite GIRG obtained by intersecting an IGIRG $G = (\calV,\calE)$ with a finite cube $Q_n$ of volume $n$, and let $u_n,v_n$ be two vertices chosen uniformly at random from $\calV \cap Q_n$. For the polylogarithmic case we must prove \eqref{eq:finite-polylog}.
For this, first we prove the slightly different statement that for two uniformly random \emph{positions} $x_n,y_n\in Q_n$, 
\begin{align}\label{eq:replace-giant-with-infinite2}
    \lim_{n\to\infty} \pr\big( d_{\calC}^{G_n}(x_n,y_n) > (\log |x_n-y_n|)^{\Delta_0+\varepsilon} \mid x_n, y_n \in \calC_\infty  \big) =0.
\end{align}
Compared to \eqref{eq:finite-polylog}, there are two differences. First, $\calC_\infty$ replaces $\calC^{\scriptscriptstyle{(n)}}_{\max}$ in the conditioning. By~\cite[Theorem~3.11]{komjathy2020explosion} there is a constant $\rho >0$ such that a.a.s.\! $|\calC^{\scriptscriptstyle{(n)}}_{\max}| \ge \rho |\calV \cap Q_n| \ge \rho|\calC_\infty \cap Q_n|$, and on the other hand $\lim_{n\rightarrow\infty} \pr\big(\calC_{\max}^{(n)} \subseteq \calC_{\infty}\big)=1$ since $\calC_\infty$ is unique. Hence, the probability of the conditions $\pr(u_n,v_n\in \calC^{\scriptscriptstyle{(n)}}_{\max})$ and $\pr(u_n,v_n \in \calC_\infty)$ differ by at most a constant factor, which means that~\eqref{eq:finite-polylog} is equivalent to conditioning on $\{u_n,v_n \in \calC_\infty\}$.
Secondly, in \eqref{eq:finite-polylog} we draw two random vertices $u_n,v_n$ from $\calV\cap Q_n$, while in \eqref{eq:replace-giant-with-infinite2} we draw two random positions $x_n,y_n$ and condition on those being in the vertex set. This changes the number of vertices in $Q_n$ from $\textrm{Poisson}(n)$ to $\textrm{Poisson}(n)\!+\!2$. The total variation distance of these two distributions is vanishing as $n\to\infty$, so this difference can also be ignored, and proving \eqref{eq:replace-giant-with-infinite2} implies \eqref{eq:finite-polylog}.

To prove~\eqref{eq:replace-giant-with-infinite2}, let $C>0$ be the constant from Lemma~\ref{lem:polylog-deviation}, let $0<\zeta\lls q,\mpar$ and consider the event $\calA_{\text{pos}}(x_n, y_n)$ that $|x_n-y_n| \ge \log n$ and that $x_n,y_n$ have distance at least $2\sqrt{d}\zeta n^{1/d}$ from the boundary of $Q_n$, a box of side-length $n^{1/d}$. Then, since $\zeta\lls q, \mpar$,
\begin{align}\label{eq:proof-of-main-theorems1}
\pr(\calA_{\text{pos}}(x_n, y_n)) \ge 1-q/2.
\end{align}
Consider now any given realisation $x_n,y_n\in Q_n$ of the random positions that satisfy $\calA_{\text{pos}}(x_n, y_n)$. By Lemma~\ref{lem:polylog-deviation} applied with $\eps_{\ref{lem:polylog-deviation}}:=\eps/2, q_{\ref{lem:polylog-deviation}}:=q/2$, conditional on $x_n,y_n\in \calC_\infty$ there is a path $\pi_{x_n, y_n}$ from $x_n$ to $y_n$ with $\mathrm{dev}(\pi_{x_n,y_n})\le \zeta|x_n-y_n| + C \le 2\sqrt{d}\zeta n^{1/d}$ and cost at most $\calC(\pi) \le (\log|x_n-y_n|)^{\Delta_0+\eps/2}+C$ with probability at least $1-q/2$. Since $\calA_{\text{pos}}(x_n, y_n)$ holds, the deviation bound of $\pi_{x_n, y_n}$ ensures that the path  $\pi_{x_n,y_n}$ lies fully within $Q_n$ and thus in $G_n$. Moreover, since $|x_n-y_n|\ge \log n$ and $n$ is sufficiently large, $\calC(\pi) \le (\log|x_n-y_n|)^{\Delta_0+\eps/2}+C \le (\log|x_n-y_n|)^{\Delta_0+\eps}$. Hence, for all  $n$ large enough, whenever $x_n,y_n$ satisfies $\calA_{\text{pos}}(x_n, y_n)$,  
\begin{align}\label{eq:proof-of-main-theorems2}
\pr\big( d_{\calC}^{G_n}(x_n,y_n) \le (\log |x_n-y_n|)^{\Delta_0+\varepsilon} \mid x_n, y_n \in \calC_\infty\big) \ge 1-q/2.
\end{align}
Since $q$ was arbitrary, together with~\eqref{eq:proof-of-main-theorems1}, this proves~\eqref{eq:replace-giant-with-infinite2} and concludes the proof for the polylogarithmic case of Theorem~\ref{thm:finite_graph} (including the extensions for $\alpha=\infty$, and/or $\beta =\infty$). The polynomial case is identical except that we use Lemma~\ref{lem:polynomial-deviation} instead of Lemma~\ref{lem:polylog-deviation}.
\end{proof}

\paragraph{Acknowledgments}
We thank Zylan Benjert for generating the simulations and the heatmaps in Figure~\ref{fig:simulation}.

\appendix
\section{Appendix}
\subsection{Chernoff bound and proofs related to nets in Section \ref{sec:nets}}\label{sec:chernoff}
\begin{theorem}\label{thm:chernoff}{(\textbf{Chernoff bounds})}
Let $X_1,\ldots X_k$ be independent Bernoulli distributed random variables, and define $X:=\sum_{i=1}^k X_i$ and $m:=\E[X]$. Then, for all $\lambda\in(0,1]$ and all $t\ge 2em$,
\begin{align*}
    \pr(X\leqslant (1-\lambda)m)\leqslant e^{-m\lambda^2/2},\quad\quad
    \pr(X\geqslant (1+\lambda)m)\leqslant e^{-m\lambda^2/3}, \quad\quad
    \pr(X\ge t) \le 2^{-t}.
\end{align*}
The same bounds hold when $X$ is instead a Poisson variable with mean $m$.
\end{theorem}
We restate the following claim before proving it:
\ClaimPolylogRegime*

\begin{proof}[Proof of Claim~\ref{claim:nets-mu-bound}]
    We start by showing \eqref{eq:mui-bound}. 
    We write $I(w) = [w_-, w_+)$. 
    The definition of a base-$2$-cover (Definition~\ref{def:base-2-cover}) ensures that $w_-,w_+ \in [w/2,2w]$ and $w_+/w_-=2$.
    Thus by the lower bound~\eqref{eq:nets-ell-bound-0} on $w_0$, for all $w\in[w_0, f(r_R)]$,
    \begin{align*}
        \pr(W \in I(w)) &=
        \ell(w_-)w_-^{-(\tau-1)} - \ell(w_+)w_+^{-(\tau-1)}\\
        &\ge \ell(w)\Big(\frac{99}{100}w_-^{-(\tau-1)} - \frac{101}{100}w_+^{-(\tau-1)}\Big)\\
        &= \ell(w)w_-^{-(\tau-1)}\Big(\frac{99}{100} - \frac{101}{100}\cdot 2^{-(\tau-1)}\Big).
    \end{align*}
    Since $\tau > 2$ and $w_- \le w$, it follows that
    \[
        \mu_i(I(w)) \ge (r_i')^d\ell(w) w^{-(\tau-1)}/2^{\tau+1}.
    \]
    The required lower bound on $\mu_i(I(w))$ then follows by the lower bound (P1) in Definition \ref{def:nets-partition} on $r_i'$.
    By a very similar argument to the lower bound, \eqref{eq:nets-ell-bound-0} also implies that
    \begin{align*}
        \pr(W \in I(w)) 
        \le \frac{101}{100}\ell(w)w_-^{-(\tau-1)} \le 2^\tau\ell(w)w^{-(\tau-1)},
    \end{align*}
    so the required upper bound on $\mu_i(I(w))$ likewise follows by the upper bound (P1) on $r_i'$.
    It remains to show \eqref{eq:z-bound}. 
    First we show that $f(r_R) \ge f(r_{R-1}) \ge \ldots \ge f(r_1) \ge w_0$. Indeed, let $j \in [R]$. By the definition of $f$ in \eqref{eq:nets-f-def},
\begin{align}\label{eq:f-ratio}
        \frac{f(r_j)}{f(r_{j-1})} = \left(\frac{r_j}{r_{j-1}}\right)^{d/(\tau-1)} \cdot \left(\frac{1 \wedge \inf\big\{\ell(x)\colon x \in [w_0,r_j^{d/(\tau-1)}]\big\}}{1 \wedge \inf\big\{\ell(x)\colon x \in [w_0,r_{j-1}^{d/(\tau-1)}]\big\}}\right)^{1/(\tau-1)}.
    \end{align}
    To bound the second factor, we will first observe that since $r_{j-1} \le r_j$, we have $\inf\{\ell(x)\colon x \in [w_0,r_j^{d/(\tau-1)}]\} \le \inf\{\ell(x)\colon x \in [w_0,r_{j-1}^{d/(\tau-1)}]\}$. By considering the two possible values of $1 \wedge \inf\{\ell(x)\colon x \in [w_0,r_j^{d/(\tau-1)}]\}$ separately, it follows that we can drop the minimum with $1$ in the ratio:
    \begin{align}
\begin{split}\label{eq:f-ratio-inter}
    &\left(\frac{1 \wedge \inf\big\{\ell(x)\colon x \in [w_0,r_j^{d/(\tau-1)}]\big\}}{1 \wedge \inf\big\{\ell(x)\colon x \in [w_0,r_{j-1}^{d/(\tau-1)}]\big\}}\right)^{1/(\tau-1)} \\
    &\qquad\ge \left(\frac{\inf\big\{\ell(x)\colon x \in [w_0,r_j^{d/(\tau-1)}]\big\}}{\inf\big\{\ell(x)\colon x \in [w_0,r_{j-1}^{d/(\tau-1)}]\big\}}\right)^{1/(\tau-1)}.
    \end{split}
    \end{align}
    We now bound $\ell(x)$ on $[w_0, r_{j-1}^{d/(\tau-1)}]$ (in the denominator) by repeatedly applying~\eqref{eq:nets-ell-bound-0}. We write $\lg(\cdot)=\log_2(\cdot)$. Then, we have $r_j^{d/(\tau-1)} = 2^{\lg ((r_j/r_{j-1} )^{d/(\tau-1)})} r_{j-1}^{d/(\tau-1)}$, so we iterate the bound in \eqref{eq:nets-ell-bound-0} roughly $\lg ((r_j/r_{j-1} )^{d/(\tau-1)})$ many times to obtain that for all $x \in [r_{j-1}^{d/(\tau-1)}, r_j^{d/(\tau-1)}]$ we have 
    \[
        \ell(x) \ge \Big(\frac{99}{100}\Big)^{1 + \lg ((r_j/r_{j-1} )^{d/(\tau-1)})} \ell(r_{j-1}).
    \]
    Returning to \eqref{eq:f-ratio-inter}, the ratio of the two infima in the smaller interval $[w_0,r_{j-1}^{d/(\tau-1)}]$ is one. And 
    since $r_j \ge 2r_{j-1}$ by (R2) of Definition~\ref{def:well-spaced} (using $\delta <1/16 < 1$), it follows that
    \begin{align*}
        \left(\frac{\inf\big\{\ell(x)\colon x \in [w_0,r_j^{d/(\tau-1)}]\big\}}{\inf\big\{\ell(x)\colon x \in [w_0,r_{j-1}^{d/(\tau-1)}]\big\}}\right)^{1/(\tau-1)} 
        &\ge \Big(\frac{99}{100}\Big)^{1 + \lg ((r_j/r_{j-1} )^{d/(\tau-1)}) } \\
        &\ge \Big(\frac{1}{2}\Big)^{\frac{1}{\tau-1}\lg ((r_j/r_{j-1} )^{d/(\tau-1)}) }\\
        & = \Big(\frac{r_{j-1}}{r_j}\Big)^{d/(\tau-1)^2} .
    \end{align*}
    Combining this with~\eqref{eq:f-ratio}, \eqref{eq:f-ratio-inter}, and the fact that $\tau > 2$, we obtain
    \[
        \frac{f(r_j)}{f(r_{j-1})} \ge \Big(\frac{r_j}{r_{j-1}}\Big)^{\frac{d}{\tau-1}(1-1/(\tau-1))} \ge 1.
    \]
    Hence $f(r_R) \ge f(r_{R-1}) \ge \ldots \ge f(r_1)$,  as claimed.
      It is now relatively easy to prove the desired lower bound. From the definition of $f$ in~\eqref{eq:nets-f-def}, we have $f(r_i) \le r_i^{d/(\tau-1)}$, and  \eqref{eq:nets-small-r} in the definition of well-spacedness ensures that $f(r_i)\ge w_0$, so
    \begin{align*}
    \ell(f(r_i))f(r_i)^{-(\tau-1)} &
        =
        \frac{\ell(f(r_i))}{1 \wedge \inf\{\ell(x)\colon x \in [w_0,r_i^{d/(\tau-1)}]\}} \cdot \frac{(2d)^{2\tau+d+8}\log(16R/\delta)}{r_i^d}  \\
        &\ge \frac{(2d)^{2\tau+d+8}\log(16R/\delta)}{r_i^d}.
    \end{align*}
    Multiplying by $r_i^d$ finishes the statement of \eqref{eq:z-bound}. 
\end{proof}

\subsection{Proofs of the building block claims in Section~\ref{sec:lemmas}}\label{app:proof-building-lemmas}

For the sake of readability, we will restate each of the claims before proving them.

\ClaimSingleBridgeEdge*

\begin{proof} 
    We first bound the number of possible edges in $N_{\eta,\gamma,z,\ulw}(Z_x,Z_y)$ from below. In Def.~\ref{def:weak-net}, we will take $\eps=\delta/4$, $w=\ulw D^{z/2}$ and $r=D^\gamma$. We have $\ulw D^{z/2} \ge \ulw \ge w_0$. Since $2d\gamma > z(\tau-1)$, we have $z/2 < d\gamma/(\tau-1)$, and $\ulw \le D^\delta$ by hypothesis; for sufficiently small $\delta$, it follows that $\ulw D^{z/2} < D^\delta \cdot D^{d\gamma/(\tau-1) - 2\delta} \le D^{d\gamma/(\tau-1) - \delta/4}$; thus the requirement on $\ulw$ of Def.~\ref{def:weak-net} is satisfied. Also $D^\gamma \in [(\log\log\xi\sqrt{d})^{16/\delta}, \xi\sqrt{d}]$ by hypothesis. Thus \eqref{eq:net-defining-crit-eps} gives for $v \in \{x,y\}$:
    \[
	    |\calZ(v)| \ge D^{\gamma d(1-\delta/4)} \ell(\ulw D^{z/2})\ulw^{-(\tau-1)}D^{-z(\tau-1)/2}.
   	\]
    Since $\ulw D^{z/2} \le D^{\delta+z/2}$ and $D\ggs \delta, z$, 
    by Potter's bound
    $|\calZ(v)| \ge \ulw^{-(\tau-1)}D^{d\gamma - z(\tau-1)/2 - 3\delta\gamma d/10}$, so  $|Z_v| \ge |\calZ(v)|/4 \ge \ulw^{-(\tau-1)}D^{d\gamma - z(\tau-1)/2 - 3\delta\gamma d/10}/4$ for $v \in \{x,y\}$.
    Accounting for the possibility of overlap between $Z_x$ and $Z_y$, we obtain
    \begin{equation}\label{eq:cheap-bridge-N}
	    \big|\{\{a,b\}: a\in Z_x, b\in Z_y|\}\big|\! \ge\! \frac{(|Z_x|\wedge|Z_y|)^2}{4} \ge  \ulw^{-2(\tau-1)}D^{2d\gamma - z(\tau-1) - 3 \delta \gamma d/5}/64.
	\end{equation}
   	We now lower-bound the probability that $a\in Z_x, b\in Z_y$ forms a low-cost edge as in \eqref{eq:N-gamma-eta-z}.
    By hypothesis $|x-y| \le c_HD$, $a \in B_{D^\gamma}(x)$, $b \in B_{D^\gamma}(y)$, and $\gamma<1$, so
	    $|a-b| \le c_HD + 2D^\gamma \le (c_H+2)D.$
	Since $w_a, w_b \in [\ulw D^{z/2}/2, 2\ulw D^{z/2}]$ by \eqref{eq:calZ-def}, it follows from \eqref{eq:connection_prob} that
	\begin{equation}\label{eq:bb-conn-prob}
	    \pr\big(ab \in \calE(G') \mid V, w_V  \big) \ge \theta\underline{c}\Big(1 \wedge \frac{\ulw^2D^z}{4(c_H+2)^dD^d} \Big)^\alpha \ge \theta\underline{c}\big(1 \wedge \ulw D^{z-d} \big)^\alpha,
	\end{equation}
	where we used the assumption that $\underline w\ge 4(c_H+2)^d$ to simplify the formula.
	Since $z \in [0,d]$, and $\delta$ is small relative to $z$, if $z-d < 0$ then we may assume $z-d \le -2\delta$. Since $1 \le \ulw \le D^\delta$, the minimum in \eqref{eq:bb-conn-prob} is attained at $1$ only for $z=d$. So, for all $\{a,b\} \in Z_x\times Z_y$,
	\begin{equation}\label{eq:bb-conn-prob-2}
	    \pr\big(ab \in E(G') \mid V,w_V\big)\ge
	    \begin{cases}
	        \theta \underline{c} \mathbbm{1}\{z = d\}& \mbox{if }\alpha=\infty,\\
	        \theta \underline{c}  D^{\alpha(z-d)}& \mbox{otherwise.}
	      \end{cases}
	\end{equation}
	Since $w_a,w_b \le 2\ulw D^{z/2}$ by \eqref{eq:calZ-def}, 
	\begin{align*}
	    \pr\big(\calC(ab)\le (\ulw/10)^{3\mu}D^\eta \mid ab\in\calE(G'), V,w_V\big)
        &\ge \pr\big((4\ulw^2D^z)^{\mu}L_{ab} \le (\ulw/10)^{3\mu}D^\eta\big) \\
        &= F_L(4000^{-\mu}\ulw^{\mu} D^{\eta-\mu z}).
	\end{align*}
	If $\eta < \mu z$, then since $\delta\lls z,\eta,\mpar$ we may assume that $\eta-\mu z \le -2\mu\delta$. Since $\ulw \le D^\delta$ and $D\ggs \delta$, it follows that $4000^{-\mu}\ulw^{\mu} D^{\eta-\mu z} \le D^{-\mu\delta} \le t_0$ and hence using Assumption~\ref{assu:L} and the assumption $\ulw \ge 4000 c_1^{-1/(\mu\beta)}$ we get $F_L(4000^{-\mu}\ulw^{\mu} D^{\eta-\mu z}) \ge D^{\beta(\eta-\mu z)}$ after simplifications. If instead $\eta \ge \mu z$, then $F_L(4000^{-\mu}\ulw^{\mu} D^{\eta-\mu z}) \ge F_L(4000^{-\mu}\ulw^{\mu}) \ge 1/2$ by hypothesis. Summarising the two cases with indicators we arrive at 
	\begin{align}
	\begin{split}\label{eq:bb-cost-prob}
	    \pr\big(\calC(ab)\le (\ulw/10)^{3\mu}D^\eta \mid ab\in \calE(G'), V,w_V\big)
	    \ge 
	    \begin{cases}
	        \mathbbm{1}\{\eta \ge \mu z\}/2 & \mbox{if }\beta=\infty,\\
	        D^{0 \wedge \beta(\eta -\mu z)}/2 & \mbox{otherwise.}
	    \end{cases}
	\end{split}
	\end{align}
	With the convention that $\infty \cdot 0 = 0$, the second row equals the first row in both~\eqref{eq:bb-conn-prob-2} and~\eqref{eq:bb-cost-prob}. Combining \eqref{eq:bb-conn-prob-2} and~\eqref{eq:bb-cost-prob}, we obtain that for all $\{a,b\}\in Z_x\times Z_y$:
	\begin{equation}\label{eq:cheap-bridge-p}
	    \pr\big(\{a,b\} \in N_{\eta,\gamma,z,\ulw}(x,y) \mid V,w_V\big) \ge \theta \underline{c} D^{\alpha(z-d) + (0 \wedge \beta(\eta - \mu z))}/2.
	\end{equation}
	Given $V, w_V$, the possible edges $\{a,b\}$ lie in $N_{\eta,\gamma,z,\ulw}(x,y)$ independently. Hence by~\eqref{eq:cheap-bridge-N} and~\eqref{eq:cheap-bridge-p}, $|N_{\eta,\gamma,z}(x,y)|$ stochastically dominates a binomial random variable whose mean $m$ is the product of the two equations' right-hand sides. On bounding $\underline{c}/128 \ge D^{-\delta \gamma d/15}$, we obtain
    \[
        m \ge \theta \ulw^{-2(\tau-1)}D^{\Lambda(\eta,z)-2\delta\gamma d/3}.
    \]
    Inequality \eqref{eq:lem-cheap-bridge} follows since this binomial variable is zero with probability at most $e^{-m}$.
\end{proof}

\ClaimSingleCheapEdgeNearby*

\begin{proof}
    Let $r = 4^{-1/d}(D \wedge (KM)^{1/d})$, and define $\calZ'(x):=\calN \cap (B_r(x)\times [\tfrac12K, 2K])$. 
    We will first lower-bound $|\calZ'(x)|$ by applying Definition~\ref{def:weak-net} with $\eps=\delta/4$, $w=K$ and the same value of $r$. Observe that~\eqref{eq:KM-xi} and the fact that $K \ge w_0$ imply all the requirements of Definition~\ref{def:weak-net} except $K \le r^{d/(\tau-1)-\delta/4}$, which we now prove. By~\eqref{eq:K-D-M}, $M \ge K^{\tau-2+\tau\delta}$ and hence
    \[
    \begin{aligned}
      r^{d/(\tau-1)-\delta/4}&\ge  (KM/4)^{1/(\tau-1) - \delta/(4d)} \ge (K^{\tau-1+\tau\delta})^{1/(\tau-1)-\delta/4} /4\\
      &=  K^{1 + \delta(\tau/(\tau-1)-(\tau-1)/4-\tau\delta/4)}/4.
      \end{aligned}
    \]
    Since $\tau<3$ and $\delta\lls \tau$, the coefficient of $\delta$ is positive in the exponent so the rhs is at least $K$, as required required by Def.~\ref{def:weak-net}. 
Applying \eqref{eq:net-defining-crit-eps} followed by  Potter's bound (since $D,K\ggs \delta$) yields that
    \begin{equation}\label{eq:cheap-edge-to-nice-vertex-N}
       |\calZ'(x)| \ge \ell(K)K^{-(\tau-1)}r^{d(1-\delta/4)} \ge K^{-(\tau-1)}\big(D \wedge KM)^{1-\delta/2}.
    \end{equation}
        We now lower-bound the probability that for a $y\in \calZ'(x)$ the edge $xy$ is present and has cost at most $U$, satisfying the requirements of $\calA_{K,D,U}(x)$. 
    Let $y\in \calZ'(x)$. Since $w_x\in[M/2, 2M]$ and $w_y\in[K/2, 2K]$, by \eqref{eq:connection_prob} and the definition of $r=D \wedge (KM/4)^{1/d}$ we have
    \begin{equation}\label{eq:cheap-edge-to-nice-vertex-p1}
        \pr\big(xy\in \calE \mid V, w_V\big) \ge \theta\underline{c}(1 \wedge KM/(4r^d))^\alpha = \theta\underline{c},
    \end{equation}
   since the minimum is at the first term; also for $\alpha=\infty$.
    Moreover,  if $\beta < \infty$, we apply \eqref{eq:F_L-condition}; otherwise, since $U(KM)^{-\mu}$ is large, $F_L(U(4KM)^{-\mu}) \ge 1/2$, to estimate the cost
    \begin{align}
    \begin{split}\label{eq:cheap-edge-to-nice-vertex-p2}
        \pr\big(\cost{xy} \le U \mid xy\in \calE(G'), V,w_V\big)
        &\ge \pr\big((4KM)^{\mu}L_{xy} \le U\big) = F_L((4KM)^{-\mu}U) \\
        &\ge C(1 \wedge (U(KM)^{-\mu})^\beta),
    \end{split}
    \end{align}
for an appropriate choice of $C>0$ depending only on \mpar.
    Combining~\eqref{eq:cheap-edge-to-nice-vertex-p1} and~\eqref{eq:cheap-edge-to-nice-vertex-p2}, we obtain for any $y\in \calZ'(x)$:
    \begin{equation}\label{eq:cheap-edge-to-nice-vertex-p}
	    \pr\big(xy \in \calE(G') \mbox{ with }\cost{xy} \le U\mid V,w_V\big) \ge \theta C\underline{c}(1 \wedge (U(KM)^{-\mu})^\beta).
	\end{equation}
	   Conditioned on $(V,w_V)$, edges are present independently, so the number of low-cost edges between $x$ and $\calZ'(x)$ stochastically dominates a binomial random variable with parameters the rhs of~\eqref{eq:cheap-edge-to-nice-vertex-N} and~\eqref{eq:cheap-edge-to-nice-vertex-p}. The mean is
    \[
       \theta C\underline{c} K^{-(\tau-1)}(D^d\wedge KM)^{1-\delta/2} (1 \wedge (U(KM)^{-\mu})^\beta).
    \]
    Since $K,M,D\ggs \delta$, we may swallow the constant factor $\theta C\underline{c}$ by increasing $\delta/2$ to $\delta$. The result follows since for a binomial variable $Z$, $\pr(Z=0)\le \exp(-\mathbb E[Z])$.
\end{proof}

\ClaimWeightIncreasingPath*

\begin{proof} We will build $\pi$ vertex-by-vertex by applying Claim~\ref{claim:cheap-edge-to-nice-vertex} $q$ times. We first define target weights. Let $M_0 := M$, and for all $i \in [q]$, let
   \begin{align}\label{eq:Mi}
       M_i := M^{1/(\tau-2+2d\tau\delta)^i} \wedge K.
   \end{align}
   Since $\tau<3$ and $\delta$ is small, $\tau-2+2d\tau\delta < 1$; hence on substituting the definition of $q$ in~\eqref{eq:cond-q} into~\eqref{eq:Mi} and removing the ceiling, we obtain
   \begin{equation*}
       M^{1/(\tau-2+2d\tau\delta)^q} = \exp\Big(\log M\cdot \mathrm{e}^{{-}q\log(\tau-2+2d\tau\delta)}\Big) \ge \exp\Big(\log M\cdot \mathrm{e}^{\log\big(\frac{\log K}{\log M}\big)}\Big) = K,
   \end{equation*}
   and hence $M_q = K$. 
   By a very similar argument, $M_{q-1} < K$.
   We now define $Y_0 = y_0$, and define an arbitrary ordering on $\calN$. For all $i \in [q]$, we define $Y_i$ to be the first vertex in $\calN$ in $B_D(Y_{i-1}) \times [\tfrac{1}{2}M_i, 2M_i]$ so that  the edge $Y_{i-1}Y_i$ is present in $G'$ and has cost at most $U$. If no such vertex exists, we define $Y_j = \texttt{None}$ for all $j \ge i$. Let $\calA_i$ be the event that $Y_0,\ldots,Y_{i} \ne \texttt{None}$. Then, if $\calA_q$ occurs, the path $\pi=Y_0\ldots Y_q$ yields $V(\pi)\subseteq \calN \cap B_{qD}(y_0)$ and $\cost{\pi} \le qU$, and that $w_{Y_q} \in [\tfrac{1}{2}K,2K]$ since $M_q = K$. So, (and because $\calA_{i-1}\subseteq \calA_i$), 
   \begin{equation}
\label{eq:cheap-increasing-path-step-1}
        \pr\big(\calA_{\pi(y_0)} \mid V,w_V\big) \ge \pr\big(\calA_q \mid V,w_V\big) 
        = \prod_{i=1}^q \pr\big(\calA_i \mid \calA_{i-1}, V,w_V\big).
   \end{equation}
      We now simplify the conditioning in~\eqref{eq:cheap-increasing-path-step-1}. For all $i \in [q]$, let
   \begin{align}
        \calF_i &:= \calE(G') \cap \big\{\{x,y\}\colon x,y\in\calN,\ w_x \in [\tfrac{1}{2} M_{i-1},2M_{i-1}],\ w_y \in [\tfrac{1}{2}M_i,2M_i]\big\},\label{eq:fi-111}\\
        \calF_{\le i} &:= (\calF_1,\dots,\calF_i),\qquad\qquad p_i := \pr\big(\calA_i \mid \calA_{i-1},V,w_V\big).\nonumber
   \end{align}
   Observe that $\calA_1,\dots,\calA_i$ and $Y_1,\dots,Y_i$ are deterministic functions of $\calF_{\le i}$. Moreover, if $F$ is a possible realisation of $\calF_{\le i-1}$ such that $\calA_{i-1}(F)$ occurs, then conditioned on $\calF_{\le i-1}=F$, $\calA_i$ occurs iff $Y_i \ne \texttt{None}$. Thus, with $\calA_{K,D,U}$ from \eqref{eq:akdu}, \eqref{eq:cheap-increasing-path-step-1} implies that for all $i \in [q]$,
   \begin{align}\nonumber
       p_i&\ge \min_{F\colon \calA_{i-1}(F)\text{ occurs}} \pr\big(Y_i \ne \texttt{None} \mid \calF_{\le i-1} = F,\,V,w_V\big)\\\label{eq:cheap-increasing-path-step-1b}
       &= \min_{F\colon \calA_{i-1}(F)\text{ occurs}} \pr\big(\calA_{M_i,D,U}(Y_{i-1}(F)) \text{ occurs}\mid \calF_{\le i-1} = F,\,V,w_V\big).
   \end{align}
   Now observe that since $\tau \in (2,3)$, $\delta$ is small, and $M\ggs \delta$, for all $i \in [q-2]$, we may assume $M_{i} = M_{i-1}^{1/(\tau-2+2d\tau\delta)^i} > 4M_{i-1}$, and moreover $M_q \ge M_{q-1}$. Therefore, the intervals $[\tfrac{1}{2}M_1,2M_2],\dots,[\tfrac{1}{2}M_{q-1},2M_q]$ are all disjoint except possibly for $[\tfrac{1}{2}M_{q-2},2M_{q-1}]$ and $[\tfrac{1}{2}M_{q-1},2M_q]$. It follows that the variables $\calF_1,\ldots,\calF_q$ are determined by disjoint sets of possible edges. For $\calF_{q-1}$ and $\calF_q$ the weights $w_{q-1}$ and $w_q$ still fall into disjoint intervals, so the edge sets in $\calF_i$ in \eqref{eq:fi-111} are disjoint across $i$. So,  $\calA_{M_i,D,U}(Y_{i-1}(F))$ is independent of $\calF_{\le i-1}$ in~\eqref{eq:cheap-increasing-path-step-1b} (conditioned on $(V,w_V)$), and hence
   \begin{equation}\label{eq:cheap-increasing-path-step-2}
    p_i \ge \min_{y \colon w_y \in [M_{i-1}/2,2M_{i-1}]} \pr\big(\calA_{M_i,D,U}(y) \mid V,w_V\big).
   \end{equation}
   We now apply Claim~\ref{claim:cheap-edge-to-nice-vertex} on the rhs: take there $\delta_{\ref{claim:cheap-edge-to-nice-vertex}} := \delta$, $\theta_{\ref{lem:CETNV}} := \theta$, $M_{\ref{lem:CETNV}}:=M_{i-1}, K_{\ref{lem:CETNV}}:=M_i$, $D_{\ref{lem:CETNV}}:=D$ and $U_{\ref{lem:CETNV}}:=U$. Observe $K_{\ref{lem:CETNV}}, M_{\ref{lem:CETNV}}\ge M$; thus by hypothesis we have that $\delta_{\ref{lem:CETNV}}\lls\mpar$ is small and that $K_{\ref{lem:CETNV}},M_{\ref{lem:CETNV}},D_{\ref{lem:CETNV}}\ggs \delta, w_0$, as required by Claim~\ref{lem:CETNV}. Next, we have $(K_{\ref{lem:CETNV}}M_{\ref{lem:CETNV}})^\mu \le K^{2\mu}$, so if $\beta=\infty$ it follows that $U_{\ref{lem:CETNV}}(K_{\ref{lem:CETNV}}M_{\ref{lem:CETNV}})^{-\mu} \ge UK^{-2\mu}$ is large as required, by the assumptions before \eqref{eq:cond-q}.
   Next, $(D \wedge (M_iM_{i-1})^{1/d})/4^d \le D \le \xi\sqrt{d}$ by hypothesis,
   and $(D \wedge (M_iM_{i-1})^{1/d})/4^d \ge (M_0/2)^{2/d} \ge (\log\log\xi\sqrt{d})^{16/\delta}$ by hypothesis, so~\eqref{eq:KM-xi} holds. Next, $M_i \le M_q = K \le D^{d/2} \le D^{d/(\tau-1)-\delta}$ by hypothesis and because $\delta$ is small; and finally $M_i \le M_{i-1}^{1/(\tau-2+2d\tau\delta)} < M_{i-1}^{1/(\tau-2+\tau\delta)}$ by definition, so~\eqref{eq:K-D-M} holds. Thus the conditions of Claim~\ref{claim:cheap-edge-to-nice-vertex} all hold, and applying  \eqref{eq:cheap-edge-to-nice-vertex} to~\eqref{eq:cheap-increasing-path-step-2} yields that for all $i \in [q]$:
    \begin{align}\label{eq:cheap-increasing-path-step-3}
    p_i \ge1-\exp\Big({-}\theta M_i^{-(\tau-1)}\Big[(D^d \wedge M_iM_{i-1})^{1-\delta} (1\wedge (U(M_iM_{i-1})^{-\mu})^\beta\Big]\Big) .
    \end{align}
    Clearly $M_iM_{i-1} \le M_q^2 = K^2$; and since $K \le D^{d/2}$ and $U\ge K^{2\mu}$ by hypothesis, the first minimum is at $M_iM_{i-1}$, while the second minimum is taken at $1$ on the rhs. Hence
    \[
        p_i \ge 1-\exp\Big({-}\theta M_i^{-(\tau-2)-\delta}M_{i-1}^{1-\delta} \Big).
    \]
    Since $M_i \le M_{i-1}^{1/(\tau-2+2d\tau\delta)}$ by \eqref{eq:Mi}, $\delta$ is small, and $\tau \in (2,3)$, after simplification the exponent of $M_{i-1}$ is at least    $\delta (\tau+1-2d\tau\delta)/(\tau-2+2d\tau\delta)\ge 3\delta$, so $p_i \ge 1-\exp({-}\theta  M_{i-1}^{3\delta})$. Using this bound in~\eqref{eq:cheap-increasing-path-step-1}, we obtain that
    \begin{equation}\label{eq:cheap-increasing-path-step-4}
        \pr\big(\calA_q \mid V,w_V\big) \ge \prod_{i=1}^q \big(1-\exp({-}\theta M_i^{3\delta})\big) \ge 1 - \sum_{i=1}^q \exp({-}\theta M_i^{3\delta}).
    \end{equation}
    Recall that for all $2 \le i \le q-1$, $M_i = M_{i-1}^{1/(\tau-2+2d\tau\delta)}$, and so since $\delta$ is small and $\tau \in (2,3)$ we have $M_i \ge M_{i-1}^{1+\delta}$. Since $M_0 = M\ggs \delta,\theta$, we obtain $\exp({-}\theta M_i^{3\delta}) \le \tfrac{1}{2}\exp({-}\theta M_{i-1}^{3\delta})$. It follows from~\eqref{eq:cheap-increasing-path-step-4} that
    \begin{align*}
        \pr\big(\calA_q \mid V,w_V\big) &\ge 1 - \sum_{i=1}^{q-1}\exp({-}\theta M_i^{3\delta}) - \exp(-M_q^{3\delta})
        \ge 1 - 3\exp({-}\theta M^{3\delta}) \ge 1 - \exp({-}\theta M^\delta)
    \end{align*}
    as required, where the last step holds since $M\ggs \delta,\theta$.
\end{proof}

\ClaimCommonNeighbour*

\begin{proof}
     We define a vertex $v\in\widetilde{\calV}$ as \emph{good} if $v \in \calN \cap (B_{D}(x_0)\times [(c_H+1)^{d}D^{d/2}, 4(c_H+1)^{d}D^{d/2}])$; thus for $\calA_{x_0\star x_1}$ to occur, it suffices that there is a good vertex $v$ such that $x_0vx_1$ is a path of cost at most $D^{2\mu d}$ in $G'$. We call this a \emph{good path}. We first lower-bound the number of good vertices. By assumption, $2(c_H+1)^{d}D^{d/2} \ge D^{d/2}\ge w_0$, and since $\tau < 3$, $\delta$ is small and $D\ggs c_H,\delta$ we have $2(c_H+1)^{d}D^{d/2} \le D^{d/(\tau-1)-\delta/4}$. Since $\calN$ is a weak $(\delta/4, w_0)$ net, by \eqref{eq:net-defining-crit},
    \begin{equation}
    \begin{aligned}\label{eq:common-neighbour-N}
        &\big|\calN \cap (B_{D}(x_0)\times [(c_H+1)^{d}D^{d/2}, 4(c_H+1)^{d}D^{d/2}])\big| \\
        &\qquad\ge
        D^{d(1-\delta/4)}\ell\big(2(c_H+1)^{d}D^{d/2}\big)\big(2(c_H+1)^{d}D^{d/2}\big)^{-(\tau-1)}   \ge D^{(3-\tau-\delta)d/2}, 
    \end{aligned}
    \end{equation}
  where the last inequality follows by Potter's bound since $D\ggs c_H,\delta$. 
    We now lower-bound the probability that for a good $v\in \calN$, the edges $x_0v, vx_1$ are present and have cost at most $D^{3\mu d/2}$ in $G'$. Observe that $|x_1-v| \le |x_1-x_0| + |x_0-v| \le (c_H+1)D$. Thus, by \eqref{eq:connection_prob}, and since $G'$ is a $\theta$-percolation,  $\pr(x_1v\in \calE(G')|V, w_V)\ge \theta\underline{c}\big[1 \wedge (c_H+1)^d(D^{d/2})^2/((c_H+1)D)^d\big]^\alpha = \theta\underline{c}$, also when $\alpha=\infty$. Further, conditioned on the existence of the edge $x_1v$, 
    \begin{align*}
        \pr\big(\cost{x_1v} \le D^{3\mu d/2} \mid x_1v\in\calE(G'), V,w_V\big) &\ge \pr\big((16(c_H+1)^{d}D^{d})^{\mu}L \le D^{3\mu d/2}\big) \\
        &= F_L(16^{-\mu}(c_H+1)^{-\mu d}D^{\mu d/2}) \ge 1/2,
    \end{align*}
    where the last inequality holds (including when $\beta=\infty$) since $D\ggs c_H$. Combining the two bounds, for all good vertices $v\in \widetilde{\calV}$,
    \begin{equation*}
	    \pr\big(x_1v \in \calE(G'), \cost{x_1v} \le D^{3\mu d/2} \mid V,w_V\big) \ge \theta\underline{c}/2.
	\end{equation*}
	Since $|x_0-v| \le D$, the same lower bounds hold for the edge $x_0v$. The two events are independent conditioned on $(V,w_V)$, and since $2D^{3\mu d/2} < D^{2\mu d}$, for all good vertices $v\in \widetilde{\calV}$,
    \begin{equation}\label{eq:common-neighbour-p}
    \begin{aligned}
        &\pr\big(x_0v,x_1v\in \calE(G'), \cost{x_0vx_1} \le D^{2\mu d} \mid V,w_V\big) 
        \ge \theta^2\underline{c}^2/4.
    \end{aligned}
    \end{equation}
Conditioned on $(V,w_V)$, the presence and cost of $x_0vx_1$ vs $x_0v'x_1$ are independent, so the number of good paths between $x_0$ and $x_1$ stochastically dominates a binomial random variable with parameters given by the rhs of~\eqref{eq:common-neighbour-N} and that of~\eqref{eq:common-neighbour-p}. For a binomial variable $Z$, $\pr(Z\neq0) \ge 1-\exp(-\E[Z])$, and so we obtain \eqref{eq:common-neighbor} by absorbing the constant $\underline{c}^2/4$ by replacing $\delta$ with $2\delta$ in the exponent of $D$, using that $D\ggs \delta$.
\end{proof}

\subsection{Proofs using two rounds of exposure in Section \ref{sec:hierarchy}}\label{app:two-rounds-proof}

Here we prove Lemma~\ref{lem:hierarchy-final-weights} and Proposition~\ref{prop:path-from-hierarchy}. We restate the results for the sake of readability.

\LemmaHierarchyHighWeights*

\begin{proof}
As in Lemma~\ref{lem:hierarchy-intermediate-weights}, let $\ulw := \xi^{\gamma^{R-1}\delta}$. To construct a $(\gamma,c_H\olw^{4\mu}\xi^{\eta},\olw,c_H)$-hierarchy in $\calN$, we use two rounds of exposure. In the first round we apply Lemma~\ref{lem:hierarchy-intermediate-weights} to get $\calH_{low}=\{\tilde y_\sigma\}$, a $(\gamma,3\ulw^{3\mu}\xi^{\eta},\ulw,2)$-hierarchy with failure probability $\mathrm{err}_{\xi,\delta}$ in \eqref{eq:hierarchy-intermediate-prob}. In the second round, we use weight-increasing paths from Claim~\ref{claim:cheap-path-to-larger-weight} to connect each $\tilde y_\sigma\in \calH_{low}$ to a vertex $y_\sigma$ of weight in $[\olw,4\olw]$, transforming $\calH_{low}$ into $\calH_{high}=\{y_\sigma\}$, a $(\gamma,c_H\olw^{4\mu}\xi^{\eta},\olw,c_H)$-hierarchy.
    
We now define an iterative cost construction on $G' \sim\{\calG^{\theta}|V, w_V\}$. Let $\calF_1$ be the collection of all tuples of possible edges that could form the bridges $(P_\sigma)_{\sigma\in \{0,1\}^R}$ of a $(\gamma,3\ulw^{3\mu}\xi^{\eta},\ulw,2)$-hierarchy $\{\tilde y_{\sigma}\}_{\sigma\in\{0,1\}^R}$ fully contained in $\calN$, and $\calU_1$ be the event that the costs of all $P_\sigma$ satisfy \emph{(H\ref{item:H3})} of Definition~\ref{def:hierarchy} with $U=3\ulw^{3\mu}\xi^{\eta}$, so that $\calH_{low}$ is indeed a valid $(\gamma,3\ulw^{3\mu}\xi^{\eta},\ulw,2)$-hierarchy. Here, $\calU_1$ is measurable wrt $\sigma(V, w_V, \calE_1)$ as required by Definition~\ref{def:iter-construct} (\ref{item:iter4}), and the round-$1$ marginal costs in \eqref{eq:marginal-cost} are equal to the true edge costs (in $G'$).
For each $\sigma\in\{0,1\}^R$, let $\calP(\sigma)$ be the set of all paths $ \pi_{\tilde y_\sigma, y_\sigma}$ in $\calN$ connecting $\tilde y_{\sigma}$ to any vertex $y_{\sigma}\in \calN \cap (B_{(c_H-2)\xi^{\gamma^{R-1}}/2}(\tilde y_{\sigma}) \times [\olw, 4\olw])$. 
Given $(V, w_V, E_1)$, call a tuple $\underline t$ of edges \emph{admissible} if it contains exactly one such potential path from $\calP(\sigma)$ for each $\sigma\in\{0,1\}^{R}$, and let $\calF_2$ be the collection of all admissible tuples, with an arbitrary ordering.
Let $\calU_2$ be the event that each potential path in the chosen tuple is present in the graph under consideration, and has round-$2$ marginal cost at most $(c_H-3)\olw^{4\mu}\xi^{\eta}/2$ in \eqref{eq:marginal-cost}.
Following Def.~\ref{def:iter-construct}\eqref{item:iter5} and Prop \ref{prop:multi-round-exposure}, for $i=1,2$, we set $\calE_i^{G'}$ to be the first tuple (in the ordering of $\calF_i$) for which $\calU_i$ occurs, and we set $\calE_i^{G'}$ to $\mathtt{None}$ if no such tuple exists.
This defines an iterative cost construction on $G'$, 
$I_{G'} = ((G',\calE_1^{G'},\calF_1,\calU_1), (G',\calE_2^{G'},\calF_2,\calU_2))$. 
If $I_{G'}$ succeeds, then $\{y_\sigma\}_{ \sigma\in\{0,1\}^R}$ is a $(\gamma,c_H\olw^{4\mu}\xi^{\eta},\olw,c_H)$-hierarchy. Indeed, condition~\emph{(H\ref{item:H1})} of Def.~\ref{def:hierarchy} is satisfied by construction. By the triangle inequality, \emph{(H\ref{item:H2})} is satisfied since for all $\sigma \in \{0,1\}^i$, $y_{\sigma1} \in B_{(c_H-2)\xi^{\gamma^i}/2}(\tilde y_{\sigma1})$ and $y_{\sigma0} \in B_{(c_H-2)\xi^{\gamma^i}/2}(\tilde y_{\sigma0})$ by construction, and $|\tilde y_{\sigma1} - \tilde y_{\sigma0}| \le 2\xi^{\gamma^i}$ by \emph{(H\ref{item:H2})} since $\tilde y_\sigma$ forms a $(\gamma, 3\ulw^{3\mu}\xi^{\eta},\ulw, 2)$-hierarchy. This also implies that $\mathrm{dev}_{y_0y_1}(\calH_{high})\le c_H(\xi^\gamma+\xi^{\gamma^2}+\dots+ \xi^{\gamma^{R-1}})\le 2c_H\xi^\gamma$, since $\xi\ggs \gamma$, as required.
Finally, $\widetilde P_\sigma$ is a path between $\tilde y_{\sigma01}$ and $\tilde y_{\sigma10}$, so let $P_\sigma$ be the concatenated path $\pi_{y_{\sigma01},\tilde y_{\sigma01}}\widetilde P_\sigma \pi_{\tilde y_{\sigma10}, y_{\sigma10}}$. 
Then the total cost of $P_\sigma$ is
\begin{align*}
   \calC(P_\sigma)&\le \calC(\widetilde P_\sigma) + \mathrm{mcost}_2(\pi_{y_{\sigma01},\tilde y_{\sigma01}}) + \mathrm{mcost}_2(\pi_{\tilde y_{\sigma10},y_{\sigma10}}) \\&\le 3\ulw^{3\mu}\xi^{\eta}+2(c_H-3)\olw^{4\mu}\xi^{\eta}/2 \le c_H \olw^{4\mu} \xi^\eta,
\end{align*}
since $\ulw=\olw^{2\delta/d}$, see \eqref{eq:ulw-in-prop} vs \eqref{eq:c_H}. Let $H_1$ and $H_2$ be independent $\theta/2$-percolations of $\calG$, i.e., with distribution $\{\calG^{\theta/2}\mid V, w_V\}$, and consider the corresponding iterative cost construction $I_H = ((H_i,\calE_i^H,\calF_i,\calU_i)\colon i = 1,2)$ on $H_1,H_2$ as in Proposition \ref{prop:multi-round-exposure}, and let $\calA_{I_H}(V, w_V, E_1)$ be the event that the first round returns the edge set $\calE_1^H = E_1$. Then   Proposition~\ref{prop:multi-round-exposure} followed by a union bound gives
\begin{align}
\begin{split}\label{eq:hierarchy-final-weights-multi}
    &\pr\big(\calX_{high\textnormal{-}hierarchy}(R,\eta,y_0,y_1) \mid V,w_V\big)\ge \pr\big(I_{G'} \textnormal{ succeeds}\mid V,w_V\big) \\
    &\quad\quad \ge \pr\big(\calE_1^H \neq \texttt{None}\mid  V,w_V\big)
    \cdot \min_{E_1 \neq \mathtt{None}} \pr\big(\calE_2^H \neq \mathtt{None} \mid \calA_{I_H}(V, w_V,E_1)\big) \\
    &\quad\quad \ge 1-\pr\big(\calE_1^H = \texttt{None}\mid V,w_V\big)-
    \max_{E_1 \neq \mathtt{None}} \pr\big(\calE_2^H = \mathtt{None} \mid \calA_{I_H}(V,w_V, E_1)\big).
\end{split}
\end{align}

The event $\calE_1^H \ne \texttt{None}$ is equivalent that the graph $H_1$ contains a $(\gamma,3\ulw^{3\mu}\xi^{\eta},\ulw,2)$-hierarchy $\calH_{low}:=\{\tilde y_\sigma\}$ of depth $R$ fully contained in $\calN$ with first level $\calL_1 = \{y_0,y_1\}$. Since $H_1\sim \{\calG^{\theta/2}\mid V, w_V\}$ is a CIRG, and since the conditions here are stronger than those in Lemma~\ref{lem:hierarchy-intermediate-weights}, all requirements of Lemma~\ref{lem:hierarchy-intermediate-weights} hold with $\theta$ replaced by $\theta/2$, so the first error term in \eqref{eq:hierarchy-final-weights-multi} is at most $\mathrm{err}_{\xi,\delta}$ in \eqref{eq:hierarchy-intermediate-prob}.
It remains to upper-bound the second error term in \eqref{eq:hierarchy-final-weights-multi}. 

Let
\begin{align}
     q^\star:= \left\lceil\frac{\log(d/\delta)}{\log(1/(\tau-2+2d\tau\delta))}\right\rceil,\label{eq:q-star}
\end{align}
and for each $v\in\calN$ let $\calA_{\mathrm{path}}(v)$ be the event that $H_2$ contains a path $\pi_{v,v'}$ from $v$ to some vertex $v'\in\calN \cap (B_{4q^\star \xi^{\gamma^{R-1}}}(v) \times [\olw,4\olw])$ with cost $\calC_2(\pi_{v,v'}) \le q^\star \olw^{4\mu}\xi^{\eta}$. The value $c_H$ in~\eqref{eq:c_H} is chosen so that $4q^\star \le (c_H-2)/2$ and 
$q^\star \le (c_H-3)/2$ both hold; thus conditioned on $\calA_{I_H}(V,w_V, E_1)$, the event $\calE_2^H=\texttt{None}$ occurs only if for some $\sigma\in\{0,1\}^R$ the event $\neg\calA_{\mathrm{path}}(\tilde y_{\sigma})$ occurs. There are $2^R$ strings $\sigma\in\{0,1\}^R$, and all the events in $\calA_{I_H}(V, w_V, E_1)$ are contained in the $\sigma$-algebra generated by $H_1$, which is independent of $H_2$ given $V, w_V$. So, by a union bound,
\begin{align}
     \max_{E_1 \neq \mathtt{None}} \pr\big(\calE_2^H = \mathtt{None} \mid \calA_{I_H}(V,w_V,E_1)\big) &\le 2^R\cdot\max_{\substack{\sigma \in \{0,1\}^{R}\\E_1 \neq \mathtt{None}}}\pr\big(\neg\calA_{\mathrm{path}}(\tilde y_{\sigma}) \mid \calA_{I_H}(V,w_V,E_1)\big)\nonumber\\
     &\le 2^R\cdot\max_{\tilde y_{\sigma}\neq \mathtt{None}}\pr\big(\neg\calA_{\mathrm{path}}(\tilde y_{\sigma})\mid V,w_V\big), \label{eq:hierarchy-final-weights-3}
\end{align}
where the maximum is taken over all possible values of $\tilde y_{\sigma}$ in non-\texttt{None} $E_1$.
To bound~\eqref{eq:hierarchy-final-weights-3}, we apply Claim~\ref{claim:cheap-path-to-larger-weight} with $G'=H_2$, $\theta$ replaced by $\theta/2$, $K=2\olw$, $M=2\ulw$, $D=4\xi^{\gamma^{R-1}}$, $U=\olw^{4\mu}\xi^{\eta}$, $y_0=y_{\sigma}$, and all other variables taking their present values. Using $\ulw,\olw$ from \eqref{eq:ulw-in-prop} and \eqref{eq:c_H},
we compute
\[
    \frac{\log K}{\log M} = \frac{\log (2\olw)}{\log (2\ulw)}= \frac{\log 2 + \tfrac12\gamma^{R-1}d\log\xi}{\log 2 + \gamma^{R-1}\delta\log\xi}\le 1 + \frac{d}{2\delta} \le \frac{d}{\delta};
 \]
and therefore $q$ from~\eqref{eq:cond-q} with these choices satisfies
     \begin{align*}
        q 
        &= \bigg\lceil\frac{\log\big(\log K/\log M\big)}{\log(1/(\tau-2+2d\tau\delta))} \bigg\rceil \le 
        \bigg\lceil\frac{ \log(d/\delta)}{\log(1/(\tau-2+2d\tau\delta))} \bigg\rceil = q^\star.
     \end{align*}
Hence the event $\calA_{\pi(\tilde y_{\sigma})}$ in Claim~\ref{claim:cheap-path-to-larger-weight} is contained in $\calA_{\mathrm{path}}(\tilde y_{\sigma})$.
We now verify that the requirements of Claim~\ref{claim:cheap-path-to-larger-weight} hold in order of their appearance there. Whenever $E_1 \neq \mathtt{None}$, $\tilde y_{\sigma}$ lies in $\calN$ with weight in $[\ulw,4\ulw]=[M/2,2M]$ by construction, where $M=2\ulw >1$. Similarly, $K=2\olw>1$ and $D=4\xi^{\gamma^{R-1}}>1$ by our choices. We check the requirement $U\ge K^{2\mu}$. By definition of $\olw$ in \eqref{eq:c_H} and the choices $U=\olw^{4\mu}\xi^{\eta}$ and $K=2\ulw$, we compute
\begin{align*}
    UK^{-2\mu} = \olw^{4\mu}\xi^{\eta}(2\olw)^{-2\mu} = 2^{-2\mu} \olw^{2\mu} \xi^{\eta},
\end{align*}
which is larger than $1$ (even if $\eta\!=\!0$) since $\mu>1$ and $\olw \ge (\log\log\xi\sqrt{d})^{8d^2/\delta^2}$ and $\xi\ggs \delta$. Next, since $\delta\lls \mpar$ by hypothesis, $\olw=\xi^{\gamma^{R-1}d/2} > \xi^{\gamma^{R-1}\delta}=\ulw$ and so $K \ge M$. Moreover, $K = 2\xi^{\gamma^{R-1}d/2} \le 4^{d/2} \xi^{\gamma^{R-1}d/2}=D^{d/2}$ also holds. Since $M= 2\ulw =2 \xi^{\gamma^{R-1}\delta} \ge 2(\log\log\xi\sqrt{d})^{16d/\delta}$, $\xi \ggs \theta,\delta,w_0$, and $M\le K\le D^{d/2}$, we have $K,M,D\ggs \theta,\delta,w_0$. Clearly $D = 4\xi^{\gamma^{R-1}} < \xi \le \xi\sqrt{d}$ since $\gamma<1$ and $\xi$ is large. Next, we check $(M/2)^{2/d} = (\xi^{\gamma^{R-1}\delta})^{2/d} > \xi^{\gamma^{R-1}\delta/d} \ge (\log\log\xi\sqrt{d})^{16/\delta}$ as required. Finally, if $\beta=\infty$ then we also need that $U(KM)^{-\mu}=\olw^{4\mu}\xi^{\eta}(4\olw\ulw)^{-\mu}\ge 4^{-\mu}\olw^{2\mu}\xi^{\eta}$ is sufficiently large. This holds even when $\eta\!=\!0$ since $\olw \ggs \mpar$.
Hence, all requirements of Claim~\ref{claim:cheap-path-to-larger-weight} are satisfied, and since $\theta$ changes to $\theta/2$ and $M=2\ulw=2\xi^{\gamma^{R-1}\delta}$ in \eqref{eq:weight-increasing-path-error}, \eqref{eq:hierarchy-final-weights-3} can be bounded as
\begin{align}
\begin{split}\label{eq:hierarchy-final-weights-4}
    \max_{E_1 \neq \mathtt{None}} \pr\big(\calE_2^H = \mathtt{None} &\mid \calA_{I_H}(V, w_V, E_1)\big) \le 2^R\exp\Big({-}(\theta/2)2^{\delta}\xi^{\gamma^{R-1}\delta^2}\Big)\\
    &\le 2^R\exp\Big({-}(\log\log\xi)^{15}\Big)\le  \exp\Big({-}(\log\log\xi)^{14}\Big),
\end{split}
\end{align}
where we obtained the second row from the hypotheses $\xi^{\gamma^{R-1}\delta^2} \ge (\log\log\xi)^{16}$ and $\xi\ggs \theta$, and then from $2^R \le e^{R}$ and $R\le (\log \log \xi)^2$.
Combining~\eqref{eq:hierarchy-final-weights-4} with \eqref{eq:hierarchy-final-weights-multi} and recalling that the first error term there is at most $\exp(- (\log \log \xi)^{1/\sqrt{\delta}})$ finishes the proof of \eqref{eq:hierarchy-final-prob} since $\delta$ is small and $\xi$ is large.
\end{proof}

\PropositionPathFromHierarchy*

\begin{proof}
    To construct the path $\pi_{y_0^\star, y_1^\star}$, we again use two rounds of exposure. In the first round we apply Lemma~\ref{lem:hierarchy-final-weights} to get a $(\gamma,c_H\olw^{4\mu}\xi^{\eta},\olw,c_H)$-hierarchy $\calH_{high}:=\{y_\sigma\}$ of depth $R$ fully contained in $\calN$ with first level $\{y_0, y_1\}$. In the second round, we use Claim~\ref{claim:common-neighbour} to connect, via a common neighbour, each pair of level-$R$ siblings $y_{\sigma0}, y_{\sigma1}, \sigma \in \{0,1\}^{R-1}\setminus\{0_{R-1}, 1_{R-1}\}$. This yields a path between $y_{0_{R-1}1}=:y^{\star}_0$ and $y_{1_{R-1}0}=:y^{\star}_1$.
    
    We now define an iterative cost construction on $G'\sim \{\calG^{\theta}|V,w_V\}$. Let $\calF_1$ be the collection of all tuples of possible edges $e$ with $\mathrm{dev}_{y_0y_1}(e)\le 2c_H\xi^\gamma$ that could form the bridges of a $(\gamma,c_H\olw^{4\mu}\xi^{\eta},\olw,c_H)$-hierarchy $\calH_{high}=\{y_{\sigma}\}_{\sigma\in\{0,1\}^R}$ fully contained in $\calN$ with first level $\{y_0, y_1\}$. Moreover,
    let $\calU_1$ be the event that the costs of all $P_\sigma$ satisfy \emph{(H\ref{item:H3})} of Definition~\ref{def:hierarchy} with $U=c_H\olw^{4\mu}\xi^{\eta}$, so that $\calH_{high}$ is indeed a valid $(\gamma,c_H\olw^{4\mu}\xi^{\eta},\olw,c_H)$-hierarchy with $\mathrm{dev}_{y_0y_1}(\calH_{high})\le 2c_H\xi^\gamma$.     
    For each $\sigma\in\{0,1\}^{R-1} \setminus \{0_{R-1}, 1_{R-1}\}$, define $\calJ(\sigma)$ to be the set of all potential paths $J_\sigma$ fully contained in $\calN$ connecting $y_{\sigma0}$ and $y_{\sigma1}$ with $\mathrm{dev}_{y_0y_1}(J_\sigma)\le 3c_H\xi^\gamma$. 
     Given $(V, w_V, E_1)$, call a tuple $\underline t$ of edges \emph{admissible} if it contains exactly one such potential path from $\calJ(\sigma)$ for each $\sigma\in\{0,1\}^{R} \setminus \{0_{R-1}, 1_{R-1}\}$, and let $\calF_2$ be the collection of all admissible tuples, with an arbitrary ordering.
 Let $\calU_2$ be the event that each potential path in the chosen tuple is present in the graph under consideration, and has round-$2$ marginal cost $\mcost_2(J_\sigma) \le \olw^{4\mu}$ in \eqref{eq:marginal-cost}.
 Following Def.~\ref{def:iter-construct}\eqref{item:iter5} and Prop \ref{prop:multi-round-exposure}, for $i=1,2$, we set $\calE_i^{G'}$ to be the first tuple (in the ordering of $\calF_i$) for which $\calU_i$ occurs, and we set $\calE_i^{G'}$ to $\mathtt{None}$ if no such tuple exists.
This defines an iterative cost construction on $G'$, 
$I_{G'} = ((G',\calE_1^{G'},\calF_1,\calU_1), (G',\calE_2^{G'},\calF_2,\calU_2))$. 

  If $I_{G'}$ succeeds, then there is a path $\pi_{y_0^\star,y_1^\star}\subseteq\calN$ between $y_0^\star:=y_{0_{R-1}1}$ and $y_1^\star:=y_{1_{R-1}0}$ with $\cost{\pi_{y_0^\star,y_1^\star}} \le c_H 2^R \olw^{4\mu}\xi^{\eta}$ and $\mathrm{dev}_{y_0y_1}(\pi_{y_0^\star,y_1^\star})\le 3c_H\xi^\gamma$. Indeed, let us order the elements $y_{\sigma}$ of $\calH_{high}$ lexicographically by their index $\sigma$, omitting $y_0$ and $y_1$, i.e.\ 
    \begin{align*}
	    y_0^\star=y_{0_{R-1}1}, y_{0_{R-2}10}, y_{0_{R-2}11}, \ldots, y_{1_{R-2}00}, y_{1_{R-2}01}, y_{1_{R-1}0}=y_1^\star,
	\end{align*}
	and notice that $P_\sigma\in \calH_{high}$ is a path between every consecutive pair of the form $y_{\sigma01}, y_{\sigma10}$ while $J_\sigma$ is a path between every consecutive pair of the form $y_{\sigma00}, y_{\sigma01}$ or $y_{\sigma{10}}, y_{\sigma11}$, so the concatenation forms a connected walk $\pi^+$. We then remove any cycles from $\pi^+$, passing to an arbitrary sub-path $\pi_{y_0^\star, y_1^\star}\in \calN$. Since $\calH_{high}$ is a $(\gamma,c_H\olw^{4\mu}\xi^{\eta},\olw,c_H)$ hierarchy with first level $y_0, y_1$, by Definition~\ref{def:hierarchy} \emph{(H\ref{item:H1})} $w_{y_0^\star}, w_{y_1^\star} \in [\olw, 4\olw]$, and by \emph{(H\ref{item:H2})}, the distances $|y_0-y_{0}^\star| \le c_H\xi^{\gamma^{R-1}}$ and $|y_1-y_{1}^\star| \le c_H\xi^{\gamma^{R-1}}$ both hold.  Finally, since each edge of $\pi^+$ is contained in $\pi_{y_0^\star,y_1^\star}$ only once, its cost is at most
    \[
        \cost{\pi_{y_0^\star,y_1^\star}} \le \sum_{e \in E(\pi^+)} \cost{e} \le \sum_{\sigma\in\{0,1\}^{i}: 0\le i\le R-2} \cost{E^-(P_\sigma)} + \sum_{\sigma\in\{0,1\}^{R-1} \setminus \{0_{R-1}, 1_{R-1}\}} \mcost_2(J_\sigma).
    \]
The cost of each $E^-(P_\sigma)$ is at most $c_H \olw^{4\mu}\xi^{\eta}$ by $\calH_{high}$ (see \emph{(H\ref{item:H3})}), and $\mcost_2(J_{\sigma}) \le \olw^{4\mu}$ by construction; since $c_H \ge 1$ it follows that
    \[
        \cost{\pi_{y_0^\star,y_1^\star}} \le (2^{R-1}-1) c_H\olw^{4\mu}\xi^{\eta} + (2^{R-1}-2)\olw^{4\mu}< c_H2^R\olw^{4\mu}\xi^{\eta},
    \]
    as required by $\calX_{high\textnormal{-}path}$. The deviation bound $3c_H\xi^\gamma$ also holds since it holds individually for all $J_\sigma$ and it holds for $\calH_{high}$ already by Lemma~\ref{lem:hierarchy-final-weights}. Let $H_1,H_2$ again be independent $\{\calG^{\theta/2}\mid V, w_V\}$, and let $I_H = ((H_i,\calE_i^H,\calF_i,\calU_i)\colon i = 1,2)$ be the corresponding iterative cost construction on $H_1,H_2$ in Proposition \ref{prop:multi-round-exposure}, and let $\calA_{I_H}(V, w_V, E_1)$ be the event that the first round returns the edge set $\calE_1^H = E_1$. Then Proposition~\ref{prop:multi-round-exposure} followed by a union bound gives similarly to \eqref{eq:hierarchy-final-weights-multi} that
    \begin{align}
    \begin{split}\label{eq:path-from-hierarchy-multi}
      \pr(\calX_{high\textnormal{-}path} &\mid V,w_V) \ge   \pr\big(I_{G'} \textnormal{ succeeds} \mid V,w_V\big) \\
      &\ge 1-\pr\big(\calE_1^H = \texttt{None}\mid V,w_V\big)\
        -
        \max_{E_1 \neq \mathtt{None}} \pr\big(\calE_2^H = \mathtt{None} \mid \calA_{I_H}(V,w_V,E_1)\big).
    \end{split}
    \end{align}
   We bound both errors on the rhs. The event $\calE_1^H \neq \texttt{None}$ can be bounded using 
   Lemma~\ref{lem:hierarchy-final-weights} with $\theta$ replaced by $\theta/2$, since $H_1\sim \{\calG^{\theta/2}\mid V, w_V\}$. All the requirements of Lemma~\ref{lem:hierarchy-final-weights} are fulfilled by hypothesis, so the first term on the rhs is at most $\exp\big({-}(\log\log\xi)^{13}\big)$ by \eqref{eq:hierarchy-final-prob}.
    It remains to upper-bound the second term in \eqref{eq:path-from-hierarchy-multi}. 
    For each $x_0,x_1\in \calN$, let $\widetilde \calA_{x_0\star x_1}$ be the event that $H_2$ contains a two-edge path $x_0vx_1 \subseteq\calN$ of cost at most $\olw^{4\mu}$ with $|x_0-v| \le c_H\xi^\gamma$. If $\mathrm{dev}_{y_0y_1}(x_0)\le 2c_H\xi^\gamma$, this implies $\mathrm{dev}_{y_0y_1}(v)\le 3c_H\xi^\gamma$. Hence, conditioned on $\calA_{I_H}(V,w_V,E_1)$, the event $\calE_2^H=\texttt{None}$ occurs only if for some $\sigma\in\{0,1\}^{R-1} \setminus \{0_{R-1}, 1_{R-1}\}$ the event $\neg\widetilde \calA_{y_{\sigma0}\star y_{\sigma1}}$ occurs. Since all the events in $\calA_{I_H}(V, w_V, E_1)$ are contained in the $\sigma$-algebra generated by $H_1$, which is independent of $H_2$ given $V, w_V$, we get by a union bound that
     \begin{align}
     \begin{split}\label{eq:path-from-hierarchy-3}
         &\max_{E_1 \neq \mathtt{None}} \pr\big(\calE_2^H = \mathtt{None} \mid \calA_{I_H}(V, w_V,E_1)\big) \\
         &\qquad\qquad\le (2^{R-1}-2)\cdot\max_{\substack{\sigma \in \{0,1\}^{R-1} \setminus \{0_{R-1}, 1_{R-1}\}\\E_1 \neq \mathtt{None}}}\pr\big(\neg\widetilde \calA_{y_{\sigma0}\star y_{\sigma1}} \mid \calA_{I_H}(V, w_V, E_1)\big) \\
         &\qquad\qquad\le  (2^{R-1}-2)\cdot \max_{y_{\sigma0}, y_{\sigma1}\neq \mathtt{None}}\pr\big(\neg\widetilde \calA_{y_{\sigma0}\star y_{\sigma1}} \mid V, w_V\big),
     \end{split}
     \end{align}
     where the maximum is taken over all possible values of $y_{\sigma0},y_{\sigma1}$ occurring in non-\texttt{None} $E_1$.
     To bound~\eqref{eq:path-from-hierarchy-3}, we observe $|y_{\sigma0}-y_{\sigma1}|\le c_H \xi^{\gamma^{R-1}}$ and $w_{y_{\sigma0}}, w_{y_{\sigma1}} \in [\olw,4\olw]$ when $\sigma\in\{0,1\}^R\setminus \{0_{R-1}, 1_{R-1}\}$, by Def.~\ref{def:hierarchy} \emph{(H\ref{item:H2})} and \emph{(H\ref{item:H1})}, since $\calH_{high}$ is a $(\gamma,c_H\olw^{4\mu}\xi^{\eta},\olw,c_H)$ hierarchy. Thus we apply Claim~\ref{claim:common-neighbour} with $G'$ replaced by $H_2$, $\theta$ replaced by $\theta/2$, $D=\xi^{\gamma^{R-1}}$, $x_0=y_{\sigma0}$, $x_1=y_{\sigma1}$ and all other variables taking their present values. We verify that the requirements of Claim~\ref{claim:common-neighbour} all hold in order of their appearance there. $\delta \lls \mpar$ by hypothesis and $D=\xi^{\gamma^{R-1}} \ge (\log\log\xi\sqrt{d})^{16d/\delta^2} \ge(\log\log\xi\sqrt{d})^{16/\delta}$ by assumption, so in particular $D\ggs c_H, \delta$ and $D \ge w_0^{2/d}$ since $\xi \ggs \mpar,\delta,w_0$. Also, clearly $\xi^{\gamma^{R-1}} \le \xi\sqrt{d}$. Next, we check the distance and weights of $y_{\sigma0},y_{\sigma1}$. Since $E_1\neq \texttt{None}$, they must lie in $\calN$, and as argued already, satisfy $|y_{\sigma0}-y_{\sigma1}| \le c_H\xi^{\gamma^{R-1}}=c_HD$ and $w_{y_{\sigma0}}, w_{y_{\sigma1}} \in [\olw,4\olw] = [D^{d/2}, 4D^{d/2}]$. Finally, the cost-bound in Claim \ref{claim:common-neighbour} is $D^{2\mu d}=\xi^{2\mu d \gamma^{R-1}}=\olw^{4\mu}$ exactly as we require it here, and the vertex $v$ satisfies $|x_0-v|\le D = \xi^{\gamma^{R-1}}\le c_H\xi^\gamma$, also as required. 
     Claim~\ref{claim:common-neighbour} applies and~\eqref{eq:path-from-hierarchy-3} can be bounded as
     \begin{align*}
         \max_{E_1 \neq \mathtt{None}} \pr\big(\calE_2^H = \mathtt{None} \mid \calA_{I_H}(V, w_V, E_1)\big) &\le 2^{R-1}\exp\Big({-}(\theta^2/4) \xi^{\gamma^{R-1}(3-\tau-2\delta)d/2}\Big) \\
         &\le 2^{R-1}\exp\big({-}(\log\log\xi)^{15/\delta}\big),
    \end{align*}
  where for the second row we used that $\delta \lls \mpar$, so $(3-\tau-2\delta)d/2 \ge \delta$ and by hypothesis $\xi^{\gamma^{R-1}\delta } \ge (\log\log\xi)^{16/\delta}$ and $\xi \ggs \theta$.
This, together with that the first error term in \eqref{eq:path-from-hierarchy-multi} was at most $\exp\big({-}(\log\log\xi)^{13}\big)$ concludes the proof of \eqref{eq:high-path}.
\end{proof}

\subsection{The optimisation of total cost: proofs of Corollaries~\ref{cor:computations-polylog} and~\ref{cor:computations-polynomial}}\label{app:path-proofs}

Both corollaries follow from Proposition~\ref{prop:path-from-hierarchy} with suitably-chosen parameters. Throughout, we use the convention that $\infty\cdot 0 = 0$. In Proposition~\ref{prop:path-from-hierarchy}, the values of $(\gamma,z,\eta,R)$ are not set yet (and they are not part of the model parameters $\mpar$). We introduce constraints on these parameters in Definition~\ref{def:valid} (``$(K,A)$-validity''), then we show in Lemma~\ref{lem:valid-works} that a $(K,A)$-valid assignment of values in Proposition~\ref{prop:path-from-hierarchy} yields a path between $0$ and $x$ of cost~$K$ with a multiplicative ``error'' of at most $A$.
Recall $\Lambda$, $\Phi$ and $\olw$ from \eqref{eq:Lambda-def}, \eqref{eq:Phi-def} and~\eqref{eq:c_H}, and that $\xi$ is the side-length of the box $Q$ in which the net exist.

\begin{definition}[Valid parameter choices]\label{def:valid}
The \emph{reduced Setting~\ref{set:hierarchy-common}} is Setting~\ref{set:hierarchy-common}, except without $\gamma$ being defined. Consider the reduced Setting ~\ref{set:hierarchy-common}, and let $K,A > 0$. A setting of parameters $(\gamma,z,\eta,R)$ is \emph{$(K,A)$-valid} for $\xi$ if the following conditions all hold for $\xi\ggs \mpar$, writing $\olw := \xi^{\gamma^{R-1}d/2}$:
 \begin{align}
 \gamma=\gamma(\mpar)& \in (0,1),\quad  z=z(\mpar) \in [0,d], \quad \eta=\eta(\mpar) \in [0,\infty), \label{eq:valid-trivial}\\
 R&=R(\mpar, \xi) \in [2, (\log\log\xi)^2/4]\cap \N, \mbox{ with }\label{eq:valid-R}\\
\olw^2&
\in [e^{(\log^{*3}\xi)^2}, A/\log\log\xi],\label{eq:valid-rounds}\\
 2^R\olw^{4\mu}\xi^{\eta} & \le  KA/\log\log \xi, \label{eq:valid-low-cost}\\
 \Lambda(\eta,z) &> 0 \qquad \mbox{and}\qquad \mbox{either } z=0 \mbox{ or } \Phi(\eta,z) > 0.\label{eq:valid-lambda-phi}
 \end{align}
\end{definition}

\begin{lemma}\label{lem:valid-works}
    Consider the reduced Setting~\ref{set:hierarchy-common}. Let $q,\zeta > 0$, let $0<\delta\lls q,\mpar$, and suppose that $\xi\ggs \delta, q, w_0,\zeta,\mpar$. Let $K,A > 0$, and suppose that $(\gamma,z,\eta,R)$ is $(K,A)$-valid for $\xi$. Let $\calX_{(K,A)}$ be the event that there is a path $\pi_{y_0^\star, y_x^\star}$ in $G'$ with endpoints $y_0^\star$ and $y_x^\star$ satisfying 
    \begin{align}\label{eq:valid-works-weights}
    &w_{y_0^\star}, w_{y_x^\star}\in[\olw, 4\olw],\\\label{eq:valid-works-positions}
    &y_0^\star \in B_{\olw^{3/d}}(0), \quad y_x^\star \in B_{\olw^{3/d}}(x),\\\label{eq:valid-works-costs}
    &\calC(\pi_{y_0^\star, y_x^\star})\le KA \mbox{ and }\mathrm{dev}_{0x}(\pi) \le \zeta |x|.
    \end{align}
    Then $\pr(\calX_{(K,A)} \mid V,w_V) \ge 1-q$.
\end{lemma}
\begin{proof}
    Let $y_0 := 0$, let $y_1 := x$, let $\xi := |x|$, and let $\theta := 1$. We first verify that the conditions of Proposition~\ref{prop:path-from-hierarchy} hold. Since $\delta\lls\mpar$, by~\eqref{eq:valid-trivial} we may also assume $\delta \lls \gamma, z, \eta$ as required by Proposition~\ref{prop:path-from-hierarchy}. Combined with~\eqref{eq:valid-lambda-phi} this implies that $\Lambda(\eta,z) \ge 2\sqrt\delta$ as required, and that either $z=0$ or $\Phi(\eta,z) \ge \sqrt\delta$ as required by Prop.\ \ref{prop:path-from-hierarchy}. Since $\xi \ggs \delta,\mpar$, the inequalities $\xi \ggs \gamma, z, \eta, \delta,w_0$ and $\xi^{\gamma^{R-1}} \ge (\log\log\xi\sqrt d)^{16d/\delta^2}$ by \eqref{eq:valid-rounds} and since $\olw := \xi^{\gamma^{R-1}d/2}$, which is also required. Finally, $R \in [2, (\log\log\xi)^2]$ by~\eqref{eq:valid-R} and $\gamma \in (0,1)$, $z \in [0,d]$ and $\eta \ge 0$ by~\eqref{eq:valid-trivial}.

    Suppose that the event $\calX_{\mathrm{high\textnormal{-}path}}$ of Proposition~\ref{prop:path-from-hierarchy} occurs, and let $\pi$ be a path as in the definition of $\calX_{\mathrm{high\textnormal{-}path}}$. Then $\pi$ satisfies~\eqref{eq:valid-works-weights} immediately, because $\calX_{\mathrm{high\textnormal{-}path}}$ requires that the end-vertices of the path $\pi$ have weights in $[\olw, 4\olw]$. The event also requires that the end-vertices are within distance $c_H\xi^{\gamma^{R-1}}$ from $0,x$ respectively. Since $\xi\ggs\mpar, \delta$,  $c_H\xi^{\gamma^{R-1}} \le \xi^{\gamma^{R-1}3/2} = \olw^{3/d}$, and so $\pi$ satisfies \eqref{eq:valid-works-positions}. The cost of $\pi$ is at most $c_H2^R\olw^{4\mu}\xi^\eta$; by~\eqref{eq:valid-low-cost} combined with the fact that $\xi\ggs\mpar, \delta$, it follows that $\calC(\pi) \le KA$. The event $\calX_{\mathrm{high\textnormal{-}path}}$ ensures that the deviation of $\pi$ from the section $S_{0,x}$ is at most $3c_H\xi^\gamma$ where $\xi\ggs\zeta,\mpar, \delta$ and $\gamma<1$, so \eqref{eq:valid-works-costs} follows for any $\zeta>0$ fixed. Thus 
    \[
        \pr\big(\calX_{(K,A)} \mid V,w_V\big) \ge \pr\big(\calX_{\mathrm{high\textnormal{-}path}} \mid V,w_V\big).
    \]
    By Proposition~\ref{prop:path-from-hierarchy} and the fact that $\xi\ggs q$, it follows that
    \[
        \pr\big(\calX_{(K,A)} \mid V,w_V\big) \ge 1 - 2e^{-(\log\log\xi)^{13}} \ge 1 - q.\qedhere
    \]
\end{proof}

We shall now apply Lemma~\ref{lem:valid-works} to prove Corollary~\ref{cor:computations-polylog} (which covers the polylogarithmic regime). Here, there are two possible choices of parameters $(\gamma, z, \eta, R)$ for Lemma~\ref{lem:valid-works}: if $\alpha < 2$, then we are able to build a polylogarithmic-cost path using long-range edges between low-weight vertices (Claim~\ref{claim:polylog-low-alpha} below); if $\mu< \mu_{\log}$ then we are able to build a polylogarithmic-cost path using edges between high-weight vertices (Claim~\ref{claim:polylog-low-mu} below). We then prove Corollary~\ref{cor:computations-polylog} by applying whichever parameter setting constructs a lower-cost path (in Corollary~\ref{cor:polylog-combine}). In both regimes, we need the following algebraic fact.

\begin{claim}\label{claim:valid-log-gamma-algebra}
    Let
    \begin{equation}\label{eq:R-def-long}
        R=R(\xi):= \Big\lceil\frac{\log\log\xi-(\log^{*4}\xi)^2}{\log\gamma^{-1}}\Big\rceil.
    \end{equation}
    Then for all $\gamma \in (1/2,1)$ and $\xi\ggs \gamma$, it holds that
    \begin{equation}\label{eq:xi-iteration-bounds}
        \xi^{\gamma^{R-1}} \in [e^{(\log^{*3}\xi)^2}, e^{\sqrt{\log\log\xi}}].
    \end{equation}
\end{claim}
\begin{proof}
    The value of $\xi$ is large, so using \eqref{eq:R-def-long} and that $\lceil x\rceil \le x+1$, 
    \[
        \gamma^{R-1} \ge e^{-\log\log\xi + (\log^{*4}\xi)^2} = (\log^{*3}\xi)^{\log^{*4}\xi}/\log\xi \ge (\log^{*3}\xi)^2/\log \xi.
    \]
    It follows that $\xi^{\gamma^{R-1}} \ge e^{(\log^{*3}\xi)^2}$, as required in \eqref{eq:xi-iteration-bounds}. Moreover, since $\xi\ggs\gamma$, it holds that
    \[
        \gamma^{R-1} \le e^{(\log^{*4}\xi)^2}/(\gamma^2\log\xi) \le e^{(\log^{*3}\xi)/2}/\log\xi = \sqrt{\log\log\xi}/\log\xi.
    \]
    It follows that $\xi^{\gamma^{R-1}} \le e^{\sqrt{\log\log\xi}}$, as required in \eqref{eq:xi-iteration-bounds}.
\end{proof}
The next claim finds a $(K,A)$-valid parameter setting when $\alpha<2$,  for polylogarithmic cost-bound $KA$.
\begin{claim}\label{claim:polylog-low-alpha}
    Consider the reduced Setting~\ref{set:hierarchy-common}, and fix $\eps > 0$. When $\alpha  < 2$, then writing $\Delta_\alpha := 1/(1-\log_2\alpha)$, the following assignment is $((\log \xi)^{\Delta_\alpha}, (\log\xi)^\eps)$-valid for $\xi \ggs \eps,\mpar$ and $0<\eps'\lls \eps,\mpar$:
    \begin{equation}\label{eq:choices-for-alpha}
        \gamma := \frac{\alpha}{2} + \eps'; \qquad 
        z := 0; \qquad 
        \eta := 0; \qquad 
        R:= \Big\lceil\frac{\log\log\xi-(\log^{*4}\xi)^2}{\log\gamma^{-1}}\Big\rceil.
    \end{equation}
\end{claim}
\begin{proof}
We check the requirements in Definition \ref{def:valid} one-by-one.  All the requirements of~\eqref{eq:valid-trivial} and~\eqref{eq:valid-R} are immediately satisfied except for $R \le (\log\log\xi)^2/4$, which follows from the definition since $\xi\ggs\gamma$ and $\gamma > 1/2$. Also since $\xi\ggs\gamma, \mpar$, \eqref{eq:valid-rounds} follows from Claim~\ref{claim:valid-log-gamma-algebra} since $\olw=\xi^{\gamma^{R-1}d/2}$ and $A=(\log \xi)^\eps$.
We now prove \eqref{eq:valid-low-cost}. Since $\xi\ggs\gamma$, we estimate $2^R$ using $R$ and $\gamma$ in \eqref{eq:choices-for-alpha}:
    \begin{equation}\label{eq:polylog-low-alpha-1}
        2^R \le 2^{\log\log\xi/\log\gamma^{-1}} = (\log \xi)^{\log 2/\log\gamma^{-1}}\!=  (\log \xi)^{-\log 2/\log(\alpha/2+\eps')}.
   \end{equation}
    Since $\eps'\lls \eps$, the exponent of $\log \xi$ on the rhs is  
    \begin{equation}\label{eq:polylog-low-alpha-2}
        \frac{\log 2}{-\log (\alpha/2+\eps')} \le \frac{\log 2}{\log(2/\alpha)} + \frac{\eps}{2} = \frac{1}{1-\log_2\alpha} + \frac{\eps}{2} = \Delta_\alpha + \frac{\eps}{2}.
    \end{equation}
    Moreover, since $\eta=0$ in \eqref{eq:choices-for-alpha}, the other factor $\olw^{4\mu}\xi^\eta$ in \eqref{eq:valid-low-cost} is at most (using Claim~\ref{claim:valid-log-gamma-algebra} and $\xi \ggs \eps$),
    \begin{equation}\label{eq:polylog-low-alpha-3}
        \olw^{4\mu}\xi^\eta = \xi^{2\mu d \gamma^{R-1}} \le e^{2\mu d \sqrt{\log\log\xi}} \le (\log\xi)^{\eps/2}/\log\log\xi.
    \end{equation}
   Then \eqref{eq:valid-low-cost} with $KA=(\log \xi)^{\Delta_\alpha+\eps}$ follows from~\eqref{eq:polylog-low-alpha-1}--\eqref{eq:polylog-low-alpha-3}. We next prove~\eqref{eq:valid-lambda-phi}. Using the formula in~\eqref{eq:Lambda-def}, with $z=0$ and $\eta=0$, and $\gamma=\alpha/2+\eps'$,
    \[
        \Lambda(\eta, z):= 2d\gamma-\alpha(d-z)-z(\tau-1)+\big(0\wedge\beta(\eta-\mu z)\big) = 2d(\alpha/2 + \eps') - \alpha d = 2d\eps',
    \]
    so $\Lambda(\eta,z) > 0$ as required. Since $z=0$, \eqref{eq:valid-lambda-phi} follows, so all criteria in Def.~\ref{def:valid} are satisfied.
\end{proof}
The next claim finds a $(K,A)$-valid parameter setting when $\mu<\mu_{\log}$,  for polylogarithmic cost bound $KA$. 
\begin{claim}\label{claim:polylog-low-mu}
    Consider the reduced Setting~\ref{set:hierarchy-common}, and fix $\eps > 0$. When $\mu < \mu_{\log}$, then writing $\Delta_{\beta} := 1/(1-\log_2(\tau-1+\mu\beta))$, the following assignment is $((\log \xi)^{\Delta_\beta}, (\log\xi)^\eps)$-valid for $\xi \ggs \eps,\mpar$ and $\eps'\lls \eps,\mpar$:
    \begin{equation}\label{eq:choices-for-mu}
        \gamma := \frac{\tau-1+\mu\beta}{2} + \eps'; \qquad 
        z := d; \qquad 
        \eta := 0; \qquad 
        R:= \left\lceil\frac{\log\log\xi-(\log^{*4}\xi)^2}{\log\gamma^{-1}}\right\rceil.
    \end{equation}
\end{claim}
\begin{proof}
First note that $\beta=\infty$ is not possible here, since in that case $\mu_{\log}=0$, see \eqref{eq:beta-infty-definitions}.
We check the requirements in Definition \ref{def:valid} one-by-one. Since $\tau > 2$ and $\mu\beta\ge0$,  we obtain $\gamma > 1/2 > 0$ above, and since $\mu < \mu_{\mathrm{log}} = (3-\tau)/\beta$ and $\eps'\lls\mpar$ it also holds that $\gamma < 1$; thus all the requirements of~\eqref{eq:valid-trivial} are satisfied. It is also immediate that~\eqref{eq:valid-R} is satisfied except for $R \le (\log\log\xi)^2/4$, which follows from the definition in \eqref{eq:choices-for-mu} since $\xi\ggs\gamma$ and $\gamma > 1/2$. Since $\xi\ggs\gamma, \mpar$, \eqref{eq:valid-rounds} follows from Claim~\ref{claim:valid-log-gamma-algebra} since $\olw=\xi^{\gamma^{R-1}d/2}$ and $A=(\log \xi)^\eps$ as in the previous claim.
We now prove \eqref{eq:valid-low-cost}. Analogously to \eqref{eq:polylog-low-alpha-1}: 
    \begin{equation}\label{eq:polylog-high-alpha-1}
        2^R \le 2^{\log\log\xi/\log\gamma^{-1}} = (\log \xi)^{\log 2/\log(1/\gamma)}.
    \end{equation}
    Since $\eps'\lls\eps$, now $\gamma$ is given in \eqref{eq:choices-for-mu} and
    \begin{equation}\label{eq:polylog-high-alpha-2}
        \frac{\log 2}{\log (1/\gamma)} \le \frac{\log 2}{\log(2/(\tau-1+\mu\beta))} + \frac{\eps}{2} = \frac{1}{1-\log_2(\tau-1+\mu\beta)} + \frac{\eps}{2} = \Delta_{\beta} + \frac{\eps}{2}.
    \end{equation}
    Moreover, by Claim~\ref{claim:valid-log-gamma-algebra} and since $\xi\ggs\eps$, it holds that
\begin{equation}\label{eq:polylog-high-alpha-3}
        \olw^{4\mu}\xi^\eta = \xi^{2\mu d \gamma^{R-1}} \le e^{2\mu d \sqrt{\log\log\xi}} \le (\log\xi)^{\eps/2}/\log\log\xi.
    \end{equation}
    Now \eqref{eq:valid-low-cost} follows immediately from~\eqref{eq:polylog-high-alpha-1}--\eqref{eq:polylog-high-alpha-3}. We next prove~\eqref{eq:valid-lambda-phi}. 
    Using the formula in~\eqref{eq:Lambda-def}, with $z=d$ and $\eta=0$, and $\gamma$ as in \eqref{eq:choices-for-mu},
    \[
        \Lambda(\eta, z) = 2d\gamma-\alpha(d-z)-z(\tau-1)+\big(0\wedge\beta(\eta-\mu z)\big) = 2d\gamma - d(\tau-1) - d\mu\beta = 2d\eps',
    \]
    so $\Lambda(\eta,z) > 0$ as required. This also remains true for $\alpha=\infty$ both formally with $\alpha(d-z)=\infty\cdot 0=0$ as well as intuitively, since $z=d$ means we use edges which are present with constant probability. It remains to prove $\Phi(\eta,z) > 0$.  Using the formula in~\eqref{eq:Phi-def}, and that $\gamma\wedge1/2=1/2$,
    \[
        \Phi(\eta, z) = \Big[d\gamma \wedge \frac{z}{2}\Big] + \Big[0 \wedge \beta\Big(\eta - \frac{\mu z}{2}\Big)\Big] = d(\gamma \wedge 1/2) - \beta\mu d/2 = d(1-\mu\beta)/2. 
    \]
    Since $\mu \le \mu_{\mathrm{log}} = (3-\tau)/\beta$, it follows that $\Phi(\eta,z) \ge d(\tau-2)/2$; since $\tau > 2$, \eqref{eq:valid-lambda-phi} follows.
\end{proof}

Comparing the definition of $\Delta_0$ in~\eqref{eq:Delta_0}  to those in Claims \ref{claim:polylog-low-alpha} and \ref{claim:polylog-low-mu}, we recover here that
\begin{equation}\label{eq:delta-again}
        \Delta_0 = \frac{1}{1 - \log_2 (\min\{\alpha, \tau-1+\mu\beta\})}=\min\{\Delta_\alpha, \Delta_\beta\},
    \end{equation}
 which formally remains true also when $\alpha = \infty$ or $\beta = \infty$ by \eqref{eq:alpha-infty-Delta_0}, or \eqref{eq:beta-infty-Delta_0}. Combining the two claims we obtain the following corollary:

\begin{corollary}\label{cor:polylog-combine} 
    Consider the reduced Setting~\ref{set:hierarchy-common}, fix $\eps > 0$. When either $\alpha\in(1,2)$ or  $\mu\in(\mu_{\mathrm{expl}},\mu_{\log})$ or both hold, then there exists a setting of $(\gamma,z, \eta,R)$ (depending on $\eps$) which is $((\log\xi)^{\Delta_0}, (\log\xi)^\eps)$-valid for $\xi\ggs \eps,\mpar$.
\end{corollary}
\begin{proof}
Recall that $\mu_{\mathrm{log}} = (3-\tau)/\beta$, so if $\mu_{\mathrm{expl}} < \mu < \mu_{\mathrm{log}}$ then $\beta < \infty$; thus we cannot have $\alpha = \beta = \infty$, and the formula \eqref{eq:delta-again} is valid whenever at least one of $\alpha, \beta$ is finite.   
    
We show that when the minimum in the denominator is $\alpha \le \tau-1+\mu\beta$, (so that $\Delta_0 = \Delta_\alpha$), then also $\alpha<2$ holds. Then, Claim~\ref{claim:polylog-low-alpha} directly gives a $((\log\xi)^{\Delta_\alpha},(\log\xi)^\eps)$-valid parameter setting. There are two cases: either $\mu>\mu_{\log}$, then $\alpha < 2$ must hold by the hypothesis of the lemma; or $\mu<\mu_{\log}=(3-\tau)/\beta$, so $\alpha$ being the minimum gives that $\alpha< \tau-1+\mu_{\log}\cdot\beta=2$.

Similarly, we show that when the minimum in the denominator is $\tau-1+\mu\beta<\alpha$, (so that $\Delta_0 = \Delta_\beta$), then also $\mu<\mu_{\log }$ holds. Then, Claim~\ref{claim:polylog-low-mu} directly gives a $((\log\xi)^{\Delta_\beta},(\log\xi)^\eps)$-valid parameter setting. There are again two cases: either $\alpha \ge 2$, then $\mu < \mu_{\mathrm{log}}$ must hold by the hypothesis of the lemma; or $\alpha<2$, so $\tau-1+\mu\beta$ being the minimum gives that $\tau-1 + \mu\beta<2$ and hence $\mu < (3-\tau)/\beta = \mu_{\mathrm{log}}$.\end{proof}
Finally, we are ready to prove the first main result of Section \ref{sec:hierarchy}:
\ComputationsPolyLog*

\begin{proof}
    Immediate from combining Lemma~\ref{lem:valid-works} with Corollary~\ref{cor:polylog-combine}, where the required bounds on $\olw$ in~\eqref{eq:weight-crit-cor1} follow from~\eqref{eq:valid-rounds}.
\end{proof}

We next apply Lemma~\ref{lem:valid-works} to prove Corollary~\ref{cor:computations-polynomial} that covers the polynomial regime. As with the proof of Corollary~\ref{cor:computations-polylog}, we show that multiple possible choices of parameters are valid and choose the one which yields the lowest-cost path. We start with the case where $\alpha = \beta = \infty$. Recall the definition of $\eta_0$ from \eqref{eq:eta_0}, \eqref{eq:alpha-infty-definitions}, \eqref{eq:beta-infty-definitions} and~\eqref{eq:alpha-beta-infty-definitions}.

\begin{claim}\label{claim:polynomial-infinite}
    Consider the reduced Setting~\ref{set:hierarchy-common}, and fix $\eps > 0$. When $\alpha=\beta=\infty$ and $\mu \in(\mu_{\mathrm{log}} ,\mu_{\mathrm{pol}}]$, i.e., $\eta_0=1\wedge d\mu$ in \eqref{eq:alpha-beta-infty-definitions}, then the following setting is $(\xi^{\eta_0},\xi^\eps)$-valid whenever $\eps'\lls \eps,\mpar$ (with $1/(\eps')^2$ an integer), and $\xi\ggs\eps,\eps',\mpar$:
    \begin{equation}\label{eq:choices-for-eta-inf-inf}
        \gamma := 1 - \eps';\qquad z := d;\qquad \eta := \eta_0 + \sqrt{\eps'};\qquad R := 1/(\eps')^2.
    \end{equation}
\end{claim}
\begin{proof}
    Recall from~\eqref{eq:alpha-beta-infty-definitions} that when $\alpha=\beta=\infty$, the values $\mu_{\mathrm{log}} = 0$, $\mu_{\mathrm{pol}} = 1/d$. 
    We check the requirements in Definition \ref{def:valid} one-by-one. Both   \eqref{eq:valid-trivial} and~\eqref{eq:valid-R} are immediate. Since $\eps' < 1/2$, it holds that $\gamma \in [e^{-2\eps'}, e^{-\eps'}]$; thus $\gamma^{R-1} \in [e^{-2/\eps'}, e^{-1/(2\eps')}]$ by the choice of $R$ in \eqref{eq:choices-for-eta-inf-inf}. Since $\xi\ggs\eps'$ and $\eps'\lls\eps, \mpar$, it follows that $\xi^{\gamma^{R-1}d} \in [e^{(\log^{*3}\xi)^2}, \xi^\eps/\log\log\xi]$ as required by \eqref{eq:valid-rounds}. Moreover, using that $\olw=\xi^{\gamma^{R-1}2/d}$ we estimate  $2^R\olw^{4\mu}\xi^\eta=2^R\xi^{\eta+2\mu d\gamma^{R-1}} \le \xi^{\eps/3} \cdot \xi^{\eta_0} \cdot \xi^{\eps/3} \le \xi^{\eta_0+\eps}/\log\log \xi$ and \eqref{eq:valid-low-cost} holds. 
    It remains to prove~\eqref{eq:valid-lambda-phi}. Using the formula in ~\eqref{eq:Lambda-def} with $\gamma, \eta,z$ as in \eqref{eq:choices-for-eta-inf-inf}, and that $\mu\le \mu_{\mathrm{pol}}$,
    \begin{align*}
        \Lambda(\eta, z) &= 2d\gamma-\alpha(d-z)-z(\tau-1)+\big(0\wedge\beta(\eta-\mu z)\big)\\
        &= 2d(1-\eps') - \infty\cdot 0 - d(\tau-1) + (0 \wedge \infty) = d(3-\tau-2\eps');
    \end{align*}
    since $\tau < 3$ and $\eps'\lls\mpar$, $\Lambda(\eta,z)>0$ as required. Finally, using the formula in~\eqref{eq:Phi-def} and that $\gamma\wedge 1/2=1/2$, we analogously obtain that
    \[
        \Phi(\eta, z) = \Big[d\gamma \wedge \frac{z}{2}\Big] + \Big[0 \wedge \beta\Big(\eta - \frac{\mu z}{2}\Big)\Big] = d(\gamma \wedge 1/2) + (0 \wedge \infty(d\mu/2 + \sqrt{\eps'})) = d/2 > 0, 
    \]
    so~\eqref{eq:valid-lambda-phi} follows. Hence, all criteria in Def.~\ref{def:valid} are satisfied.
\end{proof}

When at least one of $\alpha, \beta$ is non-infinite, we can find two possible optimisers: one when  $\mu<\mu_{\mu_{\mathrm{pol,\alpha}}}$ and one when $\mu<\mu_{\mathrm{pol, \beta}}$ hold in \eqref{eq:mu_pol_log}. We treat the two cases separately.
Recall $\mu_{\mathrm{pol},\beta}=1/d+(3-
\tau)/\beta$ and let $\eta_\beta:=d(\mu-\mu_{\log})$, the first term in the second row of \eqref{eq:eta_0}.
\begin{claim}\label{claim:polynomial-small-mu}
    Consider the reduced Setting~\ref{set:hierarchy-common}, and fix $\eps > 0$. When $\alpha > 2$, $\mu\in(\mu_{\mathrm{log}},\mu_{\mathrm{pol},\beta}]$, then the following setting is $(\xi^{\eta_\beta},\xi^\eps)$-valid for $\eps'\lls \eps,\mpar$ (with $1/(\eps')^2$ an integer) and $\xi\ggs \eps,\eps',\mpar$:
\begin{equation}\label{eq:choices-for-eta-mu}
        \gamma := 1 - \eps';\qquad z := d;\qquad \eta := \eta_\beta + \sqrt{\eps'};\qquad R := 1/(\eps')^2.
    \end{equation}
\end{claim}
\begin{proof}
 The $\alpha=\beta=\infty$ case was treated in Claim~\ref{claim:polynomial-infinite} with \eqref{eq:choices-for-eta-inf-inf} coinciding with \eqref{eq:choices-for-eta-mu}. We treat the cases when at least one of $\alpha,\beta$ is finite. 
    We check the requirements in Definition \ref{def:valid} one-by-one.
    Both~\eqref{eq:valid-trivial} and~\eqref{eq:valid-R} are immediate. Since $\eps'$ is small we may choose it $\eps' < 1/2$, implying that $\gamma \in [e^{-2\eps'}, e^{-\eps'}]$; thus $\gamma^{R-1} \in [e^{-2/\eps'}, e^{-1/(2\eps')}]$. Since $\xi\ggs\eps'$ and $\eps'\lls\eps,\mpar$, it follows that $\xi^{\gamma^{R-1}d} \in [e^{(\log^{*3}\xi)^2}, \xi^\eps/\log\log\xi]$ as required by \eqref{eq:valid-rounds}. Moreover, for \eqref{eq:valid-low-cost} we use that $\olw=\xi^{\gamma^{R-1}2/d}$ and estimate $2^R\olw^{4\mu}\xi^\eta = 2^R\xi^{\eta+2\mu d\gamma^{R-1}} \le \xi^{\eps/3} \cdot \xi^{\eta_{\beta}} \cdot \xi^{\eps/3} \le \xi^{\eta_{\beta}+\eps}/\log\log \xi$ and so \eqref{eq:valid-low-cost} holds.  It remains to prove~\eqref{eq:valid-lambda-phi}. 

By their definition in \eqref{eq:choices-for-eta-mu}, $z=d$ and $\eta=\eta_\beta+\sqrt{\eps'}$ where $\eta_\beta=d(\mu-\mu_{\log})=d(\mu-(3-\tau)/\beta)$, we compute $\eta - \mu z = \sqrt{\eps'} - (3-\tau)d/\beta < 0$, since $\eps'\lls\mpar$. So, using the formula in \eqref{eq:Lambda-def} with $\gamma, \eta,z$ as in \eqref{eq:choices-for-eta-mu},
 \begin{align*}
        \Lambda(\eta, z) &= 2d\gamma-\alpha(d-z)-z(\tau-1)+\big(0\wedge\beta(\eta-\mu z)\big)\\
        &= 2d(1-\eps') - d(\tau-1) + \beta \sqrt{\eps'} - (3-\tau)d = \beta\sqrt{\eps'} - 2d\eps'.
    \end{align*}
    Since $\eps'\lls\mpar$, it follows that $\Lambda(\eta,z) > 0$ as required by~\eqref{eq:valid-lambda-phi}. 
    This computation also remains valid both formally and intuitively when $\alpha=\infty$ and $\beta<\infty$, since $z=d$ and $\alpha(d-d)=0$ reflects the fact that the edges we use appear with constant probability each. When $\alpha<\infty$ and $\beta=\infty$, $\mu_{\log}=0$ and $\mu_{\mathrm{pol},\beta}=1/d$, hence $\eta-\mu z=d \mu+\sqrt{\eps'} -\mu d=\sqrt{\eps'}$, so the minimum in $0\wedge \beta(\eta-\mu z)=0$. Hence when $\beta=\infty$, since $\gamma=1-\eps'$ and $\tau-1<2$,
    \begin{align}\label{eq:beta-comp-again}
        \Lambda(\eta, z) &= 2d\gamma-\alpha(d-d)-d(\tau-1)= d(2\gamma-(\tau-1))>0.
    \end{align} 
    Finally we treat $\Phi(\eta,z)>0$. When $\beta<\infty$, using the formula in~\eqref{eq:Phi-def} and that $\gamma\wedge 1/2=1/2$, with parameters in \eqref{eq:choices-for-eta-mu} and  $\eta=\eta_\beta+\sqrt{\eps'}=\mu d - \frac{(3-\tau)d}{\beta}+\sqrt{\eps'}$, we analogously obtain that
    \begin{align*}
        \Phi(\eta,z) &= \Big[d\gamma \wedge \frac{z}{2}\Big] +  \Big[0 \wedge \beta\Big(\eta - \frac{\mu z}{2}\Big)\Big] = \frac{d}{2} + \Big[0 \wedge \beta\Big(\sqrt{\eps'} + \frac{\mu d}{2} - \frac{(3-\tau)d}{\beta}\Big)\Big].
    \end{align*}
    In case the minimum on the rhs is at $0$, $\Phi(\eta,z)>0$ and so \eqref{eq:valid-lambda-phi} is satisfied. In case the minimum is at the other term, we use that $\mu > \mu_{\mathrm{log}} = (3-\tau)/\beta$, so $\mu d/2 > (3-\tau)d/(2\beta)$, so     \[
        \Phi(\eta,z) \ge \frac{d}{2} - \beta\cdot \frac{(3-\tau)d}{2\beta} = \frac{(\tau-2)d}{2} > 0.
    \]
    and so $\tau\in(2,3)$ ensures that~\eqref{eq:valid-lambda-phi} holds again. The computation remains valid when $\alpha=\infty$ since $\Phi$ does not depend on $\alpha$. When $\alpha<\infty$ and $\beta=\infty$, the computation simplifies, and $\eta-\mu z/2>0$ holds since already $\eta-\mu z>0$ see above \eqref{eq:beta-comp-again}. Hence in this case $\Phi(\eta,z)=d\gamma\wedge z/2= d/2>0$. 
    Hence, all criteria in Definition \ref{def:valid} are satisfied with the choice in \eqref{eq:choices-for-eta-mu}.
    \end{proof}
The next claim finds minimisers whenever $\mu<\mu_{\mathrm{pol,\alpha}}$. Recall that $\mu_{\mathrm{pol},\alpha}=\tfrac{\alpha-(\tau-1)}{d(\alpha-2)}$ from \eqref{eq:mu_pol_log} and let $\eta_\alpha:=\mu/\mu_{\mathrm{pol},\alpha}$, the second term in the second row of \eqref{eq:eta_0}.
\begin{claim}\label{claim:polynomial-large-mu}
    Consider the reduced Setting~\ref{set:hierarchy-common}, and fix $\eps > 0$. When $\alpha > 2$, $\mu\in(\mu_{\mathrm{log}}, \mu_{\mathrm{pol},\alpha}]$,  then the following setting is $(\xi^{\eta_\alpha},\xi^\eps)$-valid for $\eps'\lls\eps,\mpar$ (with $1/(\eps')^2$ and integer), and $\xi\ggs\eps,\eps',\mpar$:
    \begin{equation}\label{eq:choices-for-eta-alpha}
        \gamma := 1 - \eps';\qquad z := (\eta_\alpha + \sqrt{\eps'})/\mu;\qquad \eta := \eta_\alpha + \sqrt{\eps'};\qquad R := 1/(\eps')^2.
    \end{equation}
\end{claim}
\begin{proof}
We first show that $\alpha=\infty$, $\beta<\infty$ is not possible here. From \eqref{eq:alpha-infty-definitions} it follows that $\mu_{\mathrm{pol,\alpha}}=1/d$, while $\mu_{\log}=1/d+(3-\tau)/\beta$, so for all $\beta>0$ the strict inequality $\mu_{\log}> \mu_{\mathrm{pol},\alpha}$ holds and hence the interval for $\mu$ is empty when $\alpha=\infty$. Hence $\alpha<\infty$ is necessary for the conditions to be satisfied. 
     We check the requirements of Definition \ref{def:valid} one-by-one. Using the formula for $\mu_{\mathrm{pol},\alpha}$ and $\tau <3$, we compute that $\eta_\alpha = \mu d(\alpha-2)/(\alpha-(\tau-1)) < \mu d$. Hence, since $\eps'\lls\mpar$ for all sufficiently small $\eps'$ the inequality $z\le d$ holds as required by~\eqref{eq:valid-trivial}. The other conditions of~\eqref{eq:valid-trivial} and~\eqref{eq:valid-R} are immediate. 
    Since $\gamma$ and $\eta$ is the same here and in Claim \ref{claim:polynomial-small-mu}, \eqref{eq:valid-rounds} and \eqref{eq:valid-low-cost} hold by the same argument as in Claim \ref{claim:polynomial-small-mu}. 
 It remains to prove~\eqref{eq:valid-lambda-phi}. Using the formula in ~\eqref{eq:Lambda-def} with $\gamma, \eta,z$ as in \eqref{eq:choices-for-eta-alpha}, which implies that $\eta-\mu z=0$,
    \begin{equation}
    \begin{aligned}
        \Lambda(\eta, z) &= 2d\gamma-\alpha(d-z)-z(\tau-1)+\big(0\wedge\beta(\eta-\mu z)\big)\\
        &= d(2-\alpha) - 2\eps'd + z(\alpha-(\tau-1)) + 0.
\end{aligned}
\end{equation}
This also remains valid both formally and intuitively when $\beta=\infty$ (with the convention that $\infty\cdot 0=0)$, since $\eta-\mu z=0$ reflects the fact that the random variable $L$ on the edge we use is constant order. 
We substitute $z=(\eta_\alpha+\sqrt{\eps'})/\mu$ from \eqref{eq:choices-for-eta-alpha} and $\eta_\alpha=\mu d(\alpha-2)/(\alpha-(\tau-1))$:
\begin{align*}        
     \Lambda(\eta, z)   &= d(2-\alpha) + \eta_\alpha(\alpha-(\tau-1))/\mu + \sqrt{\eps'}(\alpha-(\tau-1))/\mu - 2\eps' d\\
        &= \sqrt{\eps'}(\alpha-(\tau-1))/\mu - 2\eps'd,
    \end{align*}
    since the first two terms in the first row cancelled each other. 
    Since $\eps'\lls\mpar$, $\alpha>2$ and $\tau\in(2,3)$, $\alpha-(\tau-1)$ is positive, and so is $\mu>0$, so $\Lambda(\eta,z)>0$ as required by~\eqref{eq:valid-lambda-phi}. Finally, by~\eqref{eq:Phi-def} and since $z\le d$,
    \begin{align*}
        \Phi(\eta,z) &= \Big[d\gamma \wedge \frac{z}{2}\Big] +  \Big[0 \wedge \beta\Big(\eta - \frac{\mu z}{2}\Big)\Big] = \frac{z}{2} + 0 > 0,
    \end{align*}
    and so~\eqref{eq:valid-lambda-phi} holds. This also remains true for $\beta=\infty$ since the minimum is at $0$, meaning we use edges with constant value $L$. Hence, all criteria in Definition \ref{def:valid} are satisfied with the choice in \eqref{eq:choices-for-eta-alpha}.
\end{proof}

We are ready to prove Corollary \ref{cor:computations-polynomial}:
\ComputationsPolynomial*

\begin{proof}
Claim \ref{claim:polynomial-infinite} finds a setting of parameters that is $(|x|^{\eta_0}, |x|^{\eps})$-valid whenever $\alpha=\beta=\infty$ and $\mu\le \mu_{\mathrm{pol}}=1/d$. When at least one of $\alpha,\beta$ is non-infinite, Claims
\ref{claim:polynomial-small-mu} and \ref{claim:polynomial-large-mu} respectively find a setting of parameters that is $(|x|^{\eta_\beta}, |x|^{\eps})$-valid whenever $\mu\le \mu_{\mathrm{pol,\beta}}$ and one that is $(|x|^{\eta_\alpha}, |x|^{\eps})$-valid whenever $\mu\le \mu_{\mathrm{pol,\alpha}}$. By noting that $\eta_{\beta}\le 1$ exactly when $\mu<\mu_{\mathrm{pol}, \beta}$ and $\eta_{\alpha}\le 1$ exactly when $\mu\le \mu_{\mathrm{pol}, \alpha}$, we obtain that whenever $\mu\le\max\{\mu_{\mathrm{pol}, \alpha},\mu_{\mathrm{pol}, \beta}\}$, the two claims together find a parameter setting that is $(|x|^{\min\{\eta_\alpha, \eta_\beta\}}, |x|^{\eps})$ valid. Since $\eta_0=\min\{\eta_\alpha, \eta_\beta\}$ in \eqref{eq:eta_0}, the proof from here is immediate by applying Lemma~\ref{lem:valid-works},
where the required bounds on $\olw$ in~\eqref{eq:weight-crit-cor2} follow from~\eqref{eq:valid-rounds}.
\end{proof}

\subsection{Connecting the endpoints $0$ and $x$: proof of Claim~\ref{claim:join-Cm}}\label{app:proof-join-Cm}

In this section we give the proof of Claim~\ref{claim:join-Cm}, which we restate for convenience.

\ClaimJoinCm*

\begin{proof}
   Fix $z\in \R^d$ and let $\calA_1:=\calA_{\mathrm{dense}}(\calH, \calV_M, r,z)$ in \eqref{eq:a-dense}, so that $\pr(\neg \calA_1\mid \calF_{0,x})\le q/10$ by hypothesis.  
    Considering the definition of $\calA_{\mathrm{dense}}$ in~\eqref{eq:a-dense}, let $\calA_2$ be the event that for all $y \in B_r(z)$, $|B_{r^{1/3}}(y)\cap \calV_M| \ge r^{d/4}$. Choose fixed points $x_1,\dots,x_{\lceil r^d\rceil} \in B_r(z)$ such that $\{B_{r^{1/3}/2}(x_i)\colon i \le \lceil r^d\rceil\}$ covers $B_r(z)$. 
    For all $y$, the ball $B_{r^{1/3}}(y)$ must contain at least one ball $B_{r^{1/3}/2}(x_i)$, so if in each ball $B_{r^{1/3}/2}(x_i)$ we find at least $r^{d/4}$ vertices from $\calH\subseteq \calV_M$ then the event $\calA_2$ holds. Let  $c_d$ denote the volume of a unit-radius $d$-dimensional ball. 
By \eqref{eq:power_law}, in IGIRG $|B_{r^{1/3}/2}(x_i) \cap \calV_M|$ is a Poisson variable with mean 
\[
    2^{-d}c_d r^{d/3} \Big(\frac{\ell(M)}{M^{\tau-1}} - \frac{\ell(2M)}{(2M)^{\tau-1}}\Big)
    \ge 2r^{d/3}M^{-3(\tau-1)/2} 
    \ge 2r^{d/4}
\] 
(also conditioned on $\calF_{0,x}$), where the first inequality holds because $M \ggs \mpar$ and the second inequality holds since $r^{d/12}=w^{1/4} \ge M^{2(\tau-1)}$ and $M\ge1$. Similarly, for SFP it is a binomial variable with mean greater than $2r^{d/4}$; in either case, the Chernoff bound of Theorem~\ref{thm:chernoff} applies, and since $r\ggs q$ we have
    \begin{equation}\label{eq:njcm-a2}
    \begin{aligned}
        \pr(\calA_2\mid\calF_{0,x}) &= \pr\Big(\forall y \in B_r(z): |B_{r^{1/3}}(y)\cap \calV_M| \ge r^{d/4}\mid \calF_{0,x}\Big) \\ 
        &\ge \pr(\forall i\colon |B_{r^{1/3}/2}(x_i) \cap \calV_M| \ge r^{d/4}\mid \calF_{0,x}) \ge 1-\lceil r^d\rceil \cdot e^{-r^{d/4}/4} \ge 1-q/30.
        \end{aligned}
    \end{equation}
    Let $\calA_3$ be the event that $B_r(z)$ contains at most $2(c_dr^d+2)$ vertices. In SFP, $\pr(\calA_3\mid\calF_{0,x}) = 1$; in IGIRG, Theorem~\ref{thm:chernoff} applies. In both cases, using $r\ggs q, \mpar$,
    \begin{equation}\label{eq:njcm-a3}
        \pr(\calA_3\mid \calF_{0,x})=\pr(|B_r(z) \cap \calV|\le 2(c_dr^d+2)\mid  \calF_{0,x}) \ge 1-e^{-c_dr^d/3} \ge 1-q/30.
    \end{equation}
    Since $\pr(\neg \calA_1\mid \calF_{0,x}) \le q/10$, a union bound with \eqref{eq:njcm-a2} and \eqref{eq:njcm-a3} yields $\pr(\calA_1\cap \calA_2\cap \calA_3 \mid \calF_{0,x}) \ge 1-q/6$.  We abbreviate $\pr(\cdot \mid V,w_V,E_M)$ when we condition on the event that $\widetilde\calV = (V,w_V)$ and $\calE_M = E_M$. The events $\calF_{0,x},\calA_2,\calA_3$, and also the set $\calH$ and thus $\calA_1$ are all determined by $(V,w_V,E_M)$.
Let us call the realisation $(V,w_V,E_M)$ \emph{good} if the event $\calA_1\cap\calA_2\cap\calA_3\cap\calF_{0,x}$ holds. Then
    \begin{equation}\label{eq:njcm-goal}
        \pr(\calA_{\mathrm{down}}(w,z)\mid\calF_{0,x}) \ge 1 - q/6 - \max_{(V,w_V,E_M)\text{ good}}\pr\big(\neg\calA_{\mathrm{down}}(w,z)\mid V,w_V,E_M\big).
    \end{equation}
    Fix a good realisation $(V,w_V,E_M)$. Following $\calA_{\mathrm{down}}$, let $y_1,\dots,y_k$ be the (fixed) vertices in $B_r(z)$ with weights in $[w,4w]$, and for each $i \in [k]$ let $a_{1}^{\scriptscriptstyle{(i)}},\dots,a_{\ell_i}^{\scriptscriptstyle{(i)}}$ be the (fixed) vertices in $B_{r^{1/3}}(y_i) \cap \calH$. Thus, by definition of $\calA_{\mathrm{down}}$,
    \begin{align}
    \begin{split}\label{eq:njcm-all-edges}
        &\pr\big(\neg\calA_{\mathrm{down}}(w,z)\mid  V,w_V,E_M\big) \\
        &\qquad =\pr\big(\exists i\in[k]\colon\forall j\colon y_ia_{j}^{\scriptscriptstyle{(i)}}\notin \calE(G) \mbox{ or }\calC(y_ia_{j}^{\scriptscriptstyle{(i)}})\! >\! w^{2\mu}\mid V,w_V,E_M\big).
    \end{split}
    \end{align}
    Conditioned on $(V,w_V,E_M)$, the edges $y_ia_j^{\scriptscriptstyle{(i)}}$ are present independently since $w_{y_i} \ge w \ge M^{8(\tau-1)} > 2M$ since $\tau > 2$ and $M\ggs\mpar$. Since $w_{y_i}\in[w, 4w]$, $w_{a_j^{\scriptscriptstyle{(i)}}}$ $\in[M, 2M]$, and $|y_i-a_j^{\scriptscriptstyle{(i)}}|\le r^{1/3}=w^{1/d}$, we get using~\eqref{eq:connection_prob} and \eqref{eq:cost} that for all $i$ and $j$,
    \begin{align*}
        &\pr(y_ia_{j}^{\scriptscriptstyle{(i)}} \notin\calE(G) \mbox{ or }\calC(y_ia_{j}^{\scriptscriptstyle{(i)}}) > w^{2\mu}\mid V,w_V,E_M) 
        \le 1 - \underline{c}\big(1 \wedge \tfrac{wM}{w}\big)^\alpha + \pr\big(L> \tfrac{w^{2\mu}}{(8wM)^{\mu}}\big) \le 1 - \tfrac{\underline{c}}{2},
    \end{align*}
    where the last inequality holds because $w^{2\mu}/(8wM)^{\mu} \ge M^{8\mu(\tau-1)-\mu}/8^\mu$ tends to infinity with $M$ and  $M\ggs\mpar$. This computation also holds for $\alpha=\infty$ or $\beta=\infty$. Conditioned on a good realisation $(V,w_V,E_M)$, $\calA_1\cap \calA_2\cap\calA_3$ occurs, so for each $i$, the number of vertices $a_j^{\scriptscriptstyle{(i)}}$ is $\ell_i \ge r^{d/4}$, and the number of vertices $y_i$ is $k\le 2c_dr^d+4$. By independence across $j$, \eqref{eq:njcm-all-edges}, and a union bound,
    \begin{align*}
    \pr\big(\neg\calA_{\mathrm{down}}(w,z) \mid V,w_V,E_M\big) &\le \sum_{i\le k}\pr\big(\forall j\colon y_ia_{j}^{\scriptscriptstyle{(i)}} \notin\calE \mbox{ or }\calC( y_ia_{j}^{\scriptscriptstyle{(i)}}) > w^{2\mu}\mid V,w_V,E_M\big) \\
    &\le \sum_{i\le k}(1-\underline{c}/2)^{\ell_i} \le (2c_dr^d+4)e^{-r^{d/4}\underline{c}/2} \le q/6,
    \end{align*}
    where the last inequality holds since $r=w^{3/d}\ggs q,\mpar$. The claim then follows by~\eqref{eq:njcm-goal}.
\end{proof}

\subsection{Connecting the endpoints when \texorpdfstring{$d=1$}{d=1}}\label{app:1d-endpoints}

In this section we provide the missing proofs from Section~\ref{sec:endpoints} for $d=1$. We first show a simple variant of~\cite[Lemma~4.3]{komjathy2020stopping}; this lemma says that any suitably high-weight vertex is very likely to lie at the start of an infinite weight-increasing path.

\begin{lemma}\label{lem:1d-infinite-path}
    Consider Setting~\ref{set:joining-common} with $d=1$. Let $\theta>1$ and $\delta\in(0,1)$ with $\delta,\theta-1\lls\mpar$, and let $M_0 \ggs \theta,\delta,\mpar$. Let $z \in \R$ (or $\Z$ for SFP), and for all $i\ge 0$ define, $M_i := M_0^{\theta^i}$, $R_i := M_i^{(1+\delta)(\tau-1)}$, and $I_i := [z, z+R_i]$. Let $\calA_{\mathrm{inc}}(M_0,\theta,z)$ be the event that there is an infinite path $\pi_z = z_0z_1\dots $ in $G$ starting at $z=:z_0$ such that for all $i \ge 1$ we have $z_i \in (I_i\setminus I_{i-1}) \cap \calV_{M_i}$. Then 
    \begin{equation}\label{eq:weight-increasing-11}
        \pr(\neg\calA_{\mathrm{inc}}(M_0,\theta,z) \mid z \in \calV_{M_0}) \le \exp(-M_0^{\delta(\tau-1)/4}).
    \end{equation}
   The bound remains true if we additionally condition on $y\in \calV$ for any $y\in \R\setminus\{z\}$ for GIRG.
\end{lemma}
\begin{proof}
    The proof is very similar to~\cite[Lemma~4.3]{komjathy2020stopping}, which uses a similar construction but in more than one dimension and with less control over the weights. For all $j \ge 1$, let $\calA_{\mathrm{inc}}^j$ be the event that there is a path $\pi_z = z_0z_1\dots z_j$ in $G$ with $z_0 := z$ such that for all $i \in [j]$ we have $z_i \in (I_i\setminus I_{i-1})\cap \calV_{M_i}$. Let $\calA_{\mathrm{inc}}^0$ be the empty event. Then
    \begin{equation}\label{eq:1d-infinite-path-Ainci}
        \pr(\neg\calA_{\mathrm{inc}}(M_0,\theta,z) \mid z \in \calV_{M_0}) = \sum_{i=1}^\infty \pr(\neg\calA_{\mathrm{inc}}^i \mid \calA_{\mathrm{inc}}^{i-1} \mbox{ and } z \in \calV_{M_0}).
    \end{equation}
    We now bound each term in the sum of~\eqref{eq:1d-infinite-path-Ainci} above. Fix $i \ge 1$. Observe that $\calA_{\mathrm{inc}}^{i-1}$ only depends on $G[I_{i-1}]$. Let $G'$ be a possible value (realisation) of $G[I_{i-1}]$ which implies $\calA_{\mathrm{inc}}^{i-1}$. Then we can decompose the conditioning in \eqref{eq:1d-infinite-path-Ainci} by conditioning on events of the type $\calF_i := \{G[I_{i-1}] = G'\} \cap \{z\in \calV_{M_0}\}$ and later integrating over the possible realisations $G'$. Given $G'$ satisfying $\calA_{\mathrm{inc}}^{i-1}$, fix the vertices $z_0,\dots,z_{i-1}$ ensuring $\calA_{\mathrm{inc}}^{i-1}$. Let $\calA_{\mathrm{vert}}^i$ be the event that $|(I_i\setminus I_{i-1}) \cap \calV_{M_i}| \ge M_i^{\delta(\tau-1)/2}$; then
    \begin{equation}\label{eq:1d-infinite-path-union-1}
        \pr(\neg\calA_{\mathrm{inc}}^i \mid \calF_i) \le \pr(\neg\calA_{\mathrm{vert}}^i \mid \calF_i) + \pr(\neg\calA_{\mathrm{inc}}^i \mid \calA_{\mathrm{vert}}^i \cap \calF_i).
    \end{equation}
     By~\eqref{eq:power_law}, the number of vertices in $(I_i\setminus I_{i-1}) \cap \calV_{M_i}$ is either a Poisson variable (for IGIRG) or a binomial variable (for SFP) with mean at least
    \[
        (R_i-R_{i-1}-1)\Big(\frac{\ell(M_i)}{M_i^{\tau-1}} - \frac{\ell(2M_i)}{(2M_i)^{\tau-1}} \Big).
    \]
    Since $\ell$ is slowly-varying, $\tau > 2$, and $M_i > M_0 \ggs \mpar$, it follows that
    \[
        \E\big[|(I_i\setminus I_{i-1}) \cap \calV_{M_i}| \mid \calF_i\big] \ge \frac{R_i}{2}\cdot \frac{\ell(M_i)}{4M_i^{\tau-1}} = \frac{\ell(M_i)M_i^{\delta(\tau-1)}}{8} \ge 2 M_i^{\delta(\tau-1)/2}.
    \]
    In both IGIRG and SFP, it follows that
    \begin{equation}\label{eq:1d-infinite-path-vert-prob}
        \pr(\neg\calA_{\mathrm{vert}}^i \mid \calF_i) \le \exp(-M_i^{\delta(\tau-1)/2}).
    \end{equation}
We next lower-bound the probability that $z_{i-1}$ is connected to any given $z' \in (I_i\setminus I_{i-1}) \cap \calV_{M_i}$. 
Let $(V,w_V)$ be a possible value of $\widetilde{\calV}$ which implies $\calA_{\mathrm{vert}}^i$, and suppose that $z' \in (I_i\setminus I_{i-1}) \cap \calV_{M_i}$ for $\widetilde{\calV} = (V,w_V)$. The distance between $z_{i-1}$ and $z'$ is at most $R_i$, and vertices have weight in $[M, 2M]$ in  $\calV_M$, so by~\eqref{eq:connection_prob} (remembering that $d=1$), 
\begin{align}
\begin{split}\label{eq:1d-infinite-path-conn-prob}
\pr(z_{i-1}z'\in\calE \mid (V,w_V),\calF_i) 
&\ge \underline{c} \cdot \min\Big\{1,\tfrac{M_{i-1}M_i}{R_i}\Big\}^\alpha\\
&= \underline{c}\cdot \min\Big\{1,M_0^{\theta^{i-1}[1+\theta-\theta (1+\delta)(\tau-1)]}\Big\}^\alpha.
\end{split}
\end{align}
Observe that $1+\theta > 2$, and that since $\theta-1,\delta\lls\mpar$ and $\tau < 3$ we have $\theta (1+\delta)(\tau-1) < 2$; thus the exponent on the rhs of~\eqref{eq:1d-infinite-path-conn-prob} is positive and we obtain
    \[
        \pr(z_{i-1}z'\in\calE \mid (V,w_V),\calF_i) \ge \underline{c}.
    \]
    By $\calA_{\mathrm{vert}}^i$ (defined above \eqref{eq:1d-infinite-path-union-1}) there are at least $M_i^{\delta(\tau-1)/2}$ such vertices $z'$, each joined to $z_{i-1}$ independently. Thus,
    \begin{equation}\label{eq:1d-infinit-path-inc-prob}
        \pr(\neg\calA_{\mathrm{inc}}^i \mid \calA_{\mathrm{vert}}^i \cap \calF_i) \le \exp(-\log(1/\underline{c})M_i^{\delta(\tau-1)/2}).
    \end{equation}
    Combining~\eqref{eq:1d-infinite-path-union-1},~\eqref{eq:1d-infinite-path-vert-prob} and~\eqref{eq:1d-infinit-path-inc-prob} and using $M_i \ge M_0 \ggs \mpar, \delta$ yields
    \[
        \pr(\neg\calA_{\mathrm{inc}}^i \mid \calF_i) \le \exp(-M_i^{\delta(\tau-1)/2}) + \exp(-\log(1/\underline{c})M_i^{\delta(\tau-1)/2}) \le \exp(-M_i^{\delta(\tau-1)/3}).
    \]
    Substituting this bound into~\eqref{eq:1d-infinite-path-Ainci} and using $M_0 \ggs \mpar, \delta$ then yields the required bound of
    \[
        \pr(\neg\calA_{\mathrm{inc}}(M_0,\theta,z) \mid z \in \calV_{M_0}) \le \sum_{i=1}^\infty \exp(-M_i^{\delta(\tau-1)/3}) \le \exp(-M_0^{\delta(\tau-1)/4}).\qedhere
    \]
    The bound remains true if we additionally condition also on $y\in \calV$: there is a unique index $i$ so that $y\in I_i\setminus I_{i-1}$. The number of points in this interval changes by one, but the concentration bound in \eqref{eq:1d-infinite-path-vert-prob} still remains valid under the conditioning.
\end{proof}

We now apply Lemma~\ref{lem:1d-infinite-path} to prove the two remaining claims from Section~\ref{sec:endpoints} when $d=1$.

\twoinCinfty*
\begin{proof}\label{proof:claim-rho}
    For $d \ge 2$ this result appears as \cite[Lemma~3.10]{komjathy2022one2}. For $d=1$, we instead apply Lemma~\ref{lem:1d-infinite-path}. Wlog, suppose $a<b$. Let $M_0,\theta>1$ and $0<\delta<1$ with $\delta,\theta-1\lls\mpar$ and $M_0\ggs\delta,\theta,\mpar$. Let $\calA_{\mathrm{path}}(a)$ be the event that $a\in \calV$ lies in an infinite component of $G[[a,\infty)]$, and let $\calA_{\mathrm{path}}(b)$ be the event that $b$ lies in an infinite component of $G[(-\infty,b]]$. By Lemma~\ref{lem:1d-infinite-path}, (and the last sentence there), we have
    \begin{align*}
        \pr(\neg\calA_{\mathrm{path}}(a) \mid b \in \calV, a \in \calV_{M_0}) & \le \exp(-M_0^{\delta(\tau-1)/4}),\\
        \pr(\neg\calA_{\mathrm{path}}(b) \mid a \in \calV, b \in \calV_{M_0}) &  \le \exp(-M_0^{\delta(\tau-1)/4}).
    \end{align*}
    Thus by a union bound,
    \begin{equation}\label{eq:two-in-Cinfty-0}
        \pr(\calA_{\mathrm{path}}(a) \cap \calA_{\mathrm{path}}(b) \mid a,b\in\calV) \ge \pr(a,b \in \calV_{M_0} \mid a,b\in\calV) - 2\exp(-M_0^{\delta(\tau-1)/4}).
    \end{equation}
    By~\eqref{eq:power_law}, since $M_0\ggs\mpar$, $\tau>2$, and $\ell$ is slowly-varying, we have
    \[
        \pr(a,b\in\calV_{M_0}\mid a,b\in\calV) = \Big(\frac{\ell(M_0)}{M_0^{\tau-1}} - \frac{\ell(2M_0)}{(2M_0)^{\tau-1}}\Big)^2 \ge \Big(\frac{\ell(M_0)}{4M_0^{\tau-1}}\Big)^2 \ge \frac{1}{M_0^{3(\tau-1)}}.
    \]
    Since $M_0 \ggs \delta,\mpar$, it follows from~\eqref{eq:two-in-Cinfty-0} that
    \[
        \pr(\calA_{\mathrm{path}}(a) \cap \calA_{\mathrm{path}}(b) \mid a,b\in\calV) \ge M_0^{-3(\tau-1)} - 2\exp(-M_0^{\delta(\tau-1)/4}) > M_0^{-3(\tau-1)}/2;
    \]
    since $\calC_\infty$ is a.s.\ unique, the result therefore follows by taking $\rho := M_0^{-3(\tau-1)}/2$.
\end{proof}

\newexternaloned*
\begin{proof}\label{proof:lem:new_external} 
    We write $r_M=:r$. Since $\calH_M = B_{2r}(z) \cap \calV_M$, i.e, all vertices in $\calV_M$ in $B_{2r}(z)$ belong to $\calH_M$, also all vertices in $B_{r^{1/3}}(y)\cap \calV_M$ are in $\calH_\infty$ for all $y\in B_r(z)$, so the event $\calA_\mathrm{dense}(\calH,\calV_M,r,z)$ always occurs by definition, as required by~\eqref{eq:dense-near-1d}. 
    
    We next prove~\eqref{eq:linear-again-1d} by dividing $B_{2r}(z)$ into sub-interval ``cells'' and proving that each cell is whp both connected and joined to each of its adjacent cells in $G_M$. To this end, let $R := M^{2/d}/\sqrt{d}=M^2$, let $i_{\max} := \lceil 2r/R\rceil$ and $i_{\min} := -i_{\max}$. For all $i \in [i_{\min},i_{\max}]$, let $y_i := z+i\cdot R$ and $Q^{\scriptscriptstyle{(i)}} := [y_i,y_i+R)$; thus $Q^{\scriptscriptstyle{(i_{\min})}},\dots,Q^{\scriptscriptstyle{(i_{\max})}}$ partition $[z-R\lceil 2r/R \rceil,z+(R+1)\lceil 2r/R \rceil) \supset B_{2r}(z)$. Let $\calA_{\mathrm{path}}$ be the event that $G_M[Q^{\scriptscriptstyle{(i_{\min})}}],\dots,G_M[Q^{\scriptscriptstyle{(i_{\max})}}]$ are connected graphs containing at most $2R$ vertices and that for all $i \in [i_{\min},i_{\max}-1]$ there is at least one edge in $G_M$ from $Q^{\scriptscriptstyle{(i)}}$ to $Q^{\scriptscriptstyle{(i+1)}}$. If $\calA_{\mathrm{path}}$ occurs, then for all $a,b \in B_{2r}(z)\cap \calV_M$ there is a path $\pi_{a,b}$ from $a$ to $b$ in $G_M$ intersecting at most $\lfloor |a-b|/R\rfloor+2\le |a-b|/R+2$ many cells;   
    since each cell contains at most $2R$ vertices and each edge in $G_M$ has cost at most $M^{3\mu}$, and since $2RM^{3\mu} = \kappa$, it follows that
    \[
        \calC(\pi_{a,b}) \le (|a-b|/R+2) \cdot 2R \cdot M^{3\mu} = 2(|a-b|+2R) \cdot M^{3\mu} \le \kappa|a-b| + 2\kappa.
    \]
    Moreover, the deviation of $\pi_{a,b}$ is at most the size of the box-length, i.e., $R$, and $ R< \kappa$, i.e., it does not depend on $|a-b|$. With $\zeta=0$ and $C=2\kappa$, we have just shown that 
    \begin{equation}\label{eq:new-external-1d-path}
        \calA_{\mathrm{linear}}(\calH_M,\calH_M,r_M,\kappa,0,2\kappa,z) \subseteq \calA_{\mathrm{path}}.
    \end{equation}
We now bound $\pr(\calA_{\mathrm{path}})$ below. The same approach of dividing $\R^d$ into cells is used in~\cite{komjathy2022one2}, and for both IGIRG and SFP~\cite[Corollary 3.9(i)]{komjathy2022one2} lower bounds the probability that the conditions of $\calA_{\mathrm{path}}$ hold for a single cell $Q^{\scriptscriptstyle{(i)}}$ by $1-e^{-M^{3-\tau-\eps}}$ for some $\eps \lls \mpar$ with $M \ggs \eps$ (by coupling to an Erd\H{o}s-R\'{e}nyi graph). Combining \cite[Corollary 3.9(i)]{komjathy2022one2} with a union bound over the at most $2\cdot\lceil 2r/R\rceil+1$ cells yields that
    \[
        \pr(\calA_{\textrm{path}}) \ge 1-(2\cdot\lceil 2r/R\rceil+1) \cdot e^{-M^{3-\tau-\eps}} \ge  1 -5re^{-M^{3-\tau-\eps}}.
    \]
    Since $r = e^{(\log M)^2}$ and $M \ggs \mpar$, the $e^{-M^{3-\tau-\eps}}$ term dominates, and together with~\eqref{eq:new-external-1d-path} and $M\ggs c,q$ we obtain
    $
        \pr(\calA_{\mathrm{linear}}(\calH_M,\calH_M,r,\kappa,0,2\kappa,z)) \ge 1-q/10,
    $
    and we have proved~\eqref{eq:linear-again-1d} as required. The argument conditioned on $\calF_{0,x}$ is identical; note in particular that~\cite[Corollary 3.9(i)]{komjathy2022one2} explicitly allows for planted vertices. 

    It remains to bound $\pr(\calA_{\mathrm{near}}(\calH_M,C_M,C_M,z))$ conditioned on $z \in \calC_\infty$, see \eqref{eq:a-near} for the definition of $\calA_{\mathrm{near}}$. Here, we replaced the `usual' radius $r_M=\exp((\log M)^2)$  by $C_M=M^{2(\tau-1)+3\mu}\ll r_M$, i.e., we can find a path from $z$ to a vertex with weight $M$ within a much smaller radius from $z$ that $r_M$ would give. We first dominate $\calA_{\mathrm{near}}(\calH_M,C_M,C_M,z)$ below by events $\calA_1$ to $\calA_4$ defined as follows. Let $\rho \lls \mpar$ be as in Claim~\ref{claim:two-in-Cinfty}, and define $M_0 > 0$ satisfying $M \ggs M_0\ggs q,\rho,\mpar$, and let $r_0 := M_0^{2(\tau-1)}$ coming from Lemma~\ref{lem:1d-infinite-path}.
    By Lemma~\ref{lem:1d-infinite-path} we know that a.s.\ $\calC_\infty$ contains a vertex in $\calV_{M_0}$, and let $v_0$ be an (arbitrarily-chosen) closest such vertex to $z$ in Euclidean distance.
    We define the following events:
    \begin{enumerate}[(C1)]
        \item\label{item:a1} $\calA_1$: there is a path $\pi_{z,v_0}$ from $z$ to $v_0$ with $\calC(\pi_{z,v_0}) \le C_M/2$ and $\calV(\pi_{z,v_0}) \subseteq B_{C_M}(z)$;
        
        \item\label{item:a2} $\calA_2$: $B_{r_0}(z)$ contains a vertex in $\calV_{M_0}\cap \calC_\infty$, i.e., $v_0 \in B_{r_0}(z)$;
        
        \item\label{item:a3} $\calA_3$: every vertex $x \in B_{r_0}(z) \cap \calV_{M_0}$ has an associated path $\pi_{x\to\calV_M}$ from $x$ to some vertex in $\calV_M$ with $\calV(\pi_{x\to \calV_M}) \subset B_{C_M}(z)$; 
        and
        
        \item\label{item:a4} $\calA_4$: $\calA_2$ and $\calA_3$ both occur and $\calC(\pi_{v_0\to\calV_M}) \le M^{3\mu} \le C_M/2$.
    \end{enumerate}
    Observe that if $\calA_1$, $\calA_2$, $\calA_3$ and $\calA_4$ all occur then concatenating $\pi_{z,v_0}$ and $\pi_{v_0\to\calV_M}$ yields the path required by $\calA_{\mathrm{near}}(\calH_M,C_M,C_M,z)$; thus
    \begin{align*}
        \pr(\neg\calA_{\mathrm{near}}(\calH_M,C_M,C_M,z) \mid z \in \calC_\infty) &\le \sum_{i=1}^3 \pr(\neg\calA_i\mid z \in\calC_\infty) + \pr(\neg\calA_4\mid \calA_2,\calA_3, z \in \calC_\infty).
    \end{align*}
    By Claim~\ref{claim:two-in-Cinfty}, it follows that
    \begin{align}
    \begin{split}\label{eq:external-1d-union}
        &\pr(\neg\calA_{\mathrm{near}}(\calH_M,C_M,C_M,z) \mid z \in \calC_\infty) \le \pr(\neg\calA_1\mid z \in\calC_\infty) + \\
        &\qquad\qquad  \pr(\neg\calA_2\mid z \in\calV)/\rho
         +\pr(\neg\calA_3\mid z \in \calV)/\rho + \pr(\neg\calA_4\mid \calA_2,\calA_3,z \in \calV)/\rho.
    \end{split}  
    \end{align}
We first bound $\pr(\neg\calA_1\mid z\in\calC_\infty)$ in (C\ref{item:a1}). Given that we fixed $v_0$, let $\pi_{z,v_0}$ be an (arbitrarily-chosen) cheapest path from $z$ to $v_0$; such a path must exist whenever $z \in \calC_\infty$. Since $\calC(\pi_{z,v_0})$ and $\inf\{R>0\colon \calV(\pi_{z,v_0})\subseteq B_R(z)\}$ are a.s.\! finite random variables and since $C_M \ggs q,\mpar, M_0$, we can choose $C_M$ sufficiently large so that
    \begin{equation}\label{eq:external-1d-A1}
        \pr(\neg\calA_1 \mid z \in \calC_\infty) \le q/40.
    \end{equation}
   We next bound $\pr(\neg\calA_2\mid z \in \calV)$ in (C\ref{item:a2}). The event $\calA_2$ occurs if and only if $B_{r_0}(z) \cap \calV_{M_0} \cap \calC_\infty \ne \emptyset$. Similarly as before,  $|B_{r_0}(z) \cap \calV_{M_0}|$ is either a Poisson variable (in IGIRG) or a binomial variable (in SFP) with mean 
    \begin{equation}\label{eq:mean-vertices-in-r0}
        \E[|B_{r_0}(z) \cap \calV_{M_0}| \mid z \in \calV] \ge r_0\Big(\frac{\ell(M_0)}{M_0^{\tau-1}} - \frac{\ell(2M_0)}{(2M_0)^{\tau-1}}\Big) \ge 2M_0^{(\tau-1)/2},
    \end{equation}
    where we used $M_0 \ggs \mpar$ and the value of $r_0=M_0^{2(\tau-1)}$ for the second inequality. In particular, by Chernoff's bound,
    \begin{align}\label{eq:external-1d-Markov1}
    \pr(|B_{r_0}(z) \cap \calV_{M_0}| < M_0^{(\tau-1)/2} \mid z \in \calV) \le \exp(-M_0^{(\tau-1)/8}).
    \end{align}
    Let $\theta > 1$ and $\delta \in (0,1)$ satisfy $\delta,\theta\! -\! 1 \lls \mpar$ and $M_0 \ggs \delta,\theta$. Recall the event $\calA_{\mathrm{inc}}(M_0,\theta,x)$ about having an infinite weight-increasing path from Lemma~\ref{lem:1d-infinite-path}; this event implies $\{x\in\calC_\infty\}$. So by \eqref{eq:weight-increasing-11} in Lemma~\ref{lem:1d-infinite-path}, $\pr(x \notin \calC_\infty \mid x\in\calV_{M_0}) \le \exp(-M_0^{\delta(\tau-1)/4})$ for some $\delta \lls \mpar$ with $M_0 \ggs \delta$. By translation invariance, this implies 
    \begin{align*}
        \pr(x \notin \calC_\infty \mid x\in B_{r_0}(z)\cap\calV_{M_0}, z\in\calV) \le \exp(-M_0^{\delta(\tau-1)/4}).    
    \end{align*}
    Hence, the expected number of vertices in $\calV_{M_0}$ outside the infinite component is at most
        \begin{align*}
        &\E[|(B_{r_0}(z) \cap \calV_{M_0})\setminus\calC_\infty| \mid z \in \calV] \le \E[|B_{r_0}(z) \cap \calV_{M_0}| \mid z \in \calV] \cdot \exp(-M_0^{\delta(\tau-1)/4}) \\
        &\quad\quad \le \E[|B_{r_0}(z) \cap \calV| \mid z \in \calV] \cdot \exp(-M_0^{\delta(\tau-1)/4}) \le (2r_0+1)\exp(-M_0^{\delta(\tau-1)/4}),
    \end{align*}
 which is a crude upper bound. It follows by Markov's inequality that
    \begin{align}\begin{split}\label{eq:external-1d-Markov2}
        \pr(|(B_{r_0}(z) \cap \calV_{M_0}) \setminus \calC_\infty| \ge M_0^{(\tau-1)/2}/2 \mid z\in \calV) &\le \frac{(2r_0+1)\exp(-M_0^{\delta(\tau-1)/4})}{M_0^{(\tau-1)/2}/2} \\
        & \le \exp(-M_0^{\delta(\tau-1)/8}),
    \end{split}\end{align}
    where we used $r_0 = M_0^{2(\tau-1)}$ and $M_0 \ggs \delta,\mpar$.
    By a union bound over \eqref{eq:external-1d-Markov1} and~\eqref{eq:external-1d-Markov2},
    \begin{align}
    \begin{split}\label{eq:external-1d-A2}
        \pr(\neg\calA_2 \mid z \in \calV) 
        &\le \exp(-M_0^{(\tau-1)/8}) + \exp(-M_0^{\delta(\tau-1)/8})\le \rho q/40,
    \end{split}
    \end{align}
    for all sufficiently large $M_0$.
    We next bound $\pr(\neg\calA_3\mid z \in \calV)$ in (C\ref{item:a3}). For all $x \in B_{r_0}(z) \cap \calV_{M_0}$, let $\calA_3(x)$ be the event that $x$ has an associated path $\pi_{x\to \calV_M}$ as in $\calA_\mathrm{3}$. We restrict this path to be a weight-increasing path as in Lemma \ref{lem:1d-infinite-path}. Let $\theta > 1$ and $\delta \in (0,1)$ satisfy $\delta,\theta - 1 \lls \mpar$, and require that $M_0 \ggs \delta,\theta$. For sufficiently large $M$, we may choose $M_0$ such that $i := (\log \log M -\log \log M_0)/ \log \theta \le \log M$ is an integer, so $M_0^{\theta^i} = M$.  Then by Lemma~\ref{lem:1d-infinite-path}, for any given $x \in B_{r_0}(z)$, the weight increasing path reaches a vertex of weight $M$ at radius $R_i=M_0^{\theta^i(\tau-1)(1+\delta)}=M^{(\tau-1)(1+\delta)} < M^{2(\tau-1)+3\mu}=C_M$, and so the path is contained in $B_{C_M}(z)$, and we obtain
\begin{equation}\label{eq:weight-increasing-from-x}
   \begin{aligned}
        \pr(\neg\calA_3(x) \mid x \in \calV_{M_0},z \in \calV) &\le \pr(\neg\calA_\mathrm{inc}(M_0, \theta, x) \mid x \in \calV_{M_0},z \in \calV)\\&\le \exp(-M_0^{\delta(\tau-1)/4}).
        \end{aligned}
    \end{equation}
    It follows by a union bound that
    \begin{align*}
        \pr(\neg\calA_3 \mid z \in \calV) \le \pr(|B_{r_0}(z) \cap \calV_{M_0}| \ge r_0 \mid z \in \calV) + r_0 \exp(-M_0^{\delta(\tau-1)/4}).
    \end{align*}
   As before $|B_{r_0}(z) \cap \calV_{M_0}|$ is either a Poisson variable (in IGIRG) or a binomial variable (in SFP) with mean bounded from above     \begin{align*}
        \E(|B_{r_0}(z) \cap \calV_{M_0}| \mid z \in \calV) \le 2r_0\Big(\frac{\ell(M_0)}{M_0^{\tau-1}} - \frac{\ell(2M_0)}{(2M_0)^{\tau-1}}\Big) \le r_0 M_0^{-(\tau-1)/2}.
    \end{align*}
Therefore $\pr(|B_{r_0}(z) \cap \calV_{M_0}| \ge r_0 \mid z \in \calV) \le 2^{-r_0}$ and we get
    \begin{align}
    \begin{split}\label{eq:external-1d-A3}
        \pr(\neg\calA_3 \mid z \in \calV) & 
        \le 2^{-r_0}+r_0 \exp(-M_0^{\delta(\tau-1)/4}) \le \rho q/40,
    \end{split}
    \end{align}
    where the second inequality holds because because $r_0 = M_0^{2(\tau-1)}$ and $M_0 \ggs \delta, \rho, q, \mpar$.

 Finally we bound the last term in \eqref{eq:external-1d-union}, (see $\calA_4$ in (C\ref{item:a4})). Conditioned on the realisation of $G$, any path $\pi_{v_0\to \calV_M}$ that satisfies the weight-increasing path property in \eqref{eq:weight-increasing-from-x} has at most $\log M$ edges and all vertex weights at most $M$. So, its expected cost is at most 
  \[ \E[\calC(\pi_{v_0\to \calV_M})]\le |\calE(\pi_{v_0\to \calV_M})| \cdot M^{2\mu}\E[L] \le M^{2\mu}\E[L]\log M \le \rho q M^{3\mu}/40\] 
since $M \ggs \rho, q, \mpar$. Thus by Markov's inequality, the probability that the cost of this path is larger than $M^{3\mu}$ is at most $\rho q/40$.
    \begin{equation}\label{eq:external-1d-A4}
        \pr(\neg\calA_4 \mid \calA_2,\calA_3, z \in \calV) \le \rho q/40.
    \end{equation}
    The result therefore follows on substituting the bounds 
  \eqref{eq:external-1d-A1}, \eqref{eq:external-1d-A2}, \eqref{eq:external-1d-A3}, and \eqref{eq:external-1d-A4}  into~\eqref{eq:external-1d-union}. The argument conditioned on $\calF_{y,z}$ is identical; note in particular that in applying Lemma~\ref{lem:1d-infinite-path}, we may assume wlog that $y < z$ by symmetry.
\end{proof}

\bibliographystyle{abbrv}	
\bibliography{references}

\end{document}